\documentclass[preprint,3p]{amsart}
\usepackage[margin=1in]{geometry}
\usepackage{wrapfig}
\usepackage{amssymb}

\usepackage{float}

\usepackage{stmaryrd}
\usepackage{amsmath,amsfonts,amsthm,amstext, mathtools}
\usepackage[latin1]{inputenc}
\usepackage{fancybox}
\usepackage{niceframe}
\usepackage{array}
\usepackage{newlfont}
\usepackage{verbatim}
\usepackage[dvips]{psfrag}
\usepackage{color}
\usepackage{float}
\usepackage[scriptsize]{caption}
\usepackage{indentfirst}
\usepackage[normalem]{ulem}
\usepackage{pst-text}

\usepackage{graphicx}
\usepackage{graphics}
\usepackage{overpic}
\usepackage{hyperref}
\usepackage{pst-node}
\usepackage{tikz-cd} 

\usepackage{todonotes}

\numberwithin{equation}{section}

\graphicspath{{./Figures/}}
\usepackage{wrapfig}
\newcommand{\R}{\ensuremath{\mathbb{R}}}
\newcommand{\C}{\ensuremath{\mathbb{C}}}
\newcommand{\N}{\ensuremath{\mathbb{N}}}

\newcommand{\Z}{\ensuremath{\mathbb{Z}}}
\newcommand{\TT}{\ensuremath{\mathbb{T}}}

\newcommand{\s}{\Sigma}

\newcommand{\al}{\alpha}

\newcommand{\e}{\varepsilon}
\newcommand{\eps}{\varepsilon}

\newcommand{\La}{\Lambda}
\newcommand{\V}{\mathcal{V}}
\newcommand{\GG}{\mathcal{G}}
\newcommand{\LL}{\mathcal{L}}
\newcommand{\FF}{\mathcal{F}}
\newcommand{\PP}{\mathcal{P}}
\newcommand{\HH}{\mathcal{H}}

\newcommand{\er}{\mathcal{O}}
\newcommand{\cZ}{\mathcal{Z}}
\newcommand{\p}{\varphi}
\newcommand{\ag}{\alpha}
\newcommand{\bg}{\beta}
\newcommand{\cg}{\gamma}

\newcommand{\osc}{\textrm{osc}}

\newcommand{\n}{\ell_1}
\newcommand{\wt}{\widetilde}
\newcommand{\BB}{\mathcal{B}}
\newcommand{\EE}{\mathcal{E}}
\newcommand{\pa}{\partial}

\newcommand{\OO}{\mathcal{O}}
\newcommand{\inn}{\mathrm{in}}

\newcommand{\BFX}{\mathbf{X}}
\newcommand{\BFY}{\mathbf{Y}}
\newcommand{\CT}{\mathcal{T}}

\newtheorem {theorem} {Theorem}[section]
\newtheorem {prop}[theorem]{Proposition}
\newtheorem {corollary}[theorem]{Corollary}
\newtheorem {lemma}[theorem]{Lemma}
\newtheorem {definition}[theorem]{Definition}
\newtheorem {remark}[theorem]{Remark}

\newtheorem*{conjecture}{Conjecture}

\makeatletter
\DeclareFontFamily{U}{tipa}{}
\DeclareFontShape{U}{tipa}{m}{n}{<->tipa10}{}
\newcommand{\arc@char}{{\usefont{U}{tipa}{m}{n}\symbol{62}}}

\newcommand{\arc}[1]{\mathpalette\arc@arc{#1}}

\newcommand{\arc@arc}[2]{
	\sbox0{$\m@th#1#2$}
	\vbox{
		\hbox{\resizebox{\wd0}{\height}{\arc@char}}
		\nointerlineskip
		\box0
	}
}
\makeatother

\DeclareMathOperator{\sech}{sech}

\DeclareMathOperator{\Rp}{Re}
\DeclareMathOperator{\Ip}{Im}  
\definecolor{verde}{rgb}{0.0,0.5,0.0}
\definecolor{azul}{rgb}{0,0,128}
\definecolor{roxo}{rgb}{0.44,0.16,0.39}
\definecolor{vinho}{rgb}{0.5,0.0,0.13}
\definecolor{lilas1}{rgb}{0.6,0.33,0.73}
\definecolor{rosa}{rgb}{0.84,0.04,0.33}
\definecolor{mostarda}{rgb}{0.91,0.41,0.17}
\definecolor{mostarda2}{rgb}{1.0,0.66,0.07}

\newcommand{\RED}[1]{\textcolor{red}{#1}}

\begin{document}

\title[Small breathers of  nonlinear Klein-Gordon equations]{On small breathers of nonlinear Klein-Gordon equations via exponentially small homoclinic splitting}

\author[O. M. L. Gomide]{Ot\'avio M. L. Gomide}
\address[OMLG]{Department of Mathematics, UFG, IME\\ Goi\^ania-GO, 74690-900, Brazil}
\email{otaviomarcal@ufg.br}

\author[M. Guardia]{Marcel Guardia}
\address[MG]{ Departament de Matem\`atiques, Universitat Polit\`ecnica de Catalunya, Diagonal 647, 08028 Barcelona, Spain}
\email{marcel.guardia@upc.edu}

\author[T. M. Seara]{Tere M. Seara}
\address[TS]{Departament de Matem\`atiques, Universitat Polit\`ecnica de Catalunya, Diagonal 647, 08028 Barcelona, Spain}
\email{tere.m-seara@upc.edu }

\author[C. Zeng]{Chongchun Zeng}
\address[CZ]{School of Mathematics, Georgia Institute of Technology, Atlanta, GA 30332, USA}
\email{zengch@math.gatech.edu}

\maketitle

\begin{abstract}

Breathers are nontrivial time-periodic and spatially localized solutions of 
nonlinear dispersive partial differential equations (PDEs). Families of 
breathers have been found for certain integrable PDEs but are believed to be 
rare in non-integrable ones such as nonlinear Klein-Gordon equations. 
In this paper we show that small amplitude breathers of \emph{any} temporal frequency  do not exist  for semilinear  Klein-Gordon equations with generic analytic odd nonlinearities. 

A breather with small amplitude exists only when its temporal 
frequency  is close to be resonant with the linear Klein-Gordon dispersion relation. Our main result is that, 
for such frequencies, we rigorously identify the leading order term in the exponentially 
small (with respect to the small amplitude) obstruction to the existence of 
small breathers in terms of the so-called \emph{Stokes constant}, which depends on the nonlinearity analytically, but is independent of the frequency. This gives a rigorous justification of a formal asymptotic argument by Kruskal and Segur \cite{KS87} in the analysis of small breathers. 

We rely on the spatial dynamics approach where breathers can be seen as homoclinic orbits. The birth of such small homoclinics is analyzed via a singular perturbation setting  where a Bogdanov-Takens type bifurcation is coupled to infinitely many rapidly oscillatory directions. The leading order term of the exponentially small splitting between the stable/unstable invariant manifolds is obtained through a careful analysis of the analytic continuation of their parameterizations.  This requires the study of another limit equation in the complexified evolution variable, the so-called \emph{inner equation}.

\end{abstract}
\tableofcontents

\section{Introduction}

Breathers are nontrivial time-periodic and spatially localized solutions of nonlinear dispersive partial differential equations (PDEs). This kind of solutions play an important role in physical applications and the interest in their existence or breakdown gives rise to a fundamental problem in the mathematical study of the dynamics of such PDEs. 

So far breathers have been constructed mostly for completely integrable PDEs. As far as the authors know, the \textit{sine-Gordon equation}
\begin{equation}
\label{sinegordon} \tag{sG}
\partial_t^2u-\partial_x^2u+\sin(u)=0,
\end{equation}
is one of the first PDEs found to admit a family of breathers (see e.~g.~\cite{AKNS74}), which is given explicitly by
\begin{equation}
\label{breatherssine} 
u^{\omega}(x,t)=4\arctan\left(\dfrac{m}{\omega}\dfrac{\sin(\omega t)}{\cosh(m x)}\right), \quad m,\omega>0,\ m^2+\omega^2=1.
\end{equation}
They are viewed as the locked states of a kink and an anti-kink in the  integrable theory. Along with spatial and temporal translation, the breathers form a 3-dim surface in the infinite dimensional phase space of \eqref{sinegordon}.

\subsection{Non-existence of small amplitude breathers} \label{SS:non-existence}

The sine-Gordon equation \eqref{sinegordon} is a particular case of the family of \textit{nonlinear Klein-Gordon equations} in one space dimension. In this paper, we study the existence/non-existence of {\it small} breathers of a class of nonlinear  Klein-Gordon equations 
\begin{equation}\label{kleingordonrev}
\partial_{t}^2u-\partial_x^2u+u-\dfrac{1}{3}u^3-f(u)=0,
\end{equation}
where the nonlinearity $f$ satisfies 
\begin{equation} \label{E:f-Assump} 
f(u) \text{ is a real-analytic odd function and } f(u)=\er(u^5) \text{ near } 0.
\end{equation} 
While their signs are natural restrictions, the coefficients $1$ and $\tfrac 13$ in the above equation are not. In fact, given any nonlinear Klein-Gordon equation $(\pa_T^2 v - \pa_X^2) v + F(v)=0$ with a smooth real valued odd function $F(v)$ with $F'(0)>0$ and $F'''(0)<0$, it is always possible to rescale $v(X, T) = A u(aX, aT)$ so that $u(x, t)$ satisfies \eqref{kleingordonrev}. 

Let $\omega>0$ denote the temporal frequency of a possible breather $u(x, t)$ of \eqref{kleingordonrev}. 
A solution $u(x,t)$ of \eqref{kleingordonrev} is a breather of temporal frequency $\omega$ if $u(x,t)$ is $\tfrac{2\pi}\omega$-periodic in $t$ and 
\[
\lim_{x\rightarrow\pm\infty}u(x,\cdot)=0
\]
in some appropriate metric. Due to the Lagrangian structure of \eqref{kleingordonrev}, 
\begin{equation} \label{E:BH}
\mathbf{H} = \int_{-\frac {\pi}\omega}^{\frac {\pi}\omega}\left( \frac 12 (\pa_x u)^2 + \frac 12 (\pa_t u)^2 - \frac 12 u^2 + \frac 1{12} u^4 + F(u) \right)dt, \quad \text{ where } \;  F(u) = \int_0^u f(s) ds = \OO(|u|^6), 
\end{equation}
is a constant in $x$ for any $\tfrac{2\pi}\omega$-periodic-in-$t$ solutions of  \eqref{kleingordonrev}, which vanishes for any breather of temporal frequency $\omega$. 

Any real valued function $\tfrac{2\pi}\omega$-periodic in $t$ can be expressed as a Fourier series
\begin{equation} \label{E:Fourier-1}
u(t) = \sum_{n=-\infty}^{+\infty} \left(- \frac i2\right) u_n e^{in\omega t}, \quad u_{-n} = - \overline {u_n},
\end{equation}   
where the factor $-\tfrac i2$ is purely for the technical convenience when the problem is reduced to functions odd in $t$ represented in Fourier sine series. We denote 
\begin{equation} \label{E:norm-1}
\Pi_n [u] = u_n = \frac {i\omega}\pi \int_{-\frac \pi\omega}^{\frac \pi\omega} u(t) e^{-i n\omega t} dt, \qquad \| u\|_{\n}=  \sum_{n=-\infty}^{+\infty} |u_n| =  \sum_{n=-\infty}^{+\infty} |\Pi_n [u]|. 
\end{equation} 
Sometimes, with slight abuse of the notation, we also use $\Pi_n [u]$ to denote 
the mode $-\tfrac i2 u_n e^{in\omega t}$. Just like $\|\cdot\|_{L_t^\infty}$, the above norm $\|\cdot \|_{\n}$ is invariant under the rescaling in $t$, but controls the former. As we shall also take the advantage of the 
conservation (in the variable $x$) of $\mathbf{H}$, 
Sobolev norms like $\|\cdot \|_{H_t^k \big( (-\frac \pi\omega, \frac \pi\omega)\big)}$ will be involved as well. 

To state the first main result of the paper, we  introduce the function space $\FF_r$ of the nonlinearity $f$ under consideration and the definition of  single bump (in $x$) breathers (see Figure \ref{fig:bumps}). 

\begin{equation}\label{def:Banach:f} \begin{split}
\mathcal{F}_r=\Big\{f:  \{u \in \C: |u|<r\} & \to\C,\,f \text{ odd and real-analytic}, \\ 
& f(u)=\sum_{k\geq 2}f_k u^{2k+1},\ f_k \in \R, \ \|f\|_r := \sum_{k\geq 2}|f_k|r^{2k+1}<\infty\Big\},
\end{split} \end{equation}
which is equivalent to the Banach space of real valued sequences $(f_{k})_{k=2}^\infty$ with the above weighted $\ell_1$ norm. 

\begin{definition} \label{D:bump}
Let $\sigma\in (0, 1)$ and $\omega>0$. We say that a $\tfrac {2\pi}\omega$-periodic-in-$t$ function $u(x, t)$ is $\sigma$-multi-bump in $x$ in the $\ell_1$ norm  
if there exist $x_1<x_2<x_3<x_4<x_5$ such that 
\[
\max \{\|u(x_{j_1}, \cdot)\|_{\ell_1} \mid j_1\in 1,3,5\} \le \sigma \min \{ \|u(x_{j_2}, \cdot)\|_{\ell_1} \mid j_2 \in 2, 4\}.
\]
A function $u(x, t)$ is said to be $\sigma$-single-bump if it is not 
$\sigma$-multi-bump. 
\end{definition}

\begin{figure}[!]	
	\centering
	\begin{overpic}[width=6cm]{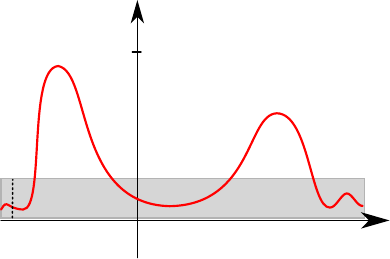}	
    \put(103,9){{\footnotesize$x$}}	
        \put(-5,18){{\footnotesize$\sigma$}}	
        \put(39,61){{\footnotesize$\|u\|_{\ell_1}$}}		
                \put(37,51){{\footnotesize$1$}}	
\end{overpic}\hspace*{0.7cm}
    \begin{overpic}[width=6cm]{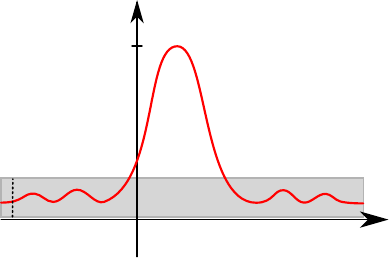}					
	    \put(103,9){{\footnotesize$x$}}	
        \put(-5,18){{\footnotesize$\sigma$}}	
	\put(39,61){{\footnotesize$\|u\|_{\ell_1}$}}	
	                \put(31,53){{\footnotesize$1$}}			
		\end{overpic}
	\caption{Multi-bump (left) and single-bump (right) functions according to Definition \ref{D:bump}.}	
	\label{fig:bumps}
\end{figure}
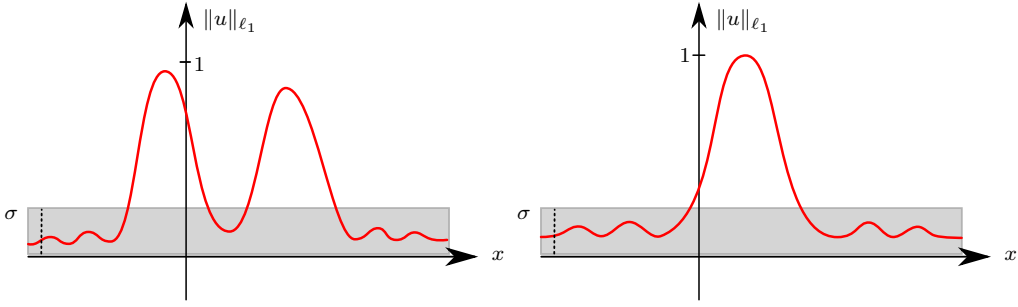 

Here $x_2$ and $x_4$ can be viewed as two bumps separated by a trough at $x_3$. The following theorem is the first main result of this paper. It will be a consequence of the more detailed Theorems \ref{T:main} and \ref{prop:Stokesconstant} below. 

\begin{theorem} \label{T:1}
Fix $r>0$. Then there exists an open and dense set $\mathcal{U}\subset\FF_r$ such that for any $f\in\mathcal{U}$ the following holds. 
For any $\sigma \in (0, 1)$, there exists $\rho^*>0$ such that there does not exist any solution $u(x, t)$ to \eqref{kleingordonrev} which:
\begin{enumerate}
\item is $\frac {2\pi}\omega$-periodic in $t$ for some $\omega>0$,
\item is $\sigma$-single-bump  in the $\ell_1$ norm  in the sense of Definition \ref{D:bump},
\item satisfies that, as $|x| \to +\infty$,   
\begin{equation} \label{E:decay-E}
		 \| u(x, \cdot)\|_{H_t^1 \big( (-\frac \pi\omega, \frac \pi\omega)\big)} + \| \pa_x u(x, \cdot)\|_{L_t^2 \big( (-\frac \pi\omega, \frac \pi\omega)\big)} \to 0, 
		\end{equation}
\item satisfies
\begin{equation} \label{E:smallness-T}
\sup_{x\in \R}  \,  \| u(x, \cdot)\|_{\ell_1} 
<\min\{1, \rho^* \omega^{\frac 12}\}.
\end{equation} 
\end{enumerate}
\end{theorem}

Some comments on Theorem \ref{T:1} are in order.
\begin{enumerate}
\item {\it About $\mathcal{U}$}: in the more detailed Theorems \ref{T:main} and \ref{prop:Stokesconstant} below we give a more precise characterization of the set $\mathcal{U}\subset\mathcal{F}_r$ where Theorem \ref{T:1} holds. Indeed, the complement of $\mathcal{U}$ in $\mathcal{F}_r$ is the preimage of zero of a certain analytic non-trivial function $C_\mathrm{in}:\FF_r\to \mathbb{C}$.

\item {\it Regarding the smallness}: it is worth pointing out that, for a function $\frac {2\pi}\omega$-periodic in $t$, while the norm $\|  \cdot\|_{\ell_1}$ is at the same level as $\|\cdot\|_{L_t^\infty}$ in scaling and controls the latter, the smallness in the theorem (see \eqref{E:smallness-T}) is measured uniformly in $\omega$ in terms of $\omega^{-\frac 12} \| u(x, \cdot)\|_{\ell_1}$. Even though this quantity looks to depend on $\omega \in \R^+$, in scaling it is comparable to $\|\cdot\|_{L_t^2 (-\frac {\pi}\omega, \frac {\pi}\omega)}$ involved in the conserved ($\mathbf{H}$). This norm is   weaker than $\|\cdot\|_{H_t^1}$, clearly the above theorem also implies the nonexistence of single-bump breather solutions, small in the energy norm (as in \eqref{E:decay-E}).

\item Theorem \ref{T:1} is only concerned with small {\it single-bump-in-$x$} breathers and it does not rule out possible small 
periodic-in-$t$ solutions which decay as $|x| \to \infty$ but with multiple bumps.
Clearly if $u(x, t)$ is $\sigma$-multi-bump, then it is also multi-bump for any $\sigma' \in (\sigma, 1)$. Hence the constant $\rho_*$ of the smallness increases as $\sigma$ increases. 

\item {\it Breathers with exponentially small tails}: while small single-bump breathers are not expected for \eqref{kleingordonrev} for most nonlinearities $f$, a  more generic phenomenon is the existence of small breathers with exponentially small (with respect to the amplitude), but non-vanishing, tails for certain values of $\omega$. 
Such solutions are usually called \emph{generalized breathers} (see \cite{L14}). Proposition \ref{prop:generalized} below  gives precise estimates on the tails of 
those generalized breathers.
\end{enumerate}

Before stating the more detailed Theorems \ref{T:main} and \ref{prop:Stokesconstant}, let us put our result into some context.

Breather type solutions represent important structures in high energy physics {\it etc.} Moreover,  they are of fundamental  importance  since  they serve as building blocks organizing the infinite dimensional dynamics of the underlying evolutionary PDE. In \cite{CLL20}, Chen, Liu, and Lu proved the soliton resolution of \eqref{sinegordon} using the integrable theory. Namely, in certain weighted Sobolev norm, solutions to \eqref{sinegordon} decay (at an algebraic rate in $t$) to a finite superposition of kinks, anti-kinks, and breathers, where breathers are the only spatially localized class. Therefore breather type structures could play a crucial role in the asymptotic dynamics of the nonlinear Klein-Gordon equations. In particular, unlike relative equilibria such as kinks, standing waves, {\it etc.}, breathers may be of arbitrarily small amplitudes and energy and thus give rise to obstacles to possible nonlinear dispersive decay or scattering of small energy solutions. (Small amplitude breathers become large in certain weighted norms adopted in some literatures, e.g. \cite{HN08, De16, CLL20} {\it etc}.)

The \eqref{sinegordon} breathers \eqref{breatherssine} are obtained based on the complete integrability of \eqref{sinegordon}. However, for non-integrable Klein-Gordon equations, the existence of (small amplitude) breathers is a completely different  problem due to the lack of effective tools such as the inverse scattering method. 
It is   a fundamental question to assert whether the  existence of breathers is a special phenomenon due to the integrability or it occurs more generally. 
In fact,  the existence of breathers for non-integrable nonlinear wave equations is expected to be rare\footnote{On the contrary breathers are more likely to exist in Hamiltonian systems on lattices, see for instance \cite{AubryM94,McKay00,Yuan02,PelinovskyPP16,PelinovskyPP20}.} (see \cite{D93,S93, BMW94}).

In the  seminal work \cite{KS87} from 1987, Kruskal and Segur used an ingenious formal asymptotic expansion to show the nonexistence of small  $\OO(\e)$ amplitude breathers in a class of nonlinear Klein-Gordon equations for $\omega$ which is $\eps^2$-close to the resonant frequency $\omega=1$. The obstacle to solving for breathers is exponentially small in $\e\ll1$. For the past more than thirty years, as far as the authors know, no rigorous justification 
of their leading order exponentially small asymptotics had been given for such nonlinear PDEs. A fundamental part of the proof of Theorem \ref{T:1} is to 
 provide a rigorous proof of  Kruskal and Segur's 
formal argument (for odd analytic nonlinearities) as well as rule out the existence of breathers for other frequencies (either close to other resonances or away from resonances).  This is stated more precisely in Theorem \ref{T:main} below. 

Among other works on small breathers, 
in \cite{KMM17}, the authors proved that small breathers odd in $x$ do not exist for \eqref{kleingordonrev} by establishing certain asymptotic stability in the phase space of odd-in-$x$ functions. The oddness is, however, contrary to the well-known examples -- the \eqref{sinegordon} breathers \eqref{breatherssine} are even in $x$.  Breathers have also been proven not to exist for some generalized  KdV equations and the Benjamin-Ono equation, see \cite{MunozP19, MunozP19b}.
In \cite{BCLS11}, small breathers of a 1-dim nonlinear wave equation in periodic (in $x$) media were obtained. They play an important role in theoretical scenarios where photonic crystals are used as optical storage. In this model, the periodic media causes the spectra of the linearized problem to be rather different, and this makes the existence of small breathers possible.  

For breathers of $\OO(1)$ amplitude, in the classical works \cite{D93, De95, BMW94}, Denzler and Birnir-McKean-Weinstein studied the rigidity of breathers, namely, the persistence of (infinite subfamilies of) breathers \eqref{breatherssine} when \eqref{sinegordon} is perturbed as 
\[
\partial^2_t u-\partial^2_x u+ \sin(u)=\e\Delta(u)+\er(\e^2),
\]
where $\Delta$ is an analytic function in a small neighborhood of $u=0$. They proved that breathers corresponding to infinitely many amplitudes $m = \sqrt {1-\omega^2}$ persist only if $\Delta (u)$ results from a trivial rescaling of \eqref{sinegordon}.
In \cite{Ki91}, \eqref{sinegordon} was also singled out as the only 1-dim nonlinear Klein-Gordon equation admitting breathers in certain form (see also \cite{Mandel21}). These rigidity results are consistent with the generic non-existence of breathers. 
Even though for small amplitude, \eqref{kleingordonrev} might also be viewed as close to \eqref{sinegordon} in the $C^\infty$ class, 
it does not help much in the analysis of the exponential small obstruction to the existence of breathers (see Theorem \ref{T:main} below) since the nonlinearity of the former is not a perturbation to that of the latter in the analytic function class. 

The nonlinear Klein-Gordon equation \eqref{kleingordonrev} is also quite different from the one studied in \cite{SW99} (see also \cite{BambusiC11}). The spatial variable in  \cite{SW99,BambusiC11} is taken in $\mathbb{R}^3$ (which gives stronger dispersion than in $\mathbb{R}$) and  an extra potential term $V(x) u$ is added. This term creates an isolated oscillatory eigenvalue of the linear problem whose interaction with the continuous spectra leads to slow radiation. In  contrast,  \eqref{kleingordonrev} does not contain a potential term and its small breathers have temporal frequencies slightly less than $1$, which is the end point of the continuous spectrum of \eqref{kleingordonrev} linearized at $0$.
In \cite{Sc20}, temporally periodic and spatially decaying solutions were found for the nonlinear Klein-Gordon equation with cubic nonlinearity, i.e. $f(u)=0$, for $x \in \R^3$. These solutions are close to some steady solutions (not necessarily small) with $\er(1/|x|)$ spatial decay. Such decay, which is too slow for the solutions to be in the energy space, is due to the 3-dim Helmholtz equation, whose solutions would only be in $L^\infty$ and oscillate if $x \in \R^1$. Hence these solutions are more analogous to the breathers with tails constructed in \cite{SZ03} in 1-dim.

\subsection{Main quantitative results: leading order of the exponentially small obstruction}
Theorem \ref{T:1} is a consequence of the more detailed 
Theorems \ref{T:main} and \ref{prop:Stokesconstant} below. In seeking small breathers, which are conceptually born from the end point of the continuous spectra of the linear Klein-Gordon equation, it is essential that the temporal frequency $\omega$ is close to  resonant. Hence, to state the more detailed theorem,  we divide $\omega \in \R^+$ into two primary classes 
\begin{equation}\label{E:intervals}
		I_k(\e_0)=\left[\sqrt{\frac{1}{k(k+\e_0^2)}}, \frac{1}{k}\right), \; k \in \N, \qquad \text{and}
		\qquad J_k(\e_0)=\left[ \frac{1}{k+1}, \sqrt{\frac{1}{k(k+\e_0^2)}}\right),\ 
		k\in\N \cup \{0\},
	\end{equation}
where $0< \eps_0 \le 1/2 $ is a parameter to be determined later. Note $J_0(\e_0) = [1, \infty)$ and  $(0,\infty)=\left(\cup_{k\in\N}I_k \right)\bigcup \left( \cup_{k \ge 0} J_k\right)$. We shall comment more on these sets in the context of spatial dynamics in Section \ref{SS:spa-dyn}. 

\begin{theorem} \label{T:main}
Fix $r>0$ and consider $f\in\FF_r$ (see \eqref{def:Banach:f}), then the following statements hold.
\begin{enumerate}
		\item There exists $\rho_1^*>0$ such that for any $\e_0 \in (0, 1/2]$, $\omega \in J_k(\e_0)$, $k \in\N \cup \{0\}$, if $u(x, t)$ is a  $\frac {2\pi}\omega$-periodic-in-$t$ solution to \eqref{kleingordonrev} satisfying 
		\begin{equation} \label{E:decay-E-1}
		 \| u(x, \cdot)\|_{H_t^1 \big( (-\frac \pi\omega, \frac \pi\omega)\big)} + \| \pa_x u(x, \cdot)\|_{L_t^2 \big( (-\frac \pi\omega, \frac \pi\omega)\big)} \to 0, 
		\end{equation}
		as $x \to +\infty$ or $-\infty$, then 
		\begin{equation} \label{E:small-0}
		\sup_{x\in \R}  
		\| u(x, \cdot)\|_{\ell_1} 
		\ge \rho^*_1 \min\{1,\, \e_0 \omega^{\frac 12}\}. 
		\end{equation}
		
		\item There exists $C_\mathrm{in}\in \C$ and  $\rho_2^*>0$ depending on $f$ such that, for any $y_0>0$, there exist 
		$\e_0, M>0$ such that for any 
		\begin{equation} \label{E:k-omega}
		\omega = \sqrt{\frac{1}{k(k+\e^2)}} \in I_k (\e_0), \;\; \forall \, k\ge 1, 
		\end{equation}
		there exist unique $\tfrac {2\pi}{k\omega}$-periodic and odd in $t$ solutions $u_{\mathrm{wk}}^\star (x, t)$, $\star=s, u$, to \eqref{kleingordonrev}, only containing Fourier modes $n \in k\Z$ with odd $\tfrac nk$ in \eqref{E:Fourier-1}, such that 
		
		\begin{enumerate} 
		\item For  $x\ge - \frac {y_0}{\e \sqrt{k} \omega}$ for $\star=s$ and $x\le  \frac  {y_0}{\e \sqrt{k} \omega}$ for $\star=u$, they can be approximated as  
		\begin{equation} \label{E:u-wk}
                 \left\| \big(1-\frac 1{(k\omega)^2}\pa_t^2\big) \left( \begin{pmatrix} u_{\mathrm{wk}}^\star (x, t) \\ \frac {\pa_x u_{\mathrm{wk}}^\star (x, 		t)}{\e \sqrt{k} \omega} \end{pmatrix} - \e \sqrt{k} \omega\begin{pmatrix} v^h (\e \sqrt{k}			\omega x) \\ (v^h)' (\e \sqrt{k}\omega x)  \end{pmatrix} \sin k\omega t \right)\right\|_{\ell_1}\le M k^{-\frac 	32} \e^3 v^h (\e \sqrt{k}\omega x),     
		\end{equation}
		where  $v^h (y) = \frac{2\sqrt{2}}{\cosh y}$;  
		
		\item They also satisfy $\Pi_k \big[\pa_x u_{\mathrm{wk}}^\star (0, \cdot)\big]=0$, $\star=s, u$, and 
		\begin{equation}\label{def:splittingmaintheorem}
\left\| 
\big( |-\pa_t^2 - 1|^{\frac 12} (u_{\mathrm{wk}}^u - u_{\mathrm{wk}}^s )+ i 
\pa_x (u_{\mathrm{wk}}^u - u_{\mathrm{wk}}^s)\big) (0,t
) - 4\sqrt{2} 
C_{\mathrm{in}}e^{-\frac {\sqrt{2k} \pi} \e} \sin 3k\omega t  \right\|_{\ell_1}
\le  \frac {M e^{-\frac {\sqrt{2k} \pi} \e}}{\tfrac 12 \log k - \log \e}.   
\end{equation} 
		
		\item A $\frac {2\pi}\omega$-periodic-in-$t$ solution $u(x, t)$ to \eqref{kleingordonrev} satisfies 
		\begin{equation} \label{E:small-1}
		\sup_{x\in \R} \|u(x, \cdot)\|_{\ell_1} \le \rho_2^*\omega^{\frac 12}
		\end{equation}
		and \eqref{E:decay-E-1} as $x \to -\infty$ (or as $x \to +\infty$) iff $u_{\mathrm{wk}}^u$ satisfies \eqref{E:small-1} and $u(x, t)= u_{\mathrm{wk}}^u (x+x_0, t+t_0)$ (or $u(x,t)= u_{\mathrm{wk}}^s (x+x_0, t+t_0)$) for some $x_0, t_0 \in \R$. Consequently, there exists a solution $u(x,t)$  to \eqref{kleingordonrev} satisfying \eqref{E:small-1} and \eqref{E:decay-E-1} as $|x| \to \infty$ iff there exists $r\in \R$ such that $u_{\mathrm{wk}}^u (x+r, t) = u_{\mathrm{wk}}^s (x, t)$ for all $x$ and $t$ and satisfies \eqref{E:small-1}. 
		\end{enumerate}\end{enumerate} 		
\end{theorem}

In this theorem, while the non-existence of small breathers is confirmed in the case of temporal frequency $\omega \in J_k(\e_0)$, the only possible (up to translations) candidates $u_{\mathrm{wk}}^\star (x, t)$, $\star=s, u$, of small breathers with $\omega \in I_k(\e_0)$ are identified along with optimal estimates. In particular, it gives a necessary and sufficient condition (in statement (2c)) on the existence of small breathers that $u_{\mathrm{wk}}^\star$, $\star=s, u$, must remain small for all $x\in \R$ and coincide after a translation. {\it The most important result of the theorem  is statement (2b)} which rigorously identifies the exponentially small leading order term and its coefficient $C_\mathrm{in}$ in the splitting of $u_{\mathrm{wk}}^\star$, $\star=s, u$, when they get close in an infinite dimensional space (of periodic functions of the variable $t$) in their first opportunity in $x$. 

From Theorem \ref{T:main}(2ab), one may verify that $x=0$ is the only critical point of $\Pi_k [u_{\mathrm{wk}}^\star (x, \cdot)]$ for $\pm x \le \frac 1{\e \sqrt{k} \omega}$, $\star = u, s$. Therefore, if $C_\mathrm{in} \ne 0$, then there does not exist $|r| \le \frac 1{\e \sqrt{k} \omega}$ such that $u_{\mathrm{wk}}^s (x, t) \equiv u_{\mathrm{wk}}^u (x +r, t)$. Hence $C_\mathrm{in} \ne 0$  excludes the existence of small single bump breathers due to \eqref{E:u-wk} and \eqref{def:splittingmaintheorem}, which are the simplest and the most natural class of small breathers including those given in \eqref{breatherssine} for \eqref{sinegordon} (see Figure \ref{fig:bumps}). 

Note then that  Theorem \ref{T:main} is conditional, since it proves the nonexistence of single-bump small amplitude  breathers \emph{provided} the constant $C_\mathrm{in}=C_\mathrm{in}(f)\in \C$ satisfies $C_\mathrm{in}\neq 0$. In particular, {\it it proves Theorem \ref{T:1}} as long as $C_\mathrm{in}(f)\neq 0$ for an open and dense set of $f\in\FF_r$.  Next theorem, proven in Section  \ref{sec:Stokes}, shows that  this is indeed  the case. In fact, we also give an explicit family of nonlinearities $f(\mu, u)$, involving a parameter $\mu$, such that $C_\mathrm{in} (f) \ne 0$ for all but a discrete set of $\mu\in\mathbb{R}$.

\begin{theorem}\label{prop:Stokesconstant}
Fix $r>0$. The map $C_\mathrm{in}:\mathcal{F}_r\to \C$ introduced in Theorem \ref{T:main}(2) is analytic and non-constant. Moreover, the set $\mathcal{U}=\mathcal{F}_r\setminus C_\mathrm{in}^{-1}(0)$ is open and dense.
\end{theorem}

As already mentioned, even if small single-bump breathers are not expected  to exist for \eqref{kleingordonrev} with most $f$, the  more generic phenomenon is the existence of small breathers with non-vanishing tails which are exponentially small with respect to the amplitude.  This is stated in the next proposition.

\begin{prop} \label{prop:generalized}
Fix $r>0$ and consider $f\in\FF_r$, then the following holds. 
There exist $\e_0, M>0$ such that  for any  $k\in \N$ and 
		$\omega =\sqrt{\frac{1}{k(k+\e^2)}}  \in I_k (\e_0)$:
\begin{enumerate} 		
		\item There always exist  $\frac {2\pi}\omega$-periodic-in-$t$ solutions $u(x, t)$ 
		such that 
		\begin{align*}
& \big\| |-\pa_t^2 - 1|^{\frac 12} (u - u_{\mathrm{wk}}^\star)(x,\cdot) 
\big\|_{L_t^2 (-\frac \pi\omega,\frac \pi\omega)} + \big\| \pa_x (u - 
u_{\mathrm{wk}}^\star)(x,\cdot)\big\|_{L_t^2 (-\frac \pi\omega,\frac \pi\omega)} 
\le  
M k^{\frac 12} e^{-\frac {\sqrt{2k} \pi} \e},   
\end{align*} 
		for both all $x\ge 0$ with $\star=s$ and $x\le 0$ with $\star=u$. 
		
\item Suppose that the constant $C_{\mathrm{in}}$ introduced in Theorem \ref{T:main} satisfies $C_{\mathrm{in}} \ne0$. Then, the  breather  with tails $u(x, t)$ given by item  (1) also satisfies
		\[
		\big\| |-\pa_t^2 - 1|^{\frac 12} (u - 
u_{\mathrm{wk}}^\star)(x,\cdot) \big\|_{L_t^2 (-\frac \pi\omega,\frac 
\pi\omega)} + \big\| \pa_x (u - u_{\mathrm{wk}}^\star)(x,\cdot)\big\|_{L_t^2 
(-\frac \pi\omega,\frac \pi\omega)} \ge  
\tfrac {|C_{\mathrm{in}}|}M k^{\frac 12} e^{-\frac {\sqrt{2k} \pi} \e},   
		\] 
		for both all $x\ge 0$ with $\star=s$ and $x\le 0$ with $\star=u$. 
\end{enumerate} 
\end{prop}

We give several remarks on Theorems \ref{T:main}, \ref{prop:Stokesconstant} and Proposition \ref{prop:generalized}.
\begin{enumerate}
 \item The constant $C_\mathrm{in}$ introduced in Theorem \ref{T:main} is often referred to as the {\it Stokes 
constant} in the literature, which is the coefficient of the leading order term 
in the  exponentially small obstruction in \eqref{def:splittingmaintheorem}. 
We emphasize that the non-existence of small single bump amplitude breathers holds for 
 {\it all} frequencies under the {\it single} condition $C_\mathrm{in}\ne 0$. This constant depends on the full jet of the real analytic nonlinearity $f$, but is independent of $k$ or $\omega$. No simple closed formula has been identified for $C_\mathrm{in}$ in the literature. We expect that one should be able to develop a computer assisted proof to check the nonvanishing of $C_\mathrm{in}$ for given nonlinearities (following the ideas developed in \cite{BCGS21}  for a 3-dimensional Hopf-zero bifurcation). In Section \ref{SS:Stokes} below, we conjecture a formula of the Stokes constant in terms of a series. 

\item  The relative scale between $u$ and $\pa_x u$ in the estimates in Theorem \ref{T:main}(2) is consistent with the quadratic part of the Hamiltonian $\mathbf{H}$, where $|-\pa_t^2 - 1|$ is somewhat degenerate of the order $\er(k^{-1} \e^2)$ when applied to the $k$-th Fourier mode $e^{ik \omega t}$. 

\item The generalized small amplitude breathers given by Proposition \ref{prop:generalized} have frequency $\omega$ slightly smaller than each $\tfrac 1k$. This $\frac 1k$ is consistent with the fact that small \eqref{sinegordon} breathers \eqref{breatherssine} have periods slightly greater than $2k\pi$.
Each of these breather-like solutions  to \eqref{kleingordonrev} with exponentially small tails is the superposition of a small exponentially localized-in-$x$ wave of order $\er(\e k^{-\frac 12} e^{- \e k^{-\frac 12} |x|})$ with an $L^\infty_{xt}$ correction up to  order $\er(k^{\frac 12} \e^{-1} e^{-\frac {\sqrt{2k} \pi} \e})$. In the generic non-degenerate case of $C_{\mathrm{in}}\ne 0$, the infimum of the tails of such generalized breathers is also bounded below by this exponential order.

\item In contrast to the fact that the breathers of \eqref{sinegordon} form a 3-dim manifold in the infinite dimensional space of solutions, the breathers with exponentially small tails actually form a family of finite codimension (see Proposition \ref{prop:FirstBif:Generalized} below). 

\item As $u_{\mathrm{wk}}^\star$, $\star=u, s$, are special solutions of high regularity, the norm in Theorem \ref{T:main}(2ab) 
actually could be refined to be $H^n_t$ for any $n\ge 0$. In  contrast, since the set of breathers with exponentially small tails is of finite codimension in the energy space of the spatial dynamics, the norms in Proposition \ref{prop:generalized} arising from the quadratic part of the Hamiltonian $\mathbf H$ are not expected to be improved.    

\item In the generic case of $C_\mathrm{in}\neq 0$ provided by Theorem \ref{prop:Stokesconstant} which implies that \eqref{kleingordonrev} does not have small breathers, the asymptotic behavior of small solutions in the energy space $H_x^1 (\R) \times L_x^2 (\R)$  is a natural but intriguing question. 
Even though the breathers with exponentially small tails $u(x, t)$ obtained in Proposition \ref{prop:generalized} are not in the energy space $H_x^1 (\R) \times L_x^2 (\R)$, they still shed some light on the dynamics of \eqref{kleingordonrev}. Take $k=1$ for simplicity. Let $\gamma \in C_0^\infty (\R, \R)$ be a cut-off function satisfying $\gamma(s) =1$ for $|s|\le 1$ and $u(x, t)$ be a sufficiently smooth  breather like solution with exponentially small tails (see Proposition \ref{prop:generalized}). Consider the solution $\tilde u (x, t)$ to \eqref{kleingordonrev} with initial value $\gamma (\tfrac 1{\e^3} e^{-\frac {\sqrt{2} \pi}\e} x) \big( u(x, 0), \pa_t u(x, 0)\big)$. Its $H_x^1 \times L_x^2$ norm is of the order $\er (\sqrt{\e})$.
The propagation speed of  \eqref{kleingordonrev} being equal to $1$ implies that $\tilde u$ is periodic in $t$ for $|x|, |t|\le \er (\e^{-3} e^{\frac {\sqrt{2} \pi}\e})$. Hence, this exponential time scale has to be relevant in studying the asymptotic dynamics of small energy solutions of \eqref{kleingordonrev}.  
\end{enumerate}

Breathers with small tails as well as some other similar types of solutions had already been 
obtained, but often with only exponentially small {\it upper bound estimates} on the tails, instead of their precise orders or lower bounds (and without  the explicit exponent $-\frac {\sqrt{2k} \pi}\e$). 
In \cite{L14}, Lu derived breathers with tails bounded by 
$\er(e^{-\frac c\e})$ for some unspecified $c>0$. 
In a sequence 
of papers, Groves and Schneider considered small amplitude 
modulating pulse solutions to a class of semilinear \cite{Gr01} and quasilinear 
\cite{Gr05, Gr08} reversible wave  equations. These are solutions consisting of 
pulse-like spatially localized envelopes advancing in the laboratory frame and 
modulating an underlying wave-train of a fixed wave number $\xi_0>0$, which are 
time-periodic in a moving frame of reference. They would become breathers if 
$\xi_0=0$. For quasilinear reversible wave equations, Groves and Schneider 
constructed solutions $u(x, t)$ of this type  with tails bounded by $\er 
(e^{-\frac c{\sqrt{\e}}})$ but only defined  for 
$|x| \le \er (e^{\frac c{\sqrt{\e}}})$. The finite length of the domain in $x$ 
was mainly due to difficulties arising in quasilinear PDEs. In the semilinear 
case, such solutions could be derived globally in $x\in \R$ with the same $\er 
(e^{-\frac c{\sqrt{\e}}})$ 
estimates on the tails. The upper bounds of the tails in these papers 
were obtained by making the error terms small through consecutive applications 
of partial normal forms, e.~g.~as in \cite{Ne84, IL05}. 

The proof of Theorems \ref{T:1}, \ref{T:main} and Proposition \ref{prop:generalized} rely on the spatial dynamics method (see, e.g. \cite{Ki82, We85}). This method is often an effective approach in constructing certain coherent structures for nonlinear PDEs where a spatial variable $x$ plays a distinct role. In such a framework, the desired solutions are sought as special solutions in an evolutionary system where this $x$ is treated as the dynamic variable.  
 
We fix a temporal frequency $\omega>0$ and consider in the rescaled variable $\tau =\omega t$, 
\begin{equation}\label{eq:KLGtau}
\omega^2\partial_{\tau}^2u-\partial_x^2u+u-\dfrac{1}{3}u^3-f(u)=0.
\end{equation}
Considering $x$ as the evolutionary variable, it is a globally well-posed infinite dimensional Hamiltonian system  in appropriate spaces of $2\pi$-periodic-in-$\tau$ functions  where its Hamiltonian can be derived from $\mathbf{H}$. 

Breathers of \eqref{kleingordonrev} correspond to $2\pi$-periodic-in-$\tau$ solutions of \eqref{eq:KLGtau} which decay to $0$ as $|x| \to \infty$, namely, orbits homoclinic to the equilibrium $0$ due to the intersection of its stable and unstable manifolds. Note that, on the one hand, the notions of single-bump and multi-bump breathers (see Definition \ref{D:bump})
get translated to homoclinics as in  Figure \ref{fig:bumpshomo}. On the other hand, small amplitude breathers correspond to small homoclinic loops. 
From this point of view, the proof of Theorem \ref{T:main} will rely on analyzing the (finite dimensional) stable and unstable invariant manifolds of $u=0$ and on whether their  intersections lead to \emph{small} homoclinic loops. Proposition \ref{prop:generalized} will rely on analyzing the center-stable and center-unstable invariant manifolds and constructing small homoclinic loops to the center manifold.

\begin{figure}[!]	
	\centering
	\begin{overpic}[width=5cm]{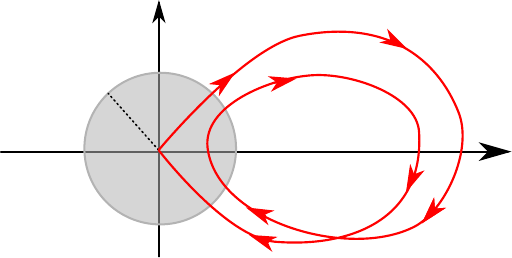}	
		\put(16,34){{\footnotesize$\sigma$}}	
	\end{overpic}\hspace*{0.7cm}
	\begin{overpic}[width=5cm]{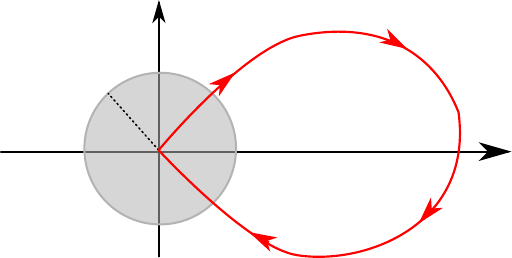}				
	
		\put(16,34){{\footnotesize$\sigma$}}	
		\end{overpic}
	\caption{Multi-bump (left) and single-bump (right) solutions in the 
spatial dynamics framework.}	
	\label{fig:bumpshomo}
\end{figure}

In the next section we first present an abstract setting for analyzing (the breakdown of) small homoclinic loops through a dynamical system approach. Then, we show how the Klein Gordon equation \eqref{eq:KLGtau} fits into this framework regarded from the spatial dynamics point of view.

\subsection{Birth of small homoclinics via  ``eigenvalue collision": exponentially small splitting of separatrices} \label{SS:BT-bifur}

Let us consider an  $N$-dimensional system, $N \le \infty$,  $(P_{\al})$ involving a parameter $\al\in I\subset\R$, which has a steady state at 0. We want to analyze whether this steady state has small homoclinic loops.

Assume the following happens in  $(P_{\al})$.

\begin{itemize}
\item [(a)] {\it ``Eigenvalue collision"} at $\al=0$. Namely, in a neighborhood 
of $0 \in \C$, there are exactly two eigenvalues $\pm \lambda (\al) \sim \pm 
\sqrt{\alpha}$ (modulo symmetries, but counting the algebraic 
multiplicity) of the linearization of $(P_{\al})$ at $0$. As $\alpha$ increases, 
they move towards $0$ from the imaginary axis $i\R$ and then move into the real axis $\R$ after coinciding at $0$ when $\al=0$. 

\item [(b)] 
For $\al>0$, the {\it normal form} of the local nonlinear system $(P_{\al})$ near $0$ projected to the 2-dimensional eigenspace $\mathcal M$ associated to $\pm \lambda(\alpha)$ is equivalent to 
\begin{equation}\label{def:singularlimitBogdanov}
\ddot u - \lambda(\al)^2 u + u^m = 0,\quad m\geq 2,
\end{equation}
where the ``$+$" sign matters only when $m$ is odd. 
Apparently this normal form system has one or two small homoclinic orbits of amplitude $\er\big(\lambda (\al)^{\frac 2{m-1}}\big)$ for $0< \al \ll1$.     

\item[(c)] The system $(P_{\al})$ has a first integral which is locally positive definite in the center manifold around the steady state.
\end{itemize}

Note that, under these assumptions, small homoclinic loops cannot exist 
either  of $\al<0$  or if $\al>0$ is ``not close to 0''. One first observes that they cannot exist  in the center manifold $W^c(0)$ since the steady state  is isolated at its level of energy inside $W^c(0)$. Hence homoclinic orbits 
exist only at the intersection of the stable and unstable invariant manifolds $W^{s, u}(0)$. However, if 
$\al$ is not small, then the dynamics inside the stable and unstable manifolds $W^{s, u}(0)$ are conjugate to the linear dynamics in a large neighborhood of $0$. Then  any orbit in them (and in particular any possible homoclinic orbits) must first go away at a uniform distance from 0. In conclusion,  small homoclinic loops can only exist for small $\al>0$.

If the whole system $(P_{\al})$ is 2-dimensional, $\al=0$ corresponds to  one type of the Bogdanov-Takens bifurcations. In this case the existence of a  conserved quantity leads to the existence of small homoclinic orbit  for all small $\al>0$.  

When $(P_{\al})$ is of higher dimensions, then the dynamics in the directions transverse to $\mathcal M$ is at a fast scale and thus $(P_{\al})$ is a typical singular perturbation system. 

If the fast dynamic is hyperbolic, then by the standard {\it normally hyperbolic} invariant manifold theory, a 2-dimensional slow manifold $\mathcal M_\al$ persists for $0< \al \ll 1$ and  the existence of small homoclinic orbit can be reduced to the above 2-dim case on $\mathcal M_\al$. 
This mechanism indeed happens in the construction of some special solutions in some nonlinear PDEs including \cite{BCLS11} (or more of the elliptic type PDE, 
see, e.g. \cite{Mi94}). 

However, if there are fast elliptic/oscillatory directions (as happens for the Klein-Gordon equation \eqref{eq:KLGtau}), then there does not necessarily exist a slow manifold and one cannot reduce $(P_{\al})$ to 2 dimensions. Without such reduction, one is forced to find small homoclinic orbits as the intersection of the low dimensional stable and unstable manifolds $W^{u, s} (0)$ of $0$ close to $\mathcal M$ in the $N$-dimensional phase space of $(P_{\al})$, but this is highly unlikely simply by  counting the dimensions. Homoclinic orbits are generated via such eigenvalue collision mechanism only in some very lucky/rare systems, such as the completely integrable \eqref{sinegordon} where the family of breathers actually can be also extended to large amplitudes.

In a generic system of the above {\it normally elliptic case} with rapid oscillations, while the stable and unstable invariant manifolds, denoted by $W^{u, s} (0)$, do not intersect, the existence of a conserved quantity whose Hessian is positive definite in the center direction of the linearized $(P_{\al})$ at $0$ often ensures the transverse intersection of the center-stable and center-unstable manifolds $W^{cs, cu}(0)$. 
This intersection yields a finite co-dimensional tube homoclinic to the center manifold, which corresponds to generalized breathers for the Klein-Gordon equation \eqref{eq:KLGtau}). Moreover, the  distance between $W^{u, s} (0)$ determines how close this homoclinic tube is to $0$ which, in the present paper,  corresponds to how small the tails of the generalized breathers of the nonlinear Klein-Gordon equations can be.

Regarding the  distance between $W^{u, s} (0)$, the strong averaging effect of the fast oscillations makes $W^{u, s} (0)$ very close to each other -- usually $\mathcal{O}(\al^n)$ if $(P_{\al})$ has finite smoothness and $\mathcal{O}(e^{-\frac 1{\al^\delta}})$ if $(P_{\al})$ analytic. A leading order approximation such as the one obtained in Theorem \ref{T:main}(2ab) provides accurate information of this distance, usually called \emph{splitting distance}\footnote{It also sheds light for the future study of scattering maps \cite{DelshamsLS08}
induced by the homoclinic tube and multi-bump homoclinics.}.       

To summarize, the mechanism of eigenvalue collision leading  to a 
Bogdanov-Takens type bifurcation embedded in a normally elliptic singular 
perturbation problem is primarily responsible for the  birth of small homoclinic 
orbits/breathers with small tails for the nonlinear Klein-Gordon equation 
\eqref{eq:KLGtau}. 
It yields exact breathers in some very special cases such as 
the completely integrable \eqref{sinegordon}. 

In fact, this general mechanism leads to  what is usually called \emph{exponentially 
small splitting of separatrixs}, a phenomenon that usually arises in analytic systems with two time scales with i.) fast oscillations and ii.) slow hyperbolic dynamics with a homoclinic loop (also called separatrix), as in the setting explained above. 
Other settings where this phenomenon occurs are the resonances of nearly integrable Hamiltonian systems and
close to the identity area preserving maps.
Analysis of such phenomena is fundamental in 
the construction of unstable behaviors in these models such as Arnold diffusion 
or chaotic dynamics.

The study of the exponentially small splitting of separatrixs goes back to the 
seminal paper by Lazutkin \cite{Lazutkin84russian}, which dealt with the standard 
map.
His strategy can be described as follows:
\begin{enumerate}
 \item 
The singular limit \eqref{def:singularlimitBogdanov} has a homoclinic orbit whose  time parameterization is analytic in a strip containing the real line  and has singularities in the  complex plane.
\item 
One looks for parameterizations of the perturbed  invariant manifolds 
which are close to the unperturbed homoclinic. 
They can be extended  to complex values of the parameter which are close to the  singularities of the unperturbed homoclinic with smallest imaginary part.
\item 
One analyzes the difference between  the perturbed stable and unstable manifolds 
close to these singularities. 
To this end, one has to look for the leading order of the perturbed invariant manifolds in these complex domains. 
Then, one is encountered with two different situations:
\begin{itemize}
\item[(i)]
In some problems, the perturbed invariant manifolds are also well approximated by the unperturbed homoclinic solution near its singularities. In this case, one can show that the classical Melnikov method gives the first order of the difference between these manifolds.
\item[(ii)]
In most of the  problems, like the problem at hand, the unperturbed homoclinic is not a good approximation of the perturbed invariant manifolds in these complex domains. Therefore, one must look for new first order approximations. 
These first  orders are solutions of the so-called \emph{inner equation}, which is a singular limit equation independent of the perturbative parameter.
The analysis of this equation gives the asymptotic formula for the difference between the invariant manifolds. In particular, it provides the Stokes constant $C_\mathrm{in}$ appearing in Theorem \ref{T:main}.
\end{itemize}
\item 
The last step is to translate the analysis in the complex domain to the real parameterizations of the invariant manifolds.
\end{enumerate}
In the present paper we apply this strategy to the nonlinear Klein-Gordon equation \eqref{eq:KLGtau}, or equivalently \eqref{kleingordonrev}. It is explained heuristically in more detail in Sections \ref{sec:heuristics} and \ref{sec:strategyA} below. 

In the last decades this strategy  or similar ones relying on analytic 
continuation of the parameterization of the invariant manifolds
has been applied to various problems mostly in {\it finite dimensions}. The first category (case 3(i) above), where the Melnikov function provides the leading order of the splitting distance, 
includes fast periodic forcing of integrable Hamiltonian systems  \cite{HolmesMS88, HolmesMS91, DelshamsS92, Gelfreich94} and non-generic unfoldings of the Hopf-zero singularity \cite{BaldomaS06}. This category can also be handled by other methods such as  direct series expansions  \cite{GallavottiGM99,Wang20, Wang21}. Problems falling into the  second category (case 3(ii) above), where one needs an inner equation, are more common.  Among them are near identity maps \cite{Gelfreich99, GelfreichS01, MartinSS11a, MartinSS11b},  resonances of nearly integrable Hamiltonian systems  
\cite{Gelfreich00,OSS03, INMA,BFGS12}, and generic local bifurcations 
\cite{ Lombardi97, Lombardi00, BaldomaS08, BaldomaCS13,BaldomaCS18,BaldomaCS18b,GelfreichL14, GaivaoG11}.
This category  can also be handled by a different method, the so-called 
continuous averaging by Treschev \cite{Treshev97}.

Exponentially small splitting phenomena also arise  in the construction of solitary and traveling waves in PDEs and lattices 
\cite{Amick91,Eckhaus92,SunS93, Lombardi97,Sun99, Lombardi00,Tovbis00,TovbisP06, 
OxtobyPB06}. However, in all the above papers involving leading order analysis of the exponentially smallness the fast oscillatory dimentions are 
finite (often two). As far as the authors know, the present paper is \emph{the first one dealing with an  infinite number of oscillatory directions}.

\subsection{The spatial dynamics approach for the Klein-Gordon equation}\label{SS:spa-dyn}
 
We devote this section to implement the spatial dynamics approach for equation \ref{eq:KLGtau} and to write it as a system having the features of the class of models $P_\alpha$ introduced in Section \ref{SS:BT-bifur}.
To this end, we denote
\begin{equation}\label{g}
g(u)=\dfrac{1}{3}u^3+f(u).
\end{equation}
In terms of the Fourier series expansion \eqref{E:Fourier-1} (see also \eqref{E:norm-1}), the equation \eqref{eq:KLGtau} reads
\begin{equation}\label{eqfourier}
(\partial_x^2+n^2\omega^2-1)u_n=-\Pi_n\left[ g(u)\right],\quad\ n\in\Z.
\end{equation}
The eigenvalues 
of the linearization equation at $0$, that is
\[
\pa_x^2  u - \omega^2 \pa_\tau^2 u - u =0,
\]
are $\pm \nu_n$, where 
\[
\nu_n = \sqrt{1- n^2 \omega^2}, 
\]
and their eigenfunctions can be calculated using the Fourier modes. 

Consider $\omega \in [\tfrac 1{k+1}, \tfrac 1k)$ for some $k \ge 0$. 
For $0\le |n| \le k$, the eigenvalues $\pm\nu_n \in \R \setminus \{0\}$ are hyperbolic, while the center  eigenvalues $\pm \nu_n = \pm i\vartheta_n$, $\vartheta_n=\sqrt{n^2\omega^2-1}$, correspond to $|n| \ge k+1$. 
Recall the two primary classes of intervals $I_k(\e_0)$, $k \in \N$, and $J_k(\e_0)$, $k\in\N \cup \{0\}$, of the frequency $\omega$ defined in \eqref{E:intervals} for some $\e_0\in (0, \frac 12)$. Clearly the  dimension of the hyperbolic eigenspace of $0$ increases by $1$ as the frequency $\omega$ decreases through $\tfrac 1k$ moving from $J_{k-1} (\e_0)$ into $I_k (\e_0)$.  

In the strongly hyperbolic case of $\omega \in J_k(\e_0)$, $k \ge 0$, the smallest hyperbolic eigenvalue satisfies   
\[
\nu_k > \tfrac {\e_0}{\sqrt{k + \e_0^2}} >\min\{1,  \tfrac {\e_0}{2\sqrt{k}}\}.   
\]
Based on the general local invariant manifold theory (see, e.g. Theorem 4.4 in \cite{CL88}) and this spectral gap along with the cubic nonlinearity of wave type equation \eqref{eq:KLGtau}, one expects that the local stable/unstable manifolds of $0$ are close to the stable/unstable subspaces in a neighborhood of $0$ of radius of the order $\er \big(\min\{1,  \tfrac {\e_0}{2\sqrt{k}}\}\big)$ and all orbits on both these manifolds leave such neighborhood eventually. This argument is carried out uniformly in $k$ and $\omega$ in Section \ref{S:SHyperbolic} and statement (1) of Theorem \ref{T:main} follows consequently. Therefore they cannot intersect in such a neighborhood to produce small homoclinic orbits.    

In contrast to the above case, 
when $\omega$ decreases through $\tfrac 1k$ and enters $I_k(\e_0)$, $k \ge 1$, $\nu_k$ can be arbitrarily small. The linearized \eqref{eq:KLGtau} is only weakly hyperbolic in the $k$-th modes -- the newly generated hyperbolic directions -- and small homoclinics might  be generated through a Bogdanov-Takens bifurcation as described in Section \ref{SS:BT-bifur}. The fact $\omega \in I_k(\e_0)$ is consistent with that the periods of small \eqref{sinegordon} breathers \eqref{breatherssine} are close to $2k \pi$. The different scales in $x$ in these weakly hyperbolic directions and the other much faster directions make the local dynamics of \eqref{eq:KLGtau} near $0$ a singular perturbation problem. More precisely, let  
\begin{equation}\label{epsilon}
\e=\sqrt{\frac 1k\left( \frac 1{\omega^{2}}-k^2\right)}\in (0,\e_0)
\end{equation}
and consider the following rescaling of the amplitude and $x$,
\begin{equation}\label{scaling}
u=\e \sqrt{k} \omega v\textrm{ and } y=\e\sqrt{k}\omega x.
\end{equation}
Thus, $u(x,\tau)$
satisfies \eqref{eq:KLGtau} if, and only if, $v(y,\tau)$ satisfies
\begin{equation}
\label{kleingordonv}
\partial_y^2v-\frac{1}{\e^2 k}\partial_{\tau}^2v-\dfrac{1}{\e^2k \omega^2}v+\dfrac{1}{3}v^3+\dfrac{1}{\e^3 k^{\frac 32} \omega^3 }f\left(\e \sqrt {k} \omega v\right)=0,\quad 
\end{equation}
which is a Hamiltonian PDE in the dynamical variable $y$ with the Hamiltonian 
\begin{equation}\label{Hamil}
\mathcal{H}(v,\partial_y v)= \displaystyle\int_{\mathbb{T}}\left(\dfrac{(\partial_y v)^2}{2}+\dfrac{(\partial_{\tau}v)^2}{2\e^2 k}-\dfrac{v^2}{2\e^2 k \omega^2}+\dfrac{v^4}{12}+\dfrac{F(\e \sqrt{k} \omega v)}{\e^4 k^2\omega^4}\right)d\tau.
\end{equation}
Using the projection $\Pi_n$ defined in \eqref{E:norm-1}  and denoting $\cdot=d/dy$, we obtain (see \eqref{eqfourier}),
\begin{equation}\label{vn}
\ddot{v}_n=-\dfrac{(n^2\omega^2-1)}{\e^2 k \omega^2}v_n -\dfrac{1}{\e^3 k^{\frac 32} \omega^3}\Pi_n\left[g(\e \sqrt{k} \omega v)\right], \ n\in \Z .
\end{equation}
By \eqref{epsilon},
\begin{equation*}
\label{lambdan}
\lambda_n= \sqrt {\frac 1k\left |n^2- \frac 1{\omega^{2}}\right|} 
\geq \frac 12,
\quad\text{for each }\,\, |n| \ne k.
\end{equation*}
Using this notation, 
 \eqref{vn} becomes
\begin{equation}
\label{vnshorter}
\left\{\begin{array}{l}
\ddot{v}_n=\dfrac{\lambda_n^2}{\e^2}v_n -\dfrac{1}{\e^3 k^{\frac 32} \omega^3}\Pi_n\left[g(\e \sqrt{k} \omega v)\right],\quad\ |n|< k, \vspace{0.2cm} \\
\ddot{v}_{{\pm k}}= v_{\pm k} -\dfrac{1}{\e^3 k^{\frac 32} \omega^3}\Pi_{\pm 
k}\left[g(\e \sqrt{k} \omega v)\right],\vspace{0.2cm}\\
\ddot{v}_n=-\dfrac{\lambda_n^2}{\e^2}v_n -\dfrac{1}{\e^3 k^{\frac 32} \omega^3}\Pi_n\left[g(\e\sqrt{k} \omega v)\right],\quad \ |n|> k.
\end{array}\right.
\end{equation}
Notice  $v_{-n} = - \overline{v_n}$ and $(\e k^{\frac 12} \omega)^{-3} g(\e \sqrt{k} \omega v)=\er(|v|^3)$ is smooth with bounds uniform in $\e \sqrt{k}\omega$. The   stable and unstable invariant manifolds $W^s(0)$ and $W^u(0)$ of $2k+1$ real dimensions correspond to solutions $v^s$ and $v^u$ of \eqref{vnshorter} satisfying the asymptotic conditions\footnote{The Hamiltonian restricted to the center manifold is positive definitive locally around $u=0$ for all $\omega>0$ and, therefore, all orbits backward/forward asymptotic to $u=0$ must belong to the unstable/stable manifold  (see  Corollary \ref{C:centerM} below).}
\begin{equation}
\label{asympt}
\displaystyle\lim_{y\rightarrow +\infty}v_n^s(y)=\displaystyle\lim_{y\rightarrow +\infty}\dot v_n^s(y)=\displaystyle\lim_{y\rightarrow -\infty}v_n^u(y)=\displaystyle\lim_{y\rightarrow -\infty}\dot v_n^u(y)=0,\ \quad\text{ for all }\,\, n\in \Z.
\end{equation}
The singular perturbation problem \eqref{vnshorter} can be written as
\begin{equation}
\label{vnsingular}
\left\{\begin{array}{l}
\e\dot{v}_n=\lambda_n w_n \\
\e\dot{w}_n=\lambda_n v_n - \lambda_n^{-1} \e^{-1} k^{-\frac 32} \omega^{-3}\Pi_n\left[g(\e \sqrt{k} v)\right],\quad |n|< k,
 \vspace{0.2cm}\\
\e\dot{w}_n=-\lambda_nv_n - \lambda_n^{-1} \e^{-1} k^{-\frac 32} \omega^{-3}\Pi_n\left[g(\e \sqrt{k} \omega v)\right],\ |n| >k,  \vspace{0.2cm}\\
\ddot{v}_{\pm k}= v_{\pm k} -(\e \sqrt{k}\omega)^{-3}\Pi_{\pm k}\left[g(\e \sqrt{k} \omega v)\right].
\end{array}\right.
\end{equation}
The formal singular limit of this system as $\e \to 0$ defines a critical manifold 
\[
\mathcal{M}=\{(v, w) \mid v_n=w_n=0\,\ \text{ for all }\ n\neq \pm k\} 
\]
of real dimension 4 due to $v_{-n} = - \overline{v_n}$.
The limiting dynamics on $\mathcal{M}$ is given by the Duffing equation
\begin{equation}
\label{singularlimit}
\ddot{v}_k=v_k -\dfrac{1}{3}\Pi_k \big[ \big( \mathrm{Im} (v_ke^{i k \tau})\big)^3 \big]=v_k-\dfrac 14 |v_k|^2 v_k,
\end{equation}
which is integrable with the phase symmetry. It is known that in  \eqref{singularlimit} the 2-dimensional stable and unstable manifolds of $0$ coincide. In particular, it has a unique  real homoclinic orbit to $0$ satisfying $v_k>0$ and $\dot v_k (0)=0$, which is given by (see Figure \ref{fig:duffing})
\begin{equation}
\label{homoclinic}
v_k=v^h(y)=\dfrac{2\sqrt{2}}{\cosh(y)}.
\end{equation}
\begin{figure}[!]	
	\centering
	\begin{overpic}[width=6cm]{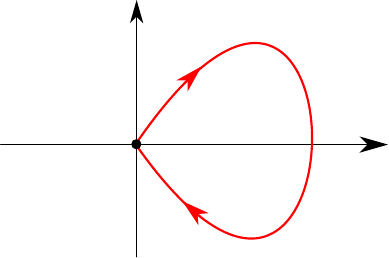}	
		\put(103,30){{\footnotesize$v_1$}}	
\put(39,62){{\footnotesize$\dot{v}_1$}}		
	\end{overpic}
	\caption{Real positive homoclinic \eqref{homoclinic} to $0$ of the 
Duffing equation \eqref{singularlimit} in the critical 
manifold $\mathcal{M}$.}	
	\label{fig:duffing}
\end{figure} 

For $0< \e \ll1$, the special solutions $u_{\mathrm{wk}}^\star$, $\star=u, s$, given in Theorem \ref{T:main}(2) and principally supported in the $k$-modes, are the ones on the $(2k+1)$-dimensional invariant manifolds $W^\star (0)$ with the weakest decay, which are natural deformations of $v^h (y) \sin k \tau$ with proper rescaling. 
In Section \ref{sec:OtherBifs} we prove that $u_{\mathrm{wk}}^\star$ is the only possible intersection of $W^\star (0)$, $\star=u, s$, in an $\er (k^{-\frac 12})$ neighborhood of $0$ (much greater than the amplitude $\er(\e k^{-\frac 12})$ of $u_{\mathrm{wk}}^\star$). 

Most of the analysis in the paper is devoted to identifying the  exponentially small leading order term of the splitting $(u_{\mathrm{wk}}^u - u_{\mathrm{wk}}^s)|_{x=0}$, where $x=0$ corresponds to the first opportunity when they get close, and deriving  the leading order coefficient $C_{\mathrm{in}}$ (see Theorem \ref{T:main}(2b)).  \\

\noindent {\bf Structure of the paper.} The core of the proof of most of the results in Theorem \ref{T:main}(2ab)  is for the case $k=1$, $\omega \in I_1 (\e_0)$, and under the oddness assumption in $t$. The main results for this particular setting are stated in  Section \ref{mainthm}: Theorem \ref{maintheorem} deals with the break up of single bump breathers and Proposition \ref{prop:FirstBif:Generalized} deals with the existence of generalized breathers. The proof of Theorem  \ref{maintheorem} is given in Section \ref{desc_sec} (Section \ref{outerdomain} - \ref{mainthmsec} contain the proofs of some of the statements in Section \ref{desc_sec}). Then,  Proposition \ref{prop:FirstBif:Generalized} is proven in Section \ref{sec:ExpSmallTails}. 
Section \ref{S:SHyperbolic} is devoted to prove the nonexistence of breathers for  frequencies which are far from resonant, that is item (1) of Theorem \ref{T:main}. Section \ref{sec:OtherBifs} explains the reduction of the general case of close to resonant frequencies to that considered in Section \ref{mainthm}: oddness in $t$ assumption and  $\omega \in I_1 (\e_0)$. This completes the prove of item (2) of Theorem \ref{T:main}.  Finally, in Section  \ref{sec:Stokes}, we  prove  Theorem \ref{prop:Stokesconstant}.

\section{Analysis of the first bifurcation ($k=1$) with oddness assumption in 
$t$} 
\label{mainthm}

We devote this section to analyze the stable and unstable manifolds of $v=0$ 
and their splitting for equation  \eqref{kleingordonv} with $k=1$ and $\omega\in 
I_1(\eps_0)$ (see \eqref{E:intervals}). We also analyze the center-stable and center-unstable manifolds to construct the generalized breathers provided by Proposition \ref{prop:generalized}. 

To make the function space setting precise, recall the norm $\|\cdot \|_{\n}$ defined in \eqref{E:norm-1} which is simply the $\n$ norm of the Fourier coefficients in $\tau$. Since 
\[
\| f_1 f_2 \|_{\n} \leq \|f_1\|_{\n} \|f_2 \|_{\n},
\]
treating $y$ as the evolution variable, the local-in-$y$ well-posedness of \eqref{kleingordonv} with $(v, \pa_y v) (y, \cdot) \in \BFX$, where 
\begin{equation} \label{E:phase-space}
\BFX := \{(v, w) \mid v, w \; \text{ are $2\pi$-periodic in $\tau$ and } 
\|(v,w) \|_\BFX :=\|v\|_{\ell_1} + \| (1+ |\pa_\tau|)^{-1}w\|_{\ell_1} 
< \infty\},
\end{equation} 
follows from a standard procedure. Here the operator $|\pa_\tau|$ is simply the multiplication of $|n|$ to the $n$-th Fourier modes for each $n$. For some results where the conservation of energy is used, we also consider the energy space $H_\tau^1 \times L_\tau^2$ which is a dense subspace of $\BFX$ where \eqref{kleingordonv} is also well-posed.  

Due to the oddness assumptions on $f$, the subspace 
\begin{equation}\label{def:Xo}
\BFX_o =  \{(v, w) \in \BFX \mid v,w \text{ are odd in } \tau \} 
\subset \BFX
\end{equation}
of $2\pi$-periodic odd functions of $\tau$ is invariant under the flow of \eqref{kleingordonv}, so we first restrict the analysis to this subspace. 
For such odd functions (of real values) of $\tau$, the Fourier series \eqref{E:Fourier-1} turns out to be  
\begin{equation}\label{usin}
v(t) = \sum_{n=-\infty}^{+\infty} \left(- \frac i2\right) \mathrm{sgn}(n) v_{|n|} e^{in\tau}=\displaystyle\sum_{n\geq 1} v_n\sin(n\tau), \quad \tau = \omega t, \;\; \Pi_n [v]=v_n \in \R, \; n\in\N.
\end{equation}
With a slight abuse of notation, sometimes we may also use $\Pi_n[v]$ to denote 
the n-th mode $v_n \sin n \tau$. Later in Section \ref{sec:OtherBifs}, we 
extend the analysis to the general  setting. 

As explained in Section \ref{SS:spa-dyn}, we refer to the analysis in the setting of $k=1$ and $\omega\in I_1(\eps_0)$ as the \emph{first bifurcation}. Indeed, for $\omega\in I_1(\eps_0)$, the linearization around $v=0$ possesses (in the odd-in-$t$ functions space $\BFX_o$) a pair of weak hyperbolic eigenvalues and all the other eigenvalues are elliptic. In particular the stable  and unstable manifolds, $W^s(0)$ and $W^u(0)$, are one dimensional.

Next theorem gives an asymptotic formula for the splitting between $W^u(0)$ and $W^s(0)$ in the  cross section
\begin{equation}
\label{section}
\s=\{ (v,\partial_y v )\in \BFX_o:
\Pi_1\left[\partial_yv\right]=0 \},
\end{equation}
(see Figure 
\ref{fig:duffingperturb}). 

\begin{theorem}\label{maintheorem}
Fix $r>0$. Consider the equation \eqref{kleingordonv}  with $f\in\FF_r$
 (equivalently \eqref{vnshorter} or \eqref{vnsingular})
 for $k=1$ and $\omega\in I_1(\eps_0)$ defined as in \eqref{E:k-omega}. 
Then, there exist a constant $C_{\mathrm{\mathrm{in}}}\in\mathbb{C}$ such that for any fixed $y_0>0$ there exists $\e_0, M>0$  such that, for every $0<\e\leq\e_0$, the following statements hold. 
	\begin{enumerate}
		\item The invariant manifolds $W^u(0)$ and $W^s(0)$ 
of \eqref{kleingordonv}  
 in $\BFX_o$
correspond to unique solutions  $v^u(y,\tau)$ and $v^s(y,\tau)$ of \eqref{vnshorter} satisfying \eqref{asympt}, which are real-analytic in $y$, $2\pi$-periodic in $\tau$, and satisfy  $\Pi_1\left[\partial_y v^{u,s}\right](0)=0$, respectively. Moreover, $\Pi_{2l}\left[v^{u, s}\right]\equiv 0$, for every $l\in\N$ and
		\[
		\begin{split}
		\big\| \pa_\tau^2 \big(v^u(y,\tau)-v^h(y)\sin\tau \big)\big\|_{\ell_1} +\big\| \pa_\tau^2 \pa_y \big(v^u(y,\tau)-v^h(y)\sin\tau \big)\big\|_{\ell_1} \le M \e^2v^h(y)\quad &\text{ for }y\leq y_0, \\
		\big\| \pa_\tau^2 \big(v^s(y,\tau)-v^h(y)\sin\tau \big)\big\|_{\ell_1} +\big\| \pa_\tau^2 \pa_y \big(v^s(y,\tau)-v^h(y)\sin\tau \big)\big\|_{\ell_1} \le M \e^2v^h(y)\quad &\text{ for }y\geq -y_0,
		\end{split}
		\]
where $v^h$ is the homoclinic orbit given in \eqref{homoclinic}.
		
		\item At $y=0$, their difference satisfies   
		\begin{equation}\label{distancia}
		\left\| \big( |-\pa_\tau^2 -\omega^{-2}|^{\frac 12} (v^u - v^s) + i\e \pa_y (v^u - v^s) \big) (0, \tau)  - \frac {4\sqrt{2}}\e e^{-\frac{\pi\sqrt{2}}{\e}}C_{\mathrm{in}}\sin (3\tau)\right\|_{\ell_1} \le \frac {M e^{-\frac{\pi\sqrt{2}}{\e}}}{\e\log(\e^{-1})}.  
		\end{equation}		
	\end{enumerate}	
\end{theorem}

We highlight that Theorem \ref{maintheorem} is concerned with the distance between the stable and unstable invariant manifolds at the first crossing with  the transversal section $\s$. This does not exclude intersections at further crossings and thus existence of multi-bump breathers. 
See Figures \ref{fig:bumpshomo} and \ref{fig:duffingperturb}. 

Theorem \ref{maintheorem} proves statements in Theorem \ref{T:main}(2ab) 
for $k=1$ and $\omega\in I_1(\eps_0)$ (restricted to the  odd in $t$ setting) which deal with the one-dimensional stable and unstable manifolds . 

The next proposition analyzes the intersection between the center-stable and center-unstable invariant manifolds of $v=0$.

\begin{prop}\label{prop:FirstBif:Generalized}
Fix $r>0$. Consider the equation \eqref{kleingordonv}  with $f\in\FF_r$
  for $k=1$ and $\omega\in I_1(\eps_0)$ defined as in \eqref{E:k-omega}. 
For any fixed $y_0>0$, there exists $\e_0, M>0$  such that, for every $0<\e\leq\e_0$, the following statements hold. 
		
		Let $W \subset \Sigma$
		be the intersection near $(v^h (0)\sin \tau, 0)$ of the center-stable manifold $W^{cs} (0)$ and center-unstable manifold $W^{cu} (0)$ of \eqref{kleingordonv}   in $\BFX_o$
 when they intersect the hyperplane $\Sigma$ for the first time in $y$. Then,
		\begin{enumerate} 
		\item Let 
		\[\begin{split}
		\mathcal N=  \{ (v, \pa_y v) \mid \ & \e^{-1} \big\| |-\pa_\tau^2 -\omega^{-2}|^{\frac 12} \big(v - v^\star(0)\big) \big\|_{L^2} +  \|\pa_y v - \pa_y v^\star(0)\|_{L^2} \\
		&\leq M \big(\e^{-1} \big\| |-\pa_\tau^2 -\omega^{-2}|^{\frac 12} \big(v^u(0) - v^s(0)\big) \big\|_{L^2} +  \|\pa_y v^u(0) - \pa_y v^s (0)\|_{L^2}\big), \; \star=u,s\} 
		\end{split}.\]
		Then $W \cap \mathcal N \ne \emptyset$ and the Hamiltonian $\HH$ evaluated at  solutions in $W \cap \mathcal N$ satisfies 
		\begin{align*}
		& \tfrac 1M \inf_{W  \cap \mathcal N}\HH \le \e^{-2} \big\| |-\pa_\tau^2 -\omega^{-2}|^{\frac 12} \big(v^u(0) - v^s(0)\big) \big\|_{L^2}^2 +  \|\pa_y v^u(0) - \pa_y v^s(0)\|_{L^2}^2 \le M\inf_{W  \cap \mathcal N} \HH.
		\end{align*} 
		
		\item Each $(v, \pa_y v) \in W$ corresponds to a single bump homoclinic orbit $\big(v(y, \tau), \pa_y v(y, \tau)\big)$ to $W^c(0)$, i.e $(v, \pa_y v)$ is  asymptotic to two orbits $\big(v_c^\pm (y), 
\pa_y v_c^\pm (y)\big)$ in the center manifold $W^c(0)$ as $y\to\pm\infty$. Moreover,  $(v, \pa_y v)$ satisfies
		\begin{equation} \label{E:gB-UB}
		\tfrac 1M \HH (v, \pa_y v) \le \e^{-2} \big \| |-\pa_\tau^2 -\omega^{-2}|^{\frac 12} \big(v(y) - v^\star(y)\big)\big\|_{L^2}^2 +  \|\pa_y v(y) - \pa_y v^\star(y)\|_{L^2}^2 \leq M \HH (v, \pa_y v),
		\end{equation}
		for $y\ge -y_0$ with $\star=s$ and $y\le y_0$ with $\star=u$.  
		
		\item If $v^u = v^s$, where $v^u, v^s$ are the solutions obtained in Theorem \ref{maintheorem}, then it is a homoclinic orbit to $0$, otherwise the intersection $W^{cs} (0) \cap W^{cu} (0)$, which is codimension 2 in $\BFX_o$, is transverse in $\mathcal N$.  
		\end{enumerate}	
\end{prop}

\begin{remark} \label{R:W}
In the case $v^u \ne v^s$, the transversality of the intersection of the codimension-1 $W^{cs} (0)$ and $W^{cu} (0)$ actually implies that a dense subset of $W$ consists of functions smooth in $\tau$. See Remark \ref{R:W-1} below.
\end{remark}

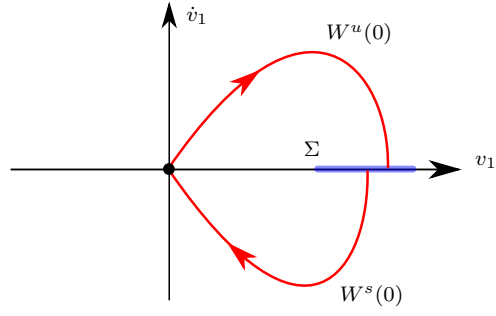
\begin{figure}[!]	
	\centering
	\begin{overpic}[width=6cm]{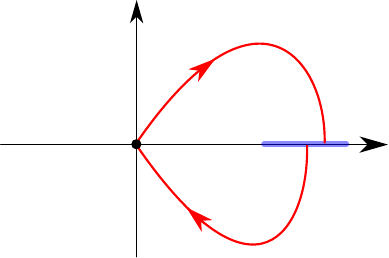}			
		
		\put(103,30){{\footnotesize$v_1$}}	
\put(65,32){{\footnotesize$\Sigma$}}	
\put(70,58){{\footnotesize$W^u(0)$}}	
\put(73,0){{\footnotesize $W^s(0)$}}
\put(39,62){{\footnotesize$\dot{v}_1$}}			
	\end{overpic}
	\caption{The (infinite dimensional) transverse section $\Sigma$ (see 
\eqref{section}) where we measure the distance between the perturbed manifolds 
$W^u(0)$ and $W^s(0)$.}	
	\label{fig:duffingperturb}
\end{figure}

This proposition  implies Proposition \ref{prop:generalized} for $k=1$ and $\omega\in I_1(\eps_0)$, which deal 
with generalized breathers with exponentially small tails. Indeed, Proposition \ref{prop:FirstBif:Generalized} implies the existence of a family of orbits homoclinic to 
the center manifold with exponentially small energy. They correspond to breather 
like solutions $u(x,t)$ of \eqref{kleingordonrev} which are $\frac 
{2\pi}\omega$-periodic in $t$ and decaying in $x$ like $\mathcal{O} ( \e e^{-\e 
|x|})$ subject to perturbations whose $L_x^\infty (H_t^1 \times L_t^2)$ norm is 
bounded by $\mathcal{O} (\tfrac{1}{\e} e^{-\frac 
{\sqrt{2}\pi}\e})$.

Theorem \ref{maintheorem} and Proposition \ref{prop:FirstBif:Generalized} are proven in Sections \ref{desc_sec} and \ref{sec:ExpSmallTails} respectively. We devote the rest of this section to give some heuristics on the proof of Theorem  \ref{maintheorem}, in particular, on why the distance between the invariant manifolds is exponentially small and on how to obtain its asymptotic formula.

\subsection{Heuristics of the proof of Theorem \ref{maintheorem}; exponentially small bounds}\label{sec:heuristics}

Looking at formula \eqref{distancia} one can see that the distance between the 
one dimensional invariant manifolds $W^u(0)$ and $W^s(0)$ of \eqref{kleingordonv} is 
exponentially small in $\e$.
In this section we give some intuition why 
and we explain which are the steps needed to obtain upper bounds on this distance. Later, in Section \ref{sec:strategyA}, we  show how to obtain the asymptotic formula \eqref{distancia} for it.

Since the invariant manifolds are one-dimensional, one can parameterize them as solutions of the second order equation \eqref{vnshorter} for $k=1$, which  satisfies 
\[
\begin{split}
\ddot{v}_1=&\,v_1 -\frac{v_1^3}{4}+\mathcal{O}_{\ell_1}(\wt \Pi v)+\mathcal{O}(\e^2)\\
\e^2 \ddot{v}_n=&\,-\lambda_n^2 v_n +\mathcal{O}_{\ell_1}(\e^2 v^3),\quad n\geq 2,
\end{split}
\]
for small $\wt \Pi v$, where we have introduced the following notation, which is also used in the forthcoming sections 
\begin{equation}\label{def:PiTilde}
 \widetilde\Pi(v)=v-\Pi_1(v)\sin\tau=\displaystyle\sum_{n\geq 2} v_n\sin(n\tau).
\end{equation}
Imposing decay at infinity (as $y\to+\infty$ for  $W^s(0)$, as $y\to-\infty$ for $W^u(0)$) and $\partial_y v_1^{u,s}(0)=0$, item (1) of Theorem \ref{maintheorem} looks natural:  the  distance between the perturbed and unperturbed manifolds $(v_1,\wt \Pi v)=(v^h,0)$ is of the same order as the perturbation (notice the singular character of the model and the different size of each component of the vector field). These estimates can be proven through a fixed point argument by using the standard Perron method.

Even if the perturbed invariant manifolds are $\OO(\e^2)$ close to the unperturbed ones, the singular character of the model makes their difference beyond all orders in $\e$, in fact exponentially small.
Let us give some heuristic ideas of why this phenomenon happens. 
We have chosen parameterizations such that $\partial_y v_1^{u,s}(0)=0$. 
Moreover, as the system conserves the Hamiltonian, both manifolds belong to the energy level of the saddle-center critical point $v=0$. 
Therefore, the difference $v_1^u-v_1^s$ at $y=0$ can be recovered from the differences projected to the rest of directions, namely $\wt \Pi (v^u-v^s)$ and $\wt \Pi (\partial_y v^u- \partial_y v^s)$. Thus, we focus on measuring these differences. 
Let us write the equations for these components as a first order equation for $n\geq 3$ (recall that $\Pi_{2l}v^{u,s}=0$ for $l\geq 0$),
\[
 \begin{split}
  \dot v_n=&\, w_n\\
  \dot w_n=&\,-\dfrac{\lambda_n^2}{\e^2}v_n +\frac{1}{\e^3 \omega^3}\Pi_n \left[g(\e \omega v)\right].
 \end{split}
\]
As the parameterizations of both invariant manifolds satisfy the same equation, their difference 
\[
(\Delta_n,\Xi_n)\triangleq (v_n^u-v_n^s,\partial_y v_n^u-\partial_y  v_n^s)
\]
 satisfies a linear equation for $n\geq 3$,
\[
 \begin{split}
  \dot\Delta_n=&\, \Xi_n\\
  \dot \Xi_n=&\,-\dfrac{\lambda_n^2}{\e^2}\Delta_n +\Pi_n\left[M(v^u,v^s)\Delta\right].
 \end{split}
\]
Since the last term is much smaller than the oscillating one, to give a heuristic idea of the phenomenon taking place, let us assume that $M=0$. 
Then, the system becomes a linear system of constant coefficients which we can diagonalize by taking
\begin{equation}\label{def:GaNThetN}
 \begin{split}
 \Gamma_n=&\,\lambda_n\Delta_n+i\eps\Xi_n\\ 
 \Theta_n=&\,\lambda_n\Delta_n-i\eps\Xi_n
 \end{split}
\end{equation}
to obtain
\[
 \begin{split}
 \dot\Gamma_n=&\,-i\dfrac{\lambda_n}{\e}\Gamma_n \\ 
 \dot\Theta_n=&\,i\dfrac{\lambda_n}{\e}\Theta_n. 
 \end{split}
\]
The solutions of this system can be easily computed as 
\[
\begin{split}
\Gamma_n(y)&=e^{-i\frac{\lambda_n}{\e}(y-y^+)}\Gamma_n(y^+)\\
\Theta_n(y)&=e^{i\frac{\lambda_n}{\e}(y-y^-)}\Theta_n(y^-)
\end{split}
 \]
for any points $y^\pm$. 

By the definition of $(\Gamma_n,\Theta_n)$ in \eqref{def:GaNThetN}, one has
\[
\begin{split}
\Gamma_n(y^+)=&\,\lambda_n(v_n^u(y^+)-v_n^s(y^+))+i\eps (\partial_yv_n^u(y^+)-\partial_y v_n^s(y^+))\\
\Theta_n(y^-)=&\,\lambda_n(v_n^u(y^-)-v_n^s(y^-))-i\eps (\partial_yv_n^u(y^-)-\partial_y v_n^s(y^-)).
\end{split}
\]
The main observation here is that, if  we are able to extend both the stable and 
unstable manifolds $v_n^{u,s}(y)$ to some complex values  $y^\pm=\pm i \sigma$, 
$\sigma>0$,  one obtains the following estimates for $y\in\mathbb{R}$ near 
$y=0$,
\[
\begin{split}
|\Gamma_n(y )|\leq &\, e^{-\frac{\lambda_n\sigma}{\e}}|\Gamma_n(i\sigma)|\\
|\Theta_n(y)|\leq &\, e^{-\frac{\lambda_n\sigma}{\e}}|\Theta_n(-i\sigma)|,
\end{split}
 \]
which are exponentially small in $\eps$ and strongly depend on the size of  the unstable/stable solutions at the complex points $y^\pm=\pm i \sigma$.
 
For the nonlinear system, we will find the solutions 
\[
v_n^{s}(y) \ \mbox{for}\ \Rp y\ge 0, \qquad  v_n^{u}(y) \ \mbox{for} \ \Rp y\le 0
\]
 as perturbations of the singular limit solution $v_1=v^h(y)$, $v_n=0$, $n\geq 2$, where $v^h(y)$ is the unperturbed homoclinic solution \eqref{homoclinic}.
As this function  has poles of order one at the points $y^\pm=\pm i \pi/2$, it is natural to expect that the optimal value for $\sigma$ is in a neighborhood on the lower side of $\pi/2$. In Theorem \ref{outerthm} to be proved  in Section \ref{outerdomain} we show that, for $y$ close to $\pm i \pi/2$, 
\begin{equation}\label{eq:outermanifolds}
v^{s,u}(y,\tau)=v^h(y)\sin \tau + \OO_{\ell_1}\left(\frac{\eps ^2}{|y^2+\frac{\pi ^2}{4}|^3}\right)=\frac{2\sqrt2}{\cosh y}\sin \tau + \OO_{\ell_1}\left(\frac{\eps ^2}{|y^2+\frac{\pi ^2}{4}|^3}\right).
\end{equation}
Therefore, we see that 
\[
|v^{u,s}(y^\pm)|\lesssim \frac{1}{\eps},  \; \text{ at }\; y_\pm =\pm i\sigma, \ \text{ with } \sigma = \frac \pi2 - \kappa\e\ \text{and some }\ \kappa>0.
\]
Consequently, $|\Gamma_n(y^+)|\lesssim \frac{1}{\eps}$, $|\Theta_n(y^-)|\lesssim \frac{1}{\eps}$. 
Therefore, one expects that for $y\in\mathbb{R}$ close to $y=0$,
\[
|\Gamma_n(y)|\lesssim  \frac{1}{\e}e^{-\frac{\lambda_n\pi}{2\e}},\qquad
|\Theta_n(y)|\lesssim \frac{1}{\e} e^{-\frac{\lambda_n\pi}{2\e}}.
\]
As $\lambda_n \ge \lambda_3=2\sqrt{2}+\mathcal{O}(\e)$ for $n\ge 3$, one obtains an upper bound for the difference 
\[
|\Gamma_n (y)|\lesssim  \frac{1}{\e}e^{-\frac{\sqrt{2}\pi}{\e}},\qquad
|\Theta_n (y)|\lesssim \frac{1}{\e} e^{-\frac{\sqrt{2}\pi}{\e}}
\]
and similar bounds are satisfied by $\Delta_n(y)=v_n^u(y)-v_n^s(y)$.

Observe that these bounds fit with the estimates \eqref{distancia} given in Theorem \ref{maintheorem}. 
However, Theorem \ref{maintheorem} gives certainly more information since it provides an 
asymptotic formula for $\Delta$.

\subsection{Strategy of the proof of Theorem \ref{maintheorem}; exponentially small asymptotics}\label{sec:strategyA}

As we have seen in Section \ref{sec:heuristics}, to obtain an asymptotic formula of $v^s-v^u$, one needs a deeper study of these functions near $y^\pm=\pm i(\pi/2-\kappa \e)$ for some $\kappa>0$. 

According to \eqref{eq:outermanifolds}, for real values of $y$, the invariant manifolds are $\eps 
^2$-perturbations of the unperturbed homoclinic orbit, but, when 
$y\mp\frac{i\pi}{2}=\OO(\eps)$ we have that both the homoclinic and error term 
become of the same size $\OO(\frac{1}{\eps})$ and therefore $v^{s,u}(y,\tau)$ 
are not well approximated by the homoclinic solution $v^h(y) \sin\tau$ anymore. 
Thus, we look for suitable leading orders of $v^{s,u}(y,\tau)$ for $y$ such that 
$y\mp\frac{i\pi}{2}=\OO(\eps)$.

We focus on the singularity $y=i \pi/2$ (the same analysis can be performed near 
 the singularity $y=-i \pi/2$ analogously). We proceed as follows. We perform 
the singular  change  to the \emph{inner variable}
\[z=\e^{-1}\left(y-i\dfrac{\pi}{2}\right)\]
and the scaling
\[\phi(z,\tau)=\e v\left(i\dfrac{\pi}{2}+\e z, \tau\right).\]
From  \eqref{kleingordonv}, one can deduce the equation satisfied by 
$\phi(z,\tau)$,
\[
\partial_z^2\phi-\partial_{\tau}^2\phi-\frac{1}{\omega^2} 
\phi+\dfrac{1}{3}\phi^3+\frac{1}{\omega^3}f(\omega \phi)=0, \qquad  
\omega=\frac{1}{\sqrt{1+\e^2}}. 
\]
The first order of  this equation corresponds to the regular limit  $\e=0$, 
which gives the so-called \emph{inner equation}
\[
\partial_z^2\phi^0-\partial_{\tau}^2\phi^0-\phi^0+\dfrac{1}{3}
(\phi^0)^3+f(\phi^0)=0.
\]
The estimates \eqref{eq:outermanifolds} show, that, after these changes, the 
stable/unstable manifolds behave as
\[
\phi^{s,u}(z,\tau)=-\frac{2\sqrt2 i}{z}\sin \tau+ 
\OO_{\ell_1}\left(\frac{1}{z^3}\right).
\]
Therefore, it is natural to look for solutions of the inner equation which 
``match'' these asymptotics.
This is  done in Theorem \ref{innerthm} below where we obtain and analyze two 
solutions $\phi^{0,u}$, $\phi^{0,s}$ of the inner equation which are the first 
order of the unstable/stable manifolds ``close to the singularity'' $y=i \pi/2$. 
They are of the form 
\begin{equation}\label{innersol0}
\begin{split}
\phi^{0,s}(z,\tau)&=-\dfrac{2\sqrt{2} i}{z}\sin\tau+\psi^{s}(z,\tau), \ 
\mbox{for} \ \Rp z>0\\
\phi^{0,u}(z,\tau)&=-\dfrac{2\sqrt{2} i}{z}\sin\tau+\psi^{u}(z,\tau), \ 
\mbox{for} \ \Rp z <0, \end{split}
\end{equation}
with $\psi^{s,u}=\OO(\frac{1}{z^3})$ in suitable complex domains satisfying $|z| 
\ge \kappa$ and containing the negative imaginary axis $\Im z\le -\kappa$ (recall that 
$z=\eps^{-1}(y-\frac{i\pi}{2})$ and therefore $y=0$ lies on this negative 
imaginary axis of $z$). Again these solutions contain only odd modes in $\tau$.

Moreover, in Theorem  \ref{innerthm} we  provide a formula for the difference of 
these two solutions which reads
\begin{equation}\label{def:diff:heur}
\phi^{0,u}(-ir,\tau)-\phi^{0,s}(-i r,\tau) = e^{-2\sqrt{2} 
r}\left(C_{\mathrm{in}}\sin 
(3\tau)+\OO_{\ell_1}\left(\frac{1}{r}\right)\right)\qquad \text{as}\qquad  r\to 
+\infty.
\end{equation}
This asymptotic formula can be obtained relying on different techniques. In the literature,  one can find proofs relying either on fixed point arguments for the difference \cite{INMA, BaldomaS08} or on Borel resummation techniques and \'Ecalle Resurgence Theory \cite{OSS03} (applied to finite dimensional inner equations). In the present paper, we obtain this formula through a different method relying on invariant manifolds and foliations for a an ill-posed associated PDE, which is more in line with the techniques used throughout the paper. Let us give here a (very) heuristic idea of the origin of this result from this new point of view.

Writing a solution of the inner equation as  $\phi^0=\sum \phi^0_n  \sin(n\tau)$ 
we obtain
\begin{equation}\label{innersystemh}
\begin{split}
\frac{d^2 }{d z^2}\phi^0_1-\frac{1}{4}(\phi^0_1)^3&= F_1\left (\phi^0\right)\\
\frac{d^2 }{d z^2}\phi^0_n+\mu_{n}^2\phi^0_n&= F_n \left(\phi^0\right), \; n\ge 
3,
\end{split}
\end{equation} where $'=d/dz$, $\mu_{n}=\sqrt{n^2-1}$, and 
$F_n$ contain higher order terms.

Let us assume that $F_n=0$ to give an heuristic idea of the process.
First, let us make the change $z=-ir$ and write system \eqref{innersystemh} as a 
first order system through the change
\[
\Psi_c (r)=\left(\phi^0_1(-ir), -i\frac{d}{d z}\phi^0_1(-ir)\right), \quad  
\Psi_{n,\pm}=-i\frac{d}{d z}\phi^0_n (-ir) \pm \sqrt{n^2-1} \phi^0_n(-ir), \quad 
\Psi^\pm = (\Psi_{2l-1, \pm})_{l=2}^\infty,
\]
which gives
\begin{equation*}
\begin{aligned}
\frac{d}{dr} \Psi_c^1&= \Psi_c^2\qquad &\frac{d}{dr} \Psi_{n, 
-}=&-\sqrt{n^2-1}\Psi_{n,-}\\
\frac{d}{dr} \Psi_c^2&= \frac{1}{4}\left(\Psi_c^1\right)^3\qquad
&\frac{d}{dr} \Psi_{n, +}=&\sqrt{n^2-1}\Psi_{n,+}.
\end{aligned}
\end{equation*} 
Observe that $\Psi=0$ is a critical point with a center manifold $W^c$  given by 
$\Psi^+=\Psi^-=0$, and a center-stable manifold $W^{cs}$ given by $\Psi^+=0$. 
Moreover, $W^{cs}$ possesses the classical stable foliation. Indeed,  given a 
point 
$\Psi=(\Psi^c,\Psi^-,0)\in W^{cs}$ there exists a point $\Psi_b=(\Psi^c,0,0)\in 
W^c$ such that
$|\Phi_r(\Psi)-\Phi_r(\Psi_b)|\le \OO(e^{-2r})$, as $2\in(0, \sqrt{3^2-1})$, 
where $\Phi_r$ is the flow on $W^{cs}$ which is well defined for $r\geq0$. The 
points whose trajectories are asymptotic to a given $\Psi_b\in W^c$ form a leaf 
of the foliation.

This foliation allows us to give an asymptotic formula for 
$\phi^{0,s}(z)-\phi^{0,u}(z)$:
\begin{itemize}
\item
The first observation is that our solutions $\phi^{0,s}(z)$, $\phi^{0,u}(z)$, 
when  restricted to the negative imaginary axis away from $0$ and 
written in these coordinates, correspond to 
$\Psi^{u,s}(r)=(\Psi^{u,s}_c,\Psi^{u,s}_-,\Psi^{u,s}_+)(r)$ satisfying
\[
\lim _{r \to+ \infty} \Psi^{s,u}(r)=0.
\]  
Therefore, they belong to $W^{cs}$ and, in this simplified model, should have the ``unstable coordinate'' $\Psi^{u,s}_+(r)\equiv 0$.
\item
The second observation is that we know, by  \eqref{innersol0}, that
\[
| \Psi^{u}(r)- \Psi^{s}(r)| \le \OO_{\ell_1}\left(\frac{1}{r^3}\right), \ \mbox{as} \ r \to+ \infty,
\]  
which implies that they should have the same ``central coordinate''  ($\Psi^{u}_c 
(r)= \Psi^{s}_c (r)$ in this simplified model) and therefore they belong to the same leaf in the stable 
foliation. One can see this fact using the  linearized 
fundamental solutions  in the central coordinates which  give:
$\Psi^{u}_c (r)- \Psi^{s}_c (r)\sim c_1 r^{-2}+c_2 r^3$ and the decay of this difference immediately gives $c_1=c_2=0$.
\item
Now  that we know that $\Psi^{u,s}(r)=(\Psi_c(r),\Psi^{u,s}_-(r),0)$, we only 
need to compute the difference in the stable coordinate $\beta_- (r) =\Psi^{u}_- 
(r) - \Psi^{s}_-(r) $ which satisfies:  
\[
\frac{d}{dr}\beta_- = A \beta_-, \quad A=\mathrm{diag}(-\sqrt{n^2-1})
\]
and this immediately implies that
\[
\beta _-(r)= e^{(r-r_0)A }\beta_- (r_0)=e^{-2\sqrt 2 (r-r_0)}\beta _{3,-}(r_0)\sin(3\tau)+\OO_{\ell_1}\left(e^{-3r}\right).
\]
Calling $C=e^{2\sqrt2 r_0}\beta _{3,-}(r_0)$
we have 
\[
\lim _{r\to +\infty} e^{2\sqrt2 r } \beta_- (r)-C\sin {3\tau}=0. 
\]
\end{itemize}
Using these ideas, in Theorem \ref{innerthm} below, we incorporate the dismissed higher order terms (see \eqref{innersystemh}) and give a complete proof of the asymptotic formula for the difference between the solutions of the inner equation. Note that the constant $C$ above corresponds to the constant $C_\mathrm{in}$ in \eqref{def:diff:heur}.

Once we obtain the difference between the inner solutions $\phi^{0,u}-\phi^{0,s}$, we must show that this difference gives indeed a first order of the difference between the perturbed invariant manifolds. That is, we must estimate the function
 $(\phi^{u}-\phi^{s})-(\phi^{0,u}-\phi^{0,s})$ in some appropriate complex 
domain.
To this end, it starts with showing that the solutions of the inner equation
$\phi^{0,u}(z)$, $\phi^{0,s}(z)$, when written in the original variables $y=i\frac{\pi}{2}+\eps z$, are good approximations of the stable and unstable solutions $v^{u}$, $v^s$ for $y$ satisfying $y\mp\frac{i\pi}{2}=\OO(\eps)$. Such analysis is done in Theorem \ref{matchingthm}.

 From such estimates, applying the ideas in Section \ref{sec:heuristics}, we 
obtain smaller exponentially small errors  at $y=0$.
 This  shows that the difference of $\phi^{0,u}-\phi^{0,s}$ provides the main 
term of the exponentially small distance between $v^u$ and $v^s$ and that the first order of this distance is given by the Stokes constant $C_\mathrm{in}$.

\section{Description of the Proof of Theorem \ref{maintheorem}}\label{desc_sec}

We describe the main steps of the proof of Theorem \ref{maintheorem} where $k=1$ and the odd symmetry of functions in $\tau$ is assumed.

\subsection{Estimates of the invariant manifolds in complex domains} \label{SS:inma}

In order to estimate the distance between the perturbed invariant manifolds $W^s(0)$ and $W^u(0)$ in $\s$, we consider  suitable  parameterizations for them. 
Since the invariant manifolds $W^u(0)$ and $W^s(0)$ are one dimensional, they are the images of solutions $v^u$ and $v^s$ of \eqref{vnshorter} with the asymptotic conditions 
\begin{equation*}
\displaystyle\lim_{y\rightarrow +\infty}v^s(y,\tau)=\displaystyle\lim_{y\rightarrow -\infty}v^u(y,\tau)=0,\ \quad \text{ for all }\,\tau\in\mathbb{T}.
\end{equation*}
We write equation \eqref{vnshorter} as
\begin{equation*}
\left\{\begin{array}{l}
\ddot{v}_1=v_1 -\dfrac{v_1^3}{4}+\left(-\dfrac{1}{\e^3 \omega^3}\Pi_1\left[g(\e \omega v)\right]+\dfrac{v_1^3}{4}\right),\vspace{0.2cm}\\
\ddot{v}_n=-\dfrac{\lambda_n^2}{\e^2}v_n -\dfrac{1}{\e^3 \omega^3}\Pi_n\left[g(\e \omega v)\right],\ n\geq 2,
\end{array}\right.
\end{equation*}
where $\Pi_n$ is the Fourier projection given by \eqref{usin} and $g$ is 
given by \eqref{g}. 

We study the solutions $v^u, v^s$ as perturbations of the homoclinic orbit $v^h(y)\sin\tau$ given by \eqref{homoclinic}, which satisfies $\ddot{v}^h=v^h -(v^h)^3/4$. Thus, we set
\begin{equation*}
\label{coord}
\xi(y,\tau)=v(y,\tau)-v^h(y)\sin\tau=\displaystyle\sum_{n\geq 1}\xi_n(y)\sin(n\tau),
\end{equation*}
whose Fourier coefficients satisfy
\begin{equation*}
\label{xin}
\left\{\begin{array}{l}
\ddot{\xi}_1= \xi_1 -\dfrac{3(v^h)^2\xi_1}{4} -\dfrac{3v^h\xi_1^2}{4}-\dfrac{\xi_1^3}{4}+\left(-\dfrac{1}{\e^3\omega^3}\Pi_1(g(\e \omega ( \xi+v^h\sin\tau)))+\dfrac{(\xi_1+v^h)^3}{4}\right),\vspace{0.2cm}\\
\ddot{\xi}_n=-\dfrac{\lambda_n^2}{\e^2}\xi_n -\dfrac{1}{\e^3 \omega^3}\Pi_n(g(\e \omega( \xi+v^h\sin\tau))),\qquad n\geq 2.
\end{array}\right.
\end{equation*}
Define the operators
\begin{align}
\mathcal{L}(\xi)=& \left(\ddot{\xi}_1-\xi_1+\dfrac{3(v^h)^2\xi_1}{4}\right)\sin\tau+\displaystyle\sum_{n\geq 2}\left(\ddot{\xi}_n+\dfrac{\lambda_n^2}{\e^2}\xi_n \right)\sin(n\tau),\label{Lop}\\
\mathcal{F}(\xi)= &-\dfrac{1}{\e^3 \omega^3}g(\e \omega( \xi+v^h\sin\tau))+\left(\dfrac{(\xi_1+v^h)^3}{4} -\dfrac{3v^h\xi_1^2}{4}-\dfrac{\xi_1^3}{4}\right)\sin\tau.\label{Fop}
\end{align}
To obtain solutions $v^\star$,  $\star=u,s$,   of \eqref{kleingordonv} satisfying \eqref{asympt} is equivalent to find  solutions $\xi^\star$ of the functional equation
\begin{equation}
\label{xinop}
\mathcal{L}(\xi)=\mathcal{F}(\xi),
\end{equation}	
satisfying 
\begin{equation}\label{asymptcond}\displaystyle\lim_{y\rightarrow-\infty}\xi^u(y,\tau)=\displaystyle\lim_{y\rightarrow\infty}\xi^s(y,\tau)= 0,\quad \text{ for all}\quad \tau\in\mathbb{T}.\end{equation}

We analyze these parameterizations in the following complex  sectorial domains, usually called \emph{outer domains},
\begin{equation}\label{outer}
\begin{split}
D^{\mathrm{out},u}_{\kappa}=&\,\left\{ y\in\C; \left|\Ip(y)\right|\leq -\tan\bg\Rp(y)+\dfrac{\pi}{2}-\kappa\e \right\}\\
D^{\mathrm{out},s}_{\kappa}=&\,\left\{ y\in\C; -y\in D^{\mathrm{out},u}_{\kappa} \right\},
\end{split}
\end{equation}
where $0<\bg<\pi/4$ is a fixed angle independent of $\e$ and $\kappa\geq 1$ (see Figure \ref{outerfig}).  Observe that $D^{\mathrm{out},\ast}_{\kappa}$, $\ast=u,s$, reach domains at a $\kappa\e$--distance of the singularities $y=\pm i\pi/2$ of $v^h$ (see Section \ref{sec:heuristics}).

\begin{figure}[!]	
	\centering
	\begin{overpic}[width=12cm]{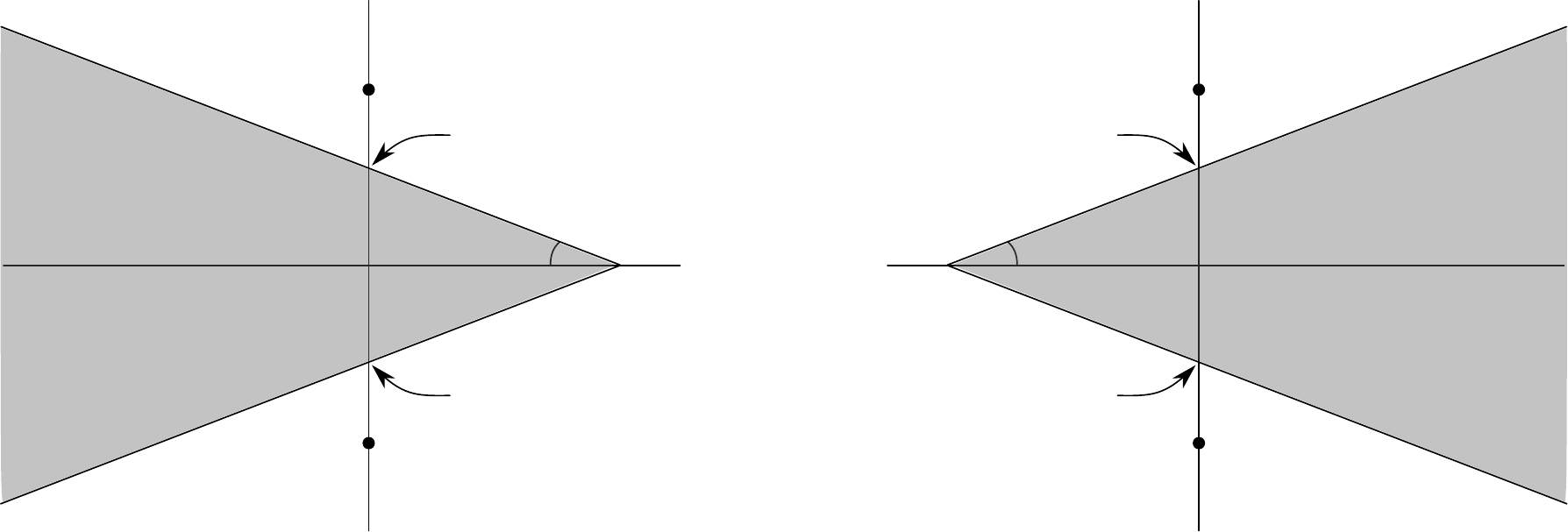}
		\put(25,27.5){{\footnotesize $i\frac{\pi}{2}$}}	
		\put(71,27.5){{\footnotesize $i\frac{\pi}{2}$}}			
		\put(25,5){{\footnotesize $-i\frac{\pi}{2}$}}	
		\put(69,5){{\footnotesize $-i\frac{\pi}{2}$}}	
		\put(30,24){{\footnotesize $i\left(\frac{\pi}{2}-\kappa\e\right)$}}		
		\put(30,8){{\footnotesize $-i\left(\frac{\pi}{2}-\kappa\e\right)$}}			
		\put(57,24){{\footnotesize $i\left(\frac{\pi}{2}-\kappa\e\right)$}}		
		\put(55,8){{\footnotesize $-i\left(\frac{\pi}{2}-\kappa\e\right)$}}	
		\put(33,17){{\footnotesize $\bg$}}	
		\put(65.5,17){{\footnotesize $\bg$}}	
		\put(3,25){{\footnotesize $D^{\mathrm{out},u}_{\kappa}$}}	
		\put(90,25){{\footnotesize $D^{\mathrm{out},s}_{\kappa}$}}						
	\end{overpic}
	\bigskip
	\caption{ Outer domains $D^{\mathrm{out},u}_{\kappa}$ and $D^{\mathrm{out},s}_{\kappa}$. }	
	\label{outerfig}
\end{figure}

The next theorem proves the existence and estimates of the functions $\xi^u,\xi^s$. It is proven in Section \ref{outerdomain}.

\begin{theorem}[Outer]\label{outerthm}
	Consider the equation \eqref{kleingordonv} with $k=1$.
	There exist $\kappa_0\geq 1$ big enough and $\e_0>0$ small enough, such that, for each $0<\e\leq \e_0$ and  $\kappa\geq \kappa_0$, the invariant manifolds $W^{\star}(0) \subset \BFX_o$ of \eqref{kleingordonv},  $\star=u,s$,  are parameterized  
	as unique solutions of equation \eqref{kleingordonv} by
	\[v^\star(y,\tau)=v^h(y)\sin\tau+\xi^\star(y,\tau),\ y\in D^{\mathrm{out},\star}_{\kappa},\ \tau\in\mathbb{T},\] where $v^h$ is given by \eqref{homoclinic} and $\xi^\star: D^{\mathrm{out},\star}_{\kappa}\times\mathbb{T}\rightarrow \C$ are  functions real-analytic in the variable $y$ such that 
	\begin{enumerate}
	\item They satisfy the asymptotic condition \eqref{asymptcond},	$\partial_y\Pi_1[\xi^\star](0)=0$ and  $\Pi_{2l}[\xi^\star](y)\equiv0$  for  $l\in \N$.
	\item There exists a constant $M_1>0$ independent of $\e$ and $\kappa$, such that
	\[
	 \begin{split}
	 \left\|\xi^\star\right\|_{\ell_1}(y)\leq &\, \dfrac{M_1\e^2}{|\cosh(y)|}\qquad \text{ for } \quad y\in D^{\mathrm{out},\star}_{\kappa}\cap\{|\Rp(y)|> 1\}\\
	 \left\|\xi^\star\right\|_{\ell_1}(y)\leq &\,\dfrac{M_1\e^2}{|y^2+\pi^2/4|^3}\qquad \text{ for } \quad y\in D^{\mathrm{out},\star}_{\kappa}\cap\{|\Rp(y)|\leq 1\}.
	 \end{split}
	\]
	\end{enumerate}	
Moreover, the derivatives of $\xi^\star$ can be bounded as 	
\begin{enumerate}
\item For  $y\in D^{\mathrm{out},\star}_{\kappa}\cap\{|\Rp(y)|> 1\}$,
\[
\left\|\partial_{\tau}^2\xi^\star\right\|_{\ell_1}(y), \  \|\pa_\tau^2 \partial_{y}\xi^\star\|(y)\leq \dfrac{M_1\e^2}{|\cosh(y)|}.\] 
\item For  $y\in D^{\mathrm{out},\star}_{\kappa}\cap\{|\Rp(y)|\leq 1\}$,
\[
\left\|\partial_{\tau}^2\xi^\star\right\|_{\ell_1}(y) \leq \dfrac{M_1\e^2}{|y^2+\pi^2/4|^3}\quad\text{and}\quad \left\|\pa_\tau^2 \partial_{y}\xi^\star\right\|_{\ell_1}(y)\leq \dfrac{M_1\e^2}{|y^2+\pi^2/4|^4}.\]
\end{enumerate}
\end{theorem}

\begin{remark} 
While the 1-dim stable and unstable manifolds of the equilibrium $0$ are determined by their exponential asymptotic behavior as $y\to \pm \infty$ where the freedom of translation in $y$ is fixed by $\pa_y \Pi_1 [\xi^{u, s}] =0$, it is important that the precise order of the error $\xi^{u, s} = \mathcal{O} (\frac {\e^2}{|y^2 +\frac {\pi^2}4|^3})$ is obtained near the singularity $y =\pm \frac \pi2i$. This does not only allows one to identify the correct scaling leading to the limit of the inner equation in the next subsection, but also uniquely fix the solutions of the inner equation optimally approximating $v^{u, s}$.     
\end{remark}

\subsection{Analysis close to the singularities}
Notice that the parameterizations $v^\star(y,\tau)$ of $W^{\star}(0)$, $\star=u,s$ given by Theorem \ref{outerthm}, are $\e^2$-close to the homoclinic orbit $v^h(y)\sin(\tau)$ for $y\in\R\cap D^{\mathrm{out},\star}_{\kappa}$. Nevertheless, at distance $\er(\e)$ of the poles $y=\pm i\pi/2$ of $v^h$,  $v^h\sim\e^{-1}$ has comparable  size to the error $\xi^\star\sim\e^{-1}$.

To obtain a first order approximation of the invariant manifolds  at distance $\er(\e)$ of the poles $y=\pm i\pi/2$ we proceed as follows.
We focus on the singularity $y=i \pi/2$ since  similar results can be derived near the singularity $y=-i \pi/2$ by the conjugacy. Consider the \emph{inner variable}
\begin{equation}
\label{innervar}
z=\e^{-1}\left(y-i\dfrac{\pi}{2}\right)
\end{equation}
and the scaling
\begin{equation}
\label{innerscaling}
\phi(z,\tau)=\e v\left(i\dfrac{\pi}{2}+\e z, \tau\right).
\end{equation}
Writing equation \eqref{kleingordonv} for $\phi(z,\tau)$ and recalling $\omega=(1+\e^2)^{-\frac 12}$, we obtain
\begin{equation}
\label{kleingordonphi}
\partial_z^2\phi-\partial_{\tau}^2\phi- \dfrac 1{\omega^2}\phi+\dfrac{1}{3}\phi^3+ \frac 1{\omega^3} f(\omega \phi)=0.
\end{equation}
This equation coincides with the original Klein-Gordon equation \eqref{eq:KLGtau}\footnote{Warning: It is the original one for $\psi=\omega\phi$, but will be analyzed near a singular complex function.}. However, notice that now the evolution variable is $\displaystyle z=x-i\frac{\pi}{2\e}$.

The first order of \eqref{kleingordonphi} corresponds to the regular limit  $\e=0$, which gives the so-called \emph{inner equation}  
\begin{equation}
\label{inner}
\partial_z^2\phi^0-\partial_{\tau}^2\phi^0-\phi^0+\dfrac{1}{3}(\phi^0)^3+f(\phi^0)=0.
\end{equation}
We are interested in identifying certain solutions of \eqref{inner} with the same first order of the outer solutions $v^{u,s}(y,\tau)$ given in Theorem \ref{outerthm} near the pole $y=i\pi/2$. Therefore, we  look for solutions $\phi^{0,\ast}(z,\tau)$, $\ast=u,s$,  of \eqref{inner} which have the same leading order expansion as $\phi^{u,s}(z,\tau)=\e v^{u,s}\left(i\left(\dfrac{\pi}{2}+\e z\right),\tau\right)$. Near the pole $y=i\pi/2$, by Theorem \ref{outerthm} we have
\[v^{u,s}(y,\tau)= v^h(y)\sin\tau+\er\left(\dfrac{\e^2}{(y-i\pi/2)^3}\right) = \dfrac{-2\sqrt{2}i}{y-i\pi/2}\sin\tau+\er(y-i\pi/2)+\er\left(\dfrac{\e^2}{(y-i\pi/2)^3}\right)\]
which, in the inner variables \eqref{innervar} and \eqref{innerscaling},  corresponds to
\begin{equation*}\label{solutionsouterinnervar}\phi^{u,s}(z,\tau)= \dfrac{-2\sqrt{2}i}{z}\sin\tau+\er(\e^2 z)+\er\left(\frac{1}{z^{3}}\right).\end{equation*} 
Taking into account the change of variables \eqref{innervar} and the shape of the outer domains \eqref{outer}, this asymptotic condition must hold for $\Ip z<0$ and $\Rp z<0$ for $\phi^{u}$ and for $\Ip z<0$ and $\Rp z>0$ for $\phi^{s}$.

Therefore we consider the \emph{inner domains}
\begin{equation}
\label{innerdomainsol}
\begin{array}{l}
D^{u,\mathrm{\mathrm{in}}}_{\theta,\kappa}=\{z\in\C; |\Ip(z)|> \tan\theta \Rp(z)+\kappa \},\vspace{0.2cm}\\
D^{s,\mathrm{\mathrm{in}}}_{\theta,\kappa}=\{z\in \C; -z\in D^{u,\mathrm{\mathrm{in}}}_{\theta,\kappa} \},
\end{array}
\end{equation}
for  $0<\theta<\pi/6$ 
and $\kappa>0$ (see Figure \ref{innerfig}), and we look for solutions of the inner equation of the form
\begin{equation*}\label{innersolution}
\phi^{0,\star}(z,\tau)=\dfrac{-2\sqrt{2}i}{z}\sin\tau+\psi^\star(z,\tau),\quad \text{with} \quad\psi^\star=\er\left(\frac{1}{z^{3}}\right),\quad \text{for}\quad (z,\tau)\in D^{\star,\mathrm{\mathrm{in}}}_{\theta,\kappa}\times\TT, \ \star=u,s.
\end{equation*}
We present the results concerning the existence of these solutions $\phi^{0,\star}$ of \eqref{kleingordonphi}, $\star=u,s$. Moreover we give an asymptotic expression for the difference $\phi^{0,u}(z,\tau)-\phi^{0,s}(z,\tau)$ as $\Ip(z)\to-\infty$, which will be crucial to compute the first order of the difference $v^u-v^s$. The following theorem will be proved in Section \ref{innersec}.

\begin{figure}[!]	
	\centering
	\begin{overpic}[width=12cm]{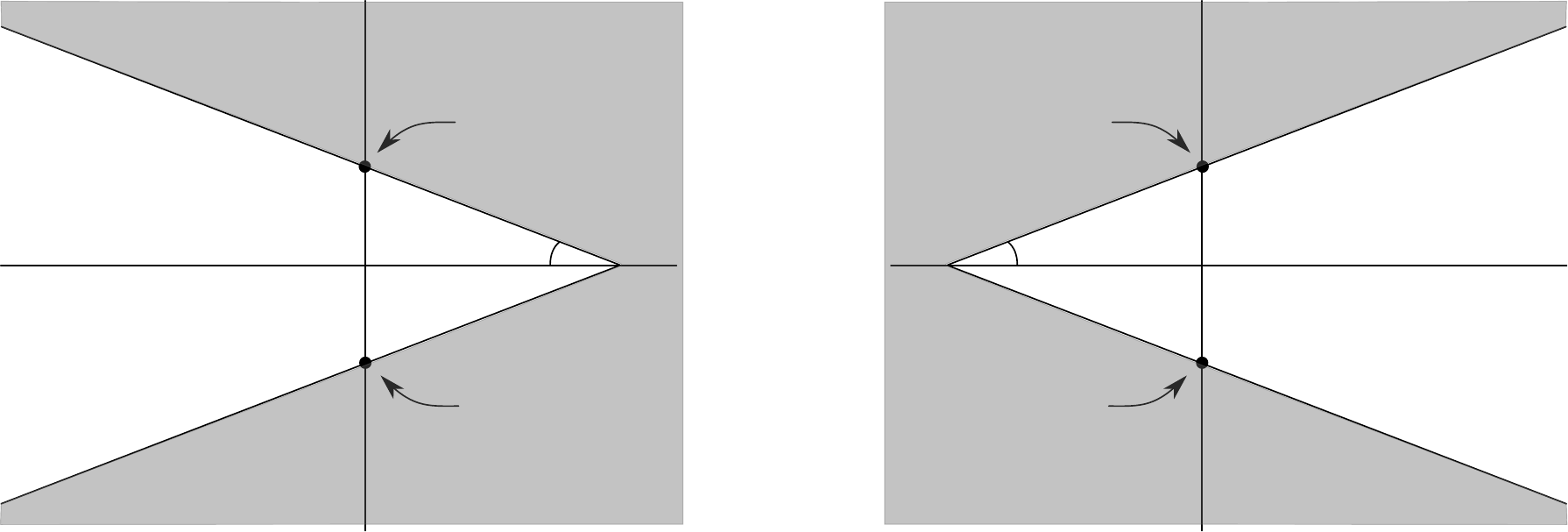}
		\put(30,25){{\footnotesize $i\kappa$}}		
		\put(30,7){{\footnotesize $-i\kappa$}}			
		\put(68,25){{\footnotesize $i\kappa$}}		
		\put(66,7){{\footnotesize $-i\kappa$}}	
		\put(33,17){{\footnotesize $\theta$}}	
		\put(65.5,17){{\footnotesize $\theta$}}	
		\put(35,31){{\footnotesize $D^{s,\mathrm{\mathrm{in}}}_{\theta,\kappa}$}}	
		\put(60,31){{\footnotesize $D^{u,\mathrm{\mathrm{in}}}_{\theta,\kappa}$}}					\end{overpic}
	\bigskip
	\caption{ Inner domains $D^{s,\mathrm{\mathrm{in}}}_{\theta,\kappa}$ and $D^{u,\mathrm{\mathrm{in}}}_{\theta,\kappa}$. }	
	\label{innerfig}
\end{figure} 

\begin{theorem}[Inner] \label{innerthm}Let $\theta\in (0, \frac \pi6)$ and $r>0$ be fixed and consider $f\in\FF_r$. There exists $\kappa_0\geq 1$ big enough such that, for each $\kappa\geq\kappa_0$,
	\begin{enumerate} 
		\item Equation \eqref{inner} has two solutions $\phi^{0,\star}: D^{\star,\mathrm{\mathrm{in}}}_{\theta,\kappa}\times\mathbb{T}\rightarrow \C$, $\star=u,s$, given by \begin{equation}\label{innersol}\phi^{0,\star}(z,\tau)=-\dfrac{2\sqrt{2} i}{z}\sin\tau+\psi^\star(z,\tau),\end{equation}
		which are analytic in the variable $z$. Moreover, $\Pi_{2l}\left[\phi^{0,\star}\right]\equiv 0$ for every $l\in\N$, and there exists a constant $M_2>0$ independent of $\kappa$ such that, for every $z\in D^{\star,\mathrm{\mathrm{in}}}_{\theta,\kappa}$ and $z' \in D^{\star,\mathrm{\mathrm{in}}}_{2\theta,4\kappa}$
		\begin{equation}\label{def:innerestimates}
		\left\|\partial_{\tau}^2\psi^{\star} \right\|_{\n}(z)\leq \dfrac{M_2}{|z|^3}, \quad \left\|\partial_{\tau}^2\pa_z \psi^{\star} \right\|_{\n}(z')\leq \dfrac{M_2}{|z'|^4}.
		\end{equation}
		
		\item The difference $\Delta\phi^0(z,\tau)= \phi^{0,u}(z,\tau)-\phi^{0,s}(z,\tau)$ is given by (see Figure \ref{difinnerfig}),
		\begin{equation}\label{diffinnersol}\Delta\phi^0(z,\tau)=e^{-i\mu_{3}z}\left(C_{\mathrm{in}}\sin(3\tau)+ \chi(z,\tau)\right),\quad z\in \mathcal{R}^{\mathrm{\mathrm{in}},+}_{\theta,\kappa}= D^{u,\mathrm{\mathrm{in}}}_{\theta,\kappa}\cap D^{s,\mathrm{\mathrm{in}}}_{\theta,\kappa}\cap\{ z;\ \Rp(z)=0, \Ip(z)<0 \}\end{equation}
		where $\mu_{3}=2\sqrt{2}$, $C_{\mathrm{in}} \in \C$ is a constant, and $\chi$ is  analytic in $z$ and satisfies  that, for $z\in \mathcal{R}^{\mathrm{\mathrm{in}},+}_{\theta,\kappa}$,
		\[
		\|\partial_{\tau}\chi\|_{\ell_1}(z)\leq \dfrac{M_2}{|z|} \quad \text{and} \quad \|\partial_z\chi\|_{\ell_1}(z)\leq \dfrac{M_2}{|z|^2}.\]

\item	The constant $C_{\mathrm{in}}=C_{\mathrm{in}}(f) \in \C$  depends analytically on $f\in\FF_r$.  
\end{enumerate}
\end{theorem}

\begin{figure}[!]	
	\centering
	\begin{overpic}[width=6cm]{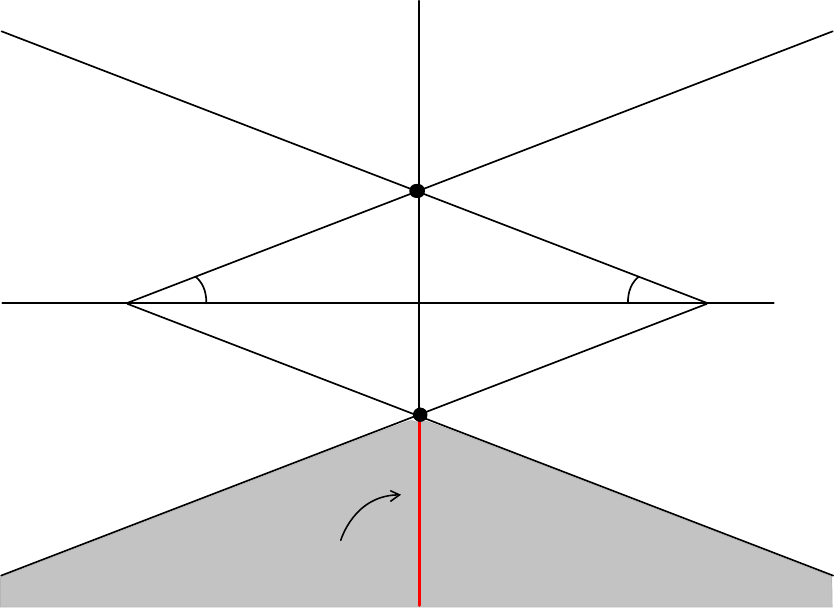}
		\put(52,21.5){{\footnotesize $-i\kappa$}}	
		\put(70,37){{\footnotesize $\theta$}}	
		\put(27,37){{\footnotesize $\theta$}}	
		\put(35,3){{\footnotesize $\mathcal{R}^{\mathrm{\mathrm{in}},+}_{\theta,\kappa}$}}	
	\end{overpic}
	\bigskip
	\caption{ Domain $\mathcal{R}^{\mathrm{\mathrm{in}},+}_{\theta,\kappa}$. }	
	\label{difinnerfig}
\end{figure} 

\begin{remark}
It is interesting to see that the stable and unstable solutions $\phi^{0, \star}$, $\star=u, s$, are identified by the $\mathcal{O} (|z|^{-1})$ decay as $\Rp z \to \pm \infty$, where the same Lyapunov-Perron approach works. The freedom of translation in $z$, which causes a variation of the order $\mathcal{O}(|z|^{-2})$ is fixed by the $\mathcal{O}(|z|^{-3})$ restriction of the error terms. The splitting $\Delta \phi^0$  between $\phi^{0, u}$ and $\phi^{0, s}$ would turn out to be the principal part of the splitting between $v^u$ and $v^s$. The leading order form of $\Delta \phi^0$ can be understood in two different perspectives. On the one hand, it is related to the Borel summation of divergent power series and the readers are referred to Section \ref{SS:Stokes} for related discussions and our conjecture on how to compute $C_{\mathrm{in}}$. 
On the other hand, along the real direction of $z$, the inner equation \eqref{inner} is hyperbolic in the PDE sense and oscillatory. 
However, when we view it along the imaginary axis, it becomes strongly hyperbolic in the dynamical systems sense and elliptic in the PDE sense (and dynamically ill-posed). 
All the originally oscillatory directions become hyperbolic in the dynamical systems sense and thus in particular the stable manifolds become infinite dimensional containing $\phi^{0, \star}$. The splitting $\Delta \phi^0$ is dominated by the weakest exponential decay rate and the Stokes constant $C_{\mathrm{in}}$ basically comes from the difference between the weakest stable coordinates of $\phi^{0, u}$ and $\phi^{0, s}$. 
\end{remark}

\begin{remark}
The constant $C_\mathrm{in} (f)$ in Theorem \ref{innerthm} is the constant appearing in Theorem \ref{maintheorem}. Theorem \ref{innerthm}(3) provides its analyticity with respect to $f$ as stated in Theorem \ref{prop:Stokesconstant}. In Section \ref{sec:Stokes}
we prove that, typically, it does not vanish. 
\end{remark}

Our next step is to prove that  the solutions of the inner equation obtained in Theorem \ref{innerthm} are good approximations of the parameterizations  $v^\ast(y,\tau)$, $\ast=u,s$, obtained in Theorem \ref{outerthm} near the pole $y=i\pi/2$. To prove this fact we introduce the following \emph{matching domains}.

Take  $0<\cg<1$, $0<\beta_1<\beta<\beta_2<\pi/4$ constants independent of $\e$ and $\kappa$. Then, we consider the points   $y_j\in\C$, $j=1,2$  satisfying
\begin{enumerate}
	\item $\Ip(y_j)=-\tan\beta_j\Rp(y_j)+\pi/2-\kappa\e$;
	\item $|y_j-i(\pi/2-\kappa\e)|=\e^\gamma$;
	\item $\Rp(y_1)<0$ and $\Rp(y_2)>0$.
\item $e^{5(\pi - \beta_1)} - e^{-5 \beta_2} \ne 0$. 
\end{enumerate}
Note that 
$\Ip(y_2)<\frac{\pi}{2}-\kappa\eps<\Ip(y_1)$.
Then, consider the following \emph{matching domains} (see Figure \ref{mchfig}),
\begin{equation}\label{matchingdomain}
\begin{array}{lcl}
D^{\mathrm{\mathrm{mch}},u}_{+,\kappa}&=&\left\{ y\in\C;\ \Ip(y)\leq -\tan\beta_1\Rp(y)+\pi/2 - \kappa\e,\ \Ip(y)\leq -\tan\beta_2\Rp(y)+\pi/2 - \kappa\e,  \right.\\
& &  \left.\Ip(y)\geq \Ip(y_1)-\tan\left(\dfrac{\beta_1+\beta_2}{2}\right)(\Rp(y)-\Rp(y_1))\right\},\vspace{0.2cm}\\
D^{\mathrm{\mathrm{mch}},s}_{+,\kappa}&=&\left\{ y\in\C; -\overline{y}\in D^{\mathrm{\mathrm{mch}},u}_{+,\kappa}\right\}.
\end{array}
\end{equation}

\begin{figure}[!]	
	\centering
	\begin{overpic}[width=14cm]{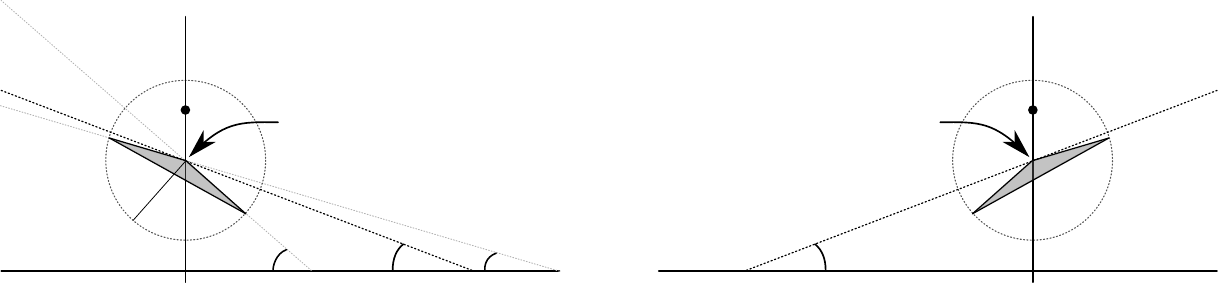}			
		\put(16,14){{\scriptsize $i\frac{\pi}{2}$}}	
		\put(24,12.5){{\footnotesize $i\left(\frac{\pi}{2}-\kappa\e\right)$}}		
		\put(85.5,14){{\scriptsize $i\frac{\pi}{2}$}}	
		\put(67,12.5){{\footnotesize $i\left(\frac{\pi}{2}-\kappa\e\right)$}}		
		\put(30,2){{\footnotesize $\bg$}}
		\put(68.5,2){{\footnotesize $\bg$}}			
		\put(20,2){{\footnotesize $\bg_2$}}	
		\put(37,2){{\footnotesize $\bg_1$}}	
		\put(9.5,3.1){{\footnotesize $c\e^{\cg}$}}
		\put(7,11){{\footnotesize $y_1$}}
		\put(20.5,4.5){{\footnotesize $y_2$}}												\end{overpic}
	\bigskip
	\caption{ Matching domains $D^{\mathrm{\mathrm{mch}},u}_{+,\kappa}$ (on the left) and $D^{\mathrm{\mathrm{mch}},s}_{+,\kappa}$ (on the right). }	
	\label{mchfig}
\end{figure} 

Notice that there exist constants $M_1, M_2>0$ independent of $\e$ and $\kappa$ such that
\[
\begin{split}
M_1\e^\cg&\leq |y_j-i\pi/2|\leq M_2\e^\cg,\qquad j=1,2,\\
M_1\kappa\e&\leq  |y-i\pi/2|\leq M_2\e^\cg,\qquad \text{for}\quad y\in D^{\mathrm{\mathrm{mch}},u}_{+,\kappa}.
\end{split}
\]
In terms of the inner variable $z$ (see \eqref{innervar}), the matching domains are given by
\[
 \mathcal{D}^{\mathrm{\mathrm{mch}},\star}_{+,\kappa}=\{z\in \C; \e z+i\pi/2\in D^{\mathrm{\mathrm{mch}},\star}_{+,\kappa} \},\ \star=u,s.
\]
Notice that,
\[
\begin{split}
M_1\e^{\cg-1}&\leq |z_j|\leq M_2\e^{\cg-1},\qquad j=1,2,\\
M_1\kappa&\leq  |z|\leq M_2\e^{\cg-1}\qquad \text{for,}\quad z\in \mathcal{D}^{\mathrm{\mathrm{mch}},u}_{+,\kappa}.
\end{split}
\]
where $z_1$ and $z_2$ are the vertices of the inner domain $y_1$ and $y_2$, respectively, expressed in the inner variable.

Next theorem estimates the difference  in the matching domains \eqref{matchingdomain} between the functions $\phi^*$, $\ast=u,s$ in \eqref{innerscaling} and the functions $\phi^{0,*}$, $\ast=u,s$, given by Theorem \ref{innerthm}. The theorem is proven in Section \ref{sec:matching}.

\begin{theorem}[Matching]\label{matchingthm}
	Fix $\cg\in(1/3,1)$. Let $\phi^{\star}(z,\tau)=\e v^{\star}(i\pi/2+ \e z,\tau)$, $\star=u,s$, where $v^{\star}$ is the parameterization obtained in Theorem \ref{outerthm}. Then, there exist $\e_0, \delta_0>0$ sufficiently small 
	such that, for each $0<\e\leq \e_0$ and $\kappa$ satisfying $\kappa\e^{1-\gamma}+ \frac {|\log \e|}{\kappa^2} \le \delta_0$, and $z\in \mathcal{D}^{\mathrm{\mathrm{mch}},\star}_{+,\kappa}$,
	$$\phi^{\star}(z,\tau)= \phi^{0,\star}(z,\tau)+ \p^{\star}(z,\tau),$$
	where $\phi^{0,\star}$ is the solution of the inner equation \eqref{inner} obtained in Theorem \ref{innerthm}, and $\varphi^\star$ satisfies that for $(z,\tau)\in \mathcal{D}^{\mathrm{\mathrm{mch}},\star}_{+,\kappa}$
	$$\|\partial_{\tau}^2 \p^{\star}\|_{\ell_1}(z)\leq \dfrac{M_3(\e^{1-\cg}+\e^{3\cg-1})|\log \e|}{|z|^2} \quad and \quad \|\pa_\tau^2 \partial_z\p^{\star}\|_{\ell_1}(z)\leq \dfrac{M_3(\e^{1-\cg}+\e^{3\cg-1})|\log \e|}{\kappa|z|^2} ,$$
	where $M_3>0$   is a constant independent of $\e$ and $\kappa$.
\end{theorem}

\begin{remark}
	Notice that $\cg=1/2$ minimizes the size of $\|\p^\star\|_{\n,2}$, $\star=u,s$, in Theorem \ref{matchingthm}. In this case, 
	$$\|\pa_\tau^2 \p^\star\|_{\n,2}\leq M |\log \e| \e^{1/2}|z|^{-2}.$$
\end{remark}

\begin{remark} 
The idea to obtain the above matching estimate is that $y_1$ and $y_2$ are connected by a segment with nontrivial slope in the complex plane, where the linear part of the problem becomes somewhat elliptic in the 1-dim variable $z$ (in the PDE sense) except in the direction of the mode $\sin \tau$. Therefore $\p^\star$, $\star=u, s$, is nicely determined by the values at $y_1$ and $y_2$ which simply come from the asymptotic form $\phi^{0, \star}$ and $\p^\star$. The order $\mathcal{O}(|z|^{-2})$ is largely determined by the mode $\sin \tau$.  
\end{remark}

\subsection{The distance between the invariant manifolds}\label{sec:Sketch:difference}

Our next step is  to give an asymptotic formula for the difference
\begin{equation}\label{diffe}
\Delta(y,\tau)= v^u(y,\tau)-v^s(y,\tau)=\xi^u(y,\tau)-\xi^s(y,\tau),
\end{equation}
where $\xi^{u,s}$ are the functions obtained in Theorem \ref{outerthm} (recall that $\Pi_{2l}[\Delta v]=0$ for every $l\geq 0$),
in the  domain (see Figure \ref{diffdom}).
\[
\mathcal{R}_{\kappa}=D^{\mathrm{out},u}_{\kappa}\cap D^{\mathrm{out},s}_{\kappa}\cap i\R,
\]
Next lemma shows that the difference  $\Delta$ satisfies a linear equation. 

\begin{figure}[!]	
	\centering
	\begin{overpic}[width=9cm]{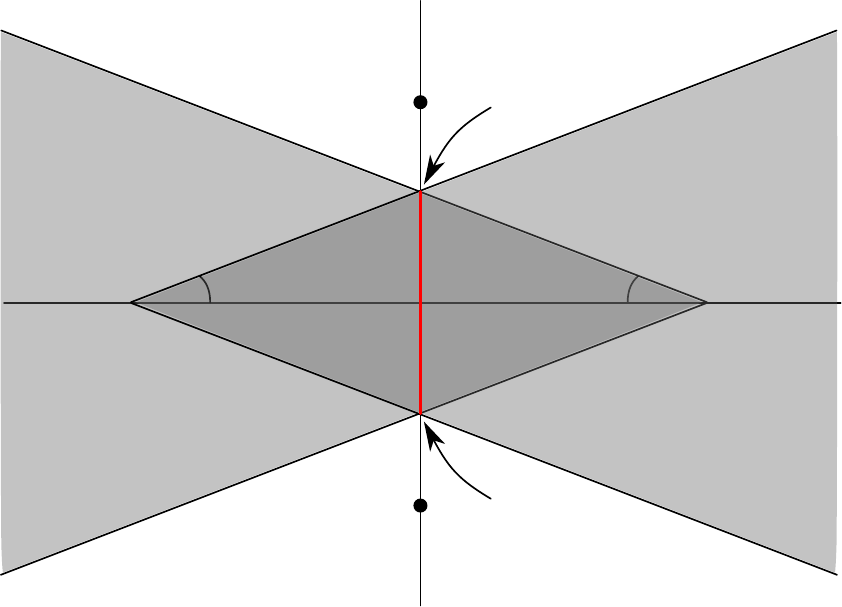}
		\put(42,59){{\footnotesize $i\frac{\pi}{2}$}}
		\put(39.5,11){{\footnotesize $-i\frac{\pi}{2}$}}		
		\put(59,60){{\footnotesize $i\left(\frac{\pi}{2}-\kappa\e\right)$}}			
		\put(58,11){{\footnotesize $-i\left(\frac{\pi}{2}-\kappa\e\right)$}}		
		\put(26,37){{\footnotesize $\bg$}}	
		\put(72,37){{\footnotesize $\bg$}}	
		\put(5,55){{\footnotesize $D^{\mathrm{out},u}_{\kappa}$}}	
		\put(88,55){{\footnotesize $D^{\mathrm{out},s}_{\kappa}$}}
		\put(44,38){{\footnotesize $\mathcal{R}_{\kappa}$}}									\end{overpic}
	\bigskip
	\caption{ Domain $\mathcal{R}_{\kappa}$. }\label{diffdom}	
\end{figure}
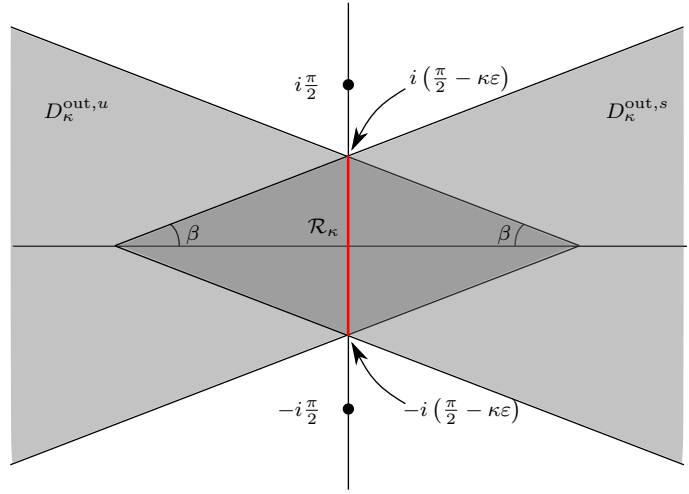 
 
\begin{lemma}\label{equationdiff}
	The function $\Delta$ introduced in \eqref{diffe} satisfies the linear equation
	\begin{equation*}\mathcal{L}(\Delta )=\Pi_1\left[ \eta_1(y,\tau)\Pi_1[\Delta ]\sin\tau+ \eta_2(y,\tau)\wt \Pi[\Delta] \right] \sin\tau+ \wt \Pi[\eta_3(y,\tau)\Delta],\end{equation*}
	where $\mathcal{L}$ is the operator given in \eqref{Lop} and $\eta_j:\mathcal{R}_\kappa\times\mathbb{T}\rightarrow\C$, $j=1,2,3$, are functions analytic in $y$. Moreover, there exists a constant $M>0$ independent of $\kappa$ and $\e$ such that
	\[ \|\eta_1\|_{\n}(y)\leq \frac{M\e^2}{|y^2+\pi^2/4|^4},\quad and \quad \|\eta_2\|_{\n}(y),\|\eta_3\|_{\n}(y)\leq \frac{M}{|y^2+\pi^2/4|^2}.\]
\end{lemma}

\begin{proof}
From \eqref{xinop} and Theorem \ref{outerthm}, we have that 
\begin{equation*}
\mathcal{L}(\Delta)=\mathcal{F}(\xi^u)-\mathcal{F}(\xi^s),
\end{equation*}
where $\mathcal{F}$ is the operator given in \eqref{Fop}. Using the expression of $\mathcal{F}$  (see also \eqref{FE}), we obtain that
	\[
	\begin{split}
	\mathcal{F}(\xi^u)-\mathcal{F}(\xi^s)	=& -\dfrac{1}{\e^3 \omega^3}\wt \Pi\left[g(\e \omega( \xi^u+v^h\sin\tau))-g(\e\omega( \xi^s +v^h\sin\tau))\right]\\
	&-\Pi_1
	\left[(\xi_1^u+v^h)^2\sin^2\tau\wt \Pi(\xi^u)-(\xi_1^s+v^h)^2\sin^2\tau\wt \Pi(\xi^s)\right]\sin\tau\\
	&-\Pi_1\left[(\xi_1^u+v^h)\sin\tau(\wt \Pi[\xi^u])^2-(\xi_1^s+v^h)\sin\tau(\wt \Pi[\xi^s])^2 \right.\\
	&\left.
	+ \dfrac{1}{3}\left((\wt \Pi[\xi^u])^3
	-(\wt \Pi[\xi^s])^3\right)\right]\sin\tau\\
	&+\left(-\dfrac{1}{\e^3 \omega^3}\Pi_1\left[f(\e\omega ( \xi^u+v^h\sin\tau))-f(\e\omega( \xi^s+v^h\sin\tau))\right] \right.\\
	&\left. -\dfrac{3v^h\left((\xi_1^u)^2-(\xi_1^s)^2\right)}{4}-\dfrac{(\xi_1^u)^3-(\xi^s_1)^3}{4}\right)\sin\tau.
	\end{split}
	\]
	The proof follows from calculations based in the power series expansion of $g$ and $f$ and
	the estimates 
	\[
	\left|v^h(y)\right|\leq \frac{M}{|y^2+\pi^2/4|}, \qquad \left\|\xi^{u,s}\right\|_{\n}(y)\leq \frac{ M\e^2}{|y^2+\pi^2/4|^3}
	\]
	obtained in Theorem \ref{outerthm}.
\end{proof}

The idea to obtain the exponentially small splitting estimate is that $y^\pm=\pm 
i(\frac{\pi}{2}-\kappa\eps)$ (see Figure \ref{diffdom}) are connected by a 
vertical segment where the linear operator $\mathcal{L}$ becomes elliptic (in 
the PDE sense) in the 1-dim variable $y$ except in the direction of the mode 
$\sin \tau$. This has two implications: a.)  the solution is determined by the 
values at the two boundary points $y^\pm$ and b.) the Green's function 
principally in the form of exponential functions leads to the desired 
splitting estimate at $y=0$. The mode $\sin \tau$ seems to be an exception. 
Recalling $\pa_y \Pi_1 [\Delta]|_{y=0} =0$, the splitting in the direction 
$\Pi_1 [\Delta] |_{y=0}$ will be handled by the conservation of energy due to 
the Hamiltonian structure. 

As explained in Section \ref{sec:heuristics}, to prove that the distance between the stable and unstable manifold is exponentially small is crucial the fact that the model considered has a conserved quantity. Indeed, if the system would not have a first integral, the distance between the invariant manifolds would be ``tipically'' of order of some power of $\e$. Therefore, in this section we must rely on the conservation of energy to analyze $\Delta$.

Let us rewrite equation \eqref{kleingordonv} as
\begin{equation*}
\left\{\begin{array}{l}
\partial_y v= w,\vspace{0.2cm}\\
\partial_y w= \dfrac{1}{\e^2}\partial_{\tau}^2v+\dfrac{1}{\e^2\omega^2}v-\dfrac{1}{3}v^3-\dfrac{1}{\e^3 \omega^3}f\left(\e\omega v\right),
\end{array}\right.
\end{equation*}
which is  Hamiltonian with respect to
\begin{equation*}
\label{Hamiltonian}
\mathcal{H}(v,w)=\dfrac{1}{\pi}\displaystyle\int_{\mathbb{T}}\left(\dfrac{w^2}{2}+\dfrac{(\partial_{\tau}v)^2}{2\e^2}-\dfrac{v^2}{2\e^2\omega^2}+\dfrac{v^4}{12}+\dfrac{F(\e \omega v)}{\e^4\omega^4}\right)d\tau,
\end{equation*}
where $F$ is an analytic function such that $F'(z)=f(z)$ and $F(z)=\er(z^6)$. 

Notice that the solutions  $v^{\star}(y,\tau)$ of \eqref{kleingordonv}, $\star=u,s$, obtained in Theorem \ref{outerthm} are contained in the  energy level  $\{\mathcal{H}=0\}$. We use the Hamiltonian $\mathcal{H}$ to obtain the variable $\Pi_1[\Delta]$ in terms of the variables $\wt \Pi[\Delta]$, $\Pi_1[\Xi]$ and  $\wt \Pi[\Xi]$  where $\Xi=\partial_y\Delta=w^u-w^s=\partial_y v^u-\partial_y v^s$.

\begin{lemma}\label{propchatissima}
	The functions $\Delta$, $\Xi$ satisfy
	\begin{equation}\label{primeirofourier}
	\Pi_1[\Delta](y)=\dfrac{\dot{v}^h(y)}{\ddot{v}^h(y)}\Pi_1[\Xi](y)+ A(\Xi)(y)+B(\wt \Pi[\Delta])(y),
	\end{equation}
	where $A$ and $B$ are linear operators such that, for  $y\in \mathcal{R}_{\kappa}$,
	\begin{enumerate}
		\item $|A(\Xi)(y)|\leq \dfrac{M\e^2}{|y^2+\pi^2/4|}\|\Xi\|_{\ell_1}(y)$
		\item $|B(\wt \Pi[\Delta])(y)|\leq M\|\wt \Pi[\Delta]\|_{\ell_1}(y)$.
	\end{enumerate}
\end{lemma}

\begin{proof}
	As the projections $\Pi_1$ and $\wt \Pi$ are orthogonal (see 
\eqref{usin} and \eqref{def:PiTilde}),  $\mathcal{H}$ is given by
	\[	
	\mathcal{H}(v,w)=\dfrac{(\Pi_1[w])^2}{2}-\dfrac{(\Pi_1[v])^2}{2}+\dfrac{1}{\pi}\displaystyle\int_{\mathbb{T}}\left(\dfrac{(\wt \Pi[w])^2}{2}+\dfrac{(\partial_{\tau}\wt \Pi[v])^2}{2\e^2}-\dfrac{(\wt \Pi[v])^2}{2\e^2\omega^2}+\dfrac{v^4}{12}+\dfrac{F(\e \omega v)}{\e^4\omega^4}\right)d\tau.
	\]
	Using that $\mathcal{H}(v^{\star},w^{\star})=0$, $\star=u,s$, integrating by parts the $\partial_\tau$ term and the Mean Value Theorem, we have that
	\[
	\begin{split}
	0=& \mathcal{H}(v^{u},w^{u})-\mathcal{H}(v^{s},w^{s})\\
	=&\dfrac{\Pi_1[w^u]+\Pi_1[w^s]}{2}\Pi_1[\Xi]-\dfrac{\Pi_1[v^u]+\Pi_1[v^s]}{2}\Pi_1[\Delta]\\
	&+\dfrac{1}{\pi}\displaystyle\int_{\mathbb{T}}\left[  \dfrac{\wt \Pi[w^u]+\wt \Pi[w^s]}{2}\wt \Pi[\Xi] -\dfrac{1}{\e^2}\dfrac{\partial_{\tau}^2\wt \Pi[v^u]+\partial_{\tau}^2\wt \Pi[v^s]}{2}\wt \Pi[\Delta]- \dfrac{\wt \Pi[v^u]+\wt \Pi[v^s]}{2\e^2\omega^2}\wt \Pi[\Delta] \right]d\tau\\
	&+\dfrac{1}{\pi}\displaystyle\int_{\mathbb{T}}\left[  \dfrac{(v^u)^3+(v^u)^2(v^s)+(v^u)(v^s)^2+(v^s)^3}{12}\Delta +\left(\dfrac{1}{\e^3 \omega^3}\displaystyle\int_0^1f(\e\omega(\sigma v^u+(1-\sigma)\sigma v^s))d\sigma\right)\Delta \right]d\tau.
	\end{split}
	\] 
	Using 
	\[
	v^\star=v^h\sin(\tau)+\xi^{\star}(y,\tau), \quad \ddot{v}^h=v^h-(v^h)^3/4  =\sqrt{2}(\cosh(2y)-3)\sech^3(y), 
	\]
and observing that $\ddot{v}^h(y)$
is strictly negative, for every $y=i\widetilde{y}$ with $\widetilde{y}\in(-\pi/2,\pi/2)$, one has 
	$$0=-\ddot{v}^h(1+{a}(y))\Pi_1[\Delta]+\dot{v}^h\Pi_1[\Xi]+ \widetilde{A}(\Xi)+\widetilde{B}(\wt \Pi[\Delta])$$
	By the estimates in Theorem \ref{outerthm} and using that $\ddot{v}^h(y)$ has a third order pole at $y=\pm i \pi/2$, we have 
	\begin{equation}\label{eq:atilde}
	\displaystyle|{a}(y)|\leq \frac{M\e^2}{|y^2+\pi^2/4|^2}\leq\dfrac{M}{\kappa^2},\qquad \text{ for }\quad y\in\mathcal{R}_{\kappa}
	\end{equation}
	and, also for $y\in\mathcal{R}_{\kappa}$, 
		\[\left|\widetilde{A}(\Xi)\right|(y)\leq \frac{M\e^2}{|y^2+\pi^2/4|^4}\|\Xi\|_{\ell_1}(y)\qquad \text{ and }\qquad
		\left|\widetilde{B}(\wt \Pi[\Delta])\right|(y)\leq \frac{M}{|y^2+\pi^2/4|^3}\|\wt \Pi[\Delta]\|_{\ell_1}(y).
	\]
	Moreover, 
	using the estimate \eqref{eq:atilde} and  taking $\kappa$ big enough, we have 
	\[
	|D(y)^{-1}|\leq M|y^2+\pi^2/4|^3, \; y\in\mathcal{R}_{\kappa}, \text{ where } D(y)=\ddot{v}^h(y)\left(1+{a}(y)\right).
	\]

	Hence, it follows that
	$$
	\Pi_1[\Delta]
	=\dfrac{\dot{v}^h\Pi_1[\Xi]+ \widetilde{A}(\Xi)+\widetilde{B}(\wt \Pi[\Delta])}{\ddot{v}^h(1+a)}
	=\dfrac{\dot{v}^h}{\ddot{v}^h}\Pi_1[\Xi]+ A(\Xi)+B(\wt \Pi[\Delta]),
	$$
	where $A$ and $B$ are the linear operators	
	\[
	\begin{split}
	A(\Xi)(y)=&\,\dfrac{\widetilde{A}(\Xi)(y)}{\ddot{v}^h(y)(1+a(y))}-\dfrac{\dot{v}^h(y)}{\ddot{v}^h(y)(1+a(y))}a(y)\Pi_1[\Xi](y)\\
	B(\wt \Pi[\Delta])=&\,\dfrac{\widetilde{B}(\wt \Pi[\Delta])}{\ddot{v}^h(y)(1+a(y))}.
	\end{split}
	\]
	The proof of the proposition follows directly from the estimates of $\widetilde{A}(\Xi)$, $\widetilde{B}(\wt \Pi[\Delta])$, $a$ and the fact that $\ddot{v}^h$ and $\dot{v}^h$ have a third and second order pole at the points $y=\pm i\pi/2$, respectively.  
\end{proof}

Lemma \ref{propchatissima} allows to study the difference between the invariant manifolds without keeping track of the component $\Delta_1$. In other words, we use coordinates $(\Pi_1w,\wt \Pi v,\wt \Pi w)$ to analyze the level of energy $\mathcal{H}=0$ and therefore we measure the difference between the functions $v^u$ and $v^s$ through the components $(\Xi_1,\wt \Pi\Delta,\wt \Pi \Xi)$. The inconvience of the energy reduction is that the equation loses the second order structure since it also depends on 
$\Xi=\partial_y\Delta$. 

To capture the exponentially small behavior of $(\wt \Pi\Delta,\wt \Pi \Xi)$ it is convenient to write the second order equation as a first order system in diagonal form. Thus, we define
\begin{equation}\label{cambio}
\begin{split}
\Gamma=&\,\displaystyle\sum_{k\geq 1}\Gamma_{2k+1}(y)\sin((2k+1)\tau),\quad \Gamma_{2k+1}=\lambda_{2k+1} \Delta_{2k+1}+i\e\Xi_{2k+1}\\
\Theta=&\,\displaystyle\sum_{k\geq 1}\Theta_{2k+1}(y)\sin((2k+1)\tau),\quad \Theta_{2k+1}= \lambda_{2k+1}\Delta_{2k+1}-i\e\Xi_{2k+1}.
\end{split}
\end{equation}
From now on we measure  the difference between the invariant manifolds (within the energy level $\mathcal{H}=0$) by the difference ``vector''
\begin{equation}\label{def:NewDelta}
 \widetilde \Delta =(\Xi_1,\Gamma,\Theta).
\end{equation}
Notice that the estimates of Theorem \ref{outerthm} imply that $\Delta$ satisfies
$$\sum_{k\geq 1} \lambda_{2k+1}^2|\Delta_{2k+1}(y)|\leq M \|\partial_\tau^2\xi^u\|_{\ell_1}(y)+M\|\partial_\tau^2\xi^s\|_{\ell_1}(y)\leq  \frac{M\e^2}{|y^2+\pi^4/4|^3},$$
along with a similar estimate on $\Xi$, therefore the functions $\Gamma$ and $\Theta$ are well defined for $y\in\mathcal{R}_{\kappa}$ and satisfy
$$\displaystyle\sum_{k\geq 1}\lambda_{2k+1}|\Gamma_{2k+1}(y)|\leq  \frac{M\e^2}{|y^2+\pi^4/4|^3} \quad and \quad \displaystyle\sum_{k\geq 1}\lambda_{2k+1}|\Theta_{2k+1}(y)|\leq  \frac{M\e^2}{|y^2+\pi^4/4|^3}.$$

\begin{prop}\label{sistfinal}
The function  $\widetilde \Delta =(\Xi_1,\Gamma,\Theta)$ satisfies the equation
	\begin{equation*}
	\widetilde\LL(\widetilde \Delta)=\mathcal{M}(\widetilde \Delta),
	\end{equation*}
	where $\widetilde\LL$ is the differential operator  
\begin{equation}\label{Nop}
\begin{array}{lcl}
\widetilde\LL(\Xi_1,\Gamma,\Theta)&=&\left( \dot{\Xi}_1-\dfrac{\dddot{v}^h}{\ddot{v}^h}\Xi_1 ,\displaystyle\sum_{k\geq 1}\left(\dot{\Gamma}_{2k+1}+i\dfrac{\lambda_{2k+1}}{\e}\Gamma_{2k+1}\right)\sin((2k+1)\tau),\right.\vspace{0.2cm}\\
&&\left.
\displaystyle\sum_{k\geq 1}\left(\dot{\Theta}_{2k+1}-i\dfrac{\lambda_{2k+1}}{\e}\Theta_{2k+1}\right)\sin((2k+1)\tau)  \right)
\end{array}
\end{equation}
and  $\mathcal{M}$ is a linear operator which can be written as
	\begin{equation}\label{Mop}
	\mathcal{M}(\Xi_1,\Gamma,\Theta)=\left(\begin{array}{c}m_W(y)\Xi_1+ \mathcal{M}_{W}(\Gamma,\Theta)\vspace{0.2cm}\\
	m_{\mathrm{osc}}(y,\tau)\Xi_1+ \mathcal{M}_{\mathrm{osc}}(\Gamma,\Theta)\vspace{0.2cm}\\
	-m_{\mathrm{osc}}(y,\tau)\Xi_1- \mathcal{M}_{\mathrm{osc}}(\Gamma,\Theta)
	\end{array}\right),
	\end{equation}
	where $m_W:\mathcal{R}_\kappa\rightarrow \C$, $m_{\mathrm{osc}}:\mathcal{R}_\kappa\times\mathbb{T}\rightarrow \C$ are functions analytic in $y$ satisfying
	\[
|m_W(y)|\leq \frac{M\e^2}{|y^2+\pi^2/4|^3}\qquad \text{ and }\qquad \|m_{\mathrm{osc}}\|_{\ell_1}(y)\leq \frac{M\e}{|y^2+\pi^2/4|}
	\]
and $ \mathcal{M}_{W},\  \mathcal{M}_{\mathrm{osc}}$ are linear operators such that, for $y\in\mathcal{R}_{\kappa}$, 
\[
\begin{split}
 \left|\mathcal{M}_{W}(\Gamma,\Theta)(y)\right|&\leq \dfrac{M}{|y^2+\pi^2/4|^2}\left(\|\Gamma\|_{\ell_1}(y)+\|\Theta\|_{\ell_1}(y)\right)\\
\left\|\mathcal{M}_{\mathrm{osc}}(\Gamma,\Theta)\right\|_{\ell_1}(y)&\leq \dfrac{M\e}{|y^2+\pi^2/4|^2}\left(\|\Gamma\|_{\ell_1}(y)+\|\Theta\|_{\ell_1}(y)\right),
\end{split}
\]
where $M>0$ is a constant independent of $\e$ and $\kappa$ .
\end{prop}

\begin{proof}
	From \eqref{cambio} and Proposition \ref{equationdiff}, we have that, for each $k\geq1$,
	\begin{equation}\label{gamma2k}
	\begin{array}{lcl}
	\dot{\Gamma}_{2k+1}&=& \lambda_{2k+1}\Xi_{2k+1}+i\e \ddot{\Delta}_{2k+1}\vspace{0.2cm}\\
	&=& \lambda_{2k+1}\Xi_{2k+1}+i\e \left(-\dfrac{\lambda_{2k+1}^2}{\e^2}\Delta_{2k+1} + \Pi_{2k+1}\left[ \eta_3(y,\tau)\Delta \right]\right)\vspace{0.2cm}\\
	&=& -i\dfrac{\lambda_{2k+1}}{\e}\Gamma_{2k+1} +i\e\Pi_{2k+1}\left[ \eta_3(y,\tau)\Delta  \right].
	\end{array}
	\end{equation}
	Analogously, for each $k\geq 1$,	
	\begin{equation}\label{theta2k}
	\dot{\Theta}_{2k+1}=i\dfrac{\lambda_{2k+1}}{\e}\Theta_{2k+1}-i\e\Pi_{2k+1}\left[ \eta_3(y,\tau)\Delta\right].
	\end{equation}
	Moreover, for the variable $\Xi_1$, by \eqref{Lop} and Proposition \ref{equationdiff}, we have that
	\begin{equation*}
	\dot{\Xi}_1=\left(1-\dfrac{3(v^h)^2}{4}\right)\Delta_1+\Pi_1\left[ \eta_1(y,\tau)\Delta_1\sin(\tau)+ \eta_2(y,\tau)\wt \Pi[\Delta] \right].
	\end{equation*}
Using  \eqref{primeirofourier} for $\Delta_1$ and $\displaystyle\left(1-\dfrac{3(v^h)^2}{4}\right)\dot{v}^h=\dddot{v}^h,$ we obtain
	\[
	\begin{split}
	\dot{\Xi}_1=&\,\dfrac{\dddot{v}^h}{\ddot{v}^h}\Xi_1+\left(1-\dfrac{3(v^h)^2}{4}\right)\left( A(\Xi)+B(\wt \Pi[\Delta])\right)\\
	&\,+\Pi_1\left[ \eta_1(y,\tau)\left(\dfrac{\dot{v}^h}{\ddot{v}^h}\Xi_1+ A(\Xi)+B(\wt \Pi[\Delta])\right)\sin(\tau)+ \eta_2(y,\tau)\wt \Pi[\Delta]\right]\\
	=&\,\dfrac{\dddot{v}^h}{\ddot{v}^h}\Xi_1+\left(1-\dfrac{3(v^h)^2}{4}\right) A(\Xi_1 \sin(\tau))+\Pi_1\left[ \eta_1(y,\tau)\left(\dfrac{\dot{v}^h}{\ddot{v}^h}\Xi_1+ A(\Xi_1\sin(\tau))\right)\sin\tau\right]\\
	&\,+\left(1-\dfrac{3(v^h)^2}{4}\right)\left( A(\wt \Pi[\Xi])+B(\wt \Pi[\Delta])\right)\\
	&\,+\Pi_1\left[ \eta_1(y,\tau)\left( A(\wt \Pi[\Xi])+B(\wt \Pi[\Delta])\right)\sin\tau+ \eta_2(y,\tau)\wt \Pi[\Delta]\right].
	\end{split}
	\]
Finally, using \eqref{cambio}, 
	\[
	\begin{split}
	\dot{\Xi}_1=&\,\dfrac{\dddot{v}^h}{\ddot{v}^h}\Xi_1+\left(1-\dfrac{3(v^h)^2}{4}\right) A(\Xi_1\sin\tau)
	+\Pi_1\left[ \eta_1(y,\tau)\left(\dfrac{\dot{v}^h}{\ddot{v}^h}\Xi_1+ A(\Xi_1\sin\tau)\right)\sin\tau\right]\\
	&+\left(1-\dfrac{3(v^h)^2}{4}\right)\left( \dfrac{1}{2i\e}A(\Gamma-\Theta)+B\left(\displaystyle\sum_{k\geq 1}\dfrac{\Gamma_{2k+1}+\Theta_{2k+1}}{2\lambda_{2k+1}}\sin((2k+1)\tau)\right)\right)\\
	&+\Pi_1\left[ \eta_1(y,\tau)\left( \dfrac{1}{2i\e}A(\Gamma-\Theta)\right)\sin(\tau)\right.
	\left.+\eta_1(y,\tau)B\left(\displaystyle\sum_{k\geq 1}\dfrac{\Gamma_{2k+1}+\Theta_{2k+1}}{2\lambda_{2k+1}}\sin((2k+1)\tau)\right)\sin\tau\right.\\
	&\left.+ \eta_2(y,\tau)\left(\displaystyle\sum_{n\geq 2}\dfrac{\Gamma_{2k+1}+\Theta_{2k+1}}{2\lambda_{2k+1}}\sin((2k+1)\tau)\right)\right].
	\end{split}
	\]
For the other components, as
	$$
	\begin{array}{lcl}
	i\e\wt \Pi\left[ \eta_3(y,\tau)\Delta \right] &=& i\e\wt \Pi\left[\eta_3(y,\tau)\left(\left(\dfrac{\dot{v}^h}{\ddot{v}^h}\Xi_1+ A(\Xi)+B(\wt \Pi[\Delta])\right)\sin(\tau)+\wt \Pi[\Delta]\right) \right]\vspace{0.2cm}\\
	&=& i\e\wt \Pi\left[\eta_3(y,\tau)\left(\dfrac{\dot{v}^h}{\ddot{v}^h}\Xi_1+ A(\Xi_1\sin(\tau))\right)\sin(\tau)\right] \vspace{0.2cm}\\
	&&+ i\e\wt \Pi\left[\eta_3(y,\tau)\left(A(\wt \Pi[\Xi])\sin(\tau)+B(\wt \Pi[\Delta])\sin(\tau)+\wt \Pi[\Delta])\right)\right] \vspace{0.2cm}\\
	&=& i\e\wt \Pi\left[\eta_3(y,\tau)\left(\dfrac{\dot{v}^h}{\ddot{v}^h}\Xi_1+ A(\Xi_1\sin(\tau))\right)\sin(\tau)\right] \vspace{0.2cm}\\
	&&+ i\e\wt \Pi\left[\dfrac{\eta_3(y,\tau)}{2i\e}A(\Gamma-\Theta)\sin(\tau)\right.
	+\eta_3(y,\tau)B\left(\displaystyle\sum_{k\geq 1}\dfrac{\Gamma_{2k+1}+\Theta_{2k+1}}{2\lambda_{2k+1}}\sin((2k+1)\tau)\right)\sin(\tau)\vspace{0.2cm}\\
	&& + \eta_3(y,\tau)\displaystyle\sum_{k\geq 1}\dfrac{\Gamma_{2k+1}+\Theta_{2k+1}}{2\lambda_{2k+1}}\sin((2k+1)\tau) \Bigg]
	\end{array}
	$$
	the proof is concluded by  using \eqref{gamma2k} and \eqref{theta2k} and  taking
{\allowdisplaybreaks
	\begin{align}
	m_W(y)\Xi_1=&\,\left(1-\dfrac{3(v^h)^2}{4}\right) A(\Xi_1\sin\tau)+\Pi_1\left[ \eta_1(y,\tau)\left(\dfrac{\dot{v}^h}{\ddot{v}^h}\Xi_1+ A(\Xi_1\sin\tau)\right)\sin\tau\right]\notag\\
	\mathcal{M}_W(\Gamma,\Theta)=&\left(1-\dfrac{3(v^h)^2}{4}\right)\left( \dfrac{1}{2i\e}A(\Gamma-\Theta)+B\left(\displaystyle\sum_{k\geq 1}\dfrac{\Gamma_{2k+1}+\Theta_{2k+1}}{2\lambda_{2k+1}}\sin((2k+1)\tau)\right)\right)\notag\\
	&+\Pi_1\left[ \dfrac{\eta_1(y,\tau)}{2i\e}A(\Gamma-\Theta)\sin(\tau)+ \eta_1(y,\tau)B\left(\displaystyle\sum_{k\geq 1}\dfrac{\Gamma_{2k+1}+\Theta_{2k+1}}{2\lambda_{2k+1}}\sin((2k+1)\tau)\right)\sin\tau \right.\notag\\
	&\left.+\eta_2(y,\tau)\left(\displaystyle\sum_{k\geq 1}\dfrac{\Gamma_{2k+1}+\Theta_{2k+1}}{2\lambda_{2k+1}}\sin((2k+1)\tau)\right)\right]\notag\\
	m_{\mathrm{osc}}(y,\tau)\Xi_1=&i\e\wt \Pi\left[\eta_3(y,\tau)\left(\dfrac{\dot{v}^h}{\ddot{v}^h}\Xi_1+ A(\Xi_1\sin \tau)\right)\sin\tau\right]\notag\\	
	\mathcal{M}_{\mathrm{osc}}(\Gamma,\Theta)=&i\e\wt \Pi\left[\eta_3(y,\tau)\left(\dfrac{1}{2i\e}A(\Gamma-\Theta)\sin\tau+B\left(\displaystyle\sum_{k\geq 1}\dfrac{\Gamma_{2k+1}+\Theta_{2k+1}}{2\lambda_{2k+1}}\sin((2k+1)\tau)\right)\sin\tau \right.\right.\notag\\
	&\left.\left.+ \displaystyle\sum_{k\geq 1}\dfrac{\Gamma_{2k+1}+\Theta_{2k+1}}{2\lambda_{2k+1}}\sin((2k+1)\tau) \right)\right],\notag
	\end{align}}
\noindent and using the bounds for the functions $\eta_j$, $j=1,2,3$ and the operators $A$ and $B$ provided in Propositions \ref{equationdiff} and \ref{propchatissima}.
\end{proof}

We characterize  the function $\widetilde\Delta$ as the \emph{unique solution} of a certain integral equation.
To this end, we introduce some notation. Given a sequence $a=(a_{2k+1})_{k\geq 1}$, we define the functions
\begin{equation}\label{Is1}
\begin{split}
\mathcal{I}_{\mathrm{\Gamma}}(a)(y,\tau)&=\displaystyle\sum_{k\geq 1} a_{2k+1} e^{-i\frac{\lambda_{2k+1}}{\e}y}\sin((2k+1)\tau)\\
\mathcal{I}_{\mathrm{\Theta}}(a)(y,\tau)&=\displaystyle\sum_{k\geq 1} a_{2k+1} e^{i\frac{\lambda_{2k+1}}{\e}y}\sin((2k+1)\tau).
\end{split}
\end{equation}
We also define  the following linear operator, which is a right inverse of the operator $\widetilde\LL$ in \eqref{Nop},
\begin{equation}\label{Pop}
\mathcal{P}(f,g,h)=\left(\mathcal{P}^W(f),\mathcal{P}^\Gamma(g),\mathcal{P}^\Theta(h)\right),
\end{equation}
where
\begin{equation*}
\begin{split}
\mathcal{P}^W(f)&=\ddot{v}^h(y)\displaystyle\int_0^y \dfrac{f(s)}{\ddot{v}^h(s)} ds\\
\mathcal{P}^\Gamma(g)&=\displaystyle\sum_{k\geq 1}\mathcal{P}^\Gamma_{2k+1}(g)\sin((2k+1)\tau), \qquad \mathcal{P}^\Gamma_{2k+1}(g)(y)=\,\displaystyle\int_{y^+}^y e^{i\frac{\lambda_{2k+1}}{\e}(s-y)}\Pi_{2k+1}[g](s)ds\\
\mathcal{P}^\Theta(h)&=\displaystyle\sum_{k\geq 1}\mathcal{P}^\Theta_{2k+1}(h)\sin((2k+1)\tau),\qquad \mathcal{P}^\Theta_{2k+1}(h)(y)=\,\displaystyle\int_{y^-}^y e^{-i\frac{\lambda_{2k+1}}{\e}(s-y)}\Pi_{2k+1}[h](s)ds
\end{split}
\end{equation*}
and 
\begin{equation*}y^{\pm}=\pm i\left(\dfrac{\pi}{2} -\kappa\e\right).\end{equation*}

Using the just introduced functions and operators and recalling that by, Theorem \ref{outerthm}, $\Xi_1(0)=\partial_y\xi^u_1(0)-\partial_y\xi^s_1(0)=0$, it can be easily checked that the function $\widetilde\Delta$ must  satisfy the integral equation 
	\begin{equation}\label{fixedequationmod}
	\widetilde{\Delta}=(0,\mathcal{I}_{\mathrm{\Gamma}}(c),\mathcal{I}_{\mathrm{\Theta}}(d))+\widetilde{\mathcal{M}}(\widetilde{\Delta}),\qquad \text{with}\qquad \widetilde{\mathcal{M}}(\widetilde{\Delta})=\mathcal{P}\circ\mathcal{M}(\widetilde{\Delta}),
	\end{equation}
 where $\mathcal{M}$ is given by \eqref{Mop}
and $\mathcal{I}_{\mathrm{\Gamma}}(c)$, $\mathcal{I}_{\mathrm{\Theta}}(d)$ are given in \eqref{Is1} with
\begin{equation}\label{def:csandds}
 c_{2k+1}=\Gamma_{2k+1}(y^+)  e^{i\frac{\lambda_{2k+1}}{\e}y^+}\qquad \text{ and }\qquad  d_{2k+1}=\Theta_{2k+1}(y^-) e^{-i\frac{\lambda_{2k+1}}{\e}y^-},
\end{equation}
(note that $\Gamma(y^+,\tau)=\mathcal{I}_{\mathrm{\Gamma}}(c)(y^+,\tau)$ and $\Theta(y^-,\tau)=\mathcal{I}_{\mathrm{\Theta}}(d)(y^-,\tau)$).

Now we are ready to define the leading order of the function 
$\widetilde\Delta$. We first give some heuristic explanation. In Section 
\ref{mainthmsec}, we shall first show that $\wt M$ is small and thus we expect that the main term of  $\widetilde\Delta$ 
for the $(\Gamma,\Theta)$ is given by 
$(\mathcal{I}_{\mathrm{\Gamma}}(c),\mathcal{I}_{\mathrm{\Theta}}(d))$. Let us 
analyze how these functions behave. We do the reasoning for $\Gamma$ since the 
one for $\Theta$ is analogous. 
\[
\mathcal{I}_{\mathrm{\Gamma}}(c)(y,\tau)=\sum_{k\geq 1}\Gamma_{2k+1}(y^+)  e^{-i\frac{\lambda_{2k+1}}{\e}(y-y^+)}\sin((2k+1)\tau)
\]
Recalling $\lambda_3 = \sqrt{8-\e^2}$, $\mu_3 =2\sqrt{2}$, \eqref{cambio}, that $\Gamma_{2k+1}(y^+)=\lambda_{2k+1} \Delta_{2k+1}(y^+)+i\e\Xi_{2k+1}(y^+)$ and using Theorem \ref{matchingthm} to approximate the functions $v^{u,s}$ at the point $y=y^+$ by the corresponding solutions of the inner equation (see Theorem \ref{innerthm}) and the asymptotic formula for the difference between $\phi^{0,u}$ and $\phi^{0,s}$ at $z^+=(y^+-i\pi/2)/\e$, also in Theorem \ref{innerthm}, one has
\[
 \begin{split}
\Gamma_{3}(y^+)=&\,2\frac{\lambda_{3}}{\e}e^{-i\mu_{3}\frac{y^+-i\frac{\pi}{2}}{\eps}}\left(C_\mathrm{in}+\mathcal{O}\left(\frac{1}{\kappa}\right)\right)+\text{h.o.t}\\
\Gamma_{2k+1}(y^+)=&\,\frac{1}{\e}e^{-i\mu_{3}\frac{y^+-i\frac{\pi}{2}}{\eps}}\mathcal{O}\left(\frac{1}{\kappa}\right)+\text{h.o.t}.
\end{split}
\]
Therefore,
\[
\mathcal{I}_{\mathrm{\Gamma}}(c)(y)=\frac{2 \lambda_3}{\e}e^{-i2\sqrt{2}\frac{y-i\frac{\pi}{2}}{\eps}}\left(C_\mathrm{in} \sin 3\tau +\mathcal{O}\left(\frac{1}{\kappa}\right)\right)+\text{h.o.t}
\]
To prove Theorem \ref{maintheorem}, it suffices to justify the above leading order expansion of $\wt \Delta$.

\begin{prop}\label{prop:DifferenceDeltatilde}
Take $\kappa=\frac{1}{2\lambda_3}|\log\e|$. There exists $M>0$ independent of small $\e$ such that, for any $y\in\mathcal{R}_{\kappa}$, it holds 
\[
 \begin{split}
& |\Xi_1(y)|\leq \frac{M}{|y^2+\pi^2/4|^2} e^{-\frac{\lambda_3}{\e}\left(\frac{\pi}{2}-|\Ip(y)|\right)}, \\ 
& \left\|\Gamma (y,\tau) - \frac{2 \lambda_3}{\e}C_\mathrm{in} e^{-i2\sqrt{2}\frac{y-i\frac{\pi}{2}}{\eps}}\sin 3\tau \right\|_{\ell_1} \leq \frac{M}{\e|\log \e|} e^{-\frac{\lambda_3}{\e}\left(\frac{\pi}{2}-|\Ip(y)|\right)},
 \end{split}
 \]
for some constant $M$ independent of $\e$. Moreover $\Theta (\bar y,\tau) =\overline{\Gamma(y,\tau)}$ satisfies a similar estimate. 
\end{prop}

The proof of this proposition is deferred to Section \ref{mainthmsec}. 
Recall that  $\Xi_1(0)=\partial_y v^u(0)-\partial_y v^s(0)=0$ (see Theorem \ref{outerthm}). However, we also need to estimate this component for $y\in\mathcal{R}_{\kappa}$  to obtain the estimate of $(\Gamma,\Theta)$ due to the coupling (see Section \ref{mainthmsec}). 
The definition of $\Gamma$ and the above inequality imply inequality \eqref{distancia} except for the missing $\sin \tau$ mode, which easily follows from $\Xi_1(0)=0$ and the estimate on $\Pi_1[\Delta]$ given by Lemma \ref{propchatissima}.

\section{Estimates of the invariant manifolds: Proof of Theorem \ref{outerthm}} \label{outerdomain}

\subsection{Banach Spaces and Linear Operators}\label{Banachouter}

In this section we prove Theorem \ref{outerthm} through a fixed point argument in some appropriate Banach spaces. We consider only the unstable case, since the stable one is completely analogous.

Given $\kappa\geq 1$ and a real-analytic function $h:D^{\mathrm{out},u}_{\kappa}\rightarrow \C$ (see \eqref{outer}), we define
\begin{equation}
\label{normcosh}
\|h\|_{m,\ag}=\displaystyle\sup_{y\in D^{\mathrm{out},u}_{\kappa}\cap\{ \Rp(y)\leq -1 \}}|\cosh(y)^mh(y)|+ \displaystyle\sup_{y\in D^{\mathrm{out},u}_{\kappa}\cap\{ \Rp(y)\geq -1 \}}|(y^2+\pi^2/4)^\ag h(y)|,
\end{equation}
and given a function $\xi:D^{\mathrm{out},u}_{\kappa}\times\mathbb{T}\rightarrow \C$ which is  real analytic in $y\in D^{\mathrm{out},u}_{\kappa}$, we define
\begin{equation*}
\label{normcoshl1}
\|\xi\|_{\ell_1,m,\ag}=\displaystyle\sum_{n\geq 1}\|\Pi_n[\xi]\|_{m,\ag}
\end{equation*}
and  the Banach spaces
\begin{equation*}
\begin{split}
\mathcal{E}_{m,\ag}&=\{\xi:D^{\mathrm{out},u}_{\kappa}\rightarrow \C;\ \xi \textrm{ is real-analytic in }y, \textrm{ and } \|\xi\|_{m,\ag}<\infty \}\\
\mathcal{E}_{\ell_1,m,\ag}&=\{\xi:D^{\mathrm{out},u}_{\kappa}\times\mathbb{T}\rightarrow \C;\ \xi(y,\tau) \textrm{ is real-analytic in }y  \textrm{ and } \|\xi\|_{\ell_1,m,\ag}<\infty \}.
\end{split}
\end{equation*}

\begin{lemma}\label{propertiesnorm1}
	There exists $M>0$ depending only on $\beta$ such that, for any $g,h:D^{\mathrm{out},u}_{\kappa}\times\mathbb{T} \rightarrow\C$, it holds
		\begin{enumerate}
		\item If $\ag_2\geq\ag_1\geq 0$, then
		$$\|h\|_{\ell_1,m,\ag_2}\leq M\|h\|_{\ell_1,m,\ag_1}\quad and \quad \|h\|_{\ell_1,m,\ag_1}\leq \dfrac{M}{(\kappa\e)^{\ag_2-\ag_1}}\|h\|_{\ell_1,m,\ag_2}	.$$			
		\item If $\ag_1,\ag_2\geq 0$, and $\|g\|_{\ell_1,m_1,\ag_1},\|h\|_{\ell_1,m_2,\ag_2}<\infty$, then
		$$\|gh\|_{\ell_1,m_1 + m_2,\ag_1+\ag_2}\leq \|g\|_{\ell_1,m_1,\ag_1}\|h\|_{\ell_1, m_2, \ag_2}.$$	
	\end{enumerate}	
\end{lemma}
This lemma actually applies to general  functions $2\pi$-periodic in $\tau$, not just to odd functions. The proof of this lemma is straightforward and we omit it.

Firstly to solve the linear equation $\mathcal{L}\xi=h$, we  introduce the operator $\mathcal{G}(h)$ acting on the Fourier coefficients of $h$ as
\begin{equation*}
\mathcal{G}(h)=\displaystyle\sum_{n\geq 1}\mathcal{G}_n(h_n)\sin(n\tau), \quad \wt \GG (h) = \displaystyle\sum_{n\geq 2}\mathcal{G}_n(h_n)\sin(n\tau) = \wt \Pi [\GG (h)], 
\end{equation*} 
with
\begin{align}
\mathcal{G}_1(h_1)&=-\zeta_1(y)\displaystyle\int_{0}^y\zeta_2(s)h_1(s)ds+\zeta_2(y)\displaystyle\int_{-\infty}^y\zeta_1(s)h_1(s)ds\label{g1}\\
\mathcal{G}_n(h_n)&=-\dfrac{i\e}{2\lambda_n}e^{i\frac{\lambda_n}{\e}y}\displaystyle\int_{-\infty}^ye^{-i\frac{\lambda_n}{\e}s}h_n(s)ds +\dfrac{i\e}{2\lambda_n}e^{-i\frac{\lambda_n}{\e}y}\displaystyle\int_{-\infty}^ye^{i\frac{\lambda_n}{\e}s}h_n(s)ds,\,\, n\geq 2.\label{gn}
\end{align}
where
\begin{equation}
\label{zetas}
\zeta_{1}(y)=-2\sqrt{2}\dfrac{\sinh(y)}{\cosh^2(y)} \quad and \quad \zeta_{2}(y)= -\dfrac{\sqrt{2}}{16}\dfrac{\sinh(y)}{\cosh^2(y)}(6y-4\coth(y)+\sinh(2y)),
\end{equation}
are linearly independent solutions of 
\[
\ddot{\zeta}-\zeta+\dfrac{3(v^h)^2}{4}\zeta=0 \qquad \text{(see \eqref{Lop})}.
\]

\begin{remark} \label{R:integral}
When $-\infty$ is involved in the above integrals, it should be understood that 
the integral is along horizontal lines. As the integrands are analytic 
functions, integral paths may be modified to yield better estimates in certain 
cases. 
\end{remark}

\begin{prop}\label{prop_operators}
	The following statements hold.
	\begin{enumerate}
		\item $\partial_y\Pi_1(\mathcal{G}(\xi))(0)=0$.
		\item $\mathcal{G}\circ\mathcal{L}(\xi)=\mathcal{L}\circ\mathcal{G}(\xi)=\xi$.
		\item For any $m>1$ and $\alpha\ge 5$, there exists a constant $M>0$ independent of $\e$ and $\kappa$ such that, for every $h \in\mathcal{E}_{m,\ag}$,
		\begin{equation*}
		\label{desig}
		\left\|\mathcal{G}_1 (h)\right\|_{1,\ag-2} 
		+\left\|\partial_y \mathcal{G}_1 (h)\right\|_{1,\ag-1} \leq M\|h\|_{m,\ag}.  
		\end{equation*}
		\item For any $m\geq 1$, $\al\ge0$, there exists $M>0$ such that for every $n\ge 2$ and $h \in\mathcal{E}_{m,\ag}$,
\begin{equation*}
\label{desig2}
		\left\|\GG_n (h)\right\|_{m,\ag}\leq M\frac {\e^2}{\lambda_n^2} 
		\|h 
		\|_{m,\ag}, \quad \left\|\partial_y \GG_n (h)\right\|_{m,\ag}\leq M\frac {\e}{\lambda_n} \| h \|_{m,\ag}, \quad \left\|\partial_y \GG_n (h)\right\|_{m,\ag}\leq M\frac {\e^2}{\lambda_n^2} \| \partial_y h \|_{m,\ag}.   
		\end{equation*}		
		\end{enumerate}
		\end{prop}
		
		The proof of this proposition is deferred to Appendix \ref{app:Operator}. In particular, the last item indicates a gain of an extra order regularity in $\tau$ for $\GG(\xi)$ compared to general solutions to wave equations and an improvement in the estimate of $\partial_y \GG_n(h)$ when $\partial_y h \in \EE_{m, \alpha}$, which is a typical trading between the smoothness and the smallness in problems involving rapid oscillations.

\subsection{Fixed Point Argument}\label{fixedouter}

Now, we use Proposition \ref{prop_operators} to rewrite \eqref{xinop} as 
$\xi=\mathcal{G}\circ\mathcal{F}(\xi),$
where $\mathcal{F}$ is given in \eqref{Fop}. We analyze the operator
\begin{equation*}
{\mathcal{F}}^\sharp=\mathcal{G}\circ\mathcal{F}
\end{equation*}
defined on the closed ball 
\[
\BB_0 (R\e^2) =\left\{ \xi \in \EE_{\ell_1, 1, 3} \mid \|\xi\|_{\ell_1, 1, 3} + \|\pa_y \xi\|_{\ell_1, 1, 4} \le R \e^2\right\}
\]
for some $R>0$. 

\begin{prop}\label{tecnico}
	There exists $M, \kappa_0, \e_0>0$, such that, if $\e\in (0, \e_0)$, $R>0$, and $\kappa> \kappa_0 R^{\frac 12}$, then 
		the operator 
		$${\mathcal{F}}^\sharp: \mathcal{E}_{\ell^1,1,3}\supset 
\mathcal{B}_0(R\e^2) \rightarrow \mathcal{E}_{\ell^1,1,3}$$ 
		is well defined and satisfies 
		\begin{align*}
		\MoveEqLeft
		\| \partial_\tau^2 {\mathcal{F}}^\sharp (0)\|_{\ell_1, 1, 3} 
		+  \|\partial_\tau^2 \partial_y {\mathcal{F}}^\sharp 
(0)\|_{\ell_1, 1, 4}\leq M \e^2, \\
		\MoveEqLeft
		\| \partial_\tau^2 \wt 
\Pi[{\mathcal{F}}^\sharp(\xi)-{\mathcal{F}}^\sharp(\xi')] \|_{\ell_1, 1, 3} + 
\| \pa_\tau^2 \partial_y\wt 
\Pi[{\mathcal{F}}^\sharp(\xi)-{\mathcal{F}}^\sharp(\xi')]\|_{\ell_1, 1, 
4} 
		\leq \frac {M}{\kappa^2}\big(\|\xi -\xi'\|_{\ell_1, 1, 3} + \|\pa_y \xi -\pa_y \xi'\|_{\ell_1, 1, 4}\big), \\
\MoveEqLeft\|\Pi_1[{\mathcal{F}}^\sharp(\xi)-{\mathcal{F}}^\sharp(\xi')]\|_{1 
,3}+ \| \partial_y \Pi_1\big[ 
{\mathcal{F}}^\sharp(\xi)-\widetilde{\mathcal{F}}(\xi') \big]\|_{1, 4} \\
		\leq &\, M \frac {1+R}{\kappa^2}\left( \|\xi -\xi'\|_{\ell_1, 1, 3} + \|\pa_y \xi -\pa_y \xi'\|_{\ell_1, 1, 4}\right)+ M \left(\|\wt \Pi[\xi -\xi']\|_{\ell_1, 1, 3} + \|\pa_y \wt \Pi[ \xi - \xi']\|_{\ell_1, 1, 4}\right).   
		\end{align*} 
	\end{prop}

Notice that the above bounds on $\partial_\tau^2 \FF^\sharp (\xi)$ immediately 
implies those on $\FF^\sharp (\xi)$ as the zeroth mode is not included. 
 
\begin{proof}
	First, we rewrite the operator $\mathcal{F}$ given in $\eqref{Fop}$, in order to make explicit some cancellations. Recall that $g(u)=u^3/3+f(u)$ is given by \eqref{g}. Then,
	\[
	\begin{array}{lcl}
	\mathcal{F}(\xi) &=& -\dfrac{1}{\e^3\omega^3}g(\e\omega( \xi+v^h\sin\tau))+\left(\dfrac{(\xi_1+v^h)^3}{4} -\dfrac{3v^h\xi_1^2}{4}-\dfrac{\xi_1^3}{4}\right)\sin\tau\vspace{0.2cm}\\
	&=& -\dfrac{1}{\e^3\omega^3}\wt \Pi\left[g(\e\omega( \xi+v^h\sin\tau))\right]+\left\{-\dfrac{1}{3}\Pi_1\left[\left( (\xi_1+v^h)\sin(\tau)+\wt \Pi(\xi)\right)^3\right] \right.\vspace{0.2cm}\\
	& &\left.-\dfrac{1}{\e^3 \omega^3}\Pi_1\left[f(\e \omega( \xi+v^h\sin\tau))\right]+\dfrac{(\xi_1+v^h)^3}{4} -\dfrac{3v^h\xi_1^2}{4}-\dfrac{\xi_1^3}{4}\right\}\sin\tau\vspace{0.2cm}\\
	&=& -\dfrac{1}{\e^3\omega^3}\wt \Pi\left[g(\e \omega( \xi+v^h\sin\tau))\right]+\left\{-\dfrac{1}{3} \Pi_1\left[(\xi_1+v^h)^3\sin^3\tau
	+3(\xi_1+v^h)^2\sin^2\tau\wt \Pi[\xi]\right.\right.\vspace{0.2cm}\\ 
	& &\left.+3(\xi_1+v^h)\sin\tau(\wt \Pi[\xi])^2+(\wt \Pi[\xi])^3\right]-\dfrac{1}{\e^3\omega^3}\Pi_1\left[f(\e\omega( \xi+v^h\sin\tau))\right]\vspace{0.2cm}\\ 
	&&
	\left.+\dfrac{(\xi_1+v^h)^3}{4} -\dfrac{3v^h\xi_1^2}{4}-\dfrac{\xi_1^3}{4}\right\}\sin\tau.
	\end{array}
	\]
	Therefore,
	\begin{equation}\label{FE}	
	\begin{split}
	\mathcal{F}(\xi)	=\,& -\dfrac{1}{\e^3\omega^3}\wt \Pi\left[g(\e\omega( \xi+v^h\sin\tau))\right]+\Bigg\{\Pi_1\bigg[
	-(\xi_1+v^h)^2(\sin^2\tau) \wt \Pi[\xi]-(\xi_1+v^h)(\sin\tau)(\wt \Pi[\xi])^2\vspace{0.2cm}\\ 
	& -\dfrac{1}{3}(\wt \Pi[\xi])^3\bigg]-\dfrac{1}{\e^3\omega^3} \Pi_1\left[f(\e\omega( \xi+v^h\sin\tau))\right] -\dfrac{3v^h\xi_1^2}{4}-\dfrac{\xi_1^3}{4}\Bigg\}\sin\tau.
	\end{split}
	\end{equation}
which implies	
\[\mathcal{F}(0)= -\dfrac{1}{\e^3\omega^3}\wt \Pi\left[g(\e\omega( v^h\sin\tau))\right] - \dfrac{1}{\e^3\omega^3}\Pi_1\left[f(\e\omega(v^h\sin\tau))\right]\sin\tau.\]
Let $g$ and $f$ have the power series expansion 
\[
g(u) = \sum_{d=1}^\infty g_{2d+1} u^{2d+1}, \quad f(u) = \sum_{d=2}^\infty g_{2d+1} u^{2d+1}, \quad g_3 = \frac 13,
\]
with a positive radius of convergence. Using Lemma \ref{propertiesnorm1} and Proposition \ref{prop_operators}, one may estimate 
\begin{align*}
\MoveEqLeft[4]\|\pa_\tau^2 \wt \GG \FF (0) \|_{\ell_1, 1,3} \lesssim \e^2 \|\wt \Pi \FF (0) \|_{\ell_1, 1, 3} \lesssim \e^2 \sum_{d=1}^\infty (\e \omega)^{2d-2} |g_{2d+1}| \|(v^h\sin \tau)^{2d+1} \|_{\ell_1, 1, 3}  \\
\lesssim &\, \e^2 \sum_{d=1}^\infty \left(\frac \omega \kappa\right)^{2d-2} |g_{2d+1}| \|(v^h \sin \tau)^{2d+1} \|_{\ell_1, 2d+1,2d+1} \lesssim \e^2 \sum_{d=1}^\infty \left(\frac \omega \kappa\right)^{2d-2} |g_{2d+1}| \|v^h \|_{1,1}^{2d+1} \lesssim \e^2,  
\end{align*}
for reasonably large $\kappa$. 
In particular, in the above the operator $\partial_\tau^2$ creates a Fourier 
multiplier of $n^2$ to the mode of $\sin n\tau$, which is cancelled by the 
$\lambda_n^{-2}$ in the estimate of $\GG_n$ in Proposition \ref{prop_operators}. 
In order to obtain the desired  estimate on  $\| \partial_y \wt \GG \FF (0) 
\|_{\ell_1, 1, 4}$, we also need 
\[
\partial_y \FF(0) = -\dfrac{1}{\e^2\omega^2}g'(\e\omega( v^h\sin\tau)) (v^h)'\sin \tau = - \sum_{d=1}^\infty (2d+1)(\e \omega)^{2d-2} g_{2d+1} ( v^h)^{2d} \partial_y v^h \sin^{2d+1} \tau 
\]
which implies 
\begin{align*}
\|\partial_y \FF(0) \|_{\ell_1, 1,4} \lesssim& \sum_{d=1}^\infty (2d+1)(\e \omega)^{2d-2} |g_{2d+1}| \| (v^h)^{2d} \partial_y v^h \sin^{2d+1} \tau\|_{\ell_1, 1,4} \\
\lesssim & \sum_{d=1}^\infty (2d+1)\left(\frac \omega \kappa\right)^{2d-2} |g_{2d+1}| \| (v^h)^{2d} \partial_y v^h \sin^{2d+1} \tau\|_{\ell_1, 2d+1, 2d+2} \\
\lesssim & \sum_{d=1}^\infty (2d+1)\left(\frac \omega \kappa\right)^{2d-2} |g_{2d+1}| \|v^h \|_{1,1}^{2d} \| \partial_y v^h \|_{1,2} \lesssim 1, 
\end{align*}
for reasonably large $\kappa$. Hence the estimates related to $\| \cdot \|_{\ell_1, 1, 4}$ estimate related to $\partial_y \wt \GG \FF (0)$ follows from Proposition \ref{prop_operators}. Again, using Lemma \ref{propertiesnorm1} and Proposition \ref{prop_operators}, one may also estimate 
\begin{align*}
\MoveEqLeft[4]\|\GG_1 \Pi_1 \FF (0) \|_{1,3} \lesssim \|\Pi_1 \FF (0) \|_{3,5} \lesssim 
\sum_{d=2}^\infty (\e \omega)^{2d-2} |g_{2d+1}| \|(v^h)^{2d+1} \|_{3,5}  \\
 \lesssim & \,(\e\omega)^2 \sum_{d=2}^\infty \left(\frac \omega 
\kappa\right)^{2d-4} |g_{2d+1}| \|(v^h)^{2d+1} \|_{2d+1,2d+1} \lesssim 
(\e\omega)^2 \sum_{d=2}^\infty \left(\frac \omega \kappa\right)^{2d-4} 
|g_{2d+1}| \|v^h \|_{1,1}^{2d+1} \lesssim \e^2,  
\end{align*}
for reasonably large $\kappa$. The estimate on $\partial_y \GG \Pi_1 [\FF(0)]$ 
is obtained in a similar fashion. The sum of these inequalities imply the 
estimate on $\FF^\sharp(0)$. 

To estimate the Lipschitz constant of $\FF^\sharp$, let 
$\xi,\xi'\in\mathcal{B}_0(R\e^2),$ we have 
	\[
	\begin{split}
	\mathcal{F}(\xi)-\mathcal{F}(\xi')=&\, -\dfrac{1}{(\e\omega) ^3}\wt \Pi\left[g(\e\omega( \xi+v^h\sin\tau))-g(\e\omega( \xi'+v^h\sin\tau))\right]\\
	&\,+\Bigg\{-\Pi_1\left[(\xi_1+v^h)^2(\sin^2\tau)(\wt \Pi[\xi]-\wt \Pi[\xi'])-\left((\xi_1+v^h)^2-(\xi'_1+v^h)^2\right)(\sin^2\tau)\wt \Pi[\xi']\right]\\  
	&\,-\Pi_1\left[(\xi_1+v^h)(\sin\tau)(\wt \Pi[\xi]^2-\wt \Pi[\xi']^2)-(\xi_1-\xi'_1)(\sin\tau)\wt \Pi[\xi']^2\right] -\dfrac{1}{3}\Pi_1\left[\wt \Pi[\xi]^3-\wt \Pi[\xi']^3\right]\\
	&\,-\dfrac{1}{(\e\omega)^3}\Pi_1\left[f(\e\omega( \xi+v^h\sin\tau))-f(\e\omega( \xi'+v^h\sin\tau))\right]-\dfrac{3v^h(\xi_1^2-(\xi_1')^2)}{4}-\dfrac{\xi_1^3-(\xi'_1)^3}{4}\Bigg\}\sin\tau.
	\end{split}
	\]
For any $d\ge 2$, $m\ge 0$, $\al \ge 0$, and $\zeta, \zeta' \in \EE_{\ell_1, m, \alpha}$, 
it is straight forward to estimate 
	\begin{equation*} 
	\| \zeta^d - (\zeta')^d \|_{\ell_1, dm, d\al} \lesssim d (\|\zeta\|_{\ell_1, m, \al}^{d-1} + \| \zeta'\|_{\ell_1, m, \al}^{d-1}) \|\zeta -\zeta'\|_{\ell_1, m,\al} 	
	\end{equation*}	
where the constant is independent of $d$. Another useful inequality is 
\begin{equation*}
\| \xi \|_{\ell_1, 1, 1} + \|\pa_y \xi \|_{\ell_1, 1, 2} \lesssim (\kappa \e)^{-2}( \| \xi \|_{\ell_1, 1, 3} + \|\pa_y \xi \|_{\ell_1, 1, 4}) \lesssim  \frac R{\kappa^2} \lesssim 1, \quad \xi \in \BB_0 (R\e^2).  
\end{equation*}
Hence one may use Lemma \ref{propertiesnorm1} and Proposition \ref{prop_operators} to estimate 
	\begin{align*}
	\MoveEqLeft[4]\|\pa_\tau^2\wt \GG [\FF(\xi) - \FF(\xi')] \|_{\ell_1, 1, 3} \lesssim \e^2 \|\wt \Pi [\FF(\xi) - \FF(\xi')] \|_{\ell_1, 1, 3}\\
	 \lesssim &\,\e^2 \sum_{d=1}^\infty (\e\omega)^{2d-2} |g_{2d+1}| \|(\xi +v^h\sin\tau)^{2d+1} - (\xi' +v^h\sin\tau)^{2d+1} \|_{\ell_1, 1, 3} \\
	 \lesssim &\, \e^2 \sum_{d=1}^\infty \left(\frac \omega \kappa\right)^{2d-2} |g_{2d+1}|\|(\xi +v^h\sin\tau)^{2d+1} - (\xi' +v^h\sin\tau)^{2d+1} \|_{\ell_1, 2d+1, 2d+1} \\
	 \lesssim &\, \e^2 \sum_{d=1}^\infty d \left(\frac \omega \kappa\right)^{2d-2} |g_{2d+1}| (1 + \|\xi \|_{\ell_1, 1, 1}^{2d} +\| \xi' \|_{\ell_1,1, 1}^{2d}) \|\xi -\xi'\|_{\ell_1, 1, 1} \\
	 \lesssim &\, \kappa^{-2} \sum_{d=1}^\infty d \left(\frac \omega \kappa\right)^{2d-2} |g_{2d+1}| 
	 \|\xi -\xi'\|_{\ell_1, 1, 3} \lesssim \kappa^{-2}  \|\xi -\xi'\|_{\ell_1, 1, 3}			
	 \end{align*}
for $\kappa \ge R$ reasonably large. To estimates $\pa_y \wt \GG [\FF(\xi) - \FF(\xi')]$, in a similar fashion one needs to compute  
	\begin{align*}
	\MoveEqLeft[4]\|\pa_y \big(\FF(\xi) - \FF(\xi')\big) \|_{\ell_1, 1, 4}\\ 
	\lesssim & \sum_{d=1}^\infty d (\e\omega)^{2d-2} |g_{2d+1}| \Big\|(\xi +v^h\sin\tau)^{2d} (\pa_y \xi + \pa_y \sin \tau)  - (\xi' +v^h\sin\tau)^{2d}(\pa_y \xi' + \pa_y \sin \tau) \Big\|_{\ell_1, 1, 4}\\
	\lesssim & \sum_{d=1}^\infty d \left(\frac \omega\kappa\right)^{2d-2} |g_{2d+1}| \Big\|(\xi +v^h\sin\tau)^{2d} (\pa_y \xi + \pa_y v^h \sin \tau) \\
	& \qquad  - (\xi' +v^h\sin\tau)^{2d}(\pa_y \xi' + \pa_y v^h \sin \tau) \Big\|_{\ell_1, 2d+1, 2d+2}\\
	\lesssim & \|\xi -\xi'\|_{\ell_1, 1, 1} +  \|\pa_y \xi - \pa_y\xi'\|_{\ell_1, 1, 2}\lesssim  (\kappa\e)^{-2} ( \|\xi -\xi'\|_{\ell_1, 1, 3} +  \|\pa_y \xi - \pa_y\xi'\|_{\ell_1, 1, 4}),
		\end{align*} 
where in the derivation of the third $\lesssim$ we applied $\|\cdot\|_{\ell_1, 1,1}$ norm to all $\xi$, $\xi'$, and $v^h$ and $\|\cdot\|_{\ell_1, 1,2}$ norm to all $\pa_y \xi$, $\pa_y \xi'$, and $\pa_y v^h$. Along with Proposition \ref{prop_operators} this inequality yields the desired estimate on $\pa_y \wt \GG [\FF(\xi) - \FF(\xi')]$. The $\GG_1$ component can be estimated much as in the above. In fact, 
	\begin{align*}
	\MoveEqLeft[4]\|\GG_1 \Pi_1 [\FF(\xi) - \FF(\xi')] \|_{1, 3} + \|\pa_y \GG_1 \Pi_1 [\FF(\xi) - \FF(\xi')] \|_{1, 4}\lesssim \|\Pi_1 [\FF(\xi) - \FF(\xi')] \|_{3, 5}\\
	\lesssim & \dfrac{1}{(\e\omega)^3} \| f(\e\omega( \xi+v^h\sin\tau)) - f(\e\omega( \xi'+v^h\sin\tau))\|_{\ell_1, 3, 5} + \| \wt \Pi[\xi]-\wt \Pi[\xi']\|_{\ell_1, 1, 3} \\
	& + ( \|\xi\|_{\ell_1, 1, 1}+ \|\xi\|_{\ell_1, 1, 1}^2+  \|\xi'\|_{\ell_1, 1, 1}+ \|\xi'\|_{\ell_1, 1, 1}^2) \|\xi -\xi'\|_{\ell_1, 1, 3}
	\end{align*}
where all the $\xi$, $\xi'$, and $v^h \sin\tau$ in front of $\xi -\xi'$ were taken the $\|\cdot \|_{\ell_1, 1, 1}$ norm. The $f$ terms can be estimated much as in the above 
\begin{align*}
\MoveEqLeft[4]\dfrac{1}{(\e\omega)^3} \| f(\e\omega( \xi+v^h\sin\tau)) - f(\e\omega( \xi'+v^h\sin\tau))\|_{\ell_1, 3, 5}\\
	 \lesssim & \sum_{d=2}^\infty (\e\omega)^{2d-2} |g_{2d+1}| (\kappa \e)^{-(2d-4)} \|(\xi +v^h\sin\tau)^{2d+1} - (\xi' +v^h\sin\tau)^{2d+1} \|_{\ell_1, 2d+1, 2d+1} \\
	  \lesssim &\, \kappa^{-2} \sum_{d=1}^\infty d (\frac \omega \kappa)^{2d-4} |g_{2d+1}| 
	 \|\xi -\xi'\|_{\ell_1, 1, 3} \lesssim \kappa^{-2}  \|\xi -\xi'\|_{\ell_1, 1, 3}	
	 \end{align*}
for $\kappa \ge R$ reasonably large.  
Summarizing the above estimates, the proposition follows. 		
\end{proof}

With the above preparations, we are ready to prove Theorem \ref{outerthm}.\\

\paragraph{\textbf{Proof of Theorem \ref{outerthm}}}
	We claim that, if $\kappa$ is sufficiently large, then 
${\mathcal{F}}^\sharp$ is a contraction on the set 
\begin{align*}
S = \{ \xi \in \EE_{\ell_1, 1, 3} \mid & \|\Pi_1 [\xi]\|_{1, 3} + \|\pa_y \Pi_1 [\xi]\|_{1, 4} \le (1+M)^2 \e^2, \\ 
&\|\wt \Pi [\xi]\|_{\ell_1, 1, 3} +\|\pa_y \wt \Pi [\xi]\|_{\ell_1, 1, 4} \le (1+M) \e^2\} \subset \mathcal{B}_0(R\e^2), \quad R= (1+M)(2+M), 
\end{align*}	
equipped with the metric 
\[
| \xi|_M := \|\Pi_1[\xi]\|_{1,3} + \|\pa_y \Pi_1 [\xi]\|_{1, 4} + (1+M) \big(\|\wt \Pi [\xi]\|_{\ell_1, 1, 3} +\|\pa_y \wt \Pi [\xi]\|_{\ell_1, 1, 4}),
\]
where $M$ is the constant from Proposition \ref{tecnico}. In fact, using Proposition \ref{tecnico} it is straight forward to estimate that, for any $\xi\in S$, 
\begin{align*}
\|\Pi_1 [{\mathcal{F}}^\sharp (\xi)]\|_{1, 3} + \|\pa_y \Pi_1 [
\FF^\sharp
(\xi)]\|_{1, 4} \le & \|\FF^\sharp (0)\|_{\ell_1, 1, 3} +  \| \pa_y \FF^\sharp
(0)\|_{\ell_1, 1, 4}+ M \frac {1+R}{\kappa^2} R\e^2 + M (1+M) \e^2 \\
\le& \left(M+ M \frac {1+R}{\kappa^2} R + M(1+M) \right) \e^2\leq (1+M)^2\eps^2,  \\
\|\wt \Pi [\FF^\sharp(\xi)]\|_{\ell_1, 1, 3} +\|\pa_y \wt \Pi [\FF^\sharp
(\xi)]\|_{\ell_1, 1, 4} \le & \|\FF^\sharp (0)\|_{\ell_1, 1, 3} +  \| \pa_y 
\FF^\sharp (0)\|_{\ell_1, 1, 4}+ \frac M{\kappa^2} R \e^2\\
\le & \left(M+\frac M{\kappa^2} R\right) \e^2\leq (1+M)\eps^2, 
\end{align*}
and for any $\xi, \xi' \in S$, 
\begin{align*}
|\FF^\sharp(\xi) - \FF^\sharp(\xi') |_M \le & \left( M \frac {1+R}{\kappa^2} + 
(1+M) \frac M{\kappa^2} \right) \big( \|\Pi_1[\xi - \xi'] \|_{1, 3} + \| \wt \Pi 
[\xi -\xi']\|_{\ell_1, 1, 3} +\|\pa_y \Pi_1[\xi - \xi'] \|_{1, 4} \\
& + \|\pa_y \wt \Pi [\xi -\xi']\|_{\ell_1, 1, 4} \big) + M( \| \wt \Pi [\xi -\xi']\|_{\ell_1, 1, 3}+ \|\pa_y \wt \Pi [\xi -\xi']\|_{\ell_1, 1, 4}) \\
\le &\left( M \frac {1+R}{\kappa^2} + (1+M) \frac M{\kappa^2} + \frac M{1+M}\right) |\xi -\xi'|_M. 
\end{align*}
Therefore our above claim holds if $\kappa$ is large and $\FF^\sharp$ has a 
unique fixed point $\xi^u \in S\subset \mathcal{B}_0(R\e^2)$. It clearly 
satisfies all desired properties in Theorem \ref{outerthm}. 
Using that  $g$ given in \eqref{g} is an odd function, 
a straightforward computation shows  that the operator $\mathcal{F}$  in \eqref{Fop}
leaves invariant the subspace of functions 
$\xi:D_{\kappa}^{\mathrm{out},u}\times\mathbb{T} \rightarrow \C$ satisfying 
$\Pi_{2l}[\xi]=0$,  $\forall l\geq0$. Consequently, 
 $\xi^u$ satisfies that $\Pi_{2l}[\xi^u]=0$, $\forall l\geq0$ which completes the proof of Theorem \ref{outerthm}.

\section{The Inner equation: Proof of Theorem \ref{innerthm}}\label{innersec}

We  look for solutions odd in $\tau$  of the inner equation \eqref{inner} as
\begin{equation}\label{phi0sine}
\phi^0=\sum_{n\geq 1} \phi_n^0 \sin(n\tau).
\end{equation}
Substituting \eqref{phi0sine} into \eqref{inner}, we obtain that
\begin{equation}\label{innerfourier}(\partial_z^2+(n^2-1))\phi_n^0+\Pi_n\left[\dfrac{1}{3}(\phi^0)^3+f(\phi^0)\right]=0,\qquad  n\geq 1.\end{equation}
As explained in Section \ref{desc_sec}, we look for solutions of the form
\begin{equation}\label{asympinner} 
\phi^0(z,\tau)=\dfrac{-2\sqrt{2}i}{z}\sin(\tau)+\psi(z,\tau) \qquad \text{with}\qquad \psi=\er\left(\frac{1}{z^{3}}\right).
 \end{equation}
Then, by \eqref{innerfourier},   $\psi(z,\tau)=\sum_{n\geq 1}\psi_n(z)\sin(n\tau)$ must satisfy
\begin{equation}\label{innersystem}
\left\{
\begin{array}{l}
\pa_{z}^2 \psi_1-\dfrac{6}{z^2}\psi_1=-\Pi_1\left[-\dfrac{8}{z^2}\sin^2(\tau)\wt \Pi\left[\psi \right]-\dfrac{2\sqrt{2}i}{z}\sin(\tau)\psi^2+\dfrac{1}{3}\psi^3 +f\left(\dfrac{-2\sqrt{2}i}{z}\sin(\tau)+\psi\right)\right],\vspace{0.2cm}\\
\pa_z^2 \psi_n+\mu_{n}^2\psi_n= -\Pi_n\left[\dfrac{1}{3}\left(\dfrac{-2\sqrt{2}i}{z}\sin(\tau)+\psi\right)^3+f\left(\dfrac{-2\sqrt{2}i}{z}\sin(\tau)+\psi\right)\right],\ n\geq 2,
\end{array}
\right.
\end{equation} 
where $'=d/dz$, and $\mu_{n}=\sqrt{n^2-1}$.

We observe that the nonlinearity $f(u) = \er(|u|^5)$ in \eqref{innerfourier} does not have to be a real analytic function, so we complexify the space $\FF_r$, $r>0$, in \eqref{def:Banach:f}, into a complex Banach space 
\begin{equation}\label{def:Banach:fc} \begin{split}
\mathcal{F}_r^c=\Big\{f:  \{u \in \C: |u|<r\} & \to\C,\,f \text{ is odd, analytic, and} \, 
f(u)=\sum_{k\geq 2}f_k u^{2k+1},\, 
\|f\|_r <\infty\Big\}, 
\end{split} \end{equation}
where 
\begin{equation} \label{E:norm-r}
\|f\|_r := \sum_{k=0}^\infty |f_k|r^k, \ \text{ for } \ f(u)=\sum_{k=0}^\infty f_k u^{k}.
\end{equation}

We define the operators
\begin{align}
\mathcal{I}(\psi)=& \left(\pa_z^2 \psi_1-\dfrac{6}{z^2}\psi_1\right)\sin(\tau)+\displaystyle\sum_{n\geq 2}\left(\pa_z^2 \psi_n+\mu_{n}^2\psi_n\right)\sin(n\tau)\label{Iop}\\
\mathcal{W}(f, \psi)=& -\Pi_1\left[-\dfrac{8}{z^2}\sin^2(\tau)\wt \Pi\left[\psi \right]-\dfrac{2\sqrt{2}i}{z}\sin(\tau)\psi^2+\dfrac{1}{3}\psi^3\right]\sin(\tau)\label{Wop}
\\
&-\wt \Pi\left[ \frac 13 \left(\dfrac{-2\sqrt{2}i}{z}\sin(\tau)+\psi\right)^3\right]-f\left(\dfrac{-2\sqrt{2}i}{z}\sin(\tau)+\psi\right)
\notag\end{align}
and notice that, for $\star=u,s$, to find a solution $\phi^{0,\star}$ of \eqref{inner} satisfying \eqref{asympinner} is equivalent to find a solution $\psi^{\star}$ of the functional equation
\begin{equation}
\label{psinop}
\mathcal{I}(\psi)=\mathcal{W}(f, \psi),
\end{equation}	
which satisfies $\psi^\star\sim\mathcal{O}(z^{-3})$ for $z\in D^{\star,\mathrm{in}}_{\theta,\kappa}$, $\star=u,s$, as defined in \eqref{innerdomainsol}.
In the remainder of this section, we look for solutions of  \eqref{psinop} with such asymptotics  through a fixed point argument and analyze their dependence on $f \in \FF_r^c$.
As before, we consider only the unstable case, since the stable one is completely analogous.

\subsection{Banach Spaces and Linear Operators}\label{Banachinner}

Given $\ag\geq0$ and an analytic function $h: D^{u,\mathrm{\mathrm{in}}}_{\theta,\kappa}\rightarrow \C$, where $D^{u,\mathrm{\mathrm{in}}}_{\theta,\kappa}$ is given in \eqref{innerdomainsol}, consider the norm
\begin{equation*}
\|h\|_{\ag}=\displaystyle\sup_{z\in D^{u,\mathrm{\mathrm{in}}}_{\theta,\kappa}}|z^{\ag}h(z)|,
\end{equation*}
and the Banach space $$\mathcal{X}_{\ag}=\{h:D^{u,\mathrm{\mathrm{in}}}_{\theta,\kappa}\rightarrow \C;\ h \textrm{ is an analytic function and } \|h\|_{\ag}<\infty \}.$$
Moreover, for $h: D^{u,\mathrm{\mathrm{in}}}_{\theta,\kappa}\times\mathbb{T}\rightarrow \C$, analytic in the variable $z$, we define
\begin{equation*}
\|h\|_{\n,\ag}=\displaystyle\sum_{n\geq 1}\|h_n\|_{\ag},
\end{equation*}
and the Banach space
$$\mathcal{X}_{\n,\ag}=\left\{h: D^{u,\mathrm{\mathrm{in}}}_{\theta,\kappa}\times\mathbb{T}\rightarrow \C;\ h \textrm{ is an analytic function in the variable }z \textrm{ and }\|h\|_{\n,\ag}<\infty \right\}.$$

\begin{lemma}\label{propertiesnorm2}
Let $r>0$. 	Given an analytic function $f: \{u \in \C :  |u| < r\} \to \C$  
	and $g,h:D^{u,\mathrm{\mathrm{in}}}_{\theta,\kappa}\times\mathbb{T}\rightarrow\C$, 
	the following statements hold for some $M$ depending only on $\theta$ and $r$, 
	\begin{enumerate}
		\item If $\ag\geq\bg\geq 0$, then 		$$\|h\|_{\ell_1,\ag-\bg}\leq \dfrac{M}{\kappa^{\bg}}\|h\|_{\ell_1,\ag}.$$			
		
		\item If $\ag,\bg\geq 0$, and $\|g\|_{\ell_1,\ag},\|h\|_{\ell_1,\bg}<\infty$, then
		$$\|gh\|_{\ell_1,\ag+\bg}\leq \|g\|_{\ell_1,\ag}\|h\|_{\ell_1,\bg}.$$	

		\item If $\ag\ge 0$, $g, h \in \mathcal{X}_{\ell_1, \ag}$ and $\|g\|_{\ell_1,0}$, $\|h\|_{\ell_1,0}< r/2$, then
		$$\|f(g)-f(h)\|_{\ell_1,\ag}\leq M \|f\|_r \|g-h\|_{\ell_1,\ag}.$$
		\item Given $n\geq 0$, if $f^{(k)}(0)=0$, for every $0\leq k\leq n-1$, and $\|g\|_{\ell_1,0}< r/2$, 
		then
		$$\|f(g)\|_{\ell_1,n\ag}\leq M \|f\|_r (\|g\|_{\ell_1,\ag})^n.$$
		$M$ also depends on $n$.
		
		\item If $h\in\mathcal{X}_{\n,\ag}$ (with respect to the inner domain $ D^{u,\mathrm{\mathrm{in}}}_{\theta,\kappa}$), then $\partial_zh\in\mathcal{X}_{\n,\ag+1}$ (with respect to the inner domain $ D^{u,\mathrm{\mathrm{in}}}_{2\theta,4\kappa}$), and
		$$\|\partial_zh\|_{\n,\ag+1}\leq M\|h\|_{\n,\ag}.$$		
	\end{enumerate}	
\end{lemma}

\begin{proof} 
Items (1)(2)(5) of this lemma are proved as Lemma 4.3 in \cite{INMA}. To prove (3) and (4), let $f(u) = \sum_{k=0}^\infty f_k u^k$. One may estimate using item (2), 
\begin{align*}
\|f(g)-f(h)\|_{\ell_1,\ag} = & \Big\| (g-h) \sum_{k=0}^\infty f_{k+1} \sum_{j=0}^k g^{k-j}h^j \Big\|_{\ell_1,\ag} \leq \sum_{k=0}^\infty |f_{k+1}| \sum_{j=0}^k \|g\|_{\ell_1,0}^{k-j} \|h\|_{\ell_1,0}^{j} \|g-h\|_{\ell_1,\ag}   \\
\leq & \sum_{k=0}^\infty (k+1) |f_{k+1}| \Big( \frac r2\Big)^k \|g-h\|_{\ell_1,\ag} = \sum_{k=0}^\infty \frac {k+1}{2^k r} |f_{k+1}| r^{k+1} \|g-h\|_{\ell_1,\ag},
\end{align*} 
which implies item (3). Again based on item (2), the proof of item (4) is similar   
\[
\|f(g)\|_{\ell_1,n\ag} = \Big\| \sum_{k=n}^\infty f_{k} g^{k}\Big\|_{\ell_1,n\ag} \leq  \sum_{k=n}^\infty |f_{k}| \|g\|_{\ell_1,0}^{k-n} \|g\|_{\ell_1, \ag}^n \leq \sum_{k=n}^\infty |f_{k}| r^{k} r^{-n} \|g\|_{\ell_1, \ag}^n
\]
and thus item (4) follows. 
\end{proof}

Now, define the linear operator acting on the Fourier coefficients of $\psi$
$$\mathcal{J}(\psi)=\displaystyle\sum_{n\geq 1}\mathcal{J}_n(\psi_n)\sin(n\tau),$$
where
\begin{equation}
\label{j1}
\begin{split}
\mathcal{J}_1(\psi_1)(z)=&\,\dfrac{z^3}{5}\displaystyle\int_{-\infty}^z\dfrac{\psi_1(s)}{s^2}ds-\dfrac{1}{5z^2}\displaystyle\int_{-\infty}^z s^3\psi_1(s)ds\\
\mathcal{J}_n(\psi_n)(z)= &\,\dfrac{1}{2i\mu_{n}}\displaystyle\int_{-\infty}^z e^{-i\mu_{n}(s-z)}\psi_n(s)ds-\dfrac{1}{2i\mu_{n}}\displaystyle\int_{-\infty}^z e^{i\mu_{n}(s-z)}\psi_n(s)ds,\ n\geq 2.
\end{split}
\end{equation}
See Remark \ref{R:integral} regarding the integral paths.  

\begin{prop}\label{prop_operatorsinner}
	Consider $\kappa\geq 1$ big enough. Given $\ag> 2$, the operator $(\pa_\tau^2)\circ \mathcal{J}:\mathcal{X}_{\ell_1,\ag+2}\rightarrow \mathcal{X}_{\ell_1,\ag}$ is well defined and the following statements hold.
	\begin{enumerate}
		\item $\mathcal{J}\circ\mathcal{I}(\psi)=\mathcal{I}\circ\mathcal{J}(\psi)=\psi$.
		\item For any $\ag >2$, there exists a constant $M>0$ independent of $\kappa$ such that, for every $h \in\mathcal{X}_{\ag+2}$,
		\begin{equation*}
		\left\| \mathcal{J}_1(h)\right\|_{\ag}\leq M\|h\|_{\ag+2}.
		\end{equation*}
		\item For any $\ag >1$, there exists a constant $M>0$ independent of $\kappa$ and $n$ such that, for every $h \in\mathcal{X}_{\ag}$,
		\begin{equation*}
		\left\|\mathcal{J}_n(h)\right\|_{\ag}\leq \frac M{\mu_n^2} \|h\|_{\ag}.
		\end{equation*}		
		\end{enumerate}
	\end{prop}
	
Again the above estimates represent the gain of one more order of derivative in $\tau$. The assumption $\ag>1$ in the above last inequality ensures the convergence of the integral in the definition of $\mathcal{J}_n$ and also allows one to adjust the path of the integral in certain ways.
	
\begin{proof}
The proof of item (1) is straightforward.
For $\mathcal{J}_n$, $n\geq 2$ and $\ag>1$, one can use the same trick as in the proof of Lemma 4.6  in \cite{INMA}, by using the Cauchy Integral Theorem to move the integral paths to the rays $\{z - s e^{\pm i \theta} : s>0\}$ for the two integrals respectively, to obtain that for $h\in\mathcal{X}_{\ag}$
and  $z\in D^{u,\mathrm{\mathrm{in}}}_{\theta,\kappa}$,
	\begin{align*}
	\left|z^{\ag}\mathcal{J}_n(h)(z)\right|= & \left| \dfrac{z^{\ag}e^{i\theta} }{2i\mu_{n}}\displaystyle\int_0^{\infty} e^{i\mu_{n}s e^{i\theta}} h(z - s e^{i \theta})ds - \dfrac{z^{\ag} e^{-i\theta}}{2i\mu_{n}}\displaystyle\int_0^{\infty} e^{-i\mu_{n}s e^{-i\theta}} h(z - s e^{ -i \theta})ds\right| \\
	\leq & \frac 1{2\mu_n} \left( \displaystyle\int_0^{\infty} e^{-\mu_{n} (\sin \theta) s } |z|^{\ag} 
	|h(z - s e^{ i \theta})| ds + \displaystyle\int_0^{\infty} e^{-\mu_{n} (\sin \theta) s } |z|^{\ag} 
	|h(z - s e^{ -i \theta})| ds \right) 
	\leq \dfrac{M}{\mu_{n}^2}\|h\|_{\ag}.
	\end{align*}
For $\mathcal{J}_1$, taking $h\in\mathcal{X}_{\ag+2}$,  $\ag>2$ and $z\in D^{u,\mathrm{\mathrm{in}}}_{\theta,\kappa}$,
\[
	\begin{split}
	\left|z^{\ag}\mathcal{J}_1(h)(z)\right|=& \left|\dfrac{z^{\ag+3}}{5}\displaystyle\int_{-\infty}^z\dfrac{h(s)}{s^2}ds-\dfrac{z^{\ag-2}}{5}\displaystyle\int_{-\infty}^z s^3h_1(s)ds\right| \\
	\leq& M\|h\|_{\ag+2}\left(\displaystyle\int_{-\infty}^z\dfrac{|z|^{\ag+3}}{|s|^{\ag+4}}ds+\displaystyle\int_{-\infty}^z \dfrac{|z|^{\ag-2}}{|s|^{\ag-1}}ds\right) \leq M\|h\|_{\ag+2}.
	\end{split}
\]
	The proof of the proposition is complete. 
\end{proof}

\subsection{The fixed point argument}\label{fixedinner}
By Proposition \ref{prop_operatorsinner}, we rewrite \eqref{psinop} as 
\begin{equation*}
\label{fixedpointpsieq}
\psi= {\mathcal{W}}^\sharp(f, \psi),\quad 
{\mathcal{W}}^\sharp=\mathcal{J}\circ\mathcal{W}, \qquad f \in \FF_r^c, \;\; \|\psi\|_{\ell_1, 0} \le r,
\end{equation*}
where $\mathcal{W}$ is given by \eqref{Wop}. In the following proposition we study some properties of the operator $\mathcal{W}^\sharp$.

\begin{prop}\label{tecnicoinner}
	Given $r>0$, for big enough $\kappa\geq \max\{1, 100 /r\}$ and $R < \min\{\kappa^2,  r\kappa^3 /100\}$, the operator 
${\mathcal{W}}^\sharp: \FF_r^c \times \mathcal{B}_0(R)
\rightarrow \mathcal{X}_{\ell_1,3}$ (where $\mathcal{B}_0(R)\subset\mathcal{X}_{\ell_1,3}$ is the ball of radius $R$) is analytic in both $f$ and $\psi$  
and the following statements hold.
	\begin{enumerate}		
		\item There exists a constant $M_1>0$ 
		depending only on $\theta$ and $r$ 
such that $\|\pa_\tau^2 {\mathcal{W}}^\sharp(f, 0)\|_{\ell_1,3}\leq 
M_1(1+ \|f\|_r).$	

		\item 
		There exists a constant $M_2>1$ 
		depending only on $\theta$ and $r$ such that, for every $\psi,\psi'\in \mathcal{B}_0(R)\subset\mathcal{X}_{\ell_1,3}$, 
$$\left\|{\mathcal{W}}^\sharp(f, \psi)-{\mathcal{W}}^\sharp(f, \psi')\right\|_{
\ell_1,3}\leq M_2\left(\dfrac{1}{\kappa^2} (1+R+\|f\|_{r}) \|\psi-\psi'\|_{\ell_1,3}+\|\wt 
\Pi[\psi]-\wt \Pi[\psi']\|_{\ell_1,3}\right).$$ 
		Furthermore, 
		$$\left\|\pa_\tau^2 \left( \wt \Pi[{\mathcal{W}}^\sharp(f, \psi)]-\wt \Pi[{\mathcal{W}}^\sharp(f, \psi')]\right)\right\|_{\ell_1,3}\leq 
\dfrac{M_2}{\kappa^2}  \left(1+R+\|f\|_{r}\right)\|\psi-\psi'\|_{\ell_1,3}.$$
	\end{enumerate}			
\end{prop}

\begin{proof}
	$\mathcal{W}(f, 0)$ is given by
	\[
	\mathcal{W}(f, 0)=-\wt \Pi\left[\frac 13 \left(\dfrac{-2\sqrt{2}i}{z}\sin(\tau)\right)^3\right]-f\left(\dfrac{-2\sqrt{2}i}{z}\sin(\tau)\right).
	\]
	Thus, since $f(z)=\er(z^5)$, it follows from Lemma \ref{propertiesnorm2}(4) that
	\[
	\begin{split}\left\|\Pi_1[\mathcal{W}(f, 0)]\right\|_{5}\leq&\, M \|f\|_r \left\|\dfrac{-2\sqrt{2}i}{z}\sin(\tau)\right\|_{\ell_1,1}^5\leq M \|f\|_r, \\
	\left\|  \wt \Pi[\mathcal{W}(f, 0)]\right\|_{\n,3}\leq&\, M\left(\left\|\dfrac{-2\sqrt{2}i}{z}\sin(\tau)\right\|_{\ell_1,1}^3+\dfrac{\|f\|_r}{\kappa^2 }\left\|\dfrac{-2\sqrt{2}i}{z}\sin(\tau)\right\|_{\ell_1,1}^5\right)\leq M\left(1+ \kappa^{-2} \|f\|_r\right).
	\end{split}
	\]
Hence, from Lemma \ref{propertiesnorm2} and Proposition \ref{prop_operatorsinner}, there exists $M_1>0$ such that
	\[
 	\begin{split}
	\left\|\pa_\tau^2{W}^\sharp(f,  0)\right\|_{\ell_1,3}&\leq 
\left\|\pa_\tau^2\mathcal{J}\left(\Pi_1[\mathcal{W}(f, 0)]
\sin(\tau)\right)\right\|_{\ell_1,3}+ \left\|\pa_\tau^2\mathcal{J}\left(\wt 
\Pi[\mathcal{W}(f, 0)]\right)\right\|_{\ell_1,3}\\
	&\leq M\left(\left\|\Pi_1[\mathcal{W}(f,  0)]\right\|_{5}+ \left\|\wt \Pi[\mathcal{W}(f,  0)]\right\|_{\ell_1,3}\right)\leq M_1 \left(1+ \|f\|_r\right).
 	\end{split}
	\]
To prove item $(2)$ on the Lipschitz property, assume that $\|\psi\|_{\n,3},\|\psi'\|_{\n,3}\leq R$, and notice that
	\[
	\begin{split}
	\mathcal{W}(f,  \psi)-\mathcal{W}(f,\psi')=& -\Pi_1\left[-\dfrac{8}{z^2}\sin^2(\tau)\left(\wt \Pi\left[\psi - \psi' \right]\right)-\dfrac{2\sqrt{2}i}{z}\sin(\tau)\left(\psi^2-(\psi')^2\right)\right.\\
	&\left.+\dfrac{1}{3}\left(\psi^3-(\psi')^3\right)\right]\sin(\tau)-\dfrac{1}{3} \wt \Pi\left[\left(\dfrac{-2\sqrt{2}i}{z}\sin(\tau)+\psi\right)^3-\left(\dfrac{-2\sqrt{2}i}{z}\sin(\tau)+\psi'\right)^3\right]\\
	&-f\left(\dfrac{-2\sqrt{2}i}{z}\sin(\tau)+\psi\right)+f\left(\dfrac{-2\sqrt{2}i}{z}\sin(\tau)+\psi'\right).
	\end{split}
	\]
	Thus, again from Lemma \ref{propertiesnorm2},  
	\[
	\begin{aligned}
	\left\|\Pi_1\left[\mathcal{W}(f, \psi)-\mathcal{W}(f, \psi')\right]\right\|_{5}\leq& \left\|\dfrac{8}{z^2}\sin^2(\tau)\right\|_{\n,2}\left\|\wt \Pi\left[\psi - \psi' \right]\right\|_{\n,3}\\
	&+\left(\left\|\dfrac{2\sqrt{2}i}{z}\sin(\tau)\right\|_{\n,1}\left\|\psi+\psi'\right\|_{\n,1}+\left\|\psi^2+\psi\psi'+(\psi')^2\right\|_{\n,2}\right)\left\|\psi-\psi'\right\|_{\n,3}\\
	&+\left\|\displaystyle\int_0^1f'\left(\dfrac{-2\sqrt{2}i}{z}\sin(\tau)+s\psi+(1-s)\psi'\right)ds\right\|_{\n,2}\left\|\psi-\psi'\right\|_{\n,3}\\
	\leq&M\left(\left\|\wt \Pi\left[\psi - \psi' \right]\right\|_{\n,3}+\dfrac{1}{\kappa^2} (R+ \|f'\|_{\frac r2})\left\|\psi-\psi'\right\|_{\n,3}\right) \\
	\leq & M\left(\left\|\wt \Pi\left[\psi - \psi' \right]\right\|_{\n,3}+\dfrac{1}{\kappa^2} (R+ \|f\|_{r})\left\|\psi-\psi'\right\|_{\n,3}\right),
	\end{aligned}
	\]
	and, recalling from \eqref{g} that $g(z)= \frac {z^3}3 + f(z) =\er(z^3)$, we have that
	\[
	\begin{aligned}
	\left\|\wt \Pi\left[\mathcal{W}(f,\psi)-\mathcal{W}(f, \psi')\right]\right\|_{\n,3}\leq&\left\|\displaystyle
	\int_0^1 g'\left(\dfrac{-2\sqrt{2}i}{z}\sin(\tau)+s\psi+(1-s)\psi'\right)ds\right\|_{\n,0}\left\|\psi -\psi' \right\|_{\n,3}\\
	\leq& \dfrac{M}{\kappa^2} \|g'\|_{\frac r2} \left\|\psi-\psi'\right\|_{\n,3} \leq  \dfrac{M}{\kappa^2} (1+\|f\|_{r}) \left\|\psi-\psi'\right\|_{\n,3}.
	\end{aligned}
	\]
	Item $(2)$ follows from the estimates above and Proposition \ref{prop_operatorsinner}.
	
Finally we prove the analyticity of $\mathcal{W}^\sharp (f, \psi)$. Since $\mathcal{J}$ is linear, it suffices to show that $\mathcal{W}(f, \psi)$ is analytic in $f \in \FF_r^c$ and $\psi \in \mathcal{B}_0 (R) \subset \mathcal{X}_{\ell_1,3}$, which is equivalent to the analyticity of $(f, \psi) \to f\big(\frac{-2\sqrt{2}i}{z}\sin(\tau)+\psi\big)$  
as the analyticity of the other terms is obvious. For any $\psi_0 \in \mathcal{B}_0 (R)$, let us denote 
\[
\varphi_0 = \frac{-2\sqrt{2}i}{z}\sin(\tau) + \psi_0,  
\]
which, due to Lemma \ref{propertiesnorm2}, satisfies 
\[ 
\|\varphi_0\|_{\ell_1, 1} \le 3 +\kappa^{-2} \|\psi_0\|_{\ell_1, 3} \le 3+ \kappa^{-2}R \le 4.
\]
Consider also $f \in \FF_r^c$,
\[ 
f (u)  = \sum_{k=0}^\infty f_k u^{k}, \quad f_k \in \C,  \quad f_{k}=0, \; \forall  k\in \{j \in \N: 2| j \text{ or } j<5\}, \quad \|f\|_r < \infty,
\]
where the coefficient sequence $(f_k)$ can be viewed as the coordinates of $f$. Near $\psi_0$, one may compute 
\begin{align*}
 f\Big(\frac{-2\sqrt{2}i}{z}\sin(\tau) + \psi_0 +\psi\Big) =& f(\varphi_0 + \psi) = \sum_{k=0}^\infty f_k (\varphi_0 + \psi)^{k} =  \sum_{k=0}^\infty f_k \sum_{j=0}^{k} \frac {k!}{j! (k-j)!}\varphi_0^{k-j} \psi^{j} \\
= &  \sum_{j=0}^\infty \Big(\sum_{k=j}^{\infty}  \frac {k!}{j! (k-j)!} f_k\varphi_0^{k-j} \Big) \psi^{j} \triangleq \sum_{j=0}^\infty A_j (\varphi_0) (f, \psi),
\end{align*}
where
\[A_j (\varphi_0) (f, \psi) = \sum_{k=j}^{\infty}  \frac {k!}{j! (k-j)!} f_k\varphi_0^{k-j} \psi^{j}. \] 
The $(j+1)$-linear transformation $A_j (\varphi_0)$ of $f$ and $\psi$ can be estimated by Lemma \ref{propertiesnorm2} as, for $j =0$,
\[
\| A_0 (\varphi_0) (f, \psi) \|_{\ell_1, 3} \le \sum_{k=5}^{\infty} |f_k| \| \varphi_0^{k} \|_{\ell_1, 3} \leq \sum_{k=5}^{\infty} |f_k| r^k r^{-k} \| \varphi_0 \|_{\ell_1, 0}^{k-3} \| \varphi_0 \|_{\ell_1, 1}^3 \leq M \sum_{k=5}^{\infty} |f_k| r^k \big(\frac{r \kappa}4\big)^{3-k} \leq M \|f\|_r,
\]
and for $j \ge 1$, 
\begin{align*}
 \| A_j (\varphi_0) (f, \psi) \|_{\ell_1, 3}& \leq  \sum_{k=\max\{j, 5\}}^{\infty}  \frac {k!}{j! (k-j)!} |f_k| \| \varphi_0^{k-j} \psi^{j}\|_{\ell_1, 3} \\
&\leq  
\| \psi\|_{\ell_1, 3}\sum_{k=\max\{j, 5\}}^{\infty} |f_k| r^k \frac {k!}{j! (k-j)!}  r^{-k} \| \varphi_0\|_{\ell_1, 0}^{k-j} \|\psi\|_{\ell_1, 0}^{j-1} \\
&\leq \| \psi\|_{\ell_1, 3}^j \sum_{k=\max\{j, 5\}}^{\infty} |f_k| r^k \frac {k! 4^{k-j}}{j! (k-j)!}  r^{-k} \kappa^{-k-2j+3}. 
\end{align*}
Since, using $\frac {k!}{(k-j)!} \le k^j$, for any $k \ge \max\{j, 5\}$, 
\begin{align*}
 \frac {k! 4^{k-j}}{j! (k-j)!}  r^{-k} \kappa^{-k-2j+3} \le & \frac {1}{j! 4^j \kappa^{2j-3}} k^j \big(\frac 4{r\kappa} \big)^k = \frac {j^j}{j! 4^j \kappa^{2j-3}} \Big(\log \frac {r\kappa}4\Big)^{-j} \Big( \Big( \frac kj \log \frac {r\kappa}4\Big) e^{-\frac kj  \log \frac {r\kappa}4} \Big)^j \\
\le & \frac {j^j \kappa^3}{j! e^j} \Big(4 \kappa^2 \log \frac {r\kappa}4\Big)^{-j} \le M j^{-\frac 12} \kappa^3 \Big(4 \kappa^2 \log \frac {r\kappa}4\Big)^{-j},  
\end{align*}
where the Stirling's approximation was used in the last step. Hence, $A_j(\varphi_0)$ is a bounded multi-linear transformation, which satisfies
\[
\| A_j (\varphi_0) (f, \psi) \|_{\ell_1, 3} \leq M j^{-\frac 12} \kappa^3 \Big(4 \kappa^2 \log \frac {r\kappa}4\Big)^{-j}  \|f\|_r \| \psi\|_{\ell_1, 3}^j. 
\]
This estimate also implies the analyticity of $f\big(\frac{-2\sqrt{2}i}{z}\sin(\tau)+\psi\big)$ in $f$ and $\psi$. 
\end{proof}

\begin{proof}[Proof of Theorem \ref{innerthm}(1)]
Much as in the proof of Theorem \ref{outerthm}, we use an equivalent norm on $\mathcal{X}_{\ell_1, 3}$ 
\[
\| \psi \|_{*} : = \|\psi_1\|_3 + 2 M_2 \|\wt \Pi [\psi]\|_{\ell_1, 3},
\]
where $M_2$ is the constant resulted in Proposition \ref{tecnicoinner}(2). Let $R_0>0$ and consider $f \in \FF_r^c$ with $\|f\|_r \le R_0$. 
Using Proposition \ref{tecnicoinner}, it is straight forward to verify that, 
with in the above norm $\|\cdot\|_*$ for sufficiently large $\kappa>0$, $\mathcal{W}^\sharp$ 
is a contraction on the closed ball of $\mathcal{X}_{\ell_1, 3}$ with radius 
$R= 3M_1 (1+2M_2) (1+R_0)$ with Lipschitz constant $\kappa^{-2} (1+R+R_0) M_2 (1+2M_2) + \frac 12 <\frac 23$. The unique fixed point $\psi^u$ depends on $f \in \FF_r^c$ analytically and gives the unstable solution 
$\phi^{0, u}$ in the form of \eqref{asympinner} which satisfies the desired 
estimates.   
Using the same arguments in the proof of Theorem \ref{outerthm}, one can conclude $\Pi_{2l}[\psi^u]\equiv 0$, $\forall l\geq0$. 
\end{proof}

\subsection{The difference between the solutions of the Inner Equation}\label{difdinnersec}

This section is devoted to prove the second and third statement of Theorem \ref{innerthm}. We consider the  two solutions $\phi_0^{u,s}$ of the inner equation \eqref{inner} which are given by \eqref{innersol} and  we study the difference 
\begin{equation*}
\Delta\psi(z,\tau)=\phi^{0,u}(z,\tau)-\phi^{0,s}(z,\tau)=\psi^{u}(z,\tau)-\psi^{s}(z,\tau),
\end{equation*}
for $z\in \mathcal{R}^{\mathrm{\mathrm{in}},+}_{\theta,\kappa}= D^{u,\mathrm{\mathrm{in}}}_{\theta,\kappa}\cap D^{s,\mathrm{\mathrm{in}}}_{\theta,\kappa}\cap\{ z:\ z\in i\R \textrm{ and }\Ip(z)<0 \}$ and $\tau\in\mathbb{T}$. For this purpose, we actually work on \eqref{inner} as an ill-posed dynamical system of real independent variable along $\mathcal{R}^{\mathrm{\mathrm{in}},+}_{\theta,\kappa}$.  

\begin{remark}
	We are interested in the behavior of the difference in the connected component $\mathcal{R}^{\mathrm{\mathrm{in}},+}_{\theta,\kappa}$ of $D^{u,\mathrm{\mathrm{in}}}_{\theta,\kappa}\cap D^{s,\mathrm{\mathrm{in}}}_{\theta,\kappa}\cap i \R$ because the change $z=\e^{-1}(y-i\pi/2)$ brings the origin $y=0$ into $z=-i\e^{-1}\pi/2\in \mathcal{R}^{\mathrm{\mathrm{in}},+}_{\theta,\kappa}$. 
\end{remark}

Let $r\gg 1$. 
We define the change of variables
\begin{equation}\label{def:ChangeToReal}
z=-ir, \quad \Psi_1 (r) = \phi_1^0 (-ir), \quad \Psi_{n\pm} (r) = \pa_r \big(\phi_n^0 (-ir)\big) \pm \sqrt{n^2 -1} \phi_n^0 (-ir), \; \; n \ge 3.
\end{equation}
That is
\begin{align} 
\phi^0 (-ir,\tau) &= \Psi_1(r) \sin \tau + \sum_{n\ge 3} \frac 1{2\sqrt{n^2-1}} (\Psi_{n+}(r) - \Psi_{n-}(r)) \sin n\tau,\label{E:Psi-phi-0}\\
\pa_r \big(\phi^0 (-ir, \tau)\big)&= \pa_r \Psi_1 (r) \sin \tau + \sum_{n\ge 3} \frac 1{2} (\Psi_{n+}(r) + \Psi_{n-}(r)) \sin n\tau. \label{E:Psi-phi-1}
\end{align}
Then, equation \eqref{inner} takes the form  
\begin{equation} \label{E:inner-1} \begin{cases}
\pa_{r}^2 \Psi_1 - \frac 14 \Psi_1^3 = F_1 (\Psi) \\
\pa_r \Psi_{n\pm}= \pm \sqrt{n^2 -1} \Psi_{n\pm} + F_{n} (\Psi),
\end{cases} \end{equation}
where 
\begin{align*}
F = (F_n)_{n=1}^\infty, \qquad F_1 (\Psi) = &\Pi_1 \big[ \, \frac 13 (\phi^0)^3 + f(\phi^0)  \big]-\frac 14 \Psi_1^3, \qquad F_{n} (\Psi) =\Pi_n \big[\frac 13 (\phi^0)^3 + f(\phi^0)  \big], \quad n \ge 3.
\end{align*}
Since $\frac 14 \Psi_1^3$ in the nonlinearity is isolated into the left side 
of \eqref{E:inner-1}, the cubic terms in  $F_1$ do not include $\Psi_1^3$. 
Note that, by item (1) of Theorem \ref{innerthm}, we can restrict to the space 
of \emph{odd} $n$'s.

Let  $\Psi^{*} (r)$, $* =u, s$, be the functions $\phi^{0,\ast}$, $\ast=u,s$, expressed in the 
coordinates  introduced in \eqref{def:ChangeToReal}. We are 
interested in $\Psi^s - \Psi^u$ as $r\to +\infty$ where, since we shall 
consider certain local invariant manifolds/foliation which are not necessarily 
analytic submanifolds, we work in the space $\ell_2$ with the smooth norm  
\[
\|\Psi\|_{\ell_2}^2 :=  |\Psi_1|^2 + |\pa_r \Psi_1|^2 + \sum_{n=3, 
\text{odd}}^\infty n^2 (|\Psi_{n+}|^2 +|\Psi_{n-}|^2)  
\]
and treat $\Psi_1, \pa_r \Psi_1, \Psi_{n\pm}$ as 2-dim real vectors. We also 
define
\[
 \Psi_c=(\Psi_1, \pa_r \Psi_1),\qquad \Psi_\pm = (\Psi_{n\pm})_{n=3}^\infty.
\]
Part (1) of Theorem \ref{innerthm} implies that $\Psi^{u, s}$ do belong 
to the $\ell_2$ space. 

It is easy to see that $F$ defines a smooth mapping on the 
$\ell_2$ space. Due to both positively and negatively unbounded exponential 
growth rates caused by the linear parts, \eqref{E:inner-1} is ill-posed both 
forward and backward in $r$. However, after multiplying a smooth  cut-off 
function based on $\|\cdot\|_{\ell_2}$ to the nonlinearities $F$, 
the standard Lyapunov-Perron approach still yields 
smooth local invariant manifolds and foliations near $\Psi=0$, including an 
infinite dimensional center-stable manifold $W^{cs}$ where \eqref{E:inner-1} is 
well-posed for $r>0$ (see e.~g.~Theorem 4.4 in \cite{CL88}), the 4-dim center manifold $W^c\subset W^{cs}$ (again see Theorem 4.4 in \cite{CL88}), and stable 
fibers inside $W^{cs}$ transverse to $W^c$ (see e.~g.~Theorem 4.3 in \cite{CLL91}).\footnote{Even though the linear operators in \cite{CLL91, CL88} are assumed to be sectorial operators generating analytic semigroups, which is not satisfied by wave type PDEs, the same proofs and results still hold when the nonlinearity are smooth mappings on the phase spaces (i.~e.~without loss of regularity), which is the case of \eqref{E:inner-1} when posed in $\ell_2$ space.} 
We shall outline a framework to derive of $W^{cs}$ and $W^c$ and the stable foliation inside $W^{cs}$ for \eqref{E:inner-1}.

Following the standard cut-off technique, take $\gamma \in C^\infty( \R, \R)$ satisfying $\mathrm{supp}(\gamma) \subset (-2, 2)$ and $\gamma|_{[-1, 1]} =1$. Let $\delta>0$ and 
\[
F^\# (\Psi) = \big(F_n^\# (\Psi) \big)_{n=1}^\infty, \quad \; F_1^\# = \gamma \left(\frac{\|\Psi\|_{\ell_2}^2}{\delta^2}\right) \left(F_1 (\Psi) + \frac 14 \Psi_1^3\right), \quad F_n^\# = \gamma \left(\frac{\|\Psi\|_{\ell_2}^2}{\delta^2}\right) (F_n (\Psi)), \; n \ge 3. 
\]
Consider 
\begin{equation} \label{E:inner-2} \begin{split}
\pa_{r}^2 \Psi_1 & = F_1^\# (\Psi) \\
\pa_r \Psi_{n\pm}&= \pm \sqrt{n^2 -1} \Psi_{n\pm} + F_{n}^\# (\Psi),
\end{split} \end{equation}
whose nonlinearity has small Lipschitz constants for $\delta \ll 1$. We shall work on the global center-stable and center manifolds and stable foliations of \eqref{E:inner-2}. 
It is clear that solutions of \eqref{E:inner-1} and \eqref{E:inner-2} coincide in the $\delta$-ball of $\ell_2$ and thus we obtain local invariant manifolds and foliations of \eqref{E:inner-1} containing $\Psi^{u, s}(r)$.\\

\noindent {\bf Center-stable manifold.} The center-stable mainfold $W^{cs} = \{ \Psi_+ = h^{cs} (\Psi_c, \Psi_-)\}$  of \eqref{E:inner-2} is represented as a graph of a mapping $h^{cs}$ satisfying 
\begin{itemize} 
\item $h^{cs} \in C^4$, $D^j h^{cs} (0)=0$, $j=0, 1, 2,$ and $h^{cs}$ is odd, i.~e.~$h^{cs} (-\Psi_c, -\Psi_-) = -h^{cs} (\Psi_c, \Psi_-)$.
\item {\it Invariance}: if $\Psi_* \in W^{cs}$, then there exists a unique solution $\Psi(r) \in W^{cs}$, $r\ge 0$, to \eqref{E:inner-2} such that $\Psi(0)=\Psi_*$ and 
\[
\Psi (\cdot) \subset \EE^{cs} := \big\{ \psi \in C^0 \big([0, \infty), \ell_2 \big) \mid 
\sup_{r\ge 0} e^{-r} \| \psi(r) \|_{\ell_2} < \infty \big\}. 
\]
\item Any solution $\Psi(\cdot) \in \EE^{cs}$ to \eqref{E:inner-2} satisfies $\Psi(r) \in W^{cs}$ for any $r\ge 0$. 
\end{itemize} 
 
To outline its construction, one observes that a solution $\Psi(r)$, $r \ge 0$, to \eqref{E:inner-2}belongs to  $ \EE^{cs}$ iff 
\[\begin{aligned} 
\Psi_1 (r)& = \Psi_1 (0) + r \pa_r \Psi_1 (0) + \int_0^r (r-\tau)  F_1^\# (\Psi(\tau) d\tau, \\
\Psi_{n-} (r)& = e^{- \sqrt{n^2 -1} r} \Psi_{n-}(0)+ \int_0^r e^{- \sqrt{n^2 -1} (r-\tau)} F_n^\# (\Psi(\tau) ) d\tau, \\ 
\Psi_{n+} (r)& = - \int_r^{+\infty} e^{\sqrt{n^2 -1} (r-\tau)} F_n^\# (\Psi(\tau)) d\tau. 
\end{aligned}\]
Denote the above righthand side as $\CT\big(\Psi_c (0), \Psi_-(0), \Psi(\cdot)\big)$. Following the proof of Theorem 4.4 (mostly consisting of Lemma 3.1 -- 3.4) in \cite{CL88} (or that of Theorem 4.2 in \cite{CLL91}), one may prove that, for $\delta \ll 1$, 
\begin{itemize}
\item[a.)] $\CT$ is a contraction in $\Psi (\cdot) \in \EE^{cs}$ possessing a unique fixed point $\Psi\big(\cdot, \Psi_c (0), \Psi_-(0)\big) \in \EE^{cs}$ depending on parameters $\Psi_c (0)$ and $\Psi_-(0)$;
\item[b.)] the mapping $h^{cs}$, defined by 
\[
h^{cs} (\Psi_c (0), \Psi_-(0)) =\left. \Psi_+ \big(r, \Psi_c (0), \Psi_-(0)\big)\right|_{r=0}
\]
from this fixed point, gives the smooth center-stable manifold $W^{cs}$ invariant under \eqref{E:inner-2}. 
\end{itemize}
The oddness of $h^{cs}$ is obtained from the fact  $\Psi(\cdot) \in \EE^{cs}$ is a solution iff so is $-\Psi(\cdot) \in \EE^{cs}$ due to the oddness of \eqref{E:inner-2}.

The property $D h^{cs} (0)=0$ always holds and corresponds to the tangency of $W^{cs}$ to the center-stable subspace. 
Here the extra $D^2 h^{cs} (0)=0$ is a natural consequence of the oddness of $h^{cs}$ from that of  \eqref{E:inner-2}. More essentially, it is implied by the lack of the quadratic nonlinearity in \eqref{E:inner-2}. 
\\

\noindent {\bf Inside $W^{cs}$: the center manifold $W^c$.} Inside the center-stable invariant manifold there is the 4-dimensional center manifold $W^c= \{ \Psi \in W^{cs} \mid \Psi_-  = h^c (\Psi_c) \}$ of \eqref{E:inner-2}, which is represented as a graph of a mapping $h^{c}$ satisfying 
\begin{itemize} 
\item $h^{c} \in C^4$ is odd, $D^j h^{c} (0)=0$, $j=0, 1, 2$.
\item {\it Invariance}: if $\Psi_* \in W^{c}$, then there exists a unique solution $\Psi(r) \in W^{c}$, $r\in \R$, to \eqref{E:inner-2} such that $\Psi(0)=\Psi_*$ and 
\[
\Psi (\cdot) \subset \EE^{c} := \big\{ \psi \in C^0 \big(\R, \ell_2 \big) \mid 
\sup_{r\in \R} e^{-|r|} \| \psi(r) \|_{\ell_2} < \infty \big\}. 
\]
\item Any solution $\Psi(\cdot) \in \EE^{c}$ to \eqref{E:inner-2} satisfies $\Psi(r) \in W^{c}$ for any $r\in \R$. 
\end{itemize} 
Due to the invariance of $W^{cs}$, when restricted to $W^{cs}$, \eqref{E:inner-2} is equivalent to 
\begin{equation} \label{E:inner-cs} \begin{cases}
\pa_{r}^2 \Psi_1 = F_1^{cs} (\Psi_c, \Psi_-) \\
\pa_r \Psi_{n-}= - \sqrt{n^2 -1} \Psi_{n-} + F_{n}^{cs} (\Psi_c, \Psi_-),
\end{cases} \end{equation} 
where 
\[
F^{cs} (\Psi_c, \Psi_-) = \big( F_n^{cs} (\Psi_c, \Psi_-) \big)_{n=1}^\infty := F^\# \big(\Psi_c, \Psi_-,   h^{cs} (\Psi_c, \Psi_-) \big). 
\]
The construction of $W^c$ is essentially that of an unstable manifold in $W^{cs}$ and thus again can follow from Theorem 4.4 in \cite{CL88} as illustrated in the above framework for $W^{cs}$.
\\

\noindent {\bf Inside $W^{cs}$: stable foliation and fiber coordinates.}
The invariant foliation theorem (e.~g.~Theorem 4.3 in \cite{CLL91}) implies that, for $\delta \ll1$, there exist  $h^s (\tilde \Psi_c, \tilde \Psi_-) \in \R^4$  (which we call the stable foliation mapping) and $\Psi= \Gamma(\tilde \Psi_c, \tilde \Psi_-)$ on $W^{cs}$ (which we call  the stable fiber coordinate system) such that 
\begin{itemize} 
\item $h^{s} \in C^4$ is odd, $D^j h^{s} (0)=0, \; j=0, 1, 2$, and $h^s( \tilde \Psi_c, 0) =0$ for any $\tilde \Psi_c$. 
\item $\Psi = (\Psi_c, \Psi_-, \Psi_+)= \Gamma( \tilde \Psi_c, \tilde \Psi_-)$ is defined as
\begin{equation} \label{E:fibercoord}
\Psi_c = \tilde \Psi_c + h^s (\tilde \Psi_c, \tilde \Psi_-), \quad \Psi_- = h^c (\tilde \Psi_c) + \tilde \Psi_-, \quad \Psi_+ = h^{cs} (\Psi_c, \Psi_-).
\end{equation}  
\item Let $\Psi_j(r) =  \Gamma( \tilde \Psi_{c, j}(r), \tilde \Psi_{-, j}(r))$, $j=1,2$, $r >0$, be solutions to \eqref{E:inner-2} and $\tilde \Psi_{c, 1}(0) = \tilde \Psi_{c, 2}(0)$, then 
\begin{itemize} 
\item $\tilde \Psi_{c, 1}(r) = \tilde \Psi_{c, 2}(r)$ for all $r\ge 0$ ({\it invariance}), and 
\item there exists $M>0$ depending only on $f$ such that 
\[
\|\tilde \Psi_{-, 1}(r) - \tilde \Psi_{-, 2}(r) \|_{\ell_2} \le M e^{-2r} \|\tilde \Psi_{-, 1}(0) - \tilde \Psi_{-, 2}(0) \|_{\ell_2}, \quad \forall r\ge 0.
\] 
\end{itemize} 
\item Consequently $W^c = \{ \Gamma(\tilde \Psi_c, 0) :  \tilde \Psi_c \in \R^4\}$ and, if $\Psi(r) =  \Gamma( \tilde \Psi_c(r), \tilde \Psi_-(r)) \in W^{cs}$, $r >0$, is a solution to \eqref{E:inner-2}, then $\Psi_b(r) =  \Gamma( \tilde \Psi_c(r), 0) \in W^{c}$ is also a solution to \eqref{E:inner-2}, called the base solution of $\Psi(r)$.  
\end{itemize}
For each $\tilde \Psi_c$, the  submanifold given by the image $\Gamma (\tilde \Psi_c, \cdot)$ is often referred to as a stable fiber. 

Note that the functions $h^{cs}$ and $h^c$ have been already obtained. Therefore, to construct $\Gamma$ we only need to show the existence of $h^s$. To this end, we only need to work  with \eqref{E:inner-cs}. Let $(\Psi_c^c (r), \Psi_-^c (r) = h^c( \Psi_c^c(r)) \in W^c$ be solution to \eqref{E:inner-cs}. One may compute that $(\Psi_c^c (r), \Psi_-^c(r)) + (\tilde \Psi_c (r), \tilde \Psi_-(r))$ where 
\[
(\tilde \Psi_c (\cdot), \tilde \Psi_-(\cdot)) \in \EE^{ss} = \big\{  (\tilde \psi_c, \tilde \psi_-) \in C^0([0, + \infty), \ell_2) \mid  \sup_{r\ge 0} e^{2r} \| (\tilde \psi_c, \tilde \psi_-) \|_{\ell_2} < \infty\},
\]
is also a solution  to \eqref{E:inner-cs} iff, for all $r \ge 0$, 
\[\begin{split} 
\tilde \Psi_1 (r)& = - \int_r^{+\infty} (r-\tau) \left( F_1^{cs} (\Psi_c^c(\tau) + \tilde \Psi_c (r), \Psi_-^c(\tau) + \tilde \Psi_- (r)) - F_1^{cs} (\Psi_c^c(\tau), \Psi_-^c(\tau))  \right) d\tau, \\
\tilde \Psi_{n-} (r)& = e^{- \sqrt{n^2 -1} r}  \tilde \Psi_- (0) + \int_0^r e^{- \sqrt{n^2 -1} (r-\tau)} \left( F_n^{cs} (\Psi_c^c(\tau) + \tilde \Psi_c (r), \Psi_-^c(\tau) + \tilde \Psi_- (r)) - F_n^{cs} (\Psi_c^c(\tau), \Psi_-^c(\tau))  \right)  d\tau.  
\end{split} \]  
Following the proof of Theorem 4.3 (mostly contained in Section 3) in \cite{CLL91}), for $\delta \ll 1$, 
\begin{itemize}
\item [a.)] the above right side is a contraction in $\EE^{ss}$ possessing a unique fixed point $\big(\tilde \Psi_c \big(\cdot, \tilde \Psi_-(0)\big), \tilde \Psi_- \big(\cdot, \tilde \Psi_-(0)\big) \in \EE^{ss}$ depending on the parameter $\tilde \Psi_- (0)$; 
\item [b.)] the desired mapping $h^s$ is given by $h^{s} (\tilde \Psi_-(0)) = \tilde \Psi_c \big(r, \tilde \Psi_-(0)\big)|_{r=0}$. The oddness of $h^{s}$ is also obtained from the oddness of \eqref{E:inner-cs}.
\end{itemize}

\noindent {\bf Splitting estimates.} By item (1) of  Theorem \ref{innerthm} and \eqref{def:ChangeToReal}, the stable/unstable solutions $\Psi^{u,s}$ to \eqref{E:inner-1} satisfy $\lim_{r\to +\infty}\Psi^{u,s}(r,\tau)=0$ and therefore they belong to the center-stable manifold $W^{cs}$. Thus, we can express them in the stable fiber coordinates,
\begin{equation} \label{E:fiber-coord}
\Psi^{u,s} (r) = \Gamma\big( \tilde \Psi_c^{u, s}(r), \tilde \Psi_-^{u,s} (r) \big), \quad \tilde \Psi_c^{u,s} (r) = \big(\tilde \Psi_1^{u,s} (r), \pa_r \tilde \Psi_1^{u, s} (r) \big),
\end{equation} 
and let $\Psi_b^{u, s}$ be their base points 
\begin{equation*}
\Psi_b^{u,s} (r) = \Gamma(\tilde \Psi_c^{u, s} (r), 0) = \Big(\tilde \Psi_c^{u, s} ( 
r), h^c \big(\tilde  \Psi_c^{u, s} ( r)  \big), h^{cs} \big(\tilde \Psi_c^{u, s} 
( r) ,  h^c \big( \tilde \Psi_c^{u, s} ( r)  \big)\big)\Big) \in W^c, 
\end{equation*}
which are solutions to \eqref{E:inner-1} themselves and satisfy
\[
\|\Psi^{u, s} (r) - \Psi_b^{u, s} (r)\|_{\ell_2} \le \OO(e^{-2r}), \; \text{ as 
} r\to +\infty.  
\]

\begin{lemma} \label{L:base} 
$\tilde \Psi_c^u (r) = \tilde \Psi_c^s (r)$. 
\end{lemma} 

\begin{proof}
From item (1) of Theorem \ref{innerthm} and Lemma \ref{propertiesnorm2}, $\pa_\tau \Psi^{u,s}(r)$ have exactly the same leading order term proportional to $r^{-1} \sin \tau$ with remainders of $\OO(r^{-3})$ in $\ell_2$ metric and $\pa_r \Psi^{u, s} (r)$ with remainders of $\OO(r^{-4})$. Since $Dh^c(0)=0$ and $Dh^{cs}(0)=0$, we have, for $r \gg 1$,   
\begin{align*} 
\OO(r^{-3}) \ge & \| \Psi^u (r) - \Psi^s (r) \|_{\ell_2} \ge \|\tilde \Psi_b^u (r) - \tilde \Psi_b^s (r)\|_{\ell_2} - \OO(e^{-2r}) \ge \frac 12 |\tilde \Psi_c^u (r) -\tilde \Psi_c^s (r)| - \OO(e^{-2r}), 
\end{align*}
and thus 
\begin{equation} \label{E:approx-1} 
|\tilde \Psi_c^u (r) - \tilde \Psi_c^s (r)|  \le \OO(r^{-3}). 
\end{equation}
 Let 
\[
\big( \tilde \beta_1 (r), \pa_r \tilde \beta_1 (r)\big) =\tilde \Psi_c^u (r) - \tilde \Psi_c^s (r), \quad B (r) = \sup_{r'\ge r} (r')^3 |\tilde \Psi_c^u (r') - \tilde \Psi_c^s (r') | <\infty,
\]
where \eqref{E:approx-1} is also used. 
Recall that $\Psi_b^{u,s}(r)$ are solutions to \eqref{E:inner-1} contained in the center manifold $W^c$, governed by the dynamics of their center coordinates $\tilde \Psi_c^{u, s}(r)$.
Substituting $h^c$ and $h^{cs}$ into the term $F_1(\Psi)$ in \eqref{E:inner-1}, using $Dh^c(0)=0$ and $Dh^{cs}(0)=0$ along with the leading order expansion of $\Psi^{u, s} (r)$ corresponding to \eqref{innersol}, and observing that the cubic nonlinearity $F_1(\Psi)$ does not contain the term $\Psi_1^3$ in its Taylor expansion, we have  
\[
\pa_r^2 \tilde \beta_1 - \frac 6{r^2} \tilde \beta_1 = \tilde G(r) = \OO\left(\frac{1}{r^{3}} (|\tilde \beta_1| + |\pa_r \tilde \beta_1|) \right) \le \OO \left( \frac{B(r)}{r^{6}} \right). 
\]
As in the definition of $\mathcal{J}_1$ in \eqref{j1}, a fundamental set of solutions of $\pa_r^2 \tilde \beta_1 - \frac 6{r^2} \tilde \beta_1 =0$ are given by $r^{-2}$ and $r^3$. Therefore the general solutions of the above equation is 
\[
\tilde \beta_1 (r) = c_1 r^{-2} + c_2 r^3 +\frac{r^3}{5}\displaystyle\int_{+\infty}^r \frac{\tilde G(s)}{s^2}ds-\frac{1}{5r^2}\displaystyle\int_{+\infty}^r s^3\tilde G(s)ds, 
\]
which implies 
\[
|\tilde \beta_1 (r) - c_1 r^{-2} - c_2 r^3| \le \OO\big(r^{-4} B(r) \big).
\]
In the view of \eqref{E:approx-1}, we conclude $c_1 =c_2=0$ and thus $|\tilde \beta_1 (r)| \le \OO\big(r^{-4} B(r) \big)$. In turn it also implies $|\pa_r \tilde \beta_1 (r)| \le \OO\big(r^{-5} B(r) \big)$ and leads to a contradiction to the definition of $B(r)$ for $r\gg1$, unless $B\equiv 0$. The lemma is proved. 
\end{proof}

Finally we are ready to prove the estimate on the difference between $\Psi^{u,s}(r)$. 

\begin{proof}[Proof of item (2) Theorem \ref{innerthm}]
Due to Lemma \ref{L:base}, to complete the proof of the theorem, we need to estimate 
\[
\tilde \beta_- = (\tilde \beta_{n-})_{n=3}^{+\infty} := \Psi_-^u (r) - \Psi_-^s(r) = \tilde \Psi_-^u(r) - 
\tilde \Psi_-^s (r) = \big(\tilde \Psi_{n-}^u(r) - \tilde \Psi_{n-}^s (r) 
\big)_{n=3}^{+\infty},  
\]
where the second equal sign is due to the definition \eqref{E:fibercoord} of the stable fiber coordinate system $\Gamma$. From \eqref{E:inner-1} and using $\|D F (\Psi) \|_{L(\ell_2)} = \er(\|\Psi \|_{\ell_2}^2)$ for $\|\Psi \|_{\ell_2}\ll 1$ and $\|\Psi^{u, s} (r)\|_{\ell_2} = \er(\frac 1r)$ for $r\gg 1$,  we have 
\[
\pa_r \tilde \beta_{-} = A \tilde \beta_- + \wt\Pi \big[ F\big( \Gamma(\tilde \Psi_c^{u,s}(r), \tilde \Psi_-^u (r))\big) - F\big( \Gamma(\tilde \Psi_c^{u,s}(r), \tilde \Psi_-^s (r))\big) \big] \triangleq  A \tilde \beta_- + \wt A_-(r) \tilde \beta_-,
\]
where 
\[
\begin{split}
A \Psi_- &= \big(  - \sqrt{n^2 -1} \Psi_{n-} \big)_{n=3}^{+\infty}\\
\big(\wt A_-(r) \Psi_-\big)_{n} &= \Big( \int_0^1 \wt\Pi \big[ ((D F_n) \circ \Gamma) 
D_{\tilde \Psi_-} \Gamma\big]\big(\tilde \Psi_c^{u,s}(r), (1-\tau) \tilde \Psi_-^s (r) + \tau \tilde \Psi_-^u (r)\big) d\tau \Big) \Psi_-,
\end{split}
\]
which satisfies
\[
 \|\wt A_- (r)\|_{L(\ell_2)} = \OO\left(r^{-2}\right).  
\]
Consequently,
\[
\pa_r \big(e^{\sqrt{8} r} \tilde \beta_-\big) = (A + \sqrt{8}) e^{\sqrt{8} r}  \tilde \beta_- + \wt A_-(r)e^{\sqrt{8} r}  \tilde \beta_-. 
\]
As $A+ \sqrt{8} \le 0$, $ \|\wt A_- (r)\|_{L(\ell_2)} = \OO(r^{-2})$  implies 
\[
\sup_{r\ge 0} \left\{e^{\sqrt{8} r} \|\tilde \beta_-(r)\|_{\ell_2}\right\} < +\infty.
\]
Write $e^{\sqrt{8} r} \tilde \beta_-$ using the variation of constants formula, 
\[
e^{\sqrt{8} r} \tilde \beta_- (r) = e^{r (A+\sqrt{8})}
\tilde \beta_- (0) + \int_{0}^{r}   e^{ (r-r') (A+\sqrt{8})}\wt A_-(r') e^{\sqrt{8} 
r'}  \tilde \beta_-(r') dr'.  
\]
Now, since $\Pi_3(A+\sqrt{8})=0$ and 
$(I - \Pi_3)(A+\sqrt{8})\leq - \sqrt{24}+\sqrt{8}<-2$, one may estimate, for $ r\gg 1$, 
\begin{align*}
\|(I-\Pi_3) e^{\sqrt{8} r} \tilde \beta_- (r) \|_{\ell_2} \le &\, e^{-2 r} \|\tilde \beta_- (0) \|_{\ell_2} + \sup_{r''\ge 0} \|e^{\sqrt{8} r''}  \tilde \beta_-(r'') \|_{\ell_2}\int_{0}^{r}   e^{-2 (r-r') } \|\wt A_-(r')\|_{L(\ell_2)} dr'  \\
\le &\, e^{-2 r} \|\tilde \beta_- (0) \|_{\ell_2} + \int_{0}^{\frac r2} \er (e^{-2 (r-r') }) dr' + \int_{\frac r2}^r \er \left(\frac {e^{-2 (r-r') }}{r^2}\right) dr'\\
\le &\, \er \left(\frac 1{r^2} \right).  
\end{align*}
Defining
\[
\tilde C_{\inn} = 
\tilde \beta_{3-} (0) + \int_{0}^{+\infty} \Pi_3\left[\wt  A_-(r') e^{\sqrt{8} r'}  \tilde \beta_-(r')\right]dr', 
\]
which converges since $ \|\wt A_- (r)\|_{L(\ell_2)} = \OO(r^{-2})$, then we obtain 
\[
\| e^{\sqrt{8} r} \tilde \beta_{3-} (r) -\tilde C_{\inn} \|_{\ell_2} \le \sup_{r''\ge 0} \|e^{\sqrt{8} r''}  \tilde \beta_-(r'') \|_{\ell_2} \int_r^{+\infty} \|A_-(r') \|_{L(\ell_2)} dr'  \le \er \left(\frac 1r\right).
\]

To complete the proof of Theorem \ref{innerthm}, we need to estimate 
$e^{\sqrt{8} r} (\phi^{0, u} - \phi^{0, s})$ and $\pa_r \big(e^{\sqrt{8} r} 
(\phi^{0, u}- \phi^{0, s})\big)$. 
From \eqref{E:Psi-phi-0} and Lemma \ref{L:base}, 
\begin{align*}
e^{\sqrt{8} r} \Delta \phi^{0}(-ir) = & e^{\sqrt{8} r} (\phi^{0, u} (-ir)  - \phi^{0, s}(-ir)) \\
= &e^{\sqrt{8} r}(\Psi_1^u (r ) - \Psi_1^s (r)) \sin \tau + \sum_{n\ge 3} \frac {e^{\sqrt{8} r}}{2\sqrt{n^2-1}} \Big( \Psi_{n+}^u(r)  - \Psi_{n+}^s (r) - \tilde \beta_{n-} (r) \Big) \sin n\tau.
\end{align*}
Therefore,  for $C_{\inn} = -\frac {\sqrt{2}}8 \tilde C_{\inn}$, from the definition of $\tilde \beta_-(r)$ and the cubic leading order of $h^{c, s, cs}$ in the stable fiber coordinates $\Gamma(\tilde \Psi_c, \tilde \Psi_-)$, we have  
\begin{align*}
\| \pa_\tau \big(e^{\sqrt{8} r} \Delta \phi^{0}(-ir)  - C_{\inn} \sin 3\tau\big)  \|_{\ell_1} 
\le & e^{\sqrt{8} r} \big| \big(h^s (\tilde \Psi_c^{u, s}(r), \tilde \Psi_-^u(r)) - h^s ( \tilde \Psi_c^{u, s}(r), \tilde \Psi_-^{s}(r)\big) \big|  \\
&  +  \|e^{\sqrt{8} r} \big(\Psi_+^u(r) - \Psi_+^{s}(r)\big) \|_{\ell_1} + \| e^{\sqrt{8} r} (\tilde \beta_{n-})_{n>3} \|_{\ell_1} + |e^{\sqrt{8} r} \tilde \beta_{3-} (r) -\tilde C_{\inn}| \\
\le & M\big(r^{-2} e^{\sqrt{8} r}\| \tilde \beta_- (r) \|_{\ell_2} + \| e^{\sqrt{8} r} (\tilde \beta_{n-})_{n>3} \|_{\ell_2} \big)+ \er (r^{-1}) \le \er (r^{-1}). 
\end{align*} 
Similarly from \eqref{E:Psi-phi-1}
\begin{align*} 
\|\pa_r \big(e^{\sqrt{8} r} \Delta \phi^{0}(-ir))\|_{\ell_1} = & \Big\|e^{\sqrt{8} r}\sum_{n\ge 3} \Big( \frac {\sqrt{n^2-1} + \sqrt{8}}{2\sqrt{n^2-1}} ( \Psi_{n+}^u(r)  - \Psi_{n+}^s (r)) + \frac {\sqrt{n^2-1} - \sqrt{8}}{2\sqrt{n^2-1}} \tilde \beta_{n-} (r) \Big) \sin n\tau \\
& + \pa_r \big( e^{\sqrt{8} r} (\Psi_1^u (r ) - \Psi_1^s (r)) \big) \sin \tau\Big\|_{\ell_1} \\
\le & M\big(r^{-2} e^{\sqrt{8} r}\| \tilde \beta_- (r) \|_{\ell_2} + \| e^{\sqrt{8} r} (\tilde \beta_{n-})_{n>3} \|_{\ell_2} \big) \le \er (r^{-2}). 
\end{align*}
Since $\Delta \phi^0(z)$ is analytic, the estimate on $\pa_r \big(e^{\sqrt{8} r} \Delta \phi^{0}(-ir) \big)$ implies the same estimate on $\pa_z \big(e^{i\sqrt{8} z} \Delta \phi^{0}(z) \big)$ and this completes the proof of Theorem \ref{innerthm}.  
\end{proof}

Finally, we prove that the Stokes constant is analytic with respect to $f\in \FF_r^c$

\begin{proof}[Proof of item (3) of Theorem \ref{innerthm}]
This proof is based on the analytic dependence of $\phi^{0, \star} (z, \tau)$, $\star=u, s$, on $f \in \FF_r^c$ in Theorem \ref{innerthm}(1) and the asymptotics \eqref{diffinnersol} in Theorem \ref{innerthm}(2) just proven.
For $s \in (-\infty, -\kappa)$, let 
\[
\widetilde C(s, f) = \frac 1\pi   \int_{-\pi}^\pi e^{-\mu_3 s} \Delta \phi^0 (is, \tau) \sin 3\tau d\tau,
\]
which is complex analytic in $f \in \FF_r^c$. From \eqref{diffinnersol} we have $C_{\mathrm{in}} (f)= \lim_{s \to -\infty} \widetilde C(s, f)$ uniformly in $f$ with $\|f\|_r \le R_0$. So we obtain the analyticity of $C_{\mathrm{in}}(f)$ in $f \in \FF_r^c$, which also implies its analyticity in $f \in \FF_r$.  
\end{proof}

\section{Complex matching estimates: Proof of Theorem \ref{matchingthm}}\label{sec:matching}

As usual, we consider only the unstable case, and in order to simplify the notation, we omit the superscript ``$u$" of the solutions. Moreover, in this section, we use  the domain $D^{\mathrm{\mathrm{mch}},u}_{+,\kappa}$ instead of $D^{u,\mathrm{\mathrm{in}}}_{\theta,\kappa}$  (see \eqref{innerdomainsol} and \eqref{matchingdomain}) but we work on the same notation for  the norms and Banach spaces introduced in Section \ref{Banachinner}.

\begin{prop}\label{errorequationlem}
Let $\phi(z, \tau)$ and $\phi^0(z, \tau)$ be solutions to \eqref{kleingordonphi} and \eqref{inner}, respectively.  
	The function $\p:D^{\mathrm{\mathrm{mch}},u}_{+,\kappa}\times\mathbb{T}\rightarrow \C $
	defined as \begin{equation}
		\label{difference}
		\varphi(z,\tau)=\phi(z,\tau)-\phi^0(z,\tau).
	\end{equation}  satisfies the following differential equation
	\begin{equation}
		\label{eqmatch}
		\mathcal{I}(\p)(z,\tau)= \left(L(\p)(z)+\widehat{L}(\wt \Pi[\p])(z)\right)\sin(\tau) + K(\p)(z,\tau)+ \mathcal{C}_{\mathrm{mch}}(z,\tau),
	\end{equation}
	where $\mathcal{I}$ is the operator given by \eqref{Iop}, $L:\mathcal{X}_{\n,\ag}\rightarrow\mathcal{X}_{\ag+4}$, $\widehat{L}:\mathcal{X}_{\n,\ag}\rightarrow\mathcal{X}_{\ag+2}$, and $K:\mathcal{X}_{\n,\ag}\rightarrow\mathcal{X}_{\n,\ag+2}$ are linear operators and $\mathcal{C}_{\mathrm{\mathrm{mch}}}:D^{\mathrm{\mathrm{mch}},u}_{+,\kappa}\times\mathbb{T}\rightarrow \C$ is an analytic function in the variable $z$.
	Moreover, $\Pi_1\circ K\equiv 0$ and there exists a constant $M>0$ independent of $\e$ and $\kappa$ such that, for $0<\cg<1$, $\e$ sufficiently small and $\kappa$ big enough
	\begin{enumerate}
		\item $\left\|\Pi_1[\mathcal{C}_{\mathrm{\mathrm{mch}}}]\right\|_{4}\leq M\e^{3\cg-1}$ and $\left\|\pa_\tau^2 \wt \Pi[\mathcal{C}_{\mathrm{\mathrm{mch}}}]\right\|_{\n,3}\leq M\e^2;$
		\item $\|L(\p)\|_{\ag+4}\leq M\|\p\|_{\n,\ag}$; 
		\item $\|\widehat{L}(\p)\|_{\ag+2}\leq M\|\p\|_{\n,\ag}$;
		\item $\| K(\p)\|_{\n,\ag+2}\leq M\| \p\|_{\n,\ag}$, $j=0,1,2$.
	\end{enumerate}
\end{prop}

\begin{proof}
	Since $\phi$ and $\phi^0$ satisfy \eqref{kleingordonphi} and \eqref{inner}, respectively, we have that $\p(z,\tau)$ satisfies
	\begin{equation}
		\label{errorequation}
		\partial_z^2\p -\partial_{\tau}^2\p-\p=\e^2\phi-\dfrac{1}{3}(\phi^3-(\phi^0)^3)- \frac 1{\omega^3} f(\omega \phi)+f(\phi^0), \quad \omega = (1+\e^2)^{-\frac 12}.
	\end{equation}
	Now, recall that $\phi(z,\tau)=\e v(i\pi/2+\e z,\tau)$, where $v(y,\tau)= v^h(y)\sin(\tau)+\xi(y,\tau)$, $v^h$ is given by \eqref{homoclinic} and $\xi$ is given by Theorem \ref{outerthm}. 	An easy computation shows that
	$$\e v^h(i\pi/2+\e z)=-\dfrac{2\sqrt{2}i}{z} +l_1(z),$$ 
	where $l_1$ is an analytic function such that $|l_1(z)|\leq M\e^2|z|$, for each $z\in D^{\mathrm{\mathrm{mch}},u}_{+,\kappa}$. Thus,
	\begin{equation}\label{cotaouter}\phi(z,\tau)=-\dfrac{2\sqrt{2}i}{z}\sin(\tau) +l_1(z)\sin(\tau)+\e\xi(i\pi/2+\e z,\tau).\end{equation}
	Using Theorem \ref{outerthm}
	and  $y=i\pi/2+\e z$, we have 
	$$\|\e\partial^2_\tau\xi(i\pi/2+\e z,\tau)\|_{\n,3}\leq \dfrac{1}{\e^2}\|\partial^2_\tau\xi(y,\tau)\|_{\n,1,3}\leq M,$$
	where $\|\cdot\|_{\n,1,3}$ is the norm introduced in Section \ref{Banachouter}.
	
	Since $M\kappa\leq |z|\leq M\e^{\cg-1}$ for every $z\in D^{\mathrm{\mathrm{mch}},u}_{+,\kappa}$, it holds 
	\begin{equation}\label{cotainner}\left\| \pa_\tau^2 \left( \phi^0(z,\tau)+\dfrac{2\sqrt{2}i}{z}\sin(\tau) \right) \right\|_{\n,3}\leq M,\end{equation}
	 and using that $\|\phi\|_{\n,1},\|\phi^0\|_{\n,1}\leq M$, $f(z)=\er(z^5)$, we obtain from the Mean Value Theorem that
	\begin{equation}\label{cubica}
		\begin{array}{lcl}
			-\dfrac{1}{3}\left(\phi^3-(\phi^0)^3\right)- f(\phi)+f(\phi^0)&=&-\dfrac{1}{3}\left(\phi^2+\phi\phi^0+(\phi^0)^2\right)\p-\p\displaystyle\int_0^1f'(s\phi+(1-s)\phi^0)ds	\vspace{0.2cm}\\
			&=&\dfrac{6}{z^2}\Pi_1[\p]\sin(\tau)- \dfrac{2}{z^2}\Pi_1[\p]\sin(3\tau) +l_2(\p) +l_3(\wt \Pi\left[\p\right]),
		\end{array}
	\end{equation}
	where $l_2:\mathcal{X}_{\n,\ag}\rightarrow \mathcal{X}_{\n,\ag+4}$ and  $l_3:\mathcal{X}_{\n,\ag}\rightarrow \mathcal{X}_{\n,\ag+2}$ are linear operators such that,
	$$\|l_2(\p)\|_{\n,\ag+4}\leq M\|\p\|_{\n,\ag}\quad and \quad \|l_3(\p)\|_{\n,\ag+2}\leq M\|\p\|_{\n,\ag}.$$
	
	The proof of the proposition follows from \eqref{errorequation}, \eqref{cotaouter}, and \eqref{cubica} and by taking,
	\begin{itemize}
		\item $	\mathcal{C}_{\mathrm{mch}}=\e^2\phi + f(\phi) - \omega^{-3} f(\omega \phi),\vspace{0.2cm}$
			\item $L(\p)= \Pi_1\left[l_2(\p)\right]\vspace{0.2cm}$,
		\item $\widehat{L}(\p)= \Pi_1\left[l_3(\wt \Pi\left[\p\right])\right]\vspace{0.2cm}$,
	
		\item $K(\p)=\wt \Pi\left[- \dfrac{2}{z^2}\Pi_1[\p]\sin(3\tau) +l_2(\p) +l_3(\wt \Pi\left[\p\right])\right].$ 
	\end{itemize}	
	\end{proof}

Let $z_j=\e^{-1}(y_j-i\pi/2)$, $j=1,2$, where $y_1$ and $y_2$ are the vertices of the matching domain $D^{\mathrm{\mathrm{mch}},u}_{+,\kappa}$ given by \eqref{matchingdomain}. Consider the following linear operator acting on the Fourier coefficients of $h(z,\tau)=\sum_{k\geq 0}h_{2k+1}(z)\sin((2k+1)\tau).$
\begin{equation} \label{oplinmatch}
\mathcal{T}(h)=\displaystyle\sum_{k\geq 0}\mathcal{T}_{2k+1}(h_{2k+1})\sin((2k+1)\tau),
\end{equation}
where 
\begin{align*} 
\mathcal{T}_1(h_1)=&\,\dfrac{z^3}{5}\displaystyle\int_{z_1}^z\dfrac{h_1(s)}{s^2}ds-\dfrac{1}{5z^2}\displaystyle\int_{z_2}^zh_1(s)s^3ds
\\
&\,-\dfrac{1}{5(z_2^5-z_1^5)}\left[ \left(z^3-\dfrac{z_2^5}{z^2}\right)\displaystyle\int_{z_2}^{z_1}h_1(s)s^3ds+ \left(z^3z_2^5-\dfrac{(z_1z_2)^5}{z^2}\right)\displaystyle\int_{z_1}^{z_2}\dfrac{h_1(s)}{s^2}ds\right]
\\
\mathcal{T}_{2k+1}(h_{2k+1})=&\,\displaystyle\int_{z_2}^z\dfrac{h_{2k+1}(s)e^{-i\mu_{2k+1}(s-z)}}{2i\mu_{2k+1}}ds -\displaystyle\int_{z_1}^z\dfrac{h_{2k+1}(s)e^{i\mu_{2k+1}(s-z)}}{2i\mu_{2k+1}}ds, 
\\
&\,+ \dfrac{\sin(\mu_{2k+1}(z_2-z))}{\sin(\mu_{2k+1}(z_1-z_2))}\displaystyle\int_{z_2}^{z_1}\dfrac{h_{2k+1}(s)e^{-i\mu_{2k+1}(s-z_1)}}{2i\mu_{2k+1}}ds
\\
&\,+\dfrac{\sin(\mu_{2k+1}(z_1-z))}{\sin(\mu_{2k+1}(z_1-z_2))}\displaystyle\int_{z_1}^{z_2}\dfrac{h_{2k+1}(s)e^{-i\mu_{2k+1}(s-z_2)}}{2i\mu_{2k+1}}ds, \qquad \textrm{ for }\,k\geq 1.
\end{align*}
Observe that $\mathcal{T}$ is chosen such that $\mathcal{I}\circ\mathcal{T}=\mathrm{Id}$ and $\mathcal{T}(h)(z_j,\tau)=0$, $j=1,2$.

Moreover, consider the analytic in $z$ function $\mathcal{Q}:D^{\mathrm{\mathrm{mch}},u}_{+,\kappa}\times\mathbb{T}\rightarrow\C$ given by
\begin{equation}
\label{indepterm}
\mathcal{Q}(z,\tau)=\displaystyle\sum_{k\geq 0}\mathcal{Q}_{2k+1}(z)\sin((2k+1)\tau),
\end{equation}
which is defined using  $\p$ in \eqref{difference} as follows, where $k\geq 1$,
\begin{align*}
\mathcal{Q}_1(z)&
=\dfrac{1}{z_2^5-z_1^5}\left(z^3(z_2^2\p_1(z_2)-z_1^2\p_1(z_1))-\dfrac{1}{z^2}\left(z_1^5z_2^2\p_1(z_2)-z_1^2z_2^5\p_1(z_1)\right)\right),\\
\mathcal{Q}_{2k+1}(z)&
=\dfrac{\sin(\mu_{2k+1}(z-z_2))}{\sin(\mu_{2k+1}(z_1-z_2))}\p_{2k+1}(z_1)-\dfrac{\sin(\mu_{2k+1}(z-z_1))}{\sin(\mu_{2k+1}(z_1-z_2))}\p_{2k+1}(z_2). 
\end{align*}
Observe that $\mathcal{Q}$ satisfies $\mathcal{I}\mathcal{Q}=0$ and $\mathcal{Q}_{2k+1}(z_j)=\varphi_{2k+1}(z_j)$, $j=1,2$.

In conclusion, observe that if $h,\widehat{\p}:D^{\mathrm{\mathrm{mch}},u}_{+,\kappa}\times\C\rightarrow \C$ are analytic in $z$ functions such that 
	\[\mathcal{I}(\widehat{\p})=h,\qquad \widehat{\p}(z_j)=\p(z_j), \quad j=1,2,\] 
	where $\p$ is given in \eqref{difference}, then, we have that 
	\begin{equation*}
	\widehat{\p}(z,\tau)= \mathcal{Q}(z,\tau)+ \mathcal{T}(h)(z,\tau),
	\end{equation*}
	where $\mathcal{T}$ and $\mathcal{Q}$ are given by \eqref{oplinmatch} and \eqref{indepterm}. In particular, as the function $\varphi$ satisfies \eqref{eqmatch} by Proposition \ref{errorequationlem}, it can be written as
\begin{equation}\label{eqmat}
\p(z,\tau)= \mathcal{Q}(z_1,z_2)(z,\tau)+\mathcal{T}\left(\mathcal{C}_{\mathrm{mch}}(z,\tau)+ \left(L(\p)(z)+\widehat{L}(\wt \Pi[\p])(z)\right)\sin(\tau) + K(\p)(z,\tau)\right).
\end{equation}
We use this expression for $\varphi$ to obtain estimates of this function for $z\in D^{\mathrm{\mathrm{mch}},u}_{+,\kappa}$.

The next lemma gives estimates for the operators $\mathcal{T}$ and $\mathcal{Q}$ given in \eqref{oplinmatch} and \eqref{indepterm}.

\begin{lemma}\label{propertiesmat}
	There exists $\delta>0$ depending only on $\beta_{1,2}$ (see \eqref{matchingdomain}), such that, for $\kappa \e^{1-\gamma} \le \delta$,  the following statements hold.
	\begin{enumerate}
		\item 
		The linear operator $\mathcal{T}_1:\mathcal{X}_{\ag}\rightarrow\mathcal{X}_{\ag-2}$ is well defined and  
		\begin{equation*}
		\left\| \mathcal{T}_1(h)\right\|_{\ag-2}\leq M\left\|h\right\|_{\ag}, \; \; \ag >4; \quad \left\| \mathcal{T}_1(h)\right\|_{2}\leq M |\log \e| \left\|h\right\|_{4}, \; \; \ag =4.
		\end{equation*}
		
		\item For $k\ge 1$ and $h\in\mathcal{X}_\ag$, with  $\ag\geq 0$, 
		\begin{equation*}
		\left\|\mathcal{T}_{2k+1} (h)\right\|_{\ag}\leq \frac M{k^2} \left\|h\right\|_{\ag}. 
		\end{equation*}
		\item 
		$\mathcal{Q}$ satisfies
		\begin{equation*}
		\left\|\mathcal{Q}_1\right\|_{\ag}\leq M\left(\e^{(\ag-3)(\cg-1)}+\e^{2+(\ag+1)(\cg-1)}\right), \; \; \ag \ge 2; \quad \; \left\| \pa_\tau^2 \wt \Pi [\mathcal{Q}] \right\|_{\ag}\leq M\e^{(\ag-3)(\cg-1)}, \; \; \ag \ge 0. 
		\end{equation*}		
	\end{enumerate}
\end{lemma}

\begin{proof}
        Due to the assumption $e^{5(\pi - \beta_1)} - e^{-5 \beta_2} \ne 0$, when $\delta$ is small, it holds 
        \begin{equation}\label{E:temp-z}
        \frac 1M \e^{\gamma-1} \le |z_1|, \, |z_2|, \, |z_1^5 -z_2^5|^{\frac 15} \le M \e^{\gamma - 1}; \quad \kappa \le |z| \le M \e^{\gamma - 1}, \;\; \forall z\in D^{\mathrm{mch,u}}_{+, \kappa}.        
        \end{equation}
	Therefore,
	\[
	\begin{split}		
	\left|\dfrac{1}{5z^2}\displaystyle\int_{z_2}^z h(s)s^3ds \right| \leq& \,\dfrac{M\|h\|_{\ag}}{|z|^2}\displaystyle\int_{z_2}^z |s|^{3-\ag}ds
	\leq  \begin{cases} M\|h\|_{\ag}|z|^{2-\ag}, \qquad &\alpha > 4, \\
	M|\log \e| \|h\|_{\ag}|z|^{2-\ag}, & \alpha = 4,
\end{cases}\end{split}\]
	and, for  $\ag \ge 4$,
	\[\begin{split}
&\left|\dfrac{z^3}{5}\displaystyle\int_{z_1}^z \dfrac{h (s)}{s^2}ds \right| \leq \, M\|h \|_{\ag}|z|^3\displaystyle\int_{z_1}^z \dfrac{1}{|s|^{2+\ag}}ds
	\leq M\|h \|_{\ag}|z|^{2-\ag}\\
	& \left|\dfrac{1}{5(z_2^5-z_1^5)} \left(z^3-\dfrac{z_2^5}{z^2}\right)\displaystyle\int_{z_2}^{z_1}h(s)s^3ds\right|\leq \, \dfrac{M\|h\|_{\ag}}{|z|^2}\displaystyle\int_{z_2}^{z_1} |s|^{3-\ag}ds \le \dfrac {M\|h\|_{\ag} \e^{(\gamma-1) (4-\ag)}}{|z|^2}
	\leq M\|h\|_{\ag}|z|^{2-\ag}\\	
	& \left|\dfrac{1}{5(z_2^5-z_1^5)}\left(z^3z_2^5-\dfrac{(z_1z_2)^5}{z^2}\right)\displaystyle\int_{z_1}^{z_2}\dfrac{h(s)}{s^2}ds\right|\leq \, M \|h\|_{\ag}\left(|z|^3+\dfrac{|z_2|^5}{|z|^2}\right)\displaystyle\int_{z_1}^{z_2} \dfrac{1}{|s|^{2+\ag}}ds\\
	& \qquad \qquad \qquad \qquad \qquad \qquad\qquad \qquad \qquad \;\;\; \le M \|h\|_{\ag}\left(|z|^3+\dfrac{|z_2|^5}{|z|^2}\right) \e^{(\gamma-1)(-1-\ag)} \leq M \|h\|_{\ag}|z|^{2-\ag},
\end{split}
	\]
	where the integral $\int_{z_1}^{z_2}$ was simply taken along the arc of the circle centered at $-i\kappa \e$. 
	Hence, we finish the proof of item (1) of the theorem. 
	
	To deal with the higher modes, we will see that
	\begin{equation}\label{sines}\left|\dfrac{\sin(\mu_{2k+1}(z_j-z))}{\sin(\mu_{2k+1}(z_1-z_2))}\right|\leq M,\ j=1,2, \ \forall\ z\in D^{\mathrm{\mathrm{mch}},u}_{+,\kappa},\ \forall\ k\geq 1. \end{equation}
	In fact, recalling that $|\sin^2(z)|=\frac{1}{2}(\cosh(2\Ip(z))-\cos(2\Rp(z)))$, we have
	\[
	\left|\dfrac{\sin(\mu_{2k+1}(z_j-z))}{\sin(\mu_{2k+1}(z_1-z_2))}\right|^2\leq\dfrac{\cosh(2\mu_{2k+1}\Ip(z_j-z))+1}{\cosh(2\mu_{2k+1}\Ip(z_1-z_2))-1}.
\]
Since $\Ip(z_1-z_2)=K\e^{\cg-1}$ and $|\Ip(z_j-z)|\leq|\Ip(z_1-z_2)| $,  we obtain \eqref{sines}.

Assume that $\ag\geq 0$. For each $z\in D^{\mathrm{\mathrm{mch}},u}_{+,\kappa}$, there exist $\beta_1^*,\beta_2^*$ (depending on $z$) between $\beta_1$ and $\beta_2$ and $t_2^*, t_1^*>0$ (depending on $z$) such that $z_2= z+ e^{-i\beta_2^*}t_2^{*}$ and $z_1= z+  e^{i(\pi - \beta_1^*)}t_1^{*}$. Thus, we have that
	\[
	\begin{split}
	\left|\displaystyle\int_{z_2}^{z}h_{2k+1}(s)e^{-i\mu_{2k+1}(s-z)}ds\right|\leq&\,\displaystyle\int^{t_2^*}_{0}\left|h_{2k+1}\left(z+e^{-i\beta_2^*}t\right)\right|e^{-\mu_{2k+1}\sin(\beta_2^*t)}dt\\
	\leq&\,\|h_{2k+1}\|_{\ag}\displaystyle\int_{0}^{t_2^*}\dfrac{e^{-\mu_{2k+1}\sin(\beta_2^*t)}}{|z+e^{-i\beta_2^*}t|^{\ag}}dt
	\leq\,\dfrac{\|h_{2k+1}\|_{\ag}}{|z|^{\ag}}\displaystyle\int_{0}^{\infty}e^{-\mu_{2k+1}\sin(\beta_2^*t)}dt\\
	\leq&\,\dfrac{M\|h_{2k+1}\|_{\ag}}{\mu_{2k+1}|z|^{\ag}}.
	\end{split}
	\]
	Analogously, we prove that 
	$$\left|\int_{z_1}^{z}h_{2k+1}(s)e^{i\mu_{2k+1}(s-z)}ds\right|\leq \dfrac{M\|h_{2k+1}\|_{\ag}}{\mu_{2k+1}|z|^{\ag}},$$
	and in particular, using that $|z_j|\geq M|z|$, $j=1,2$,
	\[
	\begin{split} 
	\left|\displaystyle\int_{z_2}^{z_1}h_{2k+1}(s)e^{-i\mu_{2k+1}(s-z_1)}ds\right|\leq &\, \dfrac{M\|h_{2k+1}\|_{\ag}}{\mu_{2k+1}|z_1|^{\ag}}\leq \dfrac{M\|h_{2k+1}\|_{\ag}}{\mu_{2k+1}|z|^{\ag}}\\
\left|\displaystyle\int_{z_1}^{z_2}h_{2k+1}(s)e^{i\mu_{2k+1}(s-z_2)}ds\right|\leq  &\,\dfrac{M\|h_{2k+1}\|_{\ag}}{\mu_{2k+1}|z_2|^{\ag}}\leq \dfrac{M\|h_{2k+1}\|_{\ag}}{\mu_{2k+1}|z|^{\ag}}.
\end{split}
\]
	Hence,	
	\begin{equation}\label{boundTn}\|\mathcal{T}_{2k+1}(h_{2k+1})\|_{\ag}\leq \dfrac{M}{\mu_{2k+1}^2}\|h_{2k+1}\|_{\ag},\ k\geq 1,\ \ag\geq 0.\end{equation} 
	Items $(2)$ follows 
	\eqref{boundTn}.	 			
		
	To estimate $\mathcal{Q}$, observe that  using \eqref{cotaouter} and \eqref{cotainner}, one has
	$$\p(z,\tau)= l_1(z)\sin\tau+b(z,\tau),\qquad\text{ with }\,b(z,\tau)=\e\xi(i\pi/2+\e z,\tau) -\left(\phi_0(z,\tau)+\frac{2\sqrt{2}i}{z}\sin\tau\right),$$
	where $l_1$ is given in \eqref{cotaouter}. Then, $\| \pa_\tau^2 b\|_{\n,3}\leq M$ and $|l_1(z)|\leq M\e^2|z|$, for each $z\in D^{\mathrm{\mathrm{mch}},u}_{+,\kappa}.$	
	Thus, 	from \eqref{E:temp-z}, we can see that
	$$
	\begin{array}{lcl}
	\left|\mathcal{Q}_1(z_1,z_2)(z)\right|&=& \left|\dfrac{1}{z_2^5-z_1^5}\left(z^3(z_2^2\p_1(z_2)-z_1^2\p_1(z_1))-\dfrac{1}{z^2}\left(z_1^5z_2^2\p_1(z_2)-z_1^2z_2^5\p_1(z_1)\right)\right)\right|\vspace{0.2cm}\\
	&\leq&  M\left(\left|\p_1(z_1)\right|+\left|\p_1(z_2)\right|+\dfrac{|z_1^2|}{|z|^2}\left|\p_1(z_1)\right|+\dfrac{|z_2^2|}{|z|^2}\left|\p_1(z_2)\right|\right)\vspace{0.2cm}\\
	&\leq&  M\left(\dfrac{1}{|z_2||z|^2}+ \e^2|z_2|+\dfrac{\e^2|z_2|^3}{|z|^2}\right).	
	\end{array}
	$$   
		Therefore for $\ag\geq 2$,
	$$
	\left\|\mathcal{Q}_1(z_1,z_2)\right\|_{\ag}\leq  M\left(\e^{(\ag-3)(\cg-1)}+\e^{2+(\ag+1)(\cg-1)}\right).
	$$   	
	Finally, from $\eqref{sines}$ and \eqref{indepterm}, we can see that, for $\ag \ge 0$ and $k\geq 1$,
	\[
	\begin{split}
	|z^{\ag}\pa_\tau^2 \mathcal{Q}_{2k+1}(z_1,z_2)(z)|=&\,\left|\dfrac{\sin(\mu_{2k+1}(z-z_2))}{\sin(\mu_{2k+1}(z_1-z_2))}z^{\ag}\pa_\tau^2 \p_{2k+1}(z_1)
	-\dfrac{\sin(\mu_{2k+1}(z-z_1))}{\sin(\mu_{2k+1}(z_1-z_2))}z^{\ag} \pa_\tau^2 \p_{2k+1}(z_2)\right|\vspace{0.2cm}\\
	\leq&\,M k^2 \|\Pi_{2k+1}[b]\|_{3}\dfrac{|z|^{\ag}}{|z_2|^3}
	\leq Mk^2 \|\Pi_{2k+1}[b]\|_{3}\e^{(\ag-3)(\cg-1)},
	\end{split}
	\]
	and thus
	$$\|\pa_\tau^2 \mathcal{Q}_{2k+1}(z_1,z_2)\|_{\ag}\leq M
	\e^{(\ag-3)(\cg-1)},\ \ag\ge 0,\ k\geq 1,$$
	which completes the proof of item (3).
\end{proof}

\begin{proof}[End of the proof of Theorem  \ref{matchingthm}] 
	To obtain the estimates for $\varphi$ stated in the theorem, we just need to estimate  $\|\varphi\|_{\ell^1,2}$.
	From \eqref{eqmat}, and Propositions \ref{errorequationlem} and \ref{propertiesmat}, we have that
	$$
	\begin{array}{lcl}
	\|\p_1\|_2&=& \left\|\mathcal{Q}_{1}(z_1,z_2)+ \mathcal{T}_{1}\left(\Pi_1\left[\mathcal{C}_{\mathrm{mch}}\right]+ L(\p)+\widehat{L}(\wt \Pi[\p])\right)\right\|_2\vspace{0.2cm}\\
	&\leq& \left\|\mathcal{Q}_{1}(z_1,z_2)\right\|_2+ M |\log \e| \left(\left\|\Pi_1\left[\mathcal{C}_{\mathrm{mch}}\right]\right\|_4+ \left\|L(\p)\right\|_4+ \left\| \widehat{L}(\wt \Pi[\p])\right\|_4\right)\vspace{0.2cm}\\
	&\leq& M(\e^{1-\cg}+\e^{2+3(\cg-1)})+ M|\log \e| \left(\e^{3\cg-1}+\left\|\p\right\|_{\n,0}+ \left\| \wt \Pi[\p]\right\|_{\n,2} \right)\vspace{0.2cm}\\
	&\leq& M\left(\e^{1-\cg}+\e^{3\cg-1} |\log \e| +\dfrac{ |\log \e|}{\kappa^2}\left\|\p\right\|_{\n,2}+ |\log \e| \left\| \wt \Pi[\p]\right\|_{\n,2} \right).	
	\end{array}
	$$
	Moreover, since $\Pi_1\circ K\equiv 0$, we have that	
	$$
	\begin{array}{lcl}
	\left\|\pa_\tau^2 \wt \Pi[\p]\right\|_{\n,2}&=& \left\|\pa_\tau^2 \wt \Pi\circ\mathcal{Q}(z_1,z_2,\varphi)+ \pa_\tau^2 \mathcal{T}\left(\wt \Pi\left[\mathcal{C}_{\mathrm{mch}}\right]+ K(\p)\right)\right\|_{\n,2}\vspace{0.2cm}\\
	&\leq& \left\|\pa_\tau^2 \wt \Pi\circ\mathcal{Q}(z_1,z_2,\varphi)\right\|_{\n,2}+ M\left(\left\|\wt \Pi\left[\mathcal{C}_{\mathrm{mch}}\right]\right\|_{\n,2}+ \left\|K(\p)\right\|_{\n,2}\right)\vspace{0.2cm}\\
	&\leq& M(\e^{1-\cg}+\e^{2+3(\cg-1)})+ M\left(\dfrac{\e^2}{\kappa}+\left\|\p\right\|_{\n,0}\right)\vspace{0.2cm}\\
	&\leq& M\left(\e^{1-\cg}+\e^{3\cg-1}+\dfrac{1}{\kappa^2}\left\|\p\right\|_{\n,2} \right).
	\end{array}
	$$
	Since $\kappa^{-2} |\log \e|$ is assumed to be small, it follows from multiplying the second inequality by $2 M |\log \e|$ and adding it to the first one that 
	$$
	\| \p_1 \|_2 + M |\log \e| \|\pa_\tau^2 \wt \Pi[\p] \|_{\n,2}\leq 2 M |\log \e| \left(\e^{1-\cg}+\e^{3\cg-1}
		\right).
	$$
	Finally, the estimate on $\pa_z \p$ could be derived by differentiating the formula of $\p$ with respect to $z$. Alternatively, from Lemma 8.1 of \cite{BFGS12}, reducing the domain $D^{\mathrm{\mathrm{mch}},u}_{+,\kappa}$ (see \eqref{matchingdomain}), with vertices $y_1$ and $y_2$ such that $|y_j-i(\pi/2-\kappa\e)|=\e^{\cg}$, $j=1,2$,  to $D^{\mathrm{\mathrm{mch}},u}_{+,2\kappa}\subset D^{\mathrm{\mathrm{mch}},u}_{+,\kappa}$ having vertices $\widetilde{y}_1$ and $\widetilde{y}_2$ such that $|\widetilde{y}_j-i(\pi/2-2\kappa\e)|=\widetilde{c}\e^{\cg}$, $j=1,2$, and $0<\widetilde{c}<1$, we obtain that
	$$\|\pa_\tau^2 \partial_z\p\|_{\n,2}\leq \dfrac{M}{\kappa} |\log \e| (\e^{1-\cg}+\e^{3\cg-1}).$$
	It completes the proof of this theorem.  In order to simplify the notation, we make no distinction between $D^{\mathrm{\mathrm{mch}},u}_{+,\kappa}$ and $D^{\mathrm{\mathrm{mch}},u}_{+,2\kappa}$. 			
\end{proof}

\section{The distance between the manifolds: Proof of Proposition \ref{prop:DifferenceDeltatilde}}\label{mainthmsec}

\subsection{Banach Space and Operators}
We devote this section to prove Proposition \ref{prop:DifferenceDeltatilde}. We start by defining the functional setting. 
Given an analytic function $f:\mathcal{R}_{\kappa}\rightarrow \C$ (see Figure 
\ref{diffdom}), we define the norm
\begin{equation*}
\|f\|_{\ag,\mathrm{exp}}=\displaystyle\sup_{y\in\mathcal{R}_{\kappa}}\left|(y^2+\pi^2/4)^{\ag} e^{\frac{\lambda_3}{\e}\left(\frac{\pi}{2}-|\Ip(y)|\right)}f(y)\right|,
\end{equation*}
and the Banach space
\begin{equation*}
\mathcal{X}_{\ag,\exp}=\{f:\mathcal{R}_{\kappa}\rightarrow \C;\ f\, \textrm{ analytic}, \|f\|_{\ag,\mathrm{exp}}<\infty\}.
\end{equation*}
Moreover, given an analytic function $f:\mathcal{R}_{\kappa}\times\mathbb{T}\rightarrow \C$ odd in $\tau \in \mathbb{T}$, we define the corresponding norm and the associated Banach space
\begin{equation*}
\begin{split}
\|f\|_{\ell_1,\ag,\mathrm{exp}}=&\displaystyle\sum_{k\geq 1} \|\Pi_{2k+1}[f]\|_{\ag,\mathrm{exp}}\\
\mathcal{X}_{\ell_1,\ag,\exp}=&\left\{f:\mathcal{R}_{\kappa}\times \mathbb{T}\rightarrow \C;\ f \textrm{ is an analytic function in the variable } y \textrm{ such that }\right.\\
&\left. \Pi_1[f]=\Pi_{2l}[f]=0,\forall l\geq 0\textrm{ and }  \|f\|_{\ell_1,\ag,\mathrm{exp}}<\infty\right\}.
\end{split}
\end{equation*}
Finally, we consider the product Banach space
\begin{equation*}
\mathcal{Y}_{\ell_1,2,\exp}=\mathcal{X}_{2,\exp}\times \mathcal{X}_{\ell_1,0,\exp}\times \mathcal{X}_{\ell_1,0,\exp},
\end{equation*} 
endowed with the weighted norm
\begin{equation*}
\llbracket(f,g,h)\rrbracket_{\ell_1,2,\exp}=\dfrac{1}{\e}\|f\|_{2,\exp}+\kappa \|g\|_{\ell_1,0,\exp}+ \kappa \|h\|_{\ell_1,0,\exp}.
\end{equation*} 
The next lemmas give estimates for the operators and functions given in Section \ref{sec:Sketch:difference}.

\begin{lemma}\label{lemmaP} The components of the operator  $\mathcal{P}$ in \eqref{Pop} have the following properties.
\begin{enumerate}
\item For $\ag=2,5$, the operator $\mathcal{P}^W:\mathcal{X}_{\ag,\exp}\rightarrow \mathcal{X}_{2,\exp}$ is well defined. Moreover,  there exists a constant $M>0$ independent of $\e$ and $\kappa$ such that,
\begin{itemize}
 \item For $h\in \mathcal{X}_{2,\exp}$, $\displaystyle\|\mathcal{P}^{W}(h)\|_{2,\exp}\leq M\e\|h\|_{2,\exp}.$
 \item For $h\in \mathcal{X}_{5,\exp}$, $\displaystyle\|\mathcal{P}^{W}(h)\|_{2,\exp}\leq \frac{M}{\e^2\kappa^3}\|h\|_{5,\exp}.$
 \end{itemize}
\item 	For $\ag>1$, the operators $\mathcal{P}^{\Gamma},\mathcal{P}^{\Theta}:\mathcal{X}_{\ell_1,\ag,\exp}\rightarrow \mathcal{X}_{\ell_1,0,\exp}$  are well-defined. Moreover, there exists a constant $M>0$ independent of $\e$ and $\kappa$ such that, for every $h\in \mathcal{X}_{\ell_1,\ag,\exp}$,
	\begin{equation*}
	\|\mathcal{P}^{\Gamma}(h)\|_{\ell_1,0,\exp},\|\mathcal{P}^{\Theta}(h)\|_{\ell_1,0,\exp}\leq \dfrac{M}{(\kappa \e)^{\ag-1}}\|h\|_{\ell_1,\ag,\exp}.
	\end{equation*}
\end{enumerate}
	\end{lemma}
\begin{proof}
We first prove item (1). 
We take $h\in \mathcal{X}_{\ell_1,2,\exp}$ and, recalling that  $\ddot{v}^h$ has a a pole of order $3$, we obtain the following estimate for $\Ip(y)>0$, 
	\[
	\begin{split}
	\left|e^{\frac{\lambda_3}{\e}\left(\frac{\pi}{2}-|\Ip(y)| \right)}|y^2+\pi^2/4|^2\ddot v^h(y)\displaystyle\int_0^y \frac{h(s)}{\ddot v^h(s)}ds\right| \leq&\,\dfrac{e^{\frac{\lambda_3}{\e}\left(\frac{\pi}{2}-|\Ip(y)| \right)}}{|y^2+\pi^2/4|}\displaystyle\int_0^y \left|\frac{h(s)}{\ddot v^h(s)}\right|ds \\
	\leq &\, M\|h\|_{2,\exp}\dfrac{e^{\frac{\lambda_3}{\e}\left(\frac{\pi}{2}-|\Ip(y)| \right)}}{|y^2+\pi^2/4|}\displaystyle\int_0^y e^{-\frac{\lambda_3}{\e}\left(\frac{\pi}{2}-|\Ip(s)| \right)}|s^2+\pi^2/4|ds\\
\leq &\, M \|h\|_{2,\exp}\dfrac{e^{\frac{\lambda_3}{\e}\left(\frac{\pi}{2}-\Ip(y) \right)}}{|y-i\pi/2|}\displaystyle\int_0^{\Ip(y)} |\sigma- \pi/2|e^{-\frac{\lambda_3}{\e}\left(\frac{\pi}{2}-\sigma\right)}d\sigma\\
\leq&\, M \|h\|_{2,\exp}\dfrac{e^{\frac{\lambda_3}{\e}\left(\frac{\pi}{2}-\Ip(y) \right)}}{|y-i\pi/2|}\displaystyle\int_{\frac{\frac{\pi}{2}-\Ip(y)}{\e}}^{\frac{\pi}{2\e}} \e re^{-\lambda_3 r}\e dr\\
\leq& \, \dfrac{M\e\|h\|_{2,\exp}}{|y-i\pi/2|}\left( \frac \e{\lambda_3}+\dfrac{\pi}{2}-\Ip(y) - e^{-\frac{\lambda_3}{\e}\Ip(y)}\left(\frac \e{\lambda_3} +\dfrac{\pi}{2}\right) \right)\\
\leq&\,  M\e \|h\|_{2,\exp}\left( \dfrac{1}{\kappa}+1 
\right)
\leq\,  M\e \|h\|_{2,\exp}.
	\end{split}
	\]
Analogously, one can obtain the same estimate for $\Ip(y)<0$. 

For $h\in \mathcal{X}_{\ell_1,5,\exp}$, one obtains,
\[
	\begin{split}
	\left|e^{\frac{\lambda_3}{\e}\left(\frac{\pi}{2}-|\Ip(y)| \right)}|y^2+\pi^2/4|^2\ddot v^h(y)\displaystyle\int_0^y \frac{h(s)}{\ddot v^h(s)}ds\right| \leq&\,\dfrac{e^{\frac{\lambda_3}{\e}\left(\frac{\pi}{2}-|\Ip(y)| \right)}}{|y^2+\pi^2/4|}\displaystyle\int_0^y \left|\frac{h(s)}{\ddot v^h(s)}\right|ds \\
	\leq &\, M\|h\|_{5,\exp}\dfrac{e^{\frac{\lambda_3}{\e}\left(\frac{\pi}{2}-|\Ip(y)| \right)}}{|y^2+\pi^2/4|}\displaystyle\int_0^y \dfrac{e^{-\frac{\lambda_3}{\e}\left(\frac{\pi}{2}-|\Ip(s)| \right)}}{|s^2+\pi^2/4|^2}ds\\
	\leq &\, \dfrac{M\|h\|_{5,\exp}}{\kappa^3\e^3} e^{-\frac{\lambda_3}{\e}|\Ip(y)| }\displaystyle\int_0^y e^{\frac{\lambda_3}{\e}|\Ip(s)| }ds\\
	\leq &\, \dfrac{M\|h\|_{5,\exp}}{\kappa^3\e^2}.
	\end{split}
	\]
	
We prove item (2) only for the operator $\mathcal{P}^{\Gamma}$, since the result for $\mathcal{P}^{\Theta}$ follows analogously. Let $h(y,\tau)=\sum_{k\geq 1}h_{2k+1}(y)\sin((2k+1)\tau)$. 
We bound each component of the operator $\mathcal{P}^{\Gamma}$ as 
	$$
	\begin{array}{lcl}
	\left|\mathcal{P}^\Gamma_{2k+1}(h_{2k+1})e^{\frac{\lambda_3}{\e}\left(\frac{\pi}{2}-|\Ip(y)|\right)}\right|&\leq&\|h_{2k+1}\|_{\ag,\exp}\displaystyle
	\int^{y^+}_y\left|e^{\frac{\lambda_3}{\e}\left(\frac{\pi}{2}-|\Ip(y)|\right)}\dfrac{e^{-\frac{\lambda_3}{\e}\left(\frac{\pi}{2}-|\Ip(s)|\right)}}{|s^2+\pi^2/4|^\ag}e^{i\frac{\lambda_{2k+1}}{\e}(s-y)}\right|ds\vspace{0.2cm}\\
	&\leq&\|h_{2k+1}\|_{\ag,\exp}\displaystyle
	\int^{\frac{\pi}{2}-\kappa\e}_{\Ip(y)} \dfrac{e^{\frac{1}{\e}\left(\lambda_3|\sigma|-\lambda_{2k+1}\sigma-(\lambda_3|\Ip(y)|-\lambda_{2k+1}\Ip(y))\right)}}{|\sigma^2-\pi^2/4|^\ag}d\sigma.
	\end{array}
	$$
Now, since the functions  $f_k(t)=\lambda_3|t|-\lambda_{2k+1}t$ are decreasing  for $t\in\mathbb{R}$ and $k\geq 1$, $\sigma >\Ip(y)$, and  recalling that $\ag>1$, we obtain 
	$$
	\left|\mathcal{P}^\Gamma_{2k+1}(h_{2k+1})e^{\frac{\lambda_3}{\e}\left(\frac{\pi}{2}-|\Ip(y)|\right)}\right|
	\leq\|h_{2k+1}\|_{\ag,\exp}\displaystyle
	\int^{\frac{\pi}{2}-\kappa\e}_{\Ip(y)}\dfrac{1}{|\sigma^2 -\pi^2/4|^\ag}d\sigma
	\leq \dfrac{M}{(\kappa\e)^{\ag-1}}\|h_{2k+1}\|_{\ag,\exp}.
	$$	
\end{proof}

In next proposition, we obtain estimates for the right hand side of equation \eqref{fixedequationmod}.

\begin{prop}\label{propgigante}
	There exists a constant $M$ independent of $\e$ and $\kappa$ such that the following statements hold.
	\begin{enumerate}
		\item The operator $\widetilde{\mathcal{M}}:\mathcal{Y}_{\ell_1,2,\exp}\rightarrow\mathcal{Y}_{\ell_1,2,\exp}$ introduced in \eqref{fixedequationmod} is well-defined and 
		\begin{equation*}
		\left\llbracket\widetilde{\mathcal{M}}\left({\Xi_1},\Gamma,\Theta\right)\right\rrbracket_{\ell_1,2,\exp}\leq \dfrac M\kappa \left\llbracket\left({\Xi_1},\Gamma,\Theta\right)\right\rrbracket_{\ell_1,2,\exp}.  
		\end{equation*}
		Moreover, denoting $\widetilde{\mathcal{M}}=(\widetilde{\mathcal{M}}_1,\widetilde{\mathcal{M}}_2,\widetilde{\mathcal{M}}_3)$, we have that
		\begin{align*}
		\left\|\widetilde{\mathcal{M}}_1\left({\Xi_1},\Gamma,\Theta\right)\right\|_{2,\exp}\leq &\, \dfrac{M}{\kappa^3}\|{\Xi_1}\|_{2,\exp}+M\e\left(\left\|\Gamma\right\|_{\ell_1,0,\exp}+\left\|\Theta\right\|_{\ell_1,0,\exp}\right),
\\
		\left\|\widetilde{\mathcal{M}_j}\left({\Xi_1},\Gamma,\Theta\right)\right\|_{\ell_1,0,\exp}\leq &\, 
		\dfrac{M}{\kappa^2 \e}\|{\Xi_1}\|_{2,\exp}+\dfrac M\kappa \left(\left\|\Gamma\right\|_{\ell_1,0,\exp}+\left\|\Theta\right\|_{\ell_1,0,\exp}\right), 
		\ j=2,3. 
		\end{align*}		
		\item The function $\wt \Delta$ defined in \eqref{def:NewDelta} 
satisfies 
		\[
		\wt \Delta = (I- \wt M)^{-1} 
\big(0,\mathcal{I}_{\mathrm{\Gamma}}(c),\mathcal{I}_{\mathrm{\Theta}}(d) \big) \ 
 \text{ and } \ \llbracket \wt \Delta - 
\big(0,\mathcal{I}_{\mathrm{\Gamma}}(c),\mathcal{I}_{\mathrm{\Theta}}(d) \big) 
\rrbracket_{\ell_1,2,\exp} \le \dfrac M\kappa \llbracket  
\big(0,\mathcal{I}_{\mathrm{\Gamma}}(c),\mathcal{I}_{\mathrm{\Theta}}(d) \big) 
\rrbracket_{\ell_1,2,\exp},\]
where $\mathcal{I}_{\mathrm{\Gamma}}(c)$, $\mathcal{I}_{\mathrm{\Theta}}(d)$ are
the functions defined in \eqref{Is1} and \eqref{def:csandds}.
	\end{enumerate}
\end{prop}

\begin{proof}
	Assume that $\left({\Xi_1},\Gamma,\Theta\right)\in\mathcal{Y}_{\ell_1,2,\exp}$. To estimate the first component of $\mathcal{M}$, using the estimates for $m_W$ and $\mathcal{M}_W$ in Proposition \ref{sistfinal} and Lemma \ref{lemmaP} for the estimates on $\mathcal{P}^W$,
	\[
	 \begin{split}
\left\|\widetilde{\mathcal{M}}_1\left({\Xi_1},\Gamma,\Theta\right)\right\|_{2,\exp}\leq &\,\left\|{\mathcal{P}}^W\left(m_W \Xi_1\right)\right\|_{2,\exp}+\left\|{\mathcal{P}}^W\left(M_W(\Gamma,\Theta\right)\right\|_{2,\exp}\\
\leq &\,\frac{M}{\e^2\kappa^3}\left\|m_W \Xi_1\right\|_{5,\exp}+M\e\left\|M_W(\Gamma,\Theta)\right\|_{2,\exp}\\
	 \leq &\,\frac{M}{\kappa^3}\left\|\Xi_1\right\|_{2,\exp}+M\e\left(\left\|\Gamma\right\|_{\ell_1,0, \exp}+\left\|\Theta\right\|_{\ell_1,0, \exp}\right).
	 \end{split}
	\]
	Now we estimate $\widetilde{\mathcal{M}}_2$. The estimates for $\widetilde{\mathcal{M}}_3$ can be done analogously. Using as before Proposition \ref{sistfinal} and Lemma \ref{lemmaP},
	\[
	 \begin{split}
\left\|\widetilde{\mathcal{M}}_2\left({\Xi_1},\Gamma,\Theta\right)\right\|_{\ell_1,0,\exp}\leq &\,\left\|{\mathcal{P}}^\Gamma\left(m_{\osc} \Xi_1\right)\right\|_{\ell_1,0,\exp}+\left\|{\mathcal{P}}^\Gamma\left(M_{\osc}(\Gamma,\Theta\right)\right\|_{\ell_1,0,\exp}\\
\leq &\,\frac{M}{\kappa^2\e^2}\left\|m_{\osc} \Xi_1\right\|_{\ell_1, 3,\exp}+\frac{M}{\kappa\e}\left\|M_{\osc}(\Gamma,\Theta)\right\|_{\ell_1,2,\exp}\\
\leq &\,\frac{M}{\kappa^2\e}\left\|\Xi_1\right\|_{\ell_1,2,\exp}+\frac{M}{\kappa}\left(\left\|\Gamma\right\|_{\ell_1, 0,\exp}+\left\|\Theta\right\|_{\ell_1,0,\exp}\right).
 \end{split}
	\]
Item (2) of the proposition is simply a direct consequence of item (1) and \eqref{fixedequationmod}. 	
\end{proof}

The rest of this section is devoted to estimating $\big(0,\mathcal{I}_{\mathrm{\Gamma}}(c),\mathcal{I}_{\mathrm{\Theta}}(d) \big)$. 

\begin{lemma}\label{lemma:EstimatesDiff}
Take $\kappa=\dfrac{1}{2\mu_{3}}|\log\e|$. There exist $\e_0>0$ and a constant 
$M>0$ independent of $\e$ such that, for each $\e\in(0,\e_0)$, 
\[
\left\| \mathcal{I}_{\Gamma}(c)-\dfrac{2\mu_{3}}{\e}C_{\mathrm{\mathrm{in}}}e^{-i\frac{\lambda_3}{\e}(y-i\pi/2)}\sin(3\tau)  \right\|_{\n,0,\exp}, \; \left\| \mathcal{I}_{\Theta}(d)-\dfrac{2\mu_{3}}{\e}\overline{C_{\mathrm{\mathrm{in}}}}e^{i\frac{\lambda_3}{\e}(y+i\pi/2)}\sin(3\tau)  \right\|_{\n,0,\exp} \leq \dfrac{M}{\e|\log\e|}.
\]
\end{lemma}

\begin{proof}
From Theorems \ref{innerthm} and \ref{matchingthm} (see also 
\eqref{innerscaling}),  the function $\Delta$ given in \eqref{diffe} can be 
written as
	$$
	\begin{array}{lcl}
	\Delta (y,\tau)&=&\dfrac{1}{\e}\phi^u\left(\dfrac{y-i\pi/2}{\e},\tau\right)- \dfrac{1}{\e}\phi^s\left(\dfrac{y-i\pi/2}{\e},\tau\right)\vspace{0.2cm}\\
	&=&\dfrac{1}{\e}\Delta\phi^0\left(\dfrac{y-i\pi/2}{\e},\tau\right)+ \dfrac{1}{\e}\p^u\left(\dfrac{y-i\pi/2}{\e},\tau\right)- \dfrac{1}{\e}\p^s\left(\dfrac{y-i\pi/2}{\e},\tau\right)\vspace{0.2cm}\\
	&=&\dfrac{1}{\e}e^{-i\mu_{3}\frac{y-i\pi/2}{\e}}\left(C_{\mathrm{\mathrm{in}}}\sin(3\tau)+ \chi\left(\dfrac{y-i\pi/2}{\e},\tau\right)\right)\vspace{0.2cm}\\
	&&+ \dfrac{1}{\e}\p^u\left(\dfrac{y-i\pi/2}{\e},\tau\right)- \dfrac{1}{\e}\p^s\left(\dfrac{y-i\pi/2}{\e},\tau\right)\vspace{0.2cm}\\	
	&=&\dfrac{1}{\e}C_{\mathrm{\mathrm{in}}}e^{-i\mu_{3}\frac{y-i\pi/2}{\e}}\sin(3\tau)+ E_1^+(y,\tau)+ E_2^+(y,\tau),
	\end{array}
	$$
	for every $y\in \mathcal{R}_{\mathrm{\mathrm{mch}},\kappa}^+=D^{\mathrm{\mathrm{mch}},u}_{+,\kappa}\cap D^{\mathrm{\mathrm{mch}},s}_{+,\kappa}\cap i\R$ and $\kappa$ satisfying assumptions in  Theorems \ref{innerthm} and \ref{matchingthm}, where $E_1^+,E_2^+:\mathcal{R}_{\mathrm{\mathrm{mch}},\kappa}\times\mathbb{T}\rightarrow\C$ are analytic functions in the variable $y$. It follows from Theorem \ref{innerthm} that
	\begin{equation}\label{E1+}
	\|\pa_\tau E_1^+\|_{\n}(y) \leq \dfrac{M |e^{-i\mu_{3}\frac{y-i\pi/2}{\e}}|}{|y-i\pi/2|} \quad and \quad \|\partial_yE_1^+\|_{\n}(y)\leq \dfrac{M |e^{-i\mu_{3}\frac{y-i\pi/2}{\e}}|}{|y-i\pi/2|^2},	\end{equation}
	and from Theorem \ref{matchingthm}, choosing $\gamma=1/2$, we obtain
	\begin{equation}\label{E2+}	
	\|\partial_{\tau}^2 E_2^+\|_{\n}(y)\leq \dfrac{M \e^{3/2} |\log \e| }{|y-i\pi/2|^2}\quad and \quad \|\pa_\tau^2 \partial_yE_2^+\|_{\n}(y)\leq \dfrac{M\e^{1/2} |\log \e|}{\kappa|y-i\pi/2|^2}.
	\end{equation}
	Analogously, since $\Delta$ is real-analytic one can deduce that for $y\in \mathcal{R}_{\mathrm{\mathrm{mch}},\kappa}^-=\{z:\bar z\in \mathcal{R}_{\mathrm{\mathrm{mch}},\kappa}^+\}$,
	\begin{equation*}
	\Delta (y,\tau)=\dfrac{1}{\e}\overline{C_{\mathrm{\mathrm{in}}}}e^{i\mu_{3}\frac{y+i\pi/2}{\e}}\sin(3\tau)+ E_1^-(y,\tau)+ E_2^-(y,\tau),\end{equation*}
	where $E_j^-(y,\tau)= \overline{ E_j^+(\bar y,\tau)}$, which satisfy
	\begin{equation*}\label{E1-}
	\begin{split}
 	\|\partial_{\tau}E_1^-\|_{\n}(y)\leq  &\,\dfrac{M |e^{i\mu_{3}\frac{y+i\pi/2}{\e}}|}{|y+i\pi/2|} \quad \text{and} \quad \|\partial_yE_1^-\|_{\n}(y)\leq \dfrac{M |e^{i\mu_{3}\frac{y+i\pi/2}{\e}}|}{|y+i\pi/2|^2},\\
	\|\partial_{\tau}^2 E_2^-\|_{\n}(y)\leq &\,\dfrac{M \e^{3/2} |\log \e|}{|y+i\pi/2|^2}\quad \text{and} \quad \|\partial_yE_2^-\|_{\n}(y)\leq\dfrac{M\e^{1/2} |\log \e|}{\kappa|y+i\pi/2|^2}.
	\end{split}
	\end{equation*}
	Using \eqref{cambio} and recalling that $\lambda_{3}=\mu_{3}+\er(\e^2)$, we obtain that for $(y,\tau)\in \mathcal{R}_{\mathrm{\mathrm{mch}},\kappa}^+\times\mathbb{T}$,
	\[
	\begin{split}
	\Gamma(y,\tau)=&\,\displaystyle\sum_{k\geq 1} (\lambda_{2k+1} \Delta_{2k+1}(y)+i\e\partial_y\Delta_{2k+1}(y))\sin((2k+1)\tau)\\
	=\,&\dfrac{2\mu_{3}}{\e}C_{\mathrm{\mathrm{in}}}e^{-i\mu_{3}\frac{y-i\pi/2}{\e}}(1+\er(\e^2))\sin(3\tau)\\
	&\,+\displaystyle\sum_{k\geq 1}\lambda_{2k+1}\Pi_{2k+1}\left[E_1^+ + 
E_2^+\right]\sin((2k+1)\tau)+i\e\wt \Pi\left[\partial_y E_1^+ +  \partial_y 
E_2^+\right](y,\tau).
	\end{split}
	\]
	Moreover, using \eqref{E1+} and \eqref{E2+}, we have that
	\[
	\begin{split}
	\left\|\displaystyle\sum_{k\geq 0}\lambda_{2k+1}\Pi_{2k+1}\left[E_1^+ + E_2^+\right]\sin((2k+1)\tau)\right\|_{\n}(y)\leq& M\left(\|\partial_{\tau}E_1^+\|_{\n}(y)+\|\partial_{\tau}E_2^+\|_{\n}(y)\right)\\
	\leq&\, M\left(\dfrac{ |e^{-i\mu_{3}\frac{y-i\pi/2}{\e}}|}{|y-i\pi/2|}+\dfrac{\e^{3/2}|\log \e|}{|y-i\pi/2|^2}\right)\\
	\left\|i\e\wt \Pi\left[\partial_y E_1^+ +  \partial_y E_2^+\right]\right\|_{\n}(y)\leq&\, M\left(\dfrac{ \e |e^{-i\mu_{3}\frac{y-i\pi/2}{\e}}|}{|y-i\pi/2|^2}+\dfrac{\e^{3/2} |\log \e| }{\kappa|y-i\pi/2|^2}\right).
	\end{split}
	\]
	Then, for  $(y,\tau)\in \mathcal{R}_{\mathrm{\mathrm{mch}},\kappa}^+\times\mathbb{T}$, $\Gamma$ satisfies
	$$\Gamma(y,\tau)=\dfrac{2\mu_{3}}{\e}C_{\mathrm{\mathrm{in}}}e^{-i\mu_{3}\frac{y-i\pi/2}{\e}}\sin(3\tau)+ E_{\Gamma}^+(y,\tau),$$
	where $E_{\Gamma}^+: \mathcal{R}_{\mathrm{\mathrm{mch}},\kappa}^+\times\mathbb{T}\rightarrow\C$ is an analytic function in the variable $\tau$ such that
	$$\left\|E_{\Gamma}^+\right\|_{\n}(y)\leq M\left(\dfrac{ |e^{-i\mu_{3}\frac{y-i\pi/2}{\e}}|}{|y-i\pi/2|}+\dfrac{\e^{3/2}|\log \e|}{|y-i\pi/2|^2}\right).$$

	Proceeding in the same way for the function 
	$$\Theta(y,\tau)=\displaystyle\sum_{k\geq 0}(\lambda_{2k+1} \Delta_{2k+1}(y)-i\e\partial_y\Delta_{2k+1}(y))\sin((2k+1)\tau),$$
	we conclude that there exists a function $E_{\Theta}^-: \mathcal{R}_{\mathrm{\mathrm{mch}},\kappa}^-\times\mathbb{T}\rightarrow\C$ analytic in the variable $y$ such that $\Theta$ can be written as
	$$\Theta(y,\tau)=\dfrac{2\mu_{3}}{\e}\overline{C_{\mathrm{\mathrm{in}}}}e^{i\mu_{3}\frac{y+i\pi/2}{\e}}\sin(3\tau)+ E_{\Theta}^-(y,\tau),\qquad \text{for }\,(y,\tau)\in \mathcal{R}_{\mathrm{\mathrm{mch}},\kappa}^-\times\mathbb{T}$$
	and
	\[
	\left\|E_{\Theta}^-\right\|_{\n}(y)\leq \, M\left(\dfrac{ |e^{i\mu_{3}\frac{y+i\pi/2}{\e}}|}{|y+i\pi/2|}+\dfrac{\e^{3/2}|\log \e|}{|y+i\pi/2|^2}\right),\qquad \text{for }y\in\mathcal{R}_{\mathrm{\mathrm{mch}},\kappa}^-.
	\]
		
	Now that we have good estimates for the functions $\Gamma$ and $\Theta$ 
in the domains $\mathcal{R}_{\mathrm{\mathrm{mch}},\kappa}^\pm$, we analyze the 
functions $\mathcal{I}_{\mathrm{\Gamma}}(c)$, 
$\mathcal{I}_{\mathrm{\Theta}}(d)$. 
Recall that $\mathcal{I}_{\Gamma}(c)(y^+)=\Gamma(y^+)$. Therefore
	$$\begin{array}{lcl}
	\left\|\mathcal{I}_{\Gamma}(c)-\dfrac{2\mu_{3}}{\e}C_{\mathrm{\mathrm{in}}}e^{-i\frac{\lambda_3}{\e}(y-i\pi/2)}\sin(3\tau)\right\|_{\n}(y^+)&=&\left\|\Gamma-\dfrac{2\mu_{3}}{\e}C_{\mathrm{\mathrm{in}}}e^{-i\frac{\lambda_3}{\e}(y-i\pi/2)}\sin(3\tau)\right\|_{\n}(y^+)\vspace{0.2cm}\\
	&=&\left\|E_{\Gamma}^+\right\|_{\n}(y^+)\vspace{0.2cm}\\
	&\leq& M\left(\dfrac{ |e^{-i\mu_{3}\frac{y^+-i\pi/2}{\e}}|}{|y^+-i\pi/2|}+\dfrac{\e^{3/2}|\log \e|}{|y^+-i\pi/2|^2}\right)\vspace{0.2cm}\\
	&\leq& M\left(\dfrac{ e^{-\mu_{3}\kappa}}{\kappa\e}+\dfrac{\e^{3/2}|\log \e|}{\kappa^2\e^2}\right),
	\end{array}
	$$
	and notice that, from \eqref{Is1}, we have that
	$$\left\| \mathcal{I}_{\Gamma}(c)-\dfrac{2\mu_{3}}{\e}C_{\mathrm{\mathrm{in}}}e^{-i\frac{\lambda_3}{\e}(y-i\pi/2)}\sin(3\tau)  \right\|_{\n,0,\exp}=e^{\lambda_3\kappa}\left\|\mathcal{I}_{\Gamma}(c)-\dfrac{2\mu_{3}}{\e}C_{\mathrm{\mathrm{in}}}e^{-i\frac{\lambda_3}{\e}(y-i\pi/2)}\sin(3\tau)\right\|_{\n}(y^+),$$
	and thus, taking $\kappa=\dfrac{1}{2\lambda_{3}}\log(\e^{-1})$, we have that
\begin{equation*}
	\left\| \mathcal{I}_{\Gamma}(c)-\dfrac{2\mu_{3}}{\e}C_{\mathrm{\mathrm{in}}}e^{-i\frac{\lambda_3}{\e}(y-i\pi/2)}\sin(3\tau)  \right\|_{\n,0,\exp}\leq M\left(\dfrac{ e^{(\lambda_3-\mu_{3})\kappa}}{\kappa\e}+\dfrac{\e^{3/2}|\log \e|e^{\lambda_3\kappa}}{\kappa^2\e^2}\right)
	\leq\dfrac{M}{\e|\log\e|}.
	\end{equation*}
	The estimate on $\mathcal{I}_{\Theta}(d)$ follows analogously and it completes the proof of the lemma. 	
\end{proof}

Proposition \ref{prop:DifferenceDeltatilde} follows directly from Proposition \ref{propgigante} and Lemma \ref{lemma:EstimatesDiff}.

\section{Breathers with exponentially small tails: Proof of Proposition \ref{prop:FirstBif:Generalized}}\label{sec:ExpSmallTails}

To prove Proposition \ref{prop:FirstBif:Generalized} we analyze the intersection of the center-stable manifold $W^{cs}(0)$ and center-unstable manifold $W^{cu}(0)$ of the zero solution which form a tube homoclinic to the center manifold $W^c(0)$ in the phase space. In the original coordinates, they correspond to an infinite dimensional family of waves of \eqref{kleingordonrev} which are $\omega$-periodic in $t$ (with $\omega$ given in \eqref{epsilon} with $k=1$) and of the order 
\[
u(x,t) = \mathcal{O}(\e e^{-\e |x|})  + \mathcal{O}(e^{- \frac {\sqrt{2}\pi}\e}).
\]
In particular, the exponentially small oscillating tails do  not decay as $|x| 
\to \infty$. The construction of such generalized breathers is largely based on 
the approach in \cite{SZ03, LZ11, L14}, so we shall adapt the problem into the 
framework in \cite{LZ11}.

We shall adopt a slightly different coordinate system and phase space in this section compared to that in Section \ref{desc_sec}. Let 
\begin{equation}\label{def:TailsChange}
q=(q_1, q_2)^T := (v_1, \pa_y v_1)^T,  \; \; Q = \big(\e^{-1}  \wt \Pi[v], \, 
(-\pa_\tau^2 - \omega^{-2})^{-\frac 12} \wt \Pi[\pa_y v] \big)^T.  
\end{equation}
In the above, the operator $(-\pa_\tau^2 - \omega^{-2})^{-\frac 12}$ is bounded uniformly in $\e$ on $\ker \Pi_1$. 
In the $(q, Q)$ variables, equation \eqref{kleingordonrev}, or equivalently  
\eqref{kleingordonv}, takes the form 
\begin{equation} \label{E:KL-fast-slow} \begin{cases} 
\pa_y q = A q + F(q, Q, \e) \\
\pa_y Q = \frac J\e Q + G(q, Q, \e), 
\end{cases} \end{equation}
where 
\[
A = \begin{pmatrix} 0 & 1 \\ 1 & 0 \end{pmatrix}, \quad J = (-\pa_\tau^2 - \omega^{-2})^{\frac 12} \begin{pmatrix} 0 & 1 \\ -1 & 0 \end{pmatrix} : H^1(\mathbb{S}^1) \cap \ker \Pi_1  \to L^2(\mathbb{S}^1) \cap \ker \Pi_1,
\]
and with $v= q_1 \sin \tau + \wt \Pi[v]$,
\[
F(q, Q, \e) = \left(0, \, -\frac 14 q_1^3 + \left(-\frac{1}{\e^3 
\omega^3}\Pi_1\left[g(\e \omega v)\right]+\dfrac{q_1^3}{4}\right)\right)^T,
\]
\[
G (q, Q, \e)= \left(0,\,  - \frac{1}{\e^3 \omega^3}(-\pa_\tau^2 - 
\omega^{-2})^{-\frac 12} \wt \Pi\left[g(\e \omega v)\right]  \right)^T. 
\]
While $q \in X:=\R^2$, we take $Y=L^2(\mathbb{S}^1) \cap \ker \Pi_1$. Apparently 
\[
J: Y \supset D(J) = Y_1 \to Y, \;\; Y_1 :=H^1(\mathbb{S}^1) \cap \ker \Pi_1, \quad J^* =-J, 
\]
where $L^2$ and $H^1$ stand for the standard Sobolev space of square integrable functions and the subspace of $L^2$ functions with square integrable first order derivatives. It is straight forward to verify that $X_1=X$, $Y$, $Y_1$, $A$, $J$, $F$, and $G$ fit into the framework of  \cite{LZ11} and satisfy all assumptions (A1--A5) in Section~2, (B1--B5) and (C1--C2) in Section~4, and (D1--D5) in Section~6 there. (In fact $G$ satisfies a stronger estimates 
\begin{equation*} 
\|G\|_{H^1} \le M,  \quad \|D_q^{l_1} D_Q^{l_2} G\|_{L\big((\R^2 \times L^2) \otimes (\otimes^{l_1+l_2-1} (\R^2\times H^1)), \R^2 \times L^2 \big)} \le M \e^{3l_2},  
\end{equation*}
for some $M>0$ independent of small $\e>0$, on any bounded set in $X \times Y_1$.)  Therefore smooth local invariant manifolds of $0$, including the 1-dim stable and unstable manifolds analyzed in details in this current paper, exist with sizes and bounds (in $(q, Q)$ variables) uniform in $\e$ (Theorems 4.2, 4.9--4.11 in \cite{LZ11}).  

In the following, we consider the homoclinic tube formed by the intersection of 
the center-stable manifold $W^{cs}(0)$ and the center-unstable manifold 
$W^{cu}(0)$. We also include the estimate of the minimal value of 
the Hamiltonian $\mathcal{H}$ on the homoclinic tube which in turn yields an 
estimate on the minimal amplitude of the oscillating tails of the corresponding 
generalized breathers. \\

\noindent $\bullet$ {\it Notation.} In this section all differentiation $D$ 
are only with respect to the variables $(q, Q)$ in the phase space, but never 
with respect to $\e$. \\

\noindent $\bullet$ {\bf The local invariant manifolds and the restriction of the Hamiltonian $\mathcal{H}$ there.} Let $q_u$ and $q_s$ be the coordinates of $q$ in the eigenvector expansion 
\[
q= q_u (1, 1)^T + q_s (1,-1)^T
\]
in term of the stable and unstable eigenvectors. According to Theorem 4.2 in 
\cite{LZ11}, locally the center-unstable (or center-stable, center) manifold 
$W^{cu}(0) \in X \times Y_1$ (or $W^{cs}(0)$, $W^c(0)$) can be represented as 
the graph of a smooth mapping $h^{cu}(\cdot, \e): \{|q_u|, \|Q\|_{Y_1} \le \delta\}\subset Y_1 \times \R \to \R$ (or 
$h^{cs}$, $h^c$):
\[
\begin{split}
W^{cu} (0) \cap \{|q_u|, |q_s|, \|Q\|_{Y_1} \le \delta\} &= \{ q_s = 
h^{cu}(q_u, Q, \e)\},\\
W^{cs} (0)  \cap \{|q_u|, |q_s|, 
\| Q\|_{Y_1} \le 
\delta\}& = \{ q_u = h^{cs} (q_s, Q, \e) 
\}, \\
W^c (0)  \cap \{|q_u|, |q_s|, \| Q\|_{Y_1} \le \delta\}& = \{ (q_u, q_s) = h^c 
(Q, \e) \}, 
\end{split}
\] 
for some $\delta>0$ independent of sufficiently small $\e>0$. Moreover 
$h^{c\star} (q_\star,Q =0, 0)$, $\star=u,s$, is well-defined and correspond to 
the 1-dim stable and unstable manifold of \eqref{singularlimit} with $k=1$. They 
satisfy the following estimates\footnote{Actually some better 
estimates have been obtained in this current paper.}. For $l\ge 1$ and some $M>0$ independent of $\e$, for $\star=u,s$,
\begin{equation*} 
 |h^{c\star} (q_\star, 0, \e) - h^{c\star} (q_\star, 0, 0)| + | D_{q_\star} 
h^{c\star} (q_\star, 0, \e)-D_{q_\star} h^{c\star} (q_\star, 0, 0)| + \| D_Q 
h^{c\star} (q_\star, 0, \e)\|_{(H^1)^*} \le M \e,  
\end{equation*}
\begin{equation*}
Dh^{c\star}(0, 0, \e)=0, \quad Dh^{c}(0, \e)=0, \quad \|D^l h^{c\star} \| + \|D^l h^{c} \| \le M.  
\end{equation*}
In the $(q_u, q_s, w)$ variables the Hamiltonian $\mathcal{H}$ defined in \eqref{Hamil} takes the form
\[
\HH (q_u, q_s, Q, \e)= -2\pi q_uq_s + \frac 12 \left\|(-\pa_\tau^2 - 
\omega^{-2})^{\frac 12} 
Q\right\|_{L^2}^2 +  \int_{\TT} \left(\frac {v^4}{12} + \dfrac{F(\e \omega 
v)}{\e^4 
\omega^4}\right)d\tau,
\]
which is smooth in $(q, Q) \in \R^2 \times Y_1$ and $\e$ due to $F(u) = 
\mathcal{O}(u^6)$ near $u=0$. Since 
\[
D_Q^2 \HH (0,0,0, \e) (Q, Q) = \|(-\pa_\tau^2 - \omega^{-2})^{\frac 12} Q\|_{L^2}^2 \ge 
\frac 12 \| Q\|_{H^1}^2,  
\]
it is straight forward to obtain the uniform quadratic positivity of $\HH$ restricted on the center manifold $W^c(0)$
\begin{equation} \label{E:HH}
D_Q^2 \big(\HH \big( h^{c} (Q,\eps), Q, \e\big) \big)(\tilde Q, \tilde Q) \ge 
\frac 13 \| \tilde Q\|_{H^1}^2,  \quad \HH \big(h^{c} (Q, \e), Q, \e\big) \ge 
\frac 16 \| Q \|_{H^1}^2, \quad \forall Q \in Y_1, \; \|Q\|_{H_1} \le \delta.  
\end{equation} 
The quadratic positivity implies that the center manifold $W^c(0)$ is unique and $0$ is stable both forward and backward in $y$ on 
$W^c(0)$ for \eqref{kleingordonv} (with $k=1$). By the conservation of energy and the invariant foliation structure 
(Theorems 5.1, 5.3, and 5.4 in \cite{LZ11}), we have that $\HH \ge 0$ on $W^{c\star}(0)$ and it achieves 
$0$ exactly at $W^\star (0)$, $\star=u, s$. Therefore, at any $U \in 
W^\star(0)$, $\|U\|_{H^1} \le \delta$, $\star=u, s$, 
\[
T_{U} W^{c\star}(0)= \ker D \HH (U, \e), \quad \ker \big( D^2 \HH (U, 
\e)|_{T_{U} W^{c\star}(0)} \big) = T_U W^\star(0). 
\] 
Here $\ker D \HH (U, \e)$ is viewed as a linear functional on $\R^2 \times Y_1$ 
and $D^2 \HH (U, \e)|_{T_{U} W^{c\star}(0)}$ a bounded linear operator on $T_{U} 
W^{c\star}(0)$ induced by the symmetric quadratic form on $T_{U} W^{c\star}(0)$. 
Moreover, for any hyperplane $P$ in the tangent space $T_U W^{c\star}(0) 
\subset \R^2 \times Y_1$ transversal to $T_U W^\star (0)$, there exists 
$\sigma>0$ such that 
\begin{equation} \label{E:D2H}
\|D^2 \HH (U, \e)|_P\|_{L(P\otimes P, \R)} \ge \sigma.
\end{equation}

\noindent $\bullet$ {\bf Analyzing $W^{cs}(0) \cap W^{cu}(0)$.} In terms of the 
$(q, Q)=(q_1,q_2,Q)$ coordinates, let 
\[
\La = \{ q_2 =0\} \subset  \R^2 \times Y_1
\]
be the hyperplane perpendicular to the unperturbed homoclinic orbit 
\[
\Gamma^h := \left\{(v^h(y), \pa_y v^h (y),0) \mid y \in \R \right\} \subset 
\R^2 
\times Y_1, \qquad \text{(see \eqref{homoclinic})},
\] 
at $U_0 = (v^h(0), 0, 0)$. 

By Theorem 2.2 and 2.3 in \cite{LZ11}, for any fixed time $T>0$ the time--$T$ 
map of 
\eqref{E:KL-fast-slow} is smooth in the phase space $\R^2 \times Y_1$ with its 
derivative bounded uniformly in $\e$. (Even though only the first 
differentiation was carefully estimated in \cite{LZ11}, the uniform in $\e$ 
bounds of the higher order derivatives simply follow from a similar argument 
inductively.) Due to the uniform in $\e$ sizes and 
bounds on $W^{cs}(0)$ and $W^{cu}(0)$, they can be extended to stripes along 
$\Gamma^h$. For $\star=u, s$, consider the  following intersections with 
$\Lambda$ for the first time after $W^{c\star}(0)$ are extended from a 
neighborhood of $0$ by the flow of \eqref{E:KL-fast-slow},
\[
\wt W^{c\star} (0)= W^{c\star} (0) \cap \La, \quad U_\star = (q_{1,\star},0, 
Q_\star) = W^\star(0) \cap \La \in \wt W^{c\star}(0). 
\]
Clearly, here $U_\star$ corresponds to the values  of the stable and unstable 
solutions $\big(v^\star (0), \pa_y v^\star(0)\big)$ analyzed in Theorem 
\ref{maintheorem}. 

 We shall start with the decomposition $\La =  \big(\R (1, 0)^T\big) \oplus Y_1$ 
to set up a coordinate system to analyze $\wt W^{c\star}(0)$, $\star 
=u, 
s$. Here in particular we notice 
\begin{equation} \label{E:DHLa-1}
 \{q=0\} \times Y_1 = \ker D \HH(U_0, 0) \cap \La, \quad  (q_1,q_2,Q)=\nabla \HH(U_0, 0) = 
\frac {10\sqrt{2}}3 \big( 1, 0,0).  
\end{equation} 
Clearly $\wt W^{c\star}(0)$ is a hypersurface in $\La$. Due to the conservation 
of the Hamiltonian $\HH$ by the flow map, it holds 
\begin{equation*}
\HH(U_\star, \e)=0, \quad T_{U_\star} W^{c\star}(0) = \ker D \HH (U_\star, \e), 
\quad T_{U_\star}  \wt W^{c\star}(0) = \ker D \HH (U_\star, \e) \cap \La, 
\end{equation*} 
which implies that locally $\{\HH(\cdot, \e)=0\}\cap\Lambda$ 
and $\wt W^{c\star}(0)$ can be expressed as the graphs of  smooth mapping from 
$Y_1 
\to \R$. In fact, due to the smoothness of $\HH$ in $\e$ and the uniform in 
$\e$ bounds of $W^{c\star}(0)$ near $0$ and the flow map, there exist 
$\delta>0$ independent of $\e$ and $\wt h^0, \wt h^{c\star}: Y_1 \to \R$, 
$\star=u,s$, such that inside the box $\{ |q_1 - v^h(0)|, \|Q - Q_\star\|_{H^1} 
\le \delta\}$ in $\La$, 
\[
\begin{split}
\{\HH(\cdot, \e)=0\}\cap\Lambda&= \{ q_1 = \wt h^0 (Q, \e)\}, 
\,\quad \wt h^{0} (Q_\star, \e)=q_{1,\star}, \star=u,s \\
\wt W^{c\star} (0)
&= \{ q_1 = \wt h^{c\star} (Q, \e)\}, \quad \wt h^{c\star} (Q_\star, \e) 
=q_{1,\star}, \star=u,s,
\end{split}
\]
where $\wt h^0$ and $\wt h^{c\star}$ along with their derivatives are bounded 
uniformly in small $\e$.

Due to \eqref{E:DHLa-1}, it is clear 
\[
(q_1, 0, Q) \in \wt W^{cu}(0) \cap \wt W^{cs}(0) \Longleftrightarrow q_1= \wt 
h^{cu} (Q, \e) = \wt h^{cs} (Q, \e)  \Longleftrightarrow \HH\big( \wt h^{cu}(Q, 
\e), 0, Q, \e\big) = \HH\big( \wt h^{cs}(Q, \e), 0, Q, \e\big),  
\]
By \eqref{E:HH} and \eqref{E:DHLa-1}, there exists $C>0$ such that 
\begin{equation} \label{E:h-quadratic}
 0\leq \HH\big(\wt h^{c\star}(Q, \e), 0, Q, \e\big) = \HH\big(\wt h^{c\star}(Q, \e), 0, Q, \e\big)-\HH\big(\wt h^{0}(Q, \e), 0, Q, \e\big)\leq C(\wt h^{c\star}(Q, \e)-\wt h^{0}(Q, 
\e)), 
\end{equation}
which implies 
\begin{equation} \label{E:wth}
\wt h^{c\star} (Q, \e) \ge \wt h^0(Q, \e), \; \text{ and } \; ``=" \text{ holds 
iff }  Q=Q_\star, \quad \star=u, s.
\end{equation}
Moreover,
from \eqref{E:D2H}, the conservation of $\HH$, and the uniform in $\e$ bound on the flow map, we have     
\begin{equation} \label{E:H-1}
C \|Q- Q_\star\|_{H^1}^2\ge \HH\big(\wt h^{c\star}(Q, \e), 0, Q, \e\big) \ge 
\frac  1C \|Q- Q_\star\|_{H^1}^2, \quad \star=u, s, 
\end{equation}

Therefore if $Q_u = Q_s$, clearly $U_u = U_s$ and $W^s(0) = W^u(0)$ which gives 
rises to a homoclinic orbit to $0$. 
In the case of $Q_u \ne Q_s$, \eqref{E:wth} implies 
\[
\wt h^{cu} (Q_s, \e) > \wt h^0 (Q_s, \e) =\wt h^{cs} (Q_s, \e)  \; \text{ and } 
\; \wt h^{cs} (Q_u, \e) > \wt h^0 (Q_u, e)=\wt h^{cu} (Q_u, \e). 
\]
Therefore, there exists $\wt Q$, e.g. on the segment connecting $Q_u$ and 
$Q_s$, such that $\wt h^{cu} (\wt Q, \e) = \wt h^{cs}(\wt Q, \e)$ and thus  
\[
(q_1,q_2,Q)=\big(\wt h^{cu} (\wt Q, \e), 0, \wt Q\big) \in \wt W^{cs}(0)\cap 
\wt W^{cu}(0) \subset W^{cu} (0) \cap W^{cs} (0). 
\]
 This completes the proof of $W^{cs}(0) \cap W^{cu}(0) \ne \emptyset$, which had 
been also obtained in \cite{L14}. Moreover, \eqref{E:h-quadratic} and \eqref{E:H-1} imply that such that  
\[
D \wt h^{c\star} (Q_\star, \e) =0, \quad D^2 \wt h^{c\star} (Q_\star, \e) \ge \tfrac 1C >0, \qquad \star =u, s.
\] 
Since $Q_u$ and $Q_s$ are exponentially close and the derivatives of $h^{c\star}$ are bounded uniformly in $\e$, we obtain the transversality of the intersection of $W^{cs}(0) \cap W^{cu}(0)$ near the above mentioned $\wt Q$ on the segment connecting $Q_u$ and $Q_s$ if $Q_u \ne Q_s$.  See Figure \ref{fig:variedade}. This completes the proof of Proposition \ref{prop:FirstBif:Generalized}(3).

\begin{remark} \label{R:W-1}
The above argument is carried out in the energy space $(v, \pa_x v) \in H_\tau^1 (\mathbb{S}^1) \times L_\tau^2 (\mathbb{S}^1)$ which heavily depends on the coercivity of the conserved energy. Hence $W^{cs}(0) \cap W^{cu}(0)$ is obtained in the energy space. In fact, locally it contains a dense a subset consisting of smooth functions of $\tau$ in the case of  $Q_u \ne Q_s$. To see this, one observes that the transversality of the intersection implies that each nearby point in $W^{cs}(0) \cap W^{cu}(0)$ can be realized as a transversal intersection of $W^{cs}(0) \cap W^{cu}(0)$ and a smooth curve $\mathcal{C}$ connecting $U_s$ and $U_u$. It is easy to see that the proof of Theorem \ref{outerthm} can be carried out in any Sobolev space of higher regularity in $\tau$, hence $U_s$ and $U_s$ are smooth in $\tau$ as well. Approximating $\mathcal{C}$ by a curve in any higher order Sobolev space, we obtain a nearby point in $W^{cs}(0) \cap W^{cu}(0)$ of higher regularity in $\tau$.
\end{remark}

\begin{figure}[!]	
	\centering
	\begin{overpic}[width=9cm]{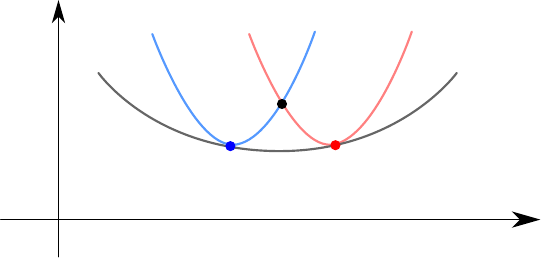}			
		
		\put(40,15){{\textcolor{blue}{\footnotesize$W^s(0)$}}}	
		\put(60,15){{\textcolor{red}{\footnotesize$W^u(0)$}}}	
		\put(18,40){{{\footnotesize$W^{cs}(0)$}}}	
		\put(78,40){{{\footnotesize$W^{cu}(0)$}}}	
		\put(80,25){{{\footnotesize$\mathcal{ H}=0$}}}				
	\end{overpic}
	\caption{Intersection between $W^{cs}(0)$ and $W^{cu}(0)$  giving rise to a breather with an exponentially small tail. }	
	\label{fig:variedade}
\end{figure}
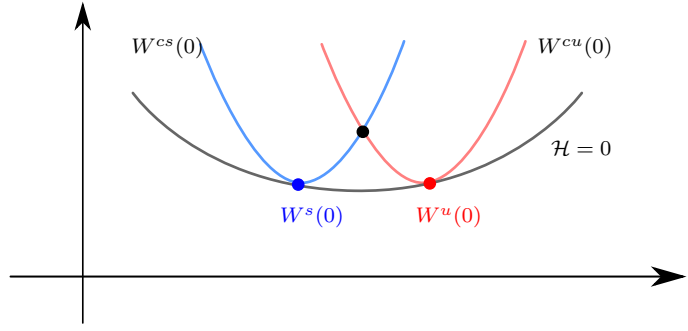 

Each orbit $\big(q(y), Q(y)\big)$ starting in $\wt W^{cs}(0) \cap \wt W^{cu}(0)$ is homoclinic to  $W^c(0)$. Due to the invariant foliation structure within $W^{c\star}(0)$ (Theorems 5.1, 5.3, and 5.4 in \cite{LZ11}), as $y\to \pm\infty$ it converges to 
two orbits in $W^c(0)$, which in the original coordinates (see \eqref{def:TailsChange}) can be written as  $\big(v^\pm_c (y), \pa_y v^\pm_c (y)\big) \subset W^c(0)$. 
Moreover, by \eqref{E:HH},  
\begin{equation*}
\tfrac 1C \HH(q, Q) = \tfrac 1C \HH(v^\pm_c, \pa_y v^\pm_c)\le \e^{-2} \|v_c^\pm\|_{H^1}^2 + \|\pa_y v_c^\pm \|_{L^2}^2 \le C \HH(v^\pm_c, \pa_y v^\pm_c)=C \HH(q, Q).
\end{equation*}
According to \eqref{E:H-1}, $\HH \big(q(0), Q(0)\big)$ can be used as an equivalent measure between the $\big(q(0), Q(0)\big)$ and $\big( v^\star (0), \pa_y v^\star (0)\big)$, $\star=u, s$. For $y$ on any finite interval, the square of the distance between $\big(q(y), Q(y)\big)$ and $\big( v^\star (y), \pa_y v^\star (y)\big)$ is proportional to $\HH\big(q(0), Q(0)\big)$ simply due to the uniform-in-$\e$ boundedness on the derivatives of the flow maps. When $\big(q(y), Q(y)\big)$ is close to $0$ (within a small $\er(1)$ distance), the distance between $\big(q(y), Q(y)\big)$ and $\big( v^\star (y), \pa_y v^\star (y)\big)$, where $\star = u$ for $y \ll 0$ and $\star=s$ for $y \gg1$, can be estimate using the stable/unstable foliations which along with their derivatives are bounded uniformly in $\e$ (see Section 5 of \cite{LZ11}). Combined with 
\[
\|Q\|_{H^1}^2 + |q|^2 \sim \big\| \e^{-1} |-\pa_\tau^2 -\omega^{-2}|^{\frac 12} v \big\|_{L^2}^2 + \|\pa_y v\|_{L^2}^2
\]
uniformly in $\e$, this finishes the proof of \eqref{E:gB-UB} and thus of Proposition \ref{prop:FirstBif:Generalized}(2). 

Finally, we estimate $\inf \HH$ on $\wt W^{cs}(0) \cap \wt W^{cu}(0)$. Let $Q \in Y_1$ such that $\| Q-Q_\star \|_{H^1} \le \delta$ and 
$\big( q_1=\wt h^{cu} (Q, \e), 0, Q\big) \in \wt W^{cs}\cap \wt W^{cu}$,  then 
\eqref{E:H-1} implies
\[
\|Q- Q_u \|_{H^1} \le C\| Q- Q_s \|_{H^1} \; \text{ and } \; \| Q- Q_s \|_{H^1} 
\le C\|Q- Q_u \|_{H^1},
\]
which further yields 
\[
C \|Q - Q_\star\|_{H^1} \ge \|Q_u -Q_s\|_{H^1}, \quad \star=u, s.
\]
Taking into account \eqref{E:H-1} again, one obtains
\[
\HH\big(\wt h^{c\star}(Q, \e), 0,Q, \e\big) \ge 
\frac  1C \|Q_u- Q_s\|_{H^1}^2. 
\]
Moreover, if such $Q$ is on the segment connecting $Q_u$ and $Q_s$, one has
\[
 \HH\big(\wt h^{c\star}(Q, \e), 0,Q, \e\big) \leq C \|Q_u- Q_s\|_{H^1}^2
\]
Therefore we obtain 
\[
C \|Q_u- Q_s\|_{H^1}^2\ge\inf_{\wt W^{cu}(0) \cap \wt W^{cs}(0) }  \HH \ge 
\frac 1C \|Q_u- Q_s\|_{H^1}^2,
\]  
and Proposition \ref{prop:FirstBif:Generalized}(1) follows from the above argument.

\section{Non-existence of small breathers: strongly hyperbolic case $\omega \in J_k(\e_0)$} \label{S:SHyperbolic}

This section is devoted to proving statement (1) of Theorem \ref{T:main}, that is the results for the case $\omega \in J_k (\e_0)$, $k\ge 0$. The other case $\omega \in I_k(\e_0)$ will be proved in Section \ref{sec:OtherBifs}. The oddness of $u$ in $t$ is not assumed to start with in these two sections. 

For any $\omega \in J_k (\e_0)$, $k> 0$, we adopt the rescaling $\tau = \omega t$ and the nonlinear Klein-Gordon equation \eqref{kleingordonrev} turns into the form of \eqref{eq:KLGtau}. Treating $x$ as the dynamic variable and recalling that $u$ is $2\pi$-periodic in $\tau$, the unknown $u(x, \tau)$ can be expanded in Fourier series 
\[
u(x, \tau) = \sum_{n=-\infty}^\infty u_n (x) e^{i n \tau}, \quad u_{-n} = \overline {u_n}. 
\]
(Note this Fourier series is different from the rest of the paper by a ratio of $-\tfrac i2$. The latter was adapted so that $u_n$ is the real coefficient of the Fourier sine series when $u$ is odd in $t$.)
The eigenvalues 
of the linearization of \eqref{eq:KLGtau} at $0$, that is
\[
\pa_x^2  u - \omega^2 \pa_\tau^2 u - u =0,
\]
are $\pm \nu_n$, where 
\begin{equation} \label{E:nu-1}
\nu_n = \sqrt{1- n^2 \omega^2}, 
\end{equation}
and their eigenfunctions can be calculated using the Fourier series. The hyperbolic eigenvalues correspond to $0\le n \le k$ and 
\begin{equation} \label{E:nu_k}
\nu_0 \ge \ldots \ge \nu_k \in \left(\frac {\e_0}{\sqrt{k + \e_0^2}}, \frac {\sqrt{2k+1}}{k+1}\right], 
\end{equation}
while the center eigenvalues correspond to $n \ge k+1$ and 
\begin{equation} \label{E:nu-2}
\nu_n = i\vartheta_n,\quad 0 \le \vartheta_{k+1} \le \vartheta_{k+2} \le \ldots. 
\end{equation}

Let $W^\star_\omega(0)$, $\star=c, s, u$, 
denote the locally invariant center, stable, and unstable manifolds of $0$ for the equation \eqref{eq:KLGtau} in the energy space $H_\tau^1 \times L_\tau^2$. Their existence and smoothness follow from standard arguments (see Theorem 4.4 in \cite{CL88}, for example) since the nonlinearity $g(u): H_\tau^1 \to H_\tau^2 \hookrightarrow L_\tau^2$ is analytic in $u$. Due to the uniqueness, $W^\star_\omega(0)$, $\star=s, u$, are also obviously the local stable and unstable manifolds of $0$ in the $\ell_1$ based phase space $(u, \pa_x u) \in \BFX$ defined in \eqref{E:phase-space}. On such finite dimensional submanifolds, different metrics including $H_\tau^1 \times L_\tau^2$ and $\| \cdot \|_\BFX$, all induce the same equivalent topology.   
Clearly $\dim W_\omega^\star (0) = 2k+1$, $\star=s, u$, while $W_\omega^c (0)$ is of codim-$(4k+2)$. 
Statement (1) of Theorem \ref{T:main} for the case of $\omega \in J_k (\e_0)$, $k\ge 0$, will be proved by showing a.) some uniform-in-$k$-and-$\omega$ estimates on the size of $W^\star_\omega(0)$ in $\ell_1$, $\star=s, u$, where the norm is dominated by the energy norm, 
and b.) no solutions converging to $0$ along $W^c_\omega(0) \subset H_\tau^1 \times L_\tau^2$. \\

\noindent $\bullet$ {\bf Estimates on the local stable/unstable manifolds for $\omega \in J_k (\e_0)$.}
Usually the sizes of the local stable/unstable manifolds in phase spaces are determined by the power nonlinearity and the minimal absolute value of the real parts of the stable/unstable eigenvalues, which is $\nu_k > \frac {\e_0}{\sqrt{k+ \e_0^2}}$ according to \eqref{E:nu_k}. We prove the following proposition on a lower bound of the sizes of $W^\star_\omega(0)$ in $\BFX$. 

\begin{prop} \label{P:size-SU} 
There exists $\rho, M>0$ such that, for any $\e_0 \in (0, 1/2)$, $\omega \in J_k(\e_0)$, $k\ge 0$, there exist $\Omega^u, \Omega^s: B_{\R^{2k+1}} (0, \rho \nu_k) \to \BFX$, where $B_{\R^{2k+1}} (0, \rho\nu_k)$ is the ball in $\R^{2k+1}$ centered at $0$ and with radius $\rho \nu_k$, such that, the image $\Omega^\star \big(B_{\R^{2k+1}} (0, \rho \nu_k)\big)$ is an open subset of $W_\omega^\star (0)$, $\star = s, u$, and  
\[
\Omega^\star (0, \tau)=0, \quad \|\Omega_1^\star (a, \cdot) - \Omega_1^\star (\tilde a, \cdot)\|_{\ell_1}  + \nu_k^{-1} \|\Omega_2^\star (a, \cdot) - \Omega_2^\star (\tilde a, \cdot)\|_{\ell_1}  \le M\nu_k^{-2} (|a|_1^2 + |\tilde a|_1^2) |a-\tilde a|_1
\]
where 
\[
\Omega^\star (a, \tau) =\left(\sum_{n=-k}^k a_n e^{ i n \tau} + \Omega_1^\star (a, \tau), \ \sum_{n=-k}^k (\mp\nu_n) a_n e^{ i n \tau}+ \Omega_2^\star (a, \tau) \right),       
\]
and $a$ and $\tilde a$ are parameters of $(2k+1)$-dim (real) satisfying 
\begin{equation} \label{E:a_k}
a=(a_{-k}, \ldots, a_k), \; \; a_n \in \C, \;\; a_{-n} = \overline{a_n}, \;\; -k\le n\le k, \quad |a|_1:= \sum_{n=-k}^k |a_n| < \rho \nu_k.  
\end{equation} 
\end{prop}

Here we identified complex numbers $a_n$ with 2-dim real vectors. These $\Omega^\star$ can be viewed as coordinate mappings of $W_\omega^\star(0)$ (see Figure \ref{figlocal}). They can actually be proved to be analytic in $a$, but our main focus here is the sizes of their domains and the error estimates. 

\begin{figure}[!]	
	\centering
	\begin{overpic}[width=7cm]{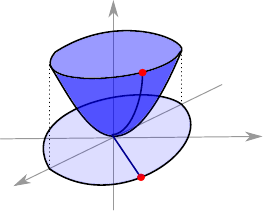}
		
		\put(75,35){{\footnotesize $B_{\R^{2k+1}}(0, \rho \nu_k) \subset\R^{2k+1}$}}	
		\put(15,65){{\footnotesize $W^s_{\omega}(0)$}}			
		\put(45,55){{\footnotesize \textcolor{red}{$\Omega^s(a)$}}}	
		\put(55,8){{\footnotesize \textcolor{red}{$a\in\partial B_{\R^{2k+1}}(0, \rho \nu_k)$}}}												
		\put(46,75){{\footnotesize $\ell_1\cap\{u_{-k}=\cdots=u_{-1}=u_{0}=u_{1}=\cdots=u_k=0\}$}}				
		
	\end{overpic}
	\bigskip
	\caption{Parameterization of the local stable manifold $W^s_{\omega}(0)$ by $\Omega^s$.}	
	\label{figlocal}
\end{figure} 

We use the classical Perron method and will only outline the argument to prove the proposition for the stable manifold. Consider the following Banach space
\[
\mathcal{E}_{\mathcal{S}}=\{h:[0,+\infty)\times\mathbb{T}\rightarrow \R;\ h \textrm{ is 
	analytic in }x, \textrm{ and } \|h\|_{\nu_k,\ell_1}<\infty \},
\]
where $$\|h\|_{\nu_k,\ell_1}=\displaystyle\sum_{n\geq 1}\|h_n\|_{\nu_k}\quad and \quad \|h_n\|_{\nu_k}=\sup_{x\geq 0}|e^{\nu_k x}h_n(x)|,$$
and define the linear operator $\mathcal{S}$ acting on the Fourier modes of a function $h(x,\tau)$
\begin{equation}
	\label{Gst}
	\mathcal{S}\left(h\right)=\displaystyle\sum_{n\geq 1}\mathcal{S}_n(h_n)\sin(n\tau), 
\end{equation}
with
\[
\begin{split}
\mathcal{S}_n(h)&=\dfrac{1}{2\nu_n}\left(\displaystyle\int_{+\infty}^xe^{
	\nu_n(x-s)}h(s)ds 
-\displaystyle\int_{0}^xe^{-\nu_n(x-s)}h(s)ds\right)\qquad \text{  for }\ 
1\leq n\leq k,\\
\mathcal{S}_n(h)&=\displaystyle\int_{+\infty}^x\dfrac{\sin(\vartheta_n(x-s))}{\vartheta_n}
h(s)ds \qquad \text{  for }\ n>k.
\end{split}
\]
where we recall $\nu_n=i\vartheta_n$.

Note that we are including in $J_k(\eps_0)$ the case $\omega=1/(k+1)$. For this value of $\omega$, one has that $\vartheta_{k+1}=0$. In this case, one can take the limit $\vartheta_{k+1}\to 0$ in $\mathcal{S}_{k+1}(h)$ to obtain
\[
 \mathcal{S}_{k+1}(h)=\displaystyle\int_{+\infty}^x(x-s)
h(s)ds.
\]

We also define the function
\[
\Xi(a, x,\tau)=\sum_{n=-k}^k a_ne^{-\nu_n x+ in\tau}, 
\]
where $a=(a_{-k},\ldots,a_k)$ are parameters satisfying \eqref{E:a_k}. 
One can check that a solution $u(x,\tau)$ of \eqref{eq:KLGtau} belongs to the stable manifold of $u=0$  if, and only if, it is a fixed point of the operator
\begin{equation*}
	\widetilde{\mathcal{S}}(a, u)=\Xi (a) +\mathcal{S}(g(u)),
\end{equation*}
for some $a$ as in  \eqref{E:a_k}, where $g$ is the nonlinearity introduced in \eqref{g}.

The following lemma is a direct consequence of the particular form of the operator $\mathcal{S}$ in \eqref{Gst} and the fact that the function $g$ is of order $3$ near $u=0$.

\begin{lemma}
	There exists $M, r_1>0$ independent of $k\ge 0$ and $\omega \in J_k (\e_0)$ such that, for any $0< r\le r_1$ and $a\in\mathbb{R}^{2k+1}$, the operator 
	$\widetilde{\mathcal{S}}: 
	B(0,r)\subset \mathcal{E}_{\mathcal{S}}\rightarrow \mathcal{E}_{\mathcal{S}}$ is a well-defined Lipschitz 
	operator which satisfies
	\[
	\left\|\widetilde{\mathcal{S}}(a, u)\right\|_{\nu_k,\ell_1}\leq |a|_1+\frac{M r^3}{\nu_k^2},\;\; \left\|\pa_x \widetilde{\mathcal{S}}(a, u)\right\|_{\nu_k,\ell_1}\leq |a|_1+\frac{M r^3}{\nu_k}, \qquad \forall \, u\in B(0,r)\subset \mathcal{E}_{\mathcal{S}}
	\]
	and  its Lipschitz constant  on $B(0, r)$ satisfies
	\[
	\mathrm{Lip}_u (\widetilde{\mathcal{S}})\leq \dfrac{Mr^2}{\nu_k^2}, \quad \mathrm{Lip}_u (\pa_x \widetilde{\mathcal{S}})\leq \dfrac{Mr^2}{\nu_k}. 
	\]
\end{lemma}

Consequently, there exists $\rho>0$ independent of $k\ge 0$ and $\omega \in J_k (\e_0)$ such that, for any $a \in B_{\R^{2k+1}}(0, \rho \nu_k)$, by taking $r = 2 |a|_1$, there exists a unique fixed point $h_* (a) \in B(0,r)\subset \mathcal{E}_{\mathcal{S}}$ of $\wt {\mathcal{S}}(a, \cdot)$ which also satisfies $h_*(a=0)=0$ and 
\begin{align*}
&\big\|\big(h_*(a) - \Xi(a)\big) - \big(h_*(\tilde a) - \Xi(\tilde a)\big) \big\|_{\nu_k,\ell_1} \le M \nu_k^{-2} (|a|_1^2 + |\tilde a|_1^2) |a-\tilde a|_1 \\
& \big\|\pa_x \big(h_*(a) - \Xi(a)\big) - \pa_x \big(h_*(\tilde a) - \Xi(\tilde a)\big) \big\|_{\nu_k,\ell_1} \le M \nu_k^{-1} (|a|_1^2 + |\tilde a|_1^2) |a-\tilde a|_1.
\end{align*} 
Let 
\[
\Omega^s (a, \tau) = \big( h_*(a, 0, \tau), \pa_x h_*(a, 0, \tau) \big). 
\]
The conclusions of Proposition \ref{P:size-SU}  follow from standard and straight forward arguments.\\ 
 
\noindent $\bullet$ {\bf Nonexistence of decaying solutions on the center manifold.
} 
Recall that, when $x$ is viewed as the dynamic variable, the nonlinear Klein-Gordon equation \eqref{eq:KLGtau} conserves the Hamiltonian $\HH(u, \pa_x u)$ where  
\[
\HH (u_1, u_2) = \int_{-\pi}^\pi \left(\frac 12 u_2^2 + \frac {\omega^2} 2 (\pa_\tau u_1)^2 -\frac 12 u_1^2 + \frac 1{12} u_1^4 + F(u_1)\right) d\tau 
\]
is smoothly defined on the energy space $H_\tau^1 \times L_\tau^2$. Let $W_\omega^c (0)$ be a center manifold of $(0, 0)$ for \eqref{eq:KLGtau}. The following lemma holds for all $\omega >0$, not just those in $I_k(\e_0)$ or $J_k(\e_0)$. 

\begin{lemma} \label{L:energy-Wc}
For any 
$\omega >0$, $(0, 0)$ is a strict local minimum of $\HH$ restricted on its local center manifold . 
\end{lemma}

\begin{proof}
There exists unique $k\ge 0$ such that $\omega \in [1/(k+1), 1/k)$. Let $\BFY^c, \BFY^h \subset H_\tau^1 \times L_\tau^2$ denote the center and hyperbolic subspaces of the linearization of  \eqref{eq:KLGtau} at $(0, 0)$
\begin{align*}
&\BFY^c = \{(u_1, u_2) \in H_\tau^1 \times L_\tau^2 \mid u_j( \tau) = \sum_{|n| \ge k+1} u_{j,n} e^{i n \tau}, \ u_{j, -n} = \overline {u_{j, n}}, \ j=1,2\}, \\ 
&\BFY^h = \{(u_1, u_2) \mid u_j( \tau) = \sum_{|n| \le k} u_{j,n} e^{i n \tau}, \ u_{j, -n} = \overline {u_{j, n}}. \ j=1,2\}.   
\end{align*} 
Locally $W_\omega^c(0)$ can be represented as the graph of a smooth mapping $\gamma^c (u_1, u_2)$ from a small neighborhood of $(0, 0)$ in $\BFY^c$ to $\BFY^h$. Due to the lack of quadratic nonlinear terms in \eqref{eq:KLGtau}, $\gamma^c$ satisfies   
\[
\gamma^c(u_1, u_2)= \mathcal{O}\big(\|u_1\|_{H_\tau^1}^3 + \|u_2\|_{L_\tau^2}^3 \big). 
\]
Due to $F(u) =\mathcal{O}(|u|^6)$ for small $u$ and the orthogonality between $\BFY^c$ and $\BFY^h$, for small $(u_1, u_2) \in \BFY^c$, 
\begin{align*}
\HH \big((u_1, u_2) + \gamma^c(u_1, u_2)\big) =& \int_{-\pi}^\pi \frac 12 u_2^2 + \frac {\omega^2} 2 (\pa_\tau u_1)^2 -\frac 12 u_1^2 + \frac 1{12} u_1^4 d\tau + \mathcal{O}\big(\|u_1\|_{H_\tau^1}^6 + \|u_2\|_{L_\tau^2}^6 \big)\\
\ge & \frac 12 \|u_2\|_{L_\tau^2}^2 + \pi \sum_{|n|\ge k+1}^\infty \vartheta_n^2 u_{1,n}^2 + \frac 1{24\pi} \|u_1\|_{L_\tau^2}^4 - \mathcal{O}\big(\|u_1\|_{H_\tau^1}^6 + \|u_2\|_{L_\tau^2}^6 \big). 
\end{align*}
If $\omega \ne \frac 1{k+1}$, then there exists $\delta>0$ such that 
\[
\frac {\vartheta_n^2}{1+n^2} \ge \delta, \ \forall |n| \ge k+1 \, \Longrightarrow \sum_{n=k+1}^\infty \vartheta_n^2 u_{1,n}^2 \ge \frac \delta{2\pi} \|u_1\|_{H_\tau^1}^2.
\]
Therefore in this case $(0, 0)$ is clearly a non-degenerate local minimum of $\HH$ on $W_\omega^c(0)$. If $\omega =\frac 1{k+1}$, then $\vartheta_{\pm (k+1)}=0$ and there exists $\delta>0$ such that 
\[
\frac {\vartheta_n^2}{1+n^2} \ge \delta, \ \forall |n| \ge k+2. 
\]
Let 
\[
\tilde u_1 = \sum_{|n|\ge k+2} u_{1,n} e^{i n \tau} 
\]
and then we have 
\begin{align*}
\HH \big((u_1, u_2) + \gamma^c(u_1, u_2)\big) \ge&  \frac 12 \|u_2\|_{L_\tau^2}^2 + \frac \delta2  \|\tilde u_1\|_{H_\tau^1}^2  + \frac 1{24\pi}  |u_{1,\pm (k+1)}|^4 - \mathcal{O}\big(\|\tilde u_1\|_{H_\tau^1}^6  + |u_{1,\pm (k+1)}|^6+ \|u_2\|_{L_\tau^2}^6 \big). 
\end{align*}
Again $(0, 0)$ is clearly a strict local minimum of $\HH$ on $W_\omega^c(0)$. 
\end{proof}

Due the conservation of $\HH$, we immediately obtain 

\begin{corollary} \label{C:centerM}
For any 
$\omega >0$, $(0, 0)$ has a locally unique center manifold $W_\omega^c(0)$ and is stable on $W_\omega^c(0)$ both forward and backward in $x$. Moreover, except $(0, 0)$ no solution on $W_\omega^c(0)$ converges to $(0, 0)$ as $x \to +\infty$ or $-\infty$. 
\end{corollary}

Finally we are ready to complete the proof of statement (1) of Theorem \ref{T:main}. \\ 

\noindent {\it Proof of statement (1) of Theorem \ref{T:main}.} Let $\e_0 \in (0, 1/2)$, $k\ge 0$, and $\omega \in J_k (\e_0)$. Since the $\|\cdot\|_{\ell_1}$ norm is invariant under a rescaling in $\tau$, we can work on \eqref{eq:KLGtau} equivalently.  Without loss of generality, assume  $u(x, \tau)$ is a solution such that $(u, \pa_x u)$ converging to $(0,0)$ in $H_\tau^1\times L_\tau^2$ as $x \to +\infty$. Such a  solution must belong to the local center-stable manifold of $(0,0)$ for $x \gg1$ (see Theorem 4.4 in \cite{CL88} or the outline of the arguments in Section \ref{difdinnersec}). It is well-known that the center-stable manifold is foliated into stable fibers based on the local center manifold $W_\omega^s(0)$ (for example, see Theorem 4.3 in \cite{CLL91}). The dynamics of all initial data on each fiber is shadowed by that of the based point on $W_\omega^c(0)$. According to Corollary \ref{C:centerM}, no non-trivial solutions on $W_\omega^c(0)$ converges to $(0,0)$ as $x\to +\infty$, the based point of the decaying solution $u(x, \tau)$ must be $(0, 0) \in W_\omega^s(0)$ and thus it belongs to the stable manifold $W_\omega^s(0)$. From Proposition \ref{P:size-SU}, locally the stable manifold $W_\omega^s(0) =\Omega^s \big(B_{\R^{2k+1}} (0, \rho \nu_k)\big)$ and $\Omega^s$ is a small perturbation of the isomorphism $\big(\Xi(a), \mathrm{diag}(\nu_{-k}, \ldots, \nu_k) \Xi(a)\big)$. In the coordinates $(a_{-k}, \ldots, a_k)$, the dynamics on $W_\omega^s(0)$ is governed by 
\[
\frac{da}{dx} = \big( D \Xi + D_a \Omega_1^s (a, \cdot)\big)^{-1} \big( -\mathrm{diag}(\nu_{-k}, \dots, \nu_k) \Xi(a) + \Omega_2^s (a, \cdot) \big) = -\mathrm{diag}(\nu_{-k}, \dots, \nu_k) a + \mathcal{O}(\nu_k \rho^2 |a|_1), 
\]

where the estimates on $\Omega_1^s$ and $\Omega_2^s$ given in Proposition \ref{P:size-SU} were used. It is straightforward to prove that, as $x$ evolves backwards, every solution on $W_\omega^s(0)$ must exit through its boundary where the $u_1$ component (corresponding to $u(x, \cdot)$ itself) satisfies $\|u_1\|_{\ell_1} \ge \frac 12 \rho \nu_k$. Finally statement (1) of Theorem \ref{T:main} follows from 
\[
\frac {\nu_k}{\omega^{\frac 12}} = \left( \frac 1 \omega - k^2 \omega\right)^{\frac 12} \ge \left( \sqrt{k(k+\e_0^2)} - \frac {k^{\frac 32}}{\sqrt{k+\e_0^2}}  \right)^{\frac 14}\ge\e_0 \left(\frac k{k+\e_0^2}\right)^{\frac 12} \ge \frac {\e_0}2, \;\text{ if } k\ge 1,
\] 
and $\nu_0=1$ if $k=0$. 
\hfill $\square$

\section{Bifurcation analysis for $\omega \in I_k(\e_0)$}\label{sec:OtherBifs}

We devote this section to the completion of the proof of Statement 2 of Theorem \ref{T:main} and Proposition \ref{prop:generalized}, that are the statements concerning $\omega \in I_k(\e_0)$. For such $\omega$, there are two pair of (weakly) hyperbolic eigenvalues along with $2k-1$ pairs of stronger ones (see \eqref{E:e-values-1}). Our strategy is to reduce the problem to $\omega \in I_1(\e_0)$ and $u(x, t)$ odd in $t$. 

We analyze the birth of small homoclinic loops taking 
\[
\omega=\sqrt{\frac{1}{k(k+\eps^2)}}\qquad \text{with}\qquad k\ge 1, \;\; 0<\eps\leq \eps_0\le  \frac 12.
\]
We expand the (real) solution $u(x, \tau)$ to the nonlinear Klein-Gordon equation \eqref{eq:KLGtau} in Fourier series in $\tau$ as  
\[
u(x, \tau) =\sum_{n=-\infty}^{+\infty} \left( -\frac i2\right) u_n (y) e^{in \tau}, \;\; u_{-n} =-\overline {u_n}, 
\]
where the $-i/2$ factor is simply for the technical convenience that, if $u(x, \tau)$ is odd in $\tau$, then $u_n(y)$, $n>0$, coincides with the coefficient in its Fourier sine series expansion. Subsequently  \eqref{eq:KLGtau} is equivalent to a coupled system of equations in the form of  
\begin{equation}
\label{vnshorter3}
\pa_x^2 u_n=\nu_n^2 u_n -\Pi_n\left[g(u)\right],\quad n \in \Z, 
\end{equation}
where $\Pi_n$ is the projection from $u(\tau)$ to the $n$-th mode $u_n$ as in the above expansion
and 
\begin{equation} \label{E:e-values-1} 
\nu_n = \sqrt{1-n^2\omega^2}, \quad 1=\nu_0 > \ldots >\nu_k = \e (k+\e^2)^{-\frac 12}, \; \; \nu_n = i \vartheta_n, \;\; \vartheta_{k+1} < \vartheta_{k+2} < \ldots, n\ge k+1, 
\end{equation} 
are same as those in \eqref{E:nu-1} and \eqref{E:nu-2}. In particular, 
\begin{equation} \label{E:e-values-2}
\nu_{k-1} = \sqrt{\frac {(2+\e^2)k -1}{k (k+\e^2)}} \ge  \sqrt{\frac 1k}, \quad \vartheta_{k+1} = \sqrt{\frac {(2-\e^2)k +1}{k (k+\e^2)}} \ge  \sqrt{\frac 1k}.   
\end{equation}
Linearizing at $u \equiv 0$, clearly $|n| \le k$ corresponds to $2k+1$ pairs of hyperbolic directions, and $|n| \ge k+1$ to codim-$(4k+2)$ center directions. From the same argument as in the proof of statement (1) of Theorem \ref{T:main} (see Section \ref{S:SHyperbolic}) based on Lemma \ref{L:energy-Wc}, a solution $u(x, \tau)$ satisfies $\|(u, \pa_x u)\|_\BFX \to 0$ as $x \to \pm \infty$ if and only if $(u, \pa_x u) \in W_\omega^\star (0)$, $\star=s, u$. Hence we shall focus on the estimates of the sizes and the splitting distance between $W_\omega^u(0)$ and  $W_\omega^s(0)$.

\subsection{Estimates on the local stable/unstable manifolds for $\omega \in I_k (\e_0)$ } \label{SS:SU-I-k} 

For a semilinear PDE like \eqref{vnshorter3}, the standard theorems (see, for example, Theorem 4.4 in \cite{CL88}) yield the existence of smooth local invariant manifolds $W_\omega^\star (0)$, $\star=s, u, c, cs, cu$, in the phase space $\BFX$ defined in \eqref{E:phase-space}. There are two issues, however. On the one hand, usually the sizes of the local invariant manifolds are generally determined by the gap between the real parts of the eigenvalues. While $\nu_n \ge k^{-\frac 12}$ for $|n| \le k-1$, the weakest stable/unstable eigenvalues $\pm \nu_k =\OO(\e k^{-\frac 12})$ of \eqref{vnshorter3} are too small for the analysis of possible breathers of amplitude $\|u \|_{\ell_1} = O(k^{-\frac 12})$. On the other hand, the ``angles" between the stable and unstable eigenfunctions in $\BFX$ of \eqref{vnshorter3} can be rather small for $n\sim k$. In this subsection, we shall outline the construction of $W_\omega^\star(0)$, $\star=s, u$, with desired estimates based on the specific structure of \eqref{eq:KLGtau}, or equivalently \eqref{vnshorter3}. Essentially our strategy is to construct $W_\omega^s(0)$ as the union of strong stable fibers based on a weak stable manifold. 

Observe 
\begin{equation*}
\mathcal{Z}_o=\left\{(u_1, u_2) \in \BFX \mid u_j (\tau)=\sum _{n\in\Z} \left( -\frac i2\right) u_{j, kn}e^{i kn\tau} = \sum _{n\in\N} u_{j, kn} \sin (kn\tau), \ u_{j, kn} \in \R,  \; j=1,2\right\},
\end{equation*}
is an invariant subspace under \eqref{eq:KLGtau}, or equivalently \eqref{vnshorter3}. Any such solution $u(x, \cdot) \in \mathcal{Z}$ is odd and actually $\frac {2\pi}{k \omega}$-periodic in $\tau$. Let 
\[
\tilde \e= k^{-\frac 12} \e \le \e_0, \quad \tilde \tau = k \tau, \quad \tilde \omega =k \omega = (1+\tilde \e^2)^{-\frac 12}, \quad y=\tilde \e \tilde \omega x, \quad u=\tilde \e \tilde \omega v,
\]  
then $v(y, \tilde \tau)$ is $2\pi$-periodic and odd in $\tilde \tau$ and satisfies 
\[
\partial_y^2v-\frac{1}{\widetilde\e^2 }\partial_{\widetilde\tau}^2v-\dfrac{1}{\widetilde\e^2 \widetilde\omega^2}v+\dfrac{1}{3}v^3+\dfrac{1}{\widetilde\e^3 \widetilde \omega^3 }f\left(\widetilde\e \widetilde \omega v\right)=0.
\]
Note that this is in the form of \eqref{kleingordonv} with $k=1$ (and note that $0<\widetilde\e\leq \e\leq \e_0$). Therefore for any $y_0\in \R$, there exists $\e_0, M>0$ (independent of $k$) such that, for $\e\in (0, \e_0]$, Theorem \ref{maintheorem} applies to imply the existence of the unique odd-in-$\tilde \tau$ stable and unstable solutions $v_{\mathrm{wk}}^\star (y, \tilde \tau)$ of \eqref{eq:KLGtau} such that $(v_{\mathrm{wk}}^\star, \pa_y v_{\mathrm{wk}}^\star) \in \BFX_o$ (see \eqref{def:Xo}) and  
\begin{equation} \label{E:9-temp-1}
\left\|(1-\pa_{\tilde \tau}^2) \left( \begin{pmatrix} v_{\mathrm{wk}}^\star (y, \tilde \tau) \\ \pa_y v_{\mathrm{wk}}^\star (y, \tilde \tau) \end{pmatrix} - \begin{pmatrix} v^h (y) \\ (v^h)' (y)  \end{pmatrix} \sin \tilde \tau \right)\right\|_{\ell_1}\le M \tilde \e^2  v^h (y), \; \text{ where } v^h (y) = \frac{2\sqrt{2}}{\cosh y},    
\end{equation}
for $y\ge - y_0$ for $\star=s$ or $y\le y_0$ for $\star=u$. One notices that we replaced $\pa_{\tilde \tau}^2$ in Theorem \ref{maintheorem} which does not change the estimates as $v_{\mathrm{wk}}^\star (y, \tilde \tau)$ are odd in $\tilde \tau$. Moreover $\pa_y \Pi_1 [v_{\mathrm{wk}}^\star (0, \cdot)] =0$ and they satisfy the exponentially small splitting estimate at $y=0$
\[
\left\| \Big( |-\pa_{\tilde \tau}^2 - \tfrac 1{\tilde \omega^2}|^{\frac 12} (v_{\mathrm{wk}}^u - v_{\mathrm{wk}}^s )+ i \tilde \e \pa_y (v_{\mathrm{wk}}^u - v_{\mathrm{wk}}^s) \Big) (0, \cdot) - \tfrac {4\sqrt{2}}{\tilde \e} C_{\mathrm{in}}e^{-\frac {\sqrt{2} \pi} {\tilde \e}} \sin 3\tilde \tau  \right\|_{\ell_1}
\le  \frac {M e^{-\frac {\sqrt{2} \pi} {\tilde \e}}}{\tilde \e \log (\tilde \e)^{-1}}.  
\]
Proposition \ref{prop:FirstBif:Generalized} there exist solutions $v(y, \tilde \tau)$ in $\BFX_o$ homoclinic to either $0$ or its center manifolds, which have bounds in terms of the values of their Hamiltonian $\HH$. From the estimates on the splitting and the $\inf \HH$ in Theorem \ref{maintheorem}, these orbits satisfy 
\begin{equation} \label{E:9-temp-2} 
\tilde \e^{-1} \big \| |-\pa_{\tilde \tau}^2 - \tfrac 1{\tilde \omega^2}|^{\frac 12} (v - v_{\mathrm{wk}}^\star )\big\|_{L_{\tilde \tau}^2 (-\pi, \pi)} + \| \pa_y (v - v_{\mathrm{wk}}^\star) \|_{L_{\tilde \tau}^2 (-\pi, \pi)}   \le M \tilde \e^{-2} e^{-\frac {\sqrt{2} \pi} {\tilde \e}},   
\end{equation} 
for $y\ge -y_0$ with $\star=s$ and $y\le y_0$ with $\star=u$. When $C_{\mathrm{in}} \ne 0$, a lower bound of the same order also holds. 

We rescale and obtain the unique stable and unstable solutions 
\[
u_{\mathrm{wk}}^\star (x, \tau) =\sqrt{k} \e \omega v_{\mathrm{wk}}^\star (\sqrt{k} \e \omega x, k \tau), 
\]
of \eqref{eq:KLGtau} such that $(u_{\mathrm{wk}}^\star, \pa_x u_{\mathrm{wk}}^\star) \in \cZ_o$ for any $x\in \R$. 
For any $x_0\in \R$, there exists $\e_0, M>0$ independent of $k$ such that, for $\e\in (0, \e_0]$, 
\begin{equation} \label{E:u-wk-1}
\left\|(1-\tfrac 1{k^2} \pa_\tau^2) \left(  \begin{pmatrix} u_{\mathrm{wk}}^\star (x, \tau) \\ \frac {\pa_x u_{\mathrm{wk}}^\star (x, \tau)}{\sqrt{k}\e \omega} \end{pmatrix} - \sqrt{k}\e \omega\begin{pmatrix} v^h (\e \sqrt{k}\omega x) \\ (v^h)' (\e \sqrt{k}\omega x)  \end{pmatrix} \sin k\tau\right)\right\|_{\ell_1}\le M k^{-\frac 12} \e^3 \omega v^h (\e \sqrt{k}\omega x), 
\end{equation}
for $x\ge - \tfrac {x_0}{\sqrt{k} \e \omega}$ for $\star=s$ or $x\le \tfrac {x_0}{\sqrt{k} \e \omega}$ for $\star=u$. Moreover $\pa_x \Pi_k [u_{\mathrm{wk}}^\star (0, \cdot)] =0$ and they satisfy the exponentially small splitting estimate
\begin{align*}
& \left\| \tfrac 1{(k\omega)^2} \Big( |-\omega^2 \pa_\tau^2 - 1|^{\frac 12} (u_{\mathrm{wk}}^u - u_{\mathrm{wk}}^s )+i (\pa_x u_{\mathrm{wk}}^u -\pa_x u_{\mathrm{wk}}^s) \Big) (0, \cdot) - 4\sqrt{2} C_{\mathrm{in}}e^{-\frac {\sqrt{2k} \pi} \e} \sin 3k\tau  \right\|_{\ell_1}
\le  \frac {M e^{-\frac {\sqrt{2k} \pi} \e}}{\tfrac 12 \log k - \log \e}.  
\end{align*} 
Since $0< 1- (k \omega)^2 < \tfrac {\e^2}k$, the stable and unstable solutions prove {\it statements (2a-b) of Theorem \ref{T:main}}. The existence and estimates of breathers with exponentially small tails (in $\cZ_o$) follow from the same rescaling and thus {\it Proposition \ref{prop:generalized}} is also proved. 

We shall prove statement (2c) of Theorem \ref{T:main} in the rest of the section. The translations (in $\tau$) of these solutions $(u_{\mathrm{wk}}^\star (x, \cdot + \theta), \partial_x u_{\mathrm{wk}}^\star (x, \cdot + \theta))$ form locally invariant  2-dim surfaces, parametrized by $x$ and $\theta \in \R/ (\tfrac {2\pi}k \Z)$, of the nonlinear Klein-Gordon equation \eqref{eq:KLGtau}, or equivalently \eqref{vnshorter3}, where solutions grow or decay at weak exponential rates. It is worth pointing out that $(u_{\mathrm{wk}}^\star (x, \cdot), \partial_x u_{\mathrm{wk}}^\star (x, \cdot))$, $x \in \R$, corresponds to only one of the two branches of the 1-dim stable/unstable manifold of \eqref{eq:KLGtau} in $\cZ_o$, while the other branch corresponds to $(-u_{\mathrm{wk}}^\star (x, \cdot), -\partial_x u_{\mathrm{wk}}^\star (x, \cdot))$. When $x=\pm \infty$ is included, the 2-dim surface generated by the translation in $\tau$ does include $0$ in the interior and the other branch (corresponding to $\theta = \tfrac \pi k$). 
Obviously they are submanifolds of the $(2k+1)$-dim stable/unstable manifolds and actually we shall construct the latter based on these weak ones (see Figure \ref{fig:center}).

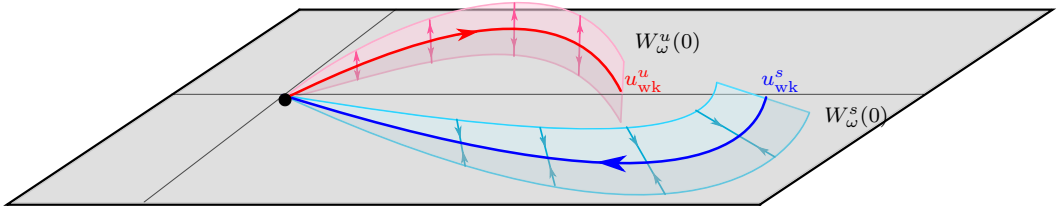
\begin{figure}[!]	
	\centering
	\begin{overpic}[width=14cm]{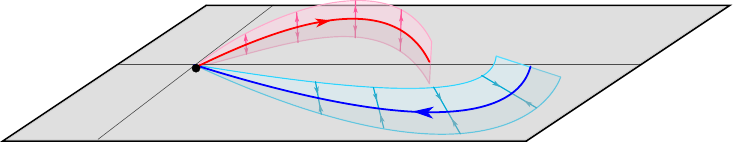}	
		\put(60,15){{\footnotesize$W^{u}_\omega(0)$}}	
		\put(78,8){{\footnotesize$W^{s}_\omega(0)$}}	
				\put(59,11.5){{\textcolor{red}{\footnotesize$u_{\mathrm{wk}}^u$}}}	
		\put(72,11.5){{\textcolor{blue}{\footnotesize$u_{\mathrm{wk}}^s$}}}	
	\end{overpic}\hspace*{0.7cm}
	\caption{We construct the stable manifold $W^{s}_\omega(0)$  as the 
union of the strong stable fibers based on  the weak stable manifold formed by 
the solution  $u_{\mathrm{wk}}^s$ (and its $\tau$-translations), which lives in 
the invariant subspace of $2\pi/k$-periodic-in-$\tau$ functions. We proceed 
analogously for the unstable manifold $W^{u}_\omega(0)$.}
 	\label{fig:center}
\end{figure} 

\begin{prop} \label{P:fibers}
There exist $M>1, \e_0, \rho_1$ independent of $k\ge 1$ and $\omega$, and unique mappings for any $\e\in (0, \e_0)$ and $\omega$ given in \eqref{E:k-omega},  
\[
\zeta^\star = \big(\zeta_1^\star (r,\theta, \delta), \zeta_2^\star (r, \theta, \delta) \big) \in \BFX, \quad \theta \in \R/ (\tfrac {2\pi}k \Z), \;\; \delta = (\delta_{1-k}, \ldots, \delta_{k-1}) \in \C^{2k-1}, \;\; \delta_{-n} = -\overline {\delta_n}, \;\; |\delta|_1 < \tfrac {\rho_1}{\sqrt{ k}}, 
\]
for $r\in \big[- \tfrac {y_0^s}{\e \sqrt{k}\omega}, +\infty\big]$, if $\star=s$ and $r \in \big[-\infty, \tfrac {y_0^u}{\e \sqrt{k}\omega}\big]$ if $\star=u$, where $y_0^\star$ are any values satisfying 
\[
\|u_{\mathrm{wk}}^u \|_{C_{y_0^u}^0} \triangleq  \|u_{\mathrm{wk}}^u \|_{C^0\big(r \in \big[-\infty, \tfrac {y_0^u}{\e \sqrt{k}\omega}\big], |\tau| \le \pi\big) } \le \tfrac {\rho_1}{\sqrt{k}}, \quad \|u_{\mathrm{wk}}^s \|_{C_{y_0^s}^0} \triangleq   \|u_{\mathrm{wk}}^s \|_{C^0\big(r \in \big[- \tfrac {y_0^s}{\e \sqrt{k}\omega}, +\infty\big], |\tau| \le \pi\big)} \le \tfrac {\rho_1}{\sqrt{k}}, 
\]
such that
\begin{align*}
& \Pi_n [\zeta_2^\star] \pm \nu_n \Pi_n[\zeta_1^\star ] =0, \; \forall \, |n|\le k-1, \;\; \star =u, s; \quad \Pi_{-n} [\zeta^\star] = \overline {\Pi_n[\zeta^\star]}, \; \forall \, n\in \Z; \quad \zeta^\star (r, \theta, 0) =0,\\
& \|\zeta_1^\star(r,\theta, \delta) - \zeta_1^\star (r, \theta, \tilde \delta) \|_{\ell_1} + \sqrt{k} \|\zeta_2^\star(r,\theta, \delta) - \zeta_2^\star (r, \theta, \tilde \delta) \|_{\ell_1}
\le k M (\|u_{\mathrm{wk}}^\star \|_{C_{y_0^\star}^0}^2 + |\delta|_1^2 + |\tilde \delta|_1^2) |\delta -\tilde \delta|_1, 
\end{align*}
and the images of $\xi^\star (r, \theta, \delta)$ is an open subset of $W_\omega^\star (0) \subset \BFX$  where 
\begin{align*}
\xi^\star (r, \theta, \delta)=\, &\big(\xi_1^\star (r, \theta, \delta), \xi_2^\star (r,\theta,  \delta) \big) =  \big(u_{\mathrm{wk}}^\star (r, \cdot +\theta), \pa_x u_{\mathrm{wk}}^\star(r, \cdot +\theta)\big) + \Xi^\star (\delta) + \zeta^\star (r, \theta, \delta), \\ 
\Xi^\star (\delta) = \, & \sum_{|n| \le k-1} \left(-\tfrac i2\right) \delta_n e^{in \tau} (1, \pm \nu_n), \;\; \star = u, s.
\end{align*}
Moreover, the orbits (of the dynamic variable $x$) on $W_\omega^\star (0)$ takes the form $\xi^\star \big(x+r, \theta, \delta(x) \big)$ with 
\[
\sum_{|n|\le k-1} \big| \pa_x \delta_n \mp \nu_n \delta_n\big| \le M\nu_{k-1}^{-1} ( \|u_{\mathrm{wk}}^\star \|_{C_{y_0^\star}^0}^2 + |\delta|_1^2) |\delta|_1, \;\; \star =u, s. 
\]
\end{prop}

\begin{remark} 
By including $r=\pm \infty$, where $u_{\mathrm{wk}}^\star (\pm\infty, \cdot)=\pa_x u_{\mathrm{wk}}^\star (\pm \infty, \cdot)=0$ for $\star=s, u$, the images of $\xi^\star$ do contain a whole open neighborhood of the zero solution in the stable/unstable manifolds $W_\omega^\star (0)\subset \BFX$. In fact, $\xi^\star (\pm \infty, \theta, \delta)$ become independent of $\theta$ and give the $(2k-1)$--dimensional strong stable/unstable manifolds corresponding to the eigenvalues $\pm \nu_n$, $|n|\le k-1$. In the $(r, \delta)$ coordinates on the invariant manifolds $W_\omega^\star(0)$, the PDE \eqref{eq:KLGtau} corresponds to a vector field whose $r$ component is always $1$ and the $\delta$ components depend on $r$ and $\delta$ which is a small perturbation to $\pm \nu_n \delta_n$. The following  proof could be carried out in the spaces with high regularity in $\tau$ such as $(1+ |\pa_\tau|)^{-N} \BFX$ for any $N \ge 0$ and thus the local invariant manifolds $W_\omega^\star (0) \subset (1+ |\pa_\tau|)^{-N} \BFX$ enjoy the same properties. The smoothness of $\zeta^\star$ in $r$ and $\theta$ is also true, for which we refer the readers to, for example, Theorem 4.3 in \cite{CLL91} for details, while we focus on the needed quantitative estimates on the sizes and the Lipschitz constant in $\delta$. 
Alternatively, one may also work on the rescaled variables as in \eqref{scaling} and obtain equivalent estimates.     
\end{remark} 

\begin{proof}[Proof of Proposition \ref{P:fibers}] 
The proof follows the standard Lyapunov-Perron method which we shall only outline for the unstable case. Given parameters $r$ and $\theta$, we see solutions to \eqref{eq:KLGtau} (or equivalently \eqref{vnshorter3}) in the form of 
\[
u(x, \tau) = u_{\mathrm{wk}}^u (r+x, \tau +\theta) 
+ U(x, \tau), \;\; x\le 0,  
\]
which decay to $0$ as $x \to -\infty$. The equation satisfied by $U$ takes the form  
\begin{equation}\label{eq:EqforU}
\LL_k U = \FF_k (U)
\end{equation}
where 
\begin{align*}
& \LL_k U = \sum_{n \in \Z} \big( (\pa_x^2 - \nu_n^2) U_n\big) e^{in \tau}, \quad  \FF_k (r, \theta, U) = g(u_{\mathrm{wk}}^u + U) - g(u_{\mathrm{wk}}^u), \; \text{ for } U(x, \tau) = \sum_{n\in \Z} U_n (x) e^{in \tau}.
\end{align*} 
Here we used the fact $u_{\mathrm{wk}}^s= u_{\mathrm{wk}}^u (r+x, \tau +\theta)$ is an exact solution. 
The decay property of $u(x, \tau)$ as $x\to+\infty$ is built into the Banach space  which $U$ belongs to 
\[
\PP = \big\{ U\in C^0 \big((-\infty, 0), \ell_1\big) \mid \|U\|_{\PP} := \sup_{x\le 0} e^{-\tfrac 23 \nu_{k-1} x} \|U(x)\|_{\ell_1} <\infty\big\}. 
\]
To set up the Lyapunov-Perron integral equation, define the linear transformation  
\begin{equation*}
\big(\mathcal{G}_k(h)\big)(x, \tau)=\displaystyle\sum_{n\in \Z}\big(\mathcal{G}_{k,n}(h_n) \big) (x) e^{in\tau}, \; \text{ where }  h(x, \tau) = \sum_{n\in \Z} h_n (x) e^{in \tau}, \quad x\le 0,
\end{equation*} 
with  
\[
\begin{aligned}
\big(\mathcal{G}_{k,n}(h_n)\big) (x)&=\dfrac 1{2\nu_{n}}e^{\nu_{n} x} \displaystyle\int_0^x e^{-\nu_{n} x'}h_n(x')dx'-\dfrac 1{ 2\nu_{n}} e^{-\nu_n x}\displaystyle\int_{-\infty}^x e^{\nu_{n} x'}h_n(x')dx',\,\, |n| \le k-1,\\
\big(\mathcal{G}_{k,n}(h_n)\big(x)\big) &=\dfrac 1{ \nu_{n}} \displaystyle\int_{-\infty}^x \sinh \big(\nu_n (x-x')\big) h_n(x') dx', 
\,\, |n|\geq k,
\end{aligned}\]
which serves as an inverse of $\LL_k$. Here we note that for $|n| >k$,  $\nu_n=i\vartheta_n$ and $\vartheta_n \ge k^{-\frac 12}$ and thus $\sinh \big(\nu_n (x-x')\big) = i \sin \big(\vartheta_n (x-x')\big)$. We also define 
\[
\wt \Xi (\delta, x, \tau) =   \sum_{|n| \le k-1} \left(-\tfrac i2\right) \delta_n e^{\nu_n x + in \tau}.   
\]
The desired solution $U$ satisfies the fixed point equation 
\[
U = \wt \FF (r, \theta, \delta, U) : = \wt \Xi (\delta) + \GG_k \big( \FF_k (r, \theta, U)\big).  
\]
Using \eqref{E:e-values-1} and \eqref{E:e-values-2}, it is straightforward to verify 
\[
\|\wt \Xi\|_{\PP} \le \frac 12 |\delta|_1, \quad \|\GG_k (h) \|_{\PP} \le \frac {100}{\nu_{k-1}^2} \|h\|_{\PP}, 
\]
\[ 
\|\FF_k (r, \theta, U) - \FF_k (r, \theta, \tilde U)\|_{\PP} \le M ( \|u_{\mathrm{wk}}^u \|_{C_{y_0^u}^0}^2 + \|U\|_{\PP}^2 + \| \tilde U\|_{\PP}^2) \|U - \tilde U\|_{\PP}. 
\]
Therefore there exists $\rho_1 >0$ independent of $\e$ and $k\ge 1$, such that, for $|\delta|_1 \le \tfrac {\rho_1}{\sqrt{k}}$, $\wt \FF$ is a contraction on the ball of radius $\tfrac {2\rho_1}{\sqrt{k}}$ in $\PP$. Let $U^u(r, \theta, \delta, x, \tau)$ be the unique fixed point of $\wt \FF$, 
\[
\zeta_1^u (r, \theta, \delta) = U^u (r, \theta, \delta, 0, \cdot) -\wt \Xi (\delta, 0, \cdot), \quad \zeta_2^u (r, \theta, \delta) = \pa_x U^u (r, \theta, \delta, 0, \cdot) - \pa_x \wt \Xi (\delta, 0, \cdot),
\]
and $\xi^u$ accordingly. The desired estimates on $\zeta^u$ follow from straightforward calculations. The invariance of $W_\omega^u(0) = \mathrm{image} (\xi^s)$ is a direct consequence of the uniqueness of the decaying solutions in $\PP$, which implies that solutions on $W_\omega^u(0)$ are parametrized by $\delta(x)$ and take the following two forms  
\begin{align*}
u(x, \cdot) = & u_{\mathrm{wk}}^u (r+x, \cdot +\theta) + U^u \big(r, \theta, \delta (0), x, \cdot\big)\\
= & \xi_1^u \big(r+x, \theta, \delta(x) \big) = u_{\mathrm{wk}}^u (r+x, \cdot +\theta) + U^u \big(r+x, \theta, \delta(x), 0, \cdot\big).  
\end{align*}
The invariance also allows us to obtain a more general identity along this solution $u(x)$ is 
\[
U^u \big(r, \theta, \delta (0), x+x', \cdot\big) = u(x+x', \cdot) - u_{\mathrm{wk}}^u (r+x+x', \cdot +\theta) = U^u \big(r+x, \theta, \delta (x), x', \cdot\big).
\]
From the definition of $\GG_k$, one may compute, for $|n|\le k-1$, 
\begin{align*}
\delta_n (x)  =& i \Pi_n \big[\big(I + \nu_n^{-1} \pa_{x'} \big) U^u \big(r+x, \theta, \delta(x), x', \cdot \big)\big]\big|_{x'=0} = i \Pi_n \big[\big(I + \nu_n^{-1} \pa_{x'} \big) U^u \big(r, \theta, \delta(0), x+x', \cdot \big)\big]\big|_{x'=0}.
\end{align*}
Therefore, differentiating this identity and using \eqref{eq:EqforU},
\[
\pa_x \delta_n (x)
=\nu_n \delta_n (x) + i\nu_n^{-1} \Pi_n [\FF_k(U^u)]\big|_{\big(r, \theta, \delta(0), x\big)}=\nu_n \delta_n (x) + i\nu_n^{-1} \Pi_n [\FF_k(U^u)]\big|_{\big(r+x, \theta, \delta(x), 0\big)}. 
\]
Letting $x=0$, we obtain the estimate on $\pa_x \delta$ straightforwardly and complete the proof of the proposition.  
\end{proof} 

The following corollary is direct consequence of the proposition and \eqref{E:u-wk-1}. 

\begin{corollary} \label{C:fibers-2}
For any $y_0\ge 0$, there exist $\e_0, \rho_1, M>0$ independent of $k$ and $\omega$, and unique mappings $\zeta^\star$, $\star=u, s$, for any $\e\in (0, \e_0)$ such that the results in Proposition \ref{P:fibers} along with $\|u_{\mathrm{wk}}^\star \|_{C_{y_0}^0} \le M  \tfrac \e{\sqrt{k}}$
hold for $r\in \big[- \tfrac {y_0}{\e \sqrt{k}\omega}, +\infty\big]$, if $\star=s$ and $r \in \big[-\infty, \tfrac {y_0}{\e \sqrt{k}\omega}\big]$ if $\star=u$. 
\end{corollary}

\subsection{Small homoclinic solutions} \label{SS:one-bump}

We first show that, regarding small breathers, $u_{\mathrm{wk}}^\star$, $\star=u, s$, is the only object that matters. 

\begin{prop} \label{P:smallB}
There exist $\e_0>0, \rho_2>0$ independent of $k\ge 1$ and $\omega$, such that, for any $\e\in (0, \e_0)$ and $\omega$ given in \eqref{E:k-omega}, a $2\pi$-periodic-in-$\tau$ solution $u(x, \tau)$ to \eqref{eq:KLGtau} exists satisfying     
\[
\sup_{x\in \R} \|u(x, \cdot)\|_{\ell_1} \le \tfrac {\rho_2}{\sqrt{k}},  \; \text{ and } \; \lim_{x\to -\infty} \| u(x, \cdot)\|_{H_\tau^1} + \| u(x, \cdot)\|_{L_\tau^2} = 0, 
\]
if and only if
\[
\sup_{x\in \R} \|u_{\mathrm{wk}}^u (x, \cdot)\|_{\ell_1} \le \tfrac {\rho_2}{\sqrt{k}} \ \text{ and } u(x, \tau) = u_{\mathrm{wk}}^u (x+r, \tau+\theta)
\]
for some $r, \theta \in \R$. Similarly $u(x, \tau) = u_{\mathrm{wk}}^s (x+x_0, \tau+\theta)$ for some $x_0, \theta \in \R$ if instead the above limit holds as $x\to +\infty$.   
\end{prop}

\begin{proof}
The ``$\Longleftarrow$" direction is obvious by \eqref{E:u-wk-1}. We shall only consider the 
``$\implies$" direction. 
Let $M\ge 1, \e_0\geq 0, \rho_1\geq 0, \xi^\star = (\xi_1^\star, \xi_2^\star), \zeta^\star = (\zeta_1^\star, \zeta_2^\star), \Xi^\star = (\Xi_1^\star, \Xi_2^\star)$, $\star=u, s$, be given by Proposition \ref{P:fibers}, and  
\[
\rho_2 = \min \{ \rho_1, M^{-1}\}/20. 
\]
We shall work on the case of $\star=u$ only as the proof for the other case is verbatim. The convergence of $u(x, \tau)$ as $x \to -\infty$ implies that $u(x, \cdot) \in W_\omega^u(0)$ which is also the unstable manifold of $0$ in the $\|\cdot\|_{\ell_1}$-based phase space $\BFX$. Hence, for all $x\ll -1$, there exists $r, \theta \in \R$ such that $u(x, \cdot) = \xi_1^u \big(r +x, \theta, \delta(x)\big)$ for some $\delta(x) \in \C^{2k-1}$ satisfying $ \delta_{-n} (x)=-\overline{\delta_n}(x)$. 

Let 
\begin{equation*}
x_1 = \sup\left\{ x\in \R \mid  \forall x'\le x, \ u(x', \cdot) = \xi_1^u \big(r +x', \theta, \delta(x')\big), \, \& \ |\delta(x')|_1, \|u_{\mathrm{wk}}^u \big(r + x', \cdot\big)\|_{\ell_1} \le \tfrac {10 \rho_2}{\sqrt{k}} \right\} \le +\infty. 
\end{equation*}
Clearly $x_1>-\infty$ since the image of $\xi^u$ is a neighborhood of $0$ in $W_\omega^u(0)$. The estimates on $\zeta^u$ given in Proposition \ref{P:fibers} and the $\frac{2\pi}k$-periodicity-in-$\tau$ of $u_{\mathrm{wk}}^u (x, \tau)$  imply that, for any $x \le x_1$,  
\[
\frac {\rho_2}{\sqrt{k}} \ge \Big\| \sum_{|n| \le k-1} \Pi_n [u(x, \cdot)] \Big\|_{\ell_1} = \Big\|\Xi_1^u \big(\delta(x)\big) +\sum_{|n| \le k-1}  \Pi_n  \big[\zeta_1^u \big( r+x, \theta,  \delta(x)\big) \big] \Big\|_{\ell_1} \ge \frac 12 |\delta(x)|_1.
\]
In turn, along with Proposition \ref{P:fibers} it yields
\[
\|u_{\mathrm{wk}}^u (r+x, \cdot )\|_{\ell_1} = \|u (x, \cdot ) - \Xi_1^u \big(\delta(x)\big) - \zeta_1^u (x+r, \theta, \delta(x)) \|_{\ell_1} \le \tfrac {\rho_2}{\sqrt{k}} + \tfrac 32 |\delta(x)|_1 \le \tfrac {4\rho_2}{\sqrt{k}}, \quad \forall x \le x_1.
\] 
The definition of $x_1$ immediately implies $x_1 =+\infty$ and, in particular, $\|u_{\mathrm{wk}}^u (x, \cdot)\|_{\ell_1} \le \tfrac {4\rho_2}{\sqrt{k}} < \tfrac {\rho_1}{\sqrt{k}}$ for all $x \in \R$. Again according to Proposition \ref{P:fibers}, $\xi^\star$ and $\zeta^\star$, $\star=u, s$, are well-defined near all $r \in \R$. Finally, from the estimate on the evolution $\pa_x \delta$ of $\delta(x)$, we have, for any $x\in \R$ and $|n| \le k-1$,
\[
|\pa_x \delta_n(x) - \nu_n \delta_n (x)|_1 \le \tfrac {20M \rho_2^2}{k \nu_{k-1}} |\delta (x)|_1 \le 20M \rho_2^2 \nu_{k-1} |\delta (x)|_1 \le  (\nu_{n}/2) |\delta (x)|_1.
\]
Therefore 
\[
|\delta_n (x)|_1 \le e^{-\frac {\nu_n}2 N} |\delta_n (x+N)|_1 \le 2 e^{-\frac {\nu_n}2 N} k^{-\frac 12} \rho_2,
\]
which implies $\delta_n (x) =0$ by taking $N \to +\infty$, and thus $u(x, \tau) = u_{\mathrm{wk}}^u (r+x, \tau +\theta)$. 
\end{proof}

As an immediate corollary, there exists a small breather solution $u(x, \tau)$ to the nonlinear Klein-Gordon equation \eqref{eq:KLGtau} satisfying \eqref{E:decay-E} as $|x| \to \infty$ and $\sup_{x\in \R} \|u(x, \cdot)\|_{\ell_1} \le \tfrac {\rho_2}{\sqrt{k}}$ iff $u_{\mathrm{wk}}^u (x+r, \tau) = u_{\mathrm{wk}}^s (x, \tau)$ for some $r\in \R$, namely $u_{\mathrm{wk}}^\star (x, \cdot)$, $\star=u, s$, are small and have the same orbits. The translation in $\tau$ is not needed since $u_{\mathrm{wk}}^\star$ are both primarily supported in the $k$-th mode if $\tau$ and odd in $\tau$. This proves Theorem \ref{T:main}(2c).

\section{The Stokes constant}\label{sec:Stokes}

We devote this section to analyze the Stokes constant $C_\mathrm{in}$ appearing in Theorems \ref{T:main} and \ref{prop:Stokesconstant}. As proved in Theorem \ref{innerthm}, $C_\mathrm{in}$ depends on the nonlinearity $f \in \FF_r$  analytically. 
In Section \ref{sec:proofstokes}, we complete the  proof of Theorem \ref{prop:Stokesconstant} by showing that $C_\mathrm{in} \ne 0$ 
in an open and dense set in $\FF_r$. In Section \ref{SS:Stokes}, we give some more discussions on $C_\mathrm{in}$ and conjecture a formula for $C_\mathrm{in}$ in terms of a power series.

\subsection{Proof of Theorem \ref{prop:Stokesconstant}}\label{sec:proofstokes}

To complete the proof of Theorem \ref{prop:Stokesconstant}, 
first we recall  the inner equation introduced in Section \ref{sec:strategyA}
\begin{equation}\label{inner0}
\partial_z^2\phi^0-\partial_{\tau}^2\phi^0-\phi^0+\dfrac{1}{3}
(\phi^0)^3+f(\phi^0)=0 ,
\end{equation}
where $f$ is a real-analytic odd function such that $f(u)=\er(u^5)$ for $|u|\ll1$.  
More concretely $f \in \mathcal{F}_r$, where  $\mathcal{F}_r$ is given in \eqref{def:Banach:f}.
Observe that  $\mathcal{F}_r$  is a Banach space with the norm $\|\cdot\|_{r}$.

Notice that \eqref{inner0} can be rewritten as
\begin{equation}\label{inner0delta}
\partial_\tau^2\phi_0-\partial_z^2\phi_0+\dfrac{1}{\sqrt{2}}\sin(\sqrt{2}\phi_0)+\Delta(\phi_0)=0,
\end{equation}
where $\Delta$ and $f$ are related through
\begin{equation}\label{perturbation}
	\Delta(u)=- \Big( \dfrac{1}{\sqrt{2}}\sin(\sqrt{2}u)-u+\dfrac{1}{3}u^3+f(u) \Big).
\end{equation}

From Theorem \ref{innerthm}, equation \eqref{inner0}, and therefore equation \eqref{inner0delta}, admit two 
solutions 
$\phi^{0,\star}:D^{\star,\mathrm{\mathrm{in}}}_{\theta,\kappa}\times\mathbb{T}\rightarrow \C$, $\star=u,s$, 
such that
$$
\displaystyle\lim_{|z|\rightarrow\infty} \phi^{0,\star}(z,\tau)=0, \ \forall  (z,\tau)   \in   D^{\star,\mathrm{\mathrm{in}}}_{\theta,\kappa}  \times  \mathbb{T}.
$$
In order to make explicit the  dependence of these solutions on $\Delta$, we shall denote $\phi^{0,\star}$ by $\phi^{0,\star}_{\Delta}$.

We also recall, that in Theorem \ref{innerthm}, we have proved  that
for 
\[
z\in \mathcal{R}^{\mathrm{\mathrm{in}},+}_{\theta,\kappa}= D^{u,\mathrm{\mathrm{in}}}_{\theta,\kappa}\cap D^{s,\mathrm{\mathrm{in}}}_{\theta,\kappa}\cap\{ z;\ z\in i\R \textrm{ and }\Ip(z)<0 \}, 
\]
we have
\begin{equation}
	\label{difmu}
	\phi^{0,u}_{\Delta}(z,\tau)-\phi^{0,s}_{\Delta}(z,\tau)=  e^{-i\mu_{3}z}\left( \mathcal{C}_{\mathrm{\mathrm{in}}}(\Delta)\sin(3\tau)+ \chi(z,\tau;\Delta)\right), \quad \mu_3=2\sqrt{2},
\end{equation}
where we have denoted $ \mathcal{C}_{\mathrm{\mathrm{in}}}(\Delta) =C_{\mathrm{\mathrm{in}}}(f)$,
the Stokes constant,
and $\chi_\Delta(z,\tau)=\chi(z,\tau;\Delta)$ are analytic functions in the variables $z$ satisfying  
$$
\| \chi_\Delta\|_{\ell_1}(z),\|\partial_{\tau}\chi_\Delta\|_{\ell_1}(z)\leq \dfrac{M_2}{|z|} \quad and \quad \|\partial_z\chi_\Delta\|_{\ell_1}(z)\leq \dfrac{M_2}{|z|^2},\ \forall z\in \mathcal{R}^{\mathrm{\mathrm{in}},+}_{\theta,\kappa}.
$$ 
For $\Delta\equiv0$, which corresponds to the sine-Gordon equation, from the explicit formula \eqref{breatherssine} and the asymptotics \eqref{innersol} of $\phi^{0,\star}$, a direct computation allows us to verify that the inner equation \eqref{inner0delta} admits  the  solutions
\begin{equation}\label{innerbreathers}
	\phi^{0,u}(z,\tau)=\phi^{0,s}(z,\tau)= \phi_{\mathrm{b}}^{0}(z,\tau)=\dfrac{4}{\sqrt{2}}\arctan\left(-\dfrac{i\sin(\tau)}{z}\right),
\end{equation}
which implies that $\mathcal{C}_{\mathrm{\mathrm{in}}}(0)= C_{\mathrm{\mathrm{in}}}(f^{\mathrm{sg}})=0$, where
	$f^{\mathrm{sg}}(\phi_0)= \RED{-} \dfrac{1}{\sqrt{2}}\sin(\sqrt{2} \phi_0)+\phi_0-\dfrac{1}{3}(\phi_0)^3$.

To prove Theorem \ref{prop:Stokesconstant}, we also consider the \textit{parameterized inner equation:}  
\begin{equation}
\label{innerpar}
\partial_\tau^2\phi_0-\partial_z^2\phi_0+\dfrac{1}{\sqrt{2}}\sin(\sqrt{2}\phi_0)+\mu\Delta(\phi_0)=0.
\end{equation}
where $\mu \in \R$ is a parameter.

Observe that equation \eqref{innerpar} corresponds to taking in equation \eqref{inner0}, the $\mu$-dependent function:
\[
f(\phi_0,\mu)=- \frac{1}{\sqrt{2}}\sin(\sqrt{2} \phi_0) + \phi_0 - \frac{\phi_0^3}{3} - \mu \Delta(\phi_0). 
\]
As equation \eqref{inner0}, with $f(\phi_0,\mu)$,
depends analytically (in fact linearly) on $\mu$, by Theorem \ref{innerthm} so do the solutions $\phi^{u,s}$ and the Stokes constant 
$C_{\mathrm{in}}(f(\cdot,\mu))=\mathcal{C}_{\mathrm{in}}(\mu \Delta)$.

For a given function $\Delta \in \mathcal{F}_r$, let us denote by
\[
c_{\mathrm{in}}^\Delta(\mu)=\mathcal{C}_{\mathrm{in}}(\mu \Delta)
\] 
which is an analytic function of $\mu$.
Consider the directional derivative
\begin{equation}\label{eq:cinlocal}
\begin{split}
c_{\mathrm{in}}^d: \mathcal{F}_r  & \to \C \\
\Delta & \mapsto  c_{\mathrm{in}}^d(\Delta)=\frac{   d c_{ \mathrm{in}}^{\Delta}   }{d \mu}(0).
\end{split}
\end{equation}
By the analyticity of $C_{\mathrm{in}}$, $c_{\mathrm{in}}^d: \FF_r \to \C$ is a bounded linear operator. We first state the following propositions. 

\begin{prop} 
For any $\Delta \notin \ker(c_{\mathrm{in}}^d)$, the set 
\[
\{ \mu \in \R : \mathcal{C}_{\mathrm{in}} (\mu \Delta) = 0 \}
\]
is a discrete subset of $\R$. 
\end{prop}

This proposition follows directly from 1.) the analyticity of $\mathcal{C}_{\mathrm{in}} (\mu \Delta)$ in $\mu$ as given in Theorem \ref{innerthm}(3) and 2.) $\mathcal{C}_{\mathrm{in}} (\mu \Delta)$ does not vanish identically due to the assumption $\Delta \notin \ker(c_{\mathrm{in}}^d)$. 

\begin{prop} \label{prop:cinlocal} 
The operator $c_{\mathrm{in}}^d$ satisfies $c_{\mathrm{in}}^d \ne 0$. 
\end{prop} 

We shall prove this proposition after we show that 
\[
\V=\{ \Delta \in \FF_r, \  \mathcal{C}_{ \mathrm{in}}(\Delta)\ne 0 \}
\]
is open and dense, which completes the proof of  Theorem \ref{prop:Stokesconstant}. In fact 
\begin{itemize}
\item $\V \subset \FF_r$ is open due to the continuity of  $\mathcal{C}_{ \mathrm{in}}(\Delta)$ in $\Delta$ as given in Theorem \ref{innerthm}. 
\item $c_{\mathrm{in}}^d \ne 0$ implies that there exists $\Delta_1 \in \FF_r \setminus \ker(c_{\mathrm{in}}^d)$. Therefore for any $\Delta \in \FF_r$, there exists $(\mu_1, \mu_2) \in \R^2$ arbitrarily close to $(1, 0)$ such that $\Delta + \mu_2 \Delta_1\notin \ker(c_{\mathrm{in}}^d)$ and $\mathcal{C}_{ \mathrm{in}}(\mu_1( \Delta + \mu_2 \Delta_1)) \ne 0$, which implies the density of $\V \subset \FF_r$. 
\item In fact, since the real dimension of $\ker(c_{\mathrm{in}}^d)$ is at most 2, the implication of the above propositions is much stronger than that $\V$ is open and dense. 
\end{itemize}

The rest of the section is devoted to prove Proposition \ref{prop:cinlocal}. 
We shall first derive a formula for  $c_{\mathrm{in}}^d(\Delta)$ defined in \eqref{eq:cinlocal} for $\Delta\in\mathcal{F}_r$. In order to do this, we consider the parameterized inner equation \eqref{innerpar}
and  we shall compute $c_{\mathrm{in}}^d(\Delta)$ 
through a {\it Melnikov-like analysis}. 
Thus, we write the solutions of equation \eqref{innerpar} as
\begin{equation}\label{expansion}
	\phi^{0,\star}_{\mu\Delta}(z,\tau)= \phi_b^{0}(z,\tau)+\mu \psi_1^{\star}(z,\tau)+\mu^2R^\star(z,\tau;\Delta,\mu), \quad \psi_1^{\star} = \Big( \frac d{d\mu} \phi^{0,\star}_{\mu\Delta} \Big) \Big|_{\mu=0},  \qquad \star=u, s,
\end{equation}
where $\phi_b^0$ is  given in \eqref{innerbreathers} and $R^\star(z,\tau;\Delta,\mu)$ is  analytic in the variable $z$ and also in the parameter $\mu$. 
Note that the functions $\psi_1^\star$ satisfy the same estimates as $\psi^\star$ in \eqref{def:innerestimates} in Theorem \ref{innerthm}.
 A direct computation shows that $\psi_1^{\star}$ satisfies the non-homogeneous linear equation
\begin{equation}
	\label{innerfirst}
	\partial_\tau^2\psi_1^{\star}-\partial_z^2\psi_1^{\star}+\cos(\sqrt{2} \phi_b^{0})\psi_1^{\star}-\Delta (\phi_b^{0})=0.
\end{equation}
As in the standard Melnikov analysis, each solution to the corresponding homogeneous equation -- the variational equation of \eqref{innerpar} around $\phi_b^0$ at $\mu=0$ -- can be used to measure the splitting between $\psi_1^u$ and $\psi_1^s$ in a certain direction, which yields the splitting of $\phi^{0,u}_{\mu\Delta}(z,\tau)$ and $\phi^{0,s}_{\mu\Delta}(z,\tau)$ in the leading order of $\mu$ for $|\mu| \ll 1$.
Next lemma gives the solutions of the variational equations around $\phi_b^0$.
\begin{lemma}
	The homogeneous linear partial differential equation
	\begin{equation}\label{homopde}
		\partial_\tau^2\xi-\partial_z^2\xi+\cos(\sqrt{2} \phi_b^{0})\xi=0
	\end{equation}
	has a family of solutions given by
	\begin{equation}\label{homosol}
		\xi_n^{\pm}(z,\tau)=\dfrac{2}{\mu_{n}^2}\left(\chi_{n}^{\pm}(z,\tau)-\chi_{-n}^{\pm}(z,\tau)\right),
	\end{equation}
	where $n\in\N$, $n\ge 2$, $\mu_{n}=\sqrt{n^2-1}$ and, for each $l\in\Z$, $\chi_{l}^{\pm}$ are the functions given by
	\begin{equation}\label{homosolchi}
		\begin{array}{ll}
			\chi_{l}^{\pm}(z,\tau)=&\vspace{0.2cm} e^{\pm i\mu_{l}z+il\tau}\left(1-\dfrac{\sin^2(\tau)}{z^2}\right)^{-1}\\
			&\times \left\{\pm\dfrac{\mu_{l}}{2z}-\dfrac{l}{2}\dfrac{\cos(\tau)\sin(\tau)}{z^2}-\dfrac{i}{4}
			\mu_l^2
			+i\dfrac{(l^2+1)}{4}\dfrac{\sin^2(\tau)}{z^2}\right\}.
		\end{array}
	\end{equation}
\end{lemma}
The proof of this lemma is obtained through a direct verification. 
In fact, the result is  a consequence of a particular case of Lemma 4 in \cite{D93}.

Next proposition gives a Melnikov integral type expression of the desired function $c_{\mathrm{in}}^d(\Delta)$:
\begin{prop}\label{melnikovformula}
For any function $\Delta\in  \mathcal{F}_{r}$, $c_{\mathrm{in}}^d(\Delta) \in \C$
satisfies
	\begin{equation}\label{derivadamu}
		c_{\mathrm{in}}^d(\Delta)=\left.\dfrac{d}{d\mu}\left(\mathcal{C}_{\mathrm{\mathrm{in}}}(\mu\Delta)\right)\right|_{\mu=0}=\dfrac{1}{2\pi i \mu_{3}}\displaystyle\int_{-\infty}^{\infty}\int_{0}^{2\pi}\Delta(\phi_b^{0}(z+s,\tau))\xi_3^+(z+s,\tau)d\tau ds,
	\end{equation}
	which is independent of $z\in \mathcal{R}^{\mathrm{\mathrm{in}},+}_{\theta,\kappa}$, 
		where $\xi_3^+$ is given in \eqref{homosol}.
\end{prop}

\begin{proof}
	Consider $\xi_3^+$ given in \eqref{homosol}. Since $\psi_1^u$ satisfies \eqref{innerfirst}, multiplying it by $\xi_3^+$, we obtain
	\begin{equation}\label{eq1}
		\xi_3^+\left(\partial_\tau^2\psi_1^{u}-\partial_z^2\psi_1^{u}+\cos(\sqrt{2} \phi_b^{0})\psi_1^{u}\right)=\xi_3^+\Delta(\phi_b^{0}).
	\end{equation}
Thus, for  $z\in D^{u,\mathrm{\mathrm{in}}}_{\theta,\kappa}$, we have 	
	\begin{equation}\label{neq1}
	{\small	\displaystyle\int_{-\infty}^0\int_0^{2\pi}	\xi_3^+\left(\partial_\tau^2\psi_1^{u}-\partial_z^2\psi_1^{u}+\cos(\sqrt{2} \phi_b^{0})\psi_1^{u}\right)(s+z,\tau) d\tau ds=	\displaystyle\int_{-\infty}^0\int_0^{2\pi}\xi_3^+\Delta(\phi_b^{0})(s+z,\tau)d\tau ds.}
\end{equation}
This integrals are well defined since $\psi_1^u$ satisfies the estimates \eqref{def:innerestimates} in Theorem \ref{innerthm}.
Integrating by parts with respect to $\tau$ twice and using that the functions are $2\pi$-periodic in $\tau$ we have that
\begin{equation}\label{neq2}
		\displaystyle\int_{-\infty}^0\int_0^{2\pi}\xi_3^+(s+z,\tau)\partial_\tau^2\psi_1^{u}(s+z,\tau)d\tau ds= 	\displaystyle\int_{-\infty}^0\int_0^{2\pi} \psi_1^{u}(s+z,\tau)\partial^2_{\tau}\xi_3^+(s+z,\tau)d\tau ds.
\end{equation}
Now, integrating by parts with respect to $s$ twice  and using the expression of $\xi_3^+$, we have that
\begin{equation}\label{neq3}
	\begin{split}	
	\displaystyle\int_{-\infty}^0\int_0^{2\pi}\xi_3^+(s+z,\tau)\partial_z^2\psi_1^{u}(s+z,\tau)d\tau ds=&\displaystyle\int_{-\infty}^0\int_0^{2\pi} \psi_1^{u}(s+z,\tau)\partial^2_{z}\xi_3^+(s+z,\tau)d\tau ds\\
	  & -\displaystyle\int_0^{2\pi}\left( \psi_1^u(z,\tau)\partial_z\xi_3^+(z,\tau)-\partial_z\psi_1^u(z,\tau)\xi_3^+(z,\tau)\right)d\tau. 
	 \end{split}
\end{equation}
Replacing \eqref{neq2},\eqref{neq3} in \eqref{neq1}  and using that  $\xi_3^+$ satisfies \eqref{homopde}, we have 
	\begin{equation}
	\label{neq4}
	\displaystyle\int_{-\infty}^0\int_0^{2\pi} \Delta(\phi_b^{0}(s+z,\tau))\xi_3^+(s+z,\tau)d\tau ds=\displaystyle\int_0^{2\pi}\left[\psi_1^u(z,\tau)\partial_z\xi_3^+(z,\tau) -\partial_z\psi_1^u(z,\tau)\xi_3^+(z,\tau)\right]d\tau.
\end{equation}
Analogously, if $z\in D^{s,\mathrm{\mathrm{in}}}_{\theta,\kappa}$, we obtain
\begin{equation}
	\label{neq5}
	\displaystyle\int_{+\infty}^0\int_0^{2\pi} \Delta(\phi_b^{0}(s+z,\tau))\xi_3^+(s+z,\tau)d\tau ds=\displaystyle\int_0^{2\pi}\left[\psi_1^s(z,\tau)\partial_z\xi_3^+(z,\tau) -\partial_z\psi_1^s(z,\tau)\xi_3^+(z,\tau)\right]d\tau.
\end{equation}	
Hence, subtracting  \eqref{neq4}, \eqref{neq5} we obtain
\begin{equation}
	\label{neq6}
	\begin{split}
		\displaystyle\int_{-\infty}^{+\infty}\int_0^{2\pi} \Delta(\phi_b^{0}(s+z,\tau))\xi_3^+(s+z,\tau)d\tau ds=&\vspace{0.2cm}\displaystyle\int_0^{2\pi}(\psi_1^u-\psi_1^s)(z,\tau)\partial_z\xi_3^+(z,\tau)d\tau \\
		&-\displaystyle\int_0^{2\pi}\partial_z(\psi_1^u-\psi_1^s)(z,\tau)\xi_3^+(z,\tau)d\tau.
	\end{split}
\end{equation}
Recall that, if $\mu=0$, then $\mathcal{C}_{\mathrm{\mathrm{in}}}(0)=0$ and $\chi(z,\tau;0)\equiv 0$ in \eqref{difmu}.
Now, using \eqref{expansion} and \eqref{difmu}, expanding $\mathcal{C}_{\mathrm{\mathrm{in}}}(\mu\Delta)$ and $\chi$ around $\mu=0$ and taking $\mu\rightarrow 0$, it follows that
\begin{equation}
	\label{neq7}
	\begin{split}
	\psi_1^u(z,\tau)-\psi_1^s(z,\tau)&=e^{-i\mu_{3}z}\left.\left(\dfrac{d}{d\mu}\mathcal{C}_{\mathrm{\mathrm{in}}}(\mu\Delta)\right|_{\mu=0}\sin(3\tau)+ \partial_\mu\chi(z,\tau;0)\right)\\&=   e^{-i\mu_{3}z}\left(c_{\mathrm{\mathrm{in}}}^d(\Delta)\sin(3\tau)+\er_{\ell_1}(z^{-1})\right).
	\end{split}
\end{equation}
Since \eqref{homosol} and \eqref{homosolchi} yield
\[
\xi_3^+(z,\tau)=e^{i\mu_{3}z}\left( \sin(3\tau)+\er_{\ell_1}(z^{-1})\right), \;\; \pa_z \xi_3^+(z,\tau)=i\mu_3 e^{i\mu_{3}z}\left( \sin(3\tau)+\er_{\ell_1}(z^{-1})\right), \; \text{ as } |z| \to \infty,
\]
a straightforward computation  of the right-hand side of \eqref{neq6} shows that, for each $z\in  \mathcal{R}^{\mathrm{\mathrm{in}},+}_{\theta,\kappa}$, 
\begin{equation}
	\label{eq5}
		\int_{-\infty}^{+\infty}\int_0^{2\pi} \Delta(\phi_b^{0}(s+z,\tau))\xi_3^+(s+z,\tau)d\tau ds	
	         =2\pi i\mu_{3}c_{\mathrm{\mathrm{in}}}^d(\Delta)+Q(z,\tau), 
\end{equation}
where $Q(z,\tau)=\er_{\ell_1}(z^{-1})$. Since the left-hand side of \eqref{eq5} does not depend on $z$ (one can just make the change of variables $\sigma=s+z$), the decay of $Q$ implies  that $Q\equiv 0$, and thus \eqref{derivadamu} holds.
\end{proof}

With a formula for $c_{\mathrm{\mathrm{in}}}^d(\Delta)$, $\Delta\in\mathcal{F}_r$, the following lemma finishes the proof of Proposition \ref{prop:cinlocal}.

\begin{lemma}
If $\Delta_0(u)=\left(-i\tan\left(\dfrac{\sqrt{2}}{4}u\right)\right)^5\left(1+ \tan^2\left(\dfrac{\sqrt{2}}{4}u\right)\right)$, then	
\[
\displaystyle c_{\mathrm{\mathrm{in}}}^d(\Delta_0)=\int_{-\infty}^{\infty}\int_{0}^{2\pi}\Delta_0(\phi_b^{0}(z+s,\tau))\xi_3^+(z+s,\tau)d\tau ds=\dfrac{26\pi^2}{15}i.
\]
\end{lemma}
\begin{proof}
	First, notice that
	\[
	\Delta_0(\phi_b^{0}(s+z,\tau))=-\dfrac{\sin^5(\tau)}{(s+z)^5}\left(1-\dfrac{\sin^2(\tau)}{(s+z)^2}\right),\]
	and
	$$\xi_3^+(s+z,\tau)=\dfrac{e^{i2\sqrt{2}(s+z)}}{8(s+z)^2\left(1-\dfrac{\sin^2(\tau)}{(s+z)^2}\right)}\left(4 \sin(\tau)+(8(s+z)^2+i4\sqrt{2}(s+z)-5)\sin(3\tau)+\sin(5\tau)\right).$$
	
	Therefore,
	$$\displaystyle\int_{-\infty}^{\infty}\int_{0}^{2\pi}\Delta_0(\phi_b^{0}(z+s,\tau))\xi_3^+(z+s,\tau)d\tau ds=\displaystyle\int_{-\infty}^{\infty}F(z,s)ds,$$
where $$F(z,s)=\dfrac{\pi e^{i2\sqrt{2}(s+z)}}{64(s+z)^7}\left(-33+ i10\sqrt{2}(s+z)+20(s+z)^2\right).$$	
	Recall that this integral is independent of $z\in \mathcal{R}^{\mathrm{\mathrm{in}},+}_{\theta,\kappa}$. Since $z=-i \kappa$, for some $\kappa>0$ sufficiently big, we have that $F(-i\kappa,s)$ has a pole at $s=i\kappa$.
	
	For $R>\kappa$ sufficiently big, consider $C_R=\{z;|z|=R,\ \Im(z)\geq 0\}$ and $L_R$ be the line segment between the points $-R$ and $R$. Let $\gamma_R(\theta)= Re^{i\theta}$, $0\leq \theta\leq \pi$, be a parameterization of $C_R$ and notice that
	$$F(-i\kappa, \gamma_R(\theta))=\dfrac{\pi e^{i2\sqrt{2}(Re^{i\theta}-i\kappa)}}{64(Re^{i\theta}-i\kappa)^7}\left(-33+ i10\sqrt{2}(Re^{i\theta}-i\kappa)+20(Re^{i\theta}-i\kappa)^2\right),$$
	which means, for some $M>0$ independent of $R$, 
	$$|F(-i\kappa, \gamma_R(\theta))|\leq M\dfrac{e^{-2\sqrt{2}R\sin(\theta)}}{|Re^{i\theta}-i\kappa|^5}.$$
	Since $0\leq\theta\leq\pi$, we have
	$$\displaystyle\lim_{R\rightarrow\infty}\int_{C_R}F(-i\kappa,s)ds=0.$$
	From Residue Theorem, it follows that
$$\displaystyle\oint_{C_R\cup L_R}F(-i\kappa,s)ds=2\pi i\textrm{Res}(F(-i\kappa,s),i\kappa)=\dfrac{26\pi^2}{15}i,$$
which implies 
$$\displaystyle\int_{-\infty}^{\infty}F(-i\kappa,s)ds=\dfrac{26\pi^2}{15}i.$$
\end{proof}

\subsection{A conjecture on the Stokes constant}\label{SS:Stokes}

As stated in Theorem \ref{prop:Stokesconstant}, generically  $C_{\mathrm{in}}$ does not vanish. We heuristically explain this fact from  points of views different compared to those given in Sections \ref{innersec} and  \ref{sec:proofstokes}.
Consider first a toy model  of \eqref{inner0} near the breather \eqref{innerbreathers} decomposed into Fourier modes in the form of \eqref{innersystemh}. A ``simplified'' equation for the third mode can be taken the form  
\begin{equation*} 
\gamma_3''(z)+\mu_{3}^2\gamma_3(z)= g(z) = \sum_{l=3}^\infty \frac {a_l}{z^l},
\end{equation*}
where we may start with the assumption that the power series on the righthand side is convergent outside a disk centered at the origin. The same proof as in Section \ref{innersec} yields two solutions  $\gamma^s$, $\gamma^u$ such that
\[
\gamma^\star (z) = \mathcal{O}(z^{-3}), \quad z \in D_{\theta, \kappa}^{\star,\mathrm{in}}, \quad \star =u,s, 
\]
where  $D_{\theta, \kappa}^{\star, \mathrm{in}}$ are the sectorial complex 
domains with vertex at $z=\pm \infty$ defined in   \eqref{innerdomainsol}.
However, in general, $\gamma^{u,s}$ cannot be extended to  analytic functions defined in a neighborhood of $\infty$. In fact,  $\gamma^{u,s}$ have the same formal asymptotic expansion $\tilde \gamma$, as $ |z|\to \infty$,
\[
\gamma^{u,s} (z)\sim \tilde \gamma (z)=\sum_{l=3}^\infty \frac {\gamma_l}{z^l}, \quad \gamma_l = \sum_{j=0}^{[\frac {l-3}2]} (-1)^j \mu_3^{-2(j+1)} \frac {(l-1)!}{(l-2j-1)!} a_{l-2j}, \qquad l \ge 3, 
\] 
where $[ b ]$ denotes the greatest integer no greater than $b$. We observe that $\tilde \gamma$ is generally a divergent series, but in the  Gevrey-1 class, namely
\[
\sup_{l} \left( \frac {|\gamma_l|}{l!}\right)^{\frac 1l} < \infty. 
\]
Hence, one may expect $\gamma^u \ne \gamma^s$ in general. 
There are several ways to see that 
and to provide an algorithm to compute their difference.\\

\noindent \textbf{Borel resummation.} 
One possibility is using the Borel resummation method as well as the Resurgence theory of \'Ecalle.
The main idea is to consider the formal Borel transform $\hat  \gamma$ (the inverse of the Laplace transform) term by term of the series
$\tilde \gamma$
\[
\hat  \gamma (\xi)=\sum_{l=3}^\infty \frac {\gamma_l}{l!}\xi ^{l-1},  \qquad l \ge 3.
\] 
If $\tilde \gamma$ is of Gevrey-1 class, then $\hat \gamma$ is a convergent series in a disk around  $\xi=0$ which gives an analytic function that we also denote $\hat \gamma$. As a first step, one can study the analytic extension of this function $\hat \gamma$ to the complex plane and its singularities. 

In our toy model, one can easily compute the equation satisfied by $\hat \gamma$,
\[
(\xi^2 +\mu_3^2)\hat \gamma (\xi)=\hat g(\xi)
\]
where $\hat g(\xi)=\sum_{l=3}^\infty \frac {a_l}{l!}\xi^{l-1}$ is an entire function.
Clearly, in this model, the only singularities  of $\hat \gamma$ are simple poles located at $\xi=\pm \mu_3 i$.

The second step to recover  the original functions is to compute the Laplace transform of the function $\hat \gamma$ along ``rays'' to infinity. 
The existence of singularities  of $\hat \gamma$ makes the Laplace transforms to be different when choosing different paths according to their decay requirements, obtaining two different functions $\gamma^{s,u}$, whose difference can be computed by means of the residuum theorem.
Again, using our toy model and assuming growth conditions on $\hat\gamma$ of the form $|\hat\gamma(z)|\lesssim e^{C|\xi|}$, one can recover $\gamma^{s,u}$ by choosing Laplace transforms of $\hat \gamma$ along positive or negative real axis
\[
\begin{split}
\gamma^s(z)&=\int _{0}^{\infty} e^{-z\xi}\hat \gamma (\xi) d\xi , \ \text{for}\quad \Rp z>C ,\\
\gamma^u(z)&=\int _{0}^{-\infty} e^{-z\xi}\hat \gamma (\xi) d\xi, \ \text{for}\quad \Rp z<-C.
\end{split}
\]
One can easily study the analytic continuation of these functions by changing the paths of integration. 
In this way we extend the functions to sectorial domains similar to $D_{\theta, \kappa}^{\star,\mathrm{in}}$ and study  its difference. For instance, for  $z=-i\kappa$, $\kappa>0$, 
for $\gamma^s$, we change the path to  $\Gamma^+=\{\arg(\xi)=\theta\}$, $\theta\in (0,\pi/2)$.
For $\gamma ^u$, we change the path to $\Gamma^-=\{\arg(\xi)=\pi-\theta\}$. Hence, we have 
\begin{equation}\label{eq:diference}
\begin{split}
\gamma^u(-i\kappa)-\gamma^s(-i\kappa) &=\int _{\Gamma^-} e^{i\kappa\xi}\hat \gamma (\xi) d\xi-\int _{\Gamma^+} e^{i\kappa\xi}\hat \gamma (\xi) d\xi \\
&=-2\pi i\mathrm {Res}(e^{i\kappa\xi}\hat \gamma (\xi),\xi= \mu_3 i) =-\pi  e^{-\mu_3 \kappa }\frac{\hat g(\mu_3 i)}{\mu_3},
\end{split}
\end{equation}
as we have assumed that $\hat\gamma$ has moderate growth at infinity   and, since $\hat \gamma(\xi)=\frac{\hat g(\xi)}{\xi^2+\mu_3^2}$, it has a simple pole at $\xi=\mu_3 i$ with residuum $\frac{\hat g(\mu_3 i)}{2\mu_3 i}$.  

Resurgence theory gives rigor to this argument when one construct the solutions $\phi^{0, \star}$, $\star=u,s$, for the full nonlinear equation \eqref{inner0}.
Roughly speaking, the constant  $C_{\mathrm{in}}$ appears in the computation of residue of the extension of such $\hat \gamma$ in the singularity closest to the origin. Using these ideas, one can develop an algorithm to compute $C_{\mathrm{in}}$.\\

\noindent \textbf{A more direct approach through Perron integrals.}
Let us end this section by proposing another approach to illustrate $\gamma^s \ne \gamma^u$ and also giving an algorithm to compute $C_{\mathrm{in}}$.
We can write an integral representation of the functions $\gamma^s$, $ \gamma^u$ using their  decay at infinity: 
\begin{align*}
& \gamma^u (z) = \dfrac{1}{2i\mu_{3}}\displaystyle\int_{-\infty}^z e^{-i\mu_{3}(s-z)} g(s) ds-\dfrac{1}{2i\mu_{3}}\displaystyle\int_{-\infty}^z e^{i\mu_{3}(s-z)} g(s) ds, \\
& \gamma^s (z) = \dfrac{1}{2i\mu_{3}}\displaystyle\int_{+\infty}^z e^{-i\mu_{3}(s-z)} g(s) ds-\dfrac{1}{2i\mu_{3}}\displaystyle\int_{+\infty}^z e^{i\mu_{3}(s-z)} g(s) ds. 
\end{align*}
For $\kappa >0$, let $B_\kappa \subset \C$ be the disk centered at $0$ with radius $\kappa$ and $S$ be the path going from $-\infty$ to $-\kappa$ along the negative real axis, then to $\kappa$ along the lower half of $\partial B_\kappa$, then to $+\infty$ along the real axis. 
By the Cauchy integral theorem we obtain 
\begin{align*}
\gamma^u (-i\kappa) - \gamma^s (-i\kappa) =&   
\dfrac{1}{2i\mu_{3}}\displaystyle\int_S e^{-i\mu_{3}(s  + 
i\kappa)} g(s) ds-\dfrac{1}{2i\mu_{3}}\displaystyle\int_S e^{i\mu_{3}(s 
+i\kappa)} g(s) ds \\
= & -\dfrac{e^{-\mu_3 \kappa} }{2i\mu_{3}}\displaystyle\oint_{\partial 
B_\kappa} e^{i\mu_{3}s} g(s) ds 
= - \pi 
e^{-\mu_3 \kappa}  \sum_{l=3}^\infty \frac {i^{l-1} \mu_3^{l-2}}{(l-1)!} a_l. 
\end{align*}
The above right side are related to $\hat g$, the Borel transform of $g$, evaluated at $i\mu_3$ and gives the difference between $\gamma^u$ and $\gamma^s$.  In fact, this formula is exactly the same as \eqref{eq:diference}.
This means that in the derivation of the stable/unstable solutions $\phi^{0, \star}$, $\star =u,s$, through the Lyapunov-Perron approach, a nonzero splitting appears even after the first iteration. \\

\noindent {\bf An algorithm to  compute $C_{\mathrm{in}}$.} As mentioned above, by the Borel-Laplace summation theory, $\phi^{0, u}$ and $\phi^{0, s}$, while analytic on their own domains and non-equal in the intersection of the domains, share the same formal series as $z \to \infty$ in suitable sectors
\begin{equation*} 
\phi^{0, \star} \sim \sum_{j=3}^\infty \frac {b_j}{z^j},
\end{equation*}
which is generally divergent, but belongs to the Gevrey-1 class,
Moreover, the right hand sides $F_n$ of \eqref{innersystemh} are also associated to formal series 
\[
F_n \left(\phi^{u,s}\right) (z) \sim \sum_{j=3}^\infty \frac {\beta_{n,j}}{z^j},
\]
in the Gevrey-1 class. 
The above considerations motivate us to make the following conjecture. 

\begin{conjecture}The constant $C_{\mathrm{in}}$ introduced in Theorem \ref{T:main} 
can be expressed as 
\[
C_{\mathrm{in}}
= - \pi \sum_{l=3}^\infty \frac {i^{l-1} \mu_3^{l-2}}{(l-1)!} 
\beta_{3,l}.
\] 
\end{conjecture}
Even though the formula of the splitting constant $C_{\mathrm{in}}$ in this 
conjecture is still very complicated, if proved, it would give an algorithm to 
compute $C_{\mathrm{in}}$ which may be implemented by numerical computations. 
The proof of this conjecture is beyond this paper.

\appendix

\section{Proof of Proposition \ref{prop_operators}}\label{app:Operator}

Unless stated otherwise, $M$ denotes any constant independent of $\kappa$ and $\e$. The proof of items (1) and (2) are straightforward using that  \eqref{Lop} acts on the Fourier coefficients of $\xi$.
To prove item (3), we consider $h\in\mathcal{E}_{m,\ag}$  and we estimate  $\GG_1(h_1)$ and $\GG_n(h_n)$ (see \eqref{g1} and \eqref{gn}). For $\GG_n$, using Lemma 5.5 in \cite{GOS}, one can see that 
	\begin{equation}\label{oppgn}
	\|\mathcal{G}_n(h)\|_{m,\ag}\leq \dfrac{M\e^2}{\lambda_n^2}\|h\|_{m,\ag},\ n\geq 2.\end{equation}
	Now we estimate $\GG_1$ given by 
	\[
\mathcal{G}_1(h)(y)=-\zeta_1(y)\displaystyle\int_{0}^y\zeta_2(s)h(s)ds+\zeta_2(y)\displaystyle\int_{-\infty}^y\zeta_1(s)h(s)ds.
	\]
First, we bound $\mathcal{G}_1(h)(y)$ for values of $y$ in $D^{\mathrm{out},u}_{\kappa}\cap\{\Rp(y)\leq -1\}$. Notice that the functions $\zeta_1(y),\zeta_2(y)$ given in \eqref{zetas} satisfy 
	\begin{equation}\label{zeta12comp}
	|\zeta_1(y)|\leq \dfrac{M}{|\cosh(y)|} \quad\textrm{ and }\quad |\zeta_2(y)|\leq M|\cosh(y)|,
	\end{equation}
	for every 
	\begin{equation*}
	y\in D^{\mathrm{out},u}_{\kappa}\cap\{y\in \C;\ |\Ip(y)|\leq -K\Rp(y)\}\qquad \text{where}\qquad K=\left(\tan(\bg)+\dfrac{\pi}{2}-\kappa\e\right).
	\end{equation*}
The second integral in $\GG_1$ 	satisfies that,  for every $y\in D^{\mathrm{out},u}_\kappa\cap\{\Rp(y)\leq-1 \}$,
\[	\left|\displaystyle\int_{-\infty}^y\zeta_1(s)h(s)ds\right|\leq  M\|h\|_{m,\ag}\displaystyle\int_{-\infty}^0\dfrac{1}{|\cosh^{m+1}(s+y)|}ds
	\leq  M\dfrac{\|h\|_{m,\ag}}{|\cosh^{m+1}(y)|}.
\]
Therefore
	\begin{equation}
	\label{segunda}
	\left|\zeta_2(y)\displaystyle\int_{-\infty}^y\zeta_1(s)h(s)ds\right|\leq \dfrac{M\|h\|_{m,\ag}}{|\cosh^{m}(y)|}.
	\end{equation}
Now, to estimate the first integral in $\GG_1$,  let $y^*$ be the unique point in the segment of line between $0$ and $y$ such that $\Rp(y^*)=-1$. Hence, it follows from \eqref{zeta12comp} that,
	\begin{enumerate}
		\item If $s$ is in the line between $0$ and $y^*$, then 
		\[
		|\zeta_2(s)h(s)|\leq\,\dfrac{M\|h\|_{m,\ag}|\cosh(s)|}{|s^2+\pi^2/4|^{\ag}}\leq\, M\|h\|_{m,\ag}.
		\]
		\item If $s$ is in the line between $y^*$ and $y$, then 
		\[|\zeta_2(s)h(s)|\leq\dfrac{M\|h\|_{m,\ag}}{|\cosh^{m-1}(s)|}.\]
	\end{enumerate}
	Thus since $m>1$, using the previous estimates, we have that
	\[
	\left|\int_0^y\zeta_2(s)h(s)ds \right|\leq\left|\displaystyle\int_{y^*}^0\zeta_2(s)h(s)ds \right|+\left|\displaystyle\int_y^{y^*}\zeta_2(s)h(s)ds \right|
	\leq M\|h\|_{m,\ag},
	\]
and consequently
	\begin{equation}
	\label{primeira1}
	\left|\zeta_1(y)\displaystyle\int_{0}^y\zeta_2(s)h(s)ds\right|\leq \dfrac{M\|h\|_{m,\ag}}{|\cosh(y)|}.
	\end{equation}
		Now, from \eqref{g1}, \eqref{segunda} and \eqref{primeira1}, we obtain that
	\begin{equation}
	\label{infinito}
	\displaystyle\sup_{y\in D^{\mathrm{out},u}_\kappa\cap\{\Rp(y)\leq-1 \}}\left|\cosh(y)\mathcal{G}_1(h)(y)\right|\leq M\|h\|_{m,\ag}.
	\end{equation}
	For the region $D^{\mathrm{out},u}_\kappa\cap\{\Rp(y)\geq-1\}$, we consider a new set of fundamental solutions  $\{\zeta_+,\zeta_-\}$ of $\mathcal{L}_1(\zeta)=0$ which has good properties at $\pm i \pi/2$. We rewrite the solutions $\zeta_1(y)$ and $\zeta_2(y)$ as linear combinations of $\zeta_+(y)$ and $\zeta_-(y)$ and use them to obtain a new expression of the operator $\mathcal{G}_1$.
We emphasize that the operator $\mathcal{G}_1$ is already defined. We  only express it in a different way.
	
\begin{lemma}
The functions 
	\begin{equation*}
	\zeta_{\pm}(y)=\zeta_1(y)\displaystyle\int_{\pm i \frac{\pi}{2}}^y\dfrac{1}{\zeta_1^2(s)}ds=-\dfrac{\sqrt{2}}{4}\dfrac{1}{\cosh^2(y)}\left(\dfrac{3y\sinh(y)}{2}-\cosh(y)+\dfrac{1}{4}\sinh(y)\sinh(2y)\mp i\dfrac{3 \pi}{4}\sinh(y)\right)
	\end{equation*}
are solutions of equation $\mathcal{L}_1(\zeta)=0$ and have the following properties.
\begin{itemize}
\item The Wronksian satisfies 
\begin{equation*}
			W(\zeta_+,\zeta_-)= \zeta_+\dot{\zeta}_--\zeta_-\dot{\zeta}_+= -i\dfrac{3\pi}{16}.
			\end{equation*}
			and therefore $\zeta_{\pm}$ are linearly independent.
			\item They can be written as
			\begin{equation}
			\label{goodvar}
			\zeta_{\pm}(y)=\dfrac{(y\mp i \pi/2)^3}{(y\pm i\pi/2)^2}\eta_{\pm}(y),
			\end{equation}
			where   $\eta_{\pm}$ are analytic functions in $D^{\mathrm{out},u}_{\kappa}\cap \{\Rp(y)\geq -1\}$ uniformly bounded (with respect to $\e$ and $\kappa$).
	\item The operator $\mathcal{G}_1$ given by \eqref{g1} can be rewritten as 
		\begin{equation*}
		\mathcal{G}_1(h)=i\dfrac{16}{3\pi}\left(-\zeta_+(y)\displaystyle\int_0^y\zeta_-(s)h(s)ds+\zeta_-(y)\displaystyle\int_0^y\zeta_+(s)h(s)ds\right) +\zeta_2(y)\displaystyle\int_{-\infty}^0\zeta_1(s)h(s)ds,
		\end{equation*}
		where $\zeta_1,\zeta_2$  are given in \eqref{zetas}.
		\end{itemize}
	\end{lemma}
	
	The proof of this lemma is a straightforward computation using the relation between $\zeta_\pm$ and $\zeta_1,\zeta_2$.

	Using this lemma, we bound $\GG_1(h)$ for $y\in D^{\mathrm{out},u}_{\kappa}$ satisfying $\Rp(y)\geq -1$.
First, notice that we can use \eqref{zeta12comp} to see that
	$$
	\begin{array}{lcl}
	\left|\displaystyle\int_{-\infty}^0\zeta_1(s)h(s)ds\right|&\leq& M\|h\|_{m,\ag} \left(\displaystyle\int_{-\infty}^{-1}\dfrac{1}{|\cosh^{m+1}(s)|}ds+\displaystyle\int_{-1}^0\dfrac{1}{|\cosh(s)(s^2+\pi^2/4)^{\ag}|}ds\right)\vspace{0.2cm}\\ 
	&\leq& M\|h\|_{m,\ag}.
	\end{array}
	$$
	From the expression of $\zeta_2(y)$ in \eqref{zetas}, we have that $\zeta_2(y)$ has poles of order 2 at $\pm i \pi/2+ i2k\pi$. Since $\ag\geq 5$, 
\begin{equation*}
	\sup_{y\in D^{\mathrm{out},u}_\kappa\cap\{\Rp(y)\geq-1 \}}\left|(y^2+\pi^2/4)^{\ag-2}\zeta_2(y)\displaystyle\int_{-\infty}^0\zeta_1(s)h(s)ds\right|\leq M\|h\|_{m,\ag}. 
	\end{equation*}
	Now, we use that $\ag\geq 5$ and equation \eqref{goodvar} to see that
	$$
	\begin{array}{lcl}
	\left|\zeta_+(y)\displaystyle\int_0^y\zeta_-(s)h(s)ds\right|&\leq&M\dfrac{|y-i\pi/2|^3}{|y+i\pi/2|^2}\displaystyle\int_0^y\dfrac{|s+i\pi/2|^3}{|s-i\pi/2|^2}|h(s)|ds\vspace{0.2cm}\\
	&\leq&M\|h\|_{m,\ag}\dfrac{|y-i\pi/2|^3}{|y+i\pi/2|^2}\displaystyle\int_0^y\dfrac{1}{|s+i\pi/2|^{\ag-3}|s-i\pi/2|^{\ag+2}}ds\vspace{0.2cm}\\
	&\leq&\dfrac{M\|h\|_{m,\ag}}{|y^2+\pi^2/4|^{\ag-2}}.
	\end{array}
	$$
We conclude that
	\begin{equation*}
	\sup_{y\in D^{\mathrm{out},u}_\kappa\cap\{\Rp(y)\geq-1 \}}\left|(y^2+\pi^2/4)^{\ag-2}\zeta_+(y)\displaystyle\int_0^y\zeta_-(s)h(s)ds\right|\leq M\|h\|_{m,\ag}. 
	\end{equation*}
	In a similar way, we can prove that
	\begin{equation*}
	\sup_{y\in D^{\mathrm{out},u}_\kappa\cap\{\Rp(y)\geq-1 \}}\left|(y^2+\pi^2/4)^{\ag-2}\zeta_-(y)\displaystyle\int_0^y\zeta_+(s)h(s)ds\right|\leq M\|h\|_{m,\ag}. 
	\end{equation*}
Therefore
	\begin{equation}
	\label{polo}
	\sup_{y\in D^{\mathrm{out},u}_\kappa\cap\{\Rp(y)\geq-1 \}}\left|(y^2+\pi^2/4)^{\ag-2}\mathcal{G}_1(h)(y)\right|\leq M\|h\|_{m,\ag}. 
	\end{equation}
	Hence, using \eqref{normcosh}, \eqref{oppgn}, \eqref{infinito} and \eqref{polo} , one obtains item (3) of Proposition \ref{prop_operators}.
	
	To prove the estimates on $\partial_{\tau}\mathcal{G}(h)$ and 
$\partial_{\tau}^2\mathcal{G}(h)$ it is sufficient to use \eqref{oppgn} and 
	$$
	\Pi_n[\partial_{\tau}^2\mathcal{G}(h)]=-n^2\Pi_n[\mathcal{G}(h)].$$
	
Finally, for  item (5), notice that
	\[\partial_y\circ\mathcal{G}_n(h)=\dfrac{1}{2}e^{i\frac{\lambda_n}{\e}y}\displaystyle\int_{-\infty}^ye^{-i\frac{\lambda_n}{\e}s}h(s)ds +\dfrac{1}{2}e^{-i\frac{\lambda_n}{\e}y}\displaystyle\int_{-\infty}^ye^{i\frac{\lambda_n}{\e}s}h(s)ds,\ n\geq 2,\]
	and thus, one can easily obtain
	\[
	\|\partial_y\circ\mathcal{G}_n(h)\|_{m,\ag}\leq \dfrac{M\e}{\lambda_n}\|h\|_{m,\ag},\ n\geq 2.
	\]
The decay of $\partial_y\circ\mathcal{G}_n(h)$ for $n\ge 2$ also implies that $\partial_y\circ\mathcal{G}_n(h) = \GG (\partial_y h)$
and thus we also have 
\[
\|\partial_y\circ\mathcal{G}_n(h)\|_{m,\ag}\leq \dfrac{M\e^2}{\lambda_n^2}\|\partial_y h\|_{m,\ag},\ n\geq 2.
\]
	For the first mode, since 
	\begin{align*}
	\partial_y\circ\mathcal{G}_1(h)=&i\dfrac{16}{3\pi}\left(-\zeta_+'(y)\displaystyle\int_0^y\zeta_-(s)h(s)ds+\zeta_-'(y)\displaystyle\int_0^y\zeta_+(s)h(s)ds\right) +\zeta_2'(y)\displaystyle\int_{-\infty}^0\zeta_1(s)h(s)ds\\
	=& -\zeta_1'(y)\displaystyle\int_{0}^y\zeta_2(s)h(s)ds+\zeta_2'(y)\displaystyle\int_{-\infty}^y\zeta_1(s)h(s)ds, 
	\end{align*}	
one has $\|\partial_y\circ \mathcal{G}_1(h)\|_{1,\ag-1}\leq M\|h\|_{m,\ag}$.

\section*{Acknowledgements}
This project has received funding from the European Research Council (ERC) under the European Union's Horizon 2020 research and innovation programme (grant agreement No 757802).
O.M.L.G. has been partially supported by the Brazilian FAPESP grants 2015/22762-5, 2016/23716-0 and 2019/01682-4, and by the Brazilian CNPq grant 438975/2018-9.
T. M. S. has also been partly supported by the Spanish MINECO-FEDER Grant PGC2018-098676-B-100 (AEI/FEDER/UE) and the Catalan grant 2017SGR1049. M. G. and T. M. S. are supported by the Catalan Institution for Research
and Advanced Studies via an ICREA Academia Prize 2019. 
C. Z. has been partially supported by the US NSF grant DMS 19000083 and DMS-2350115.
This material is based upon work
supported by the National Science Foundation under Grant No. DMS-1440140 while the authors
were in residence at the Mathematical Sciences Research Institute in Berkeley, California, during
the Fall 2018 semester.

\bibliography{references_otavio}
\bibliographystyle{abbrv}

\end{document}